\documentclass{amsart}

\usepackage{fancyhdr}
\usepackage{tikz}
\usetikzlibrary{intersections}
\usepackage{etoc}

\newcommand{\etocsettocmargins}[3][0pt]{%
\renewcommand{\etocaftertitlehook}
  {\let\oldpar\par\def\par{%
   \advance\leftskip by -#1\relax
   \advance\leftskip by #2\relax
   \advance\rightskip by #3\relax\oldpar
}}}

\usepackage[english]{babel}
\usepackage[utf8]{inputenc}
\usepackage{amsmath}
\usepackage{amsfonts}
\usepackage{mathrsfs}
\usepackage{mathtools}
\mathtoolsset{showonlyrefs}
\usepackage{enumitem}

\usepackage{esint}
\usepackage[colorlinks=true, allcolors=red, linktocpage=true]{hyperref}
\usepackage{amssymb}
\usepackage{graphicx}
\usepackage{commath}
\usepackage{cases}
\usepackage{bm}
\usepackage{dsfont}
\usepackage{xcolor}
\usepackage{verbatim}

\usepackage{amsthm}

\theoremstyle{plain}

\newtheorem{thm}{Theorem}[subsection]
\newtheorem{prop}[thm]{Proposition}
\newtheorem{lem}[thm]{Lemma}

\newtheorem{cor}[thm]{Corollary}

\newtheorem{thmx}{Theorem}

\numberwithin{figure}{subsection}

\numberwithin{equation}{subsection}

\theoremstyle{definition}

\newtheorem{defn}[thm]{Definition}

\newtheorem{rem}[thm]{Remark}

\makeatletter
\newcommand*{\rom}[1]{\expandafter\@slowromancap\romannumeral #1@}
\makeatother

\newcommand{\vp}{\varphi}
\newcommand{\ve}{\varepsilon}
\newcommand\op{\mathrm{op}}
\newcommand\BM{\mathrm{BM}}
\newcommand\bilat{\xi}
\newcommand\unilat{\zeta}

\newcommand{\N}{\mathbb{N}}

\newcommand{\R}{\mathbb{R}}

\newcommand{\Z}{\mathbb{Z}}

\newcommand{\calA}{\mathcal{A}}
\newcommand{\calB}{\mathcal{B}}

\newcommand{\calD}{\mathcal{D}}

\newcommand{\calF}{\mathcal{F}}
\newcommand{\calG}{\mathcal{G}}
\newcommand{\calH}{\mathcal{H}}
\newcommand{\calI}{\mathcal{I}}

\newcommand{\calS}{\mathcal{S}}

\newcommand{\bigconst}{N}
\newcommand{\Corona}{\Lambda}

\def\dist{\mathop\mathrm{dist}} 
\def\supp{\mathop\mathrm{supp}} 
\def\diam{\mathop\mathrm{diam}} 
\def\md{\mathop\mathrm{md}} 

\newcommand{\id}{\mathrm{Id}}

\newcommand{\COARSE}{{\theta}}

\newcommand {\cG}{\mathcal {G}}
\newcommand {\cB}{\mathcal {B}}
\newcommand {\cT}{\mathcal {T}}
\newcommand {\bR}{\mathbb {R}}
\newcommand{\away}{\alpha}
\newcommand{\toward}{\beta}

\title[Uniformly rectifiable metric spaces]{Uniformly rectifiable metric spaces: Lipschitz images, Bi-Lateral Weak Geometric Lemma and Corona Decompositions}

\author{David Bate}

\address{David Bate\\
	Mathematics Institute \\
	Zeeman Building \\
	University of Warwick \\
	Coventry, CV4 7AL.}

\email{david.bate@warwick.ac.uk}

\author{Matthew Hyde }

\address{Matthew Hyde\\
	Mathematics Institute \\
	Zeeman Building \\
	University of Warwick \\
	Coventry, CV4 7AL.}
\email{matthew.hyde@warwick.ac.uk}

\author{Raanan Schul}

\address{Raanan Schul\\
	Department of Mathematics \\
	Stony Brook University \\
	NY 11794-3651.}

\email{raanan.schul@stonybrook.edu}

\thanks{D. Bate and M. Hyde are supported by the European Union's Horizon 2020 research and innovation programme (Grant agreement No. 948021). R. Schul is supported by the National Science Foundation under Grant No. DMS-2154613} 


\begin{document}

	\maketitle

	\begin{abstract}
		In their 1991 and 1993 foundational monographs, David and Semmes  characterized uniform rectifiability for subsets of Euclidean space in a multitude of geometric and analytic ways. The fundamental geometric conditions can be naturally stated in any metric space and it has long been a question of how these concepts are related in this general setting. In this paper we prove their equivalence. Namely, we show the equivalence of Big Pieces of Lipschitz Images, Bi-lateral Weak Geometric Lemma and Corona Decomposition in any Ahlfors regular metric space. Loosely speaking, this gives a quantitative equivalence between having Lipschitz charts and approximations by nicer spaces. En route, we also study Reifenberg parametrizations. 
		
	\end{abstract}

	\etocsettocmargins{.01\linewidth}{.01\linewidth}
\etocsettocstyle{Contents}{}

\tableofcontents

\section{Introduction}

A set $E \subseteq \R^m$ with $\calH^n(E)<\infty$ is said to be
\textit{$n$-rectifiable} if it can be covered, up to a set of $\calH^n$
measure zero, by countably many Lipschitz images of $\R^n$. Here and
throughout, $\calH^n$ denotes the \textit{$n$-dimensional Hausdorff measure}. Rectifiable sets generalize the notion of $C^1$ manifolds
while still maintaining many of their nice geometric properties (in some
suitable measure theoretic sense).  This notion  goes back to
Besicovitch \cite{besicovitch1928fundamental} and
has many nice properties.
For example, a consequence of
Rademacher's Theorem is that rectifiable sets admit \textit{approximate
tangent planes} at $\calH^n$-a.e. point. Conversely, the existence of
approximate tangent planes at
$\calH^n$-a.e. point implies rectifiability, see \cite{mattila1999geometry}.
An idea that we will return to later is that the notion of tangent may
be weakened to allow the approximating tangent planes at a point to
\textit{rotate} depending on the scale, as was shown by Marstrand
\cite{marstrand1961hausdorff} and Mattila \cite{mattila1975hausdorff}. Since their inception, rectifiable sets have appeared in many areas of
analysis and geometry and unify many different topics \cite{mattila2023rectifiability}.

This article concerns itself with quantitative notions of rectifiability
and approximation by tangent planes. In the late 80s and early 90s there
was particular interest in studying the $L^2$ boundedness of certain
singular integral operators in connection with Vitushkin's conjecture
concerning geometric characterizations of \textit{removable sets} (see
\cite{tolsa2014analytic} for an introduction). In
\cite{calderon1977cauchy}, Calder\'on proved the Cauchy transform is
bounded on $L^2(\Gamma)$ for $\Gamma$ a Lipschitz graph with small
constant. The small constant assumption was later removed by Coifman,
McIntosh and Meyer \cite{coifman1982integrale}. Key to the full solution
of Vitushkin's conjecture was the development of multi-scale techniques
for rectifiability, initiated by Jones. In \cite{jones1990rectifiable}
Jones proves his Travelling Salesman Theorem, which characterizes
rectifiable curves in the plane via $\beta$-numbers, a coefficient
measuring how well a set can be approximated by a line, at each scale
and location. Okikiolu \cite{okikiolu1992characterization} later
extended this characterization to rectifiable curves in $\R^n.$ Building
on the work of Jones, David and Semmes
\cite{david1991singular,david1993analysis} developed a rich theory of
quantitative rectifiability for higher dimensional subsets of Euclidean
space, formulating a framework within which to consider
$L^2$-boundedness of $n$-dimensional singular integral operators. It was
in these works that they introduced the notion of uniform rectifiability.

\begin{defn}\label{d:UR-intro}
A set $E \subseteq \R^m$ is said to be \textit{uniformly
$n$-rectifiable} if it is \textit{Ahlfors $n$-regular} i.e. there exists
$C_0 \geq 1$ such that
\begin{align}
C_0^{-1}r^n \leq \calH^n(E \cap B(x,r)) \leq C_0 r^n
\end{align}
for all $x \in E$ and $0 < r < \diam(E)$, and has \textit{Big Pieces of
Lipschitz images (BPLI) of $\R^n$} i.e. there exist constant $\theta, L
\geq 0$ such that for each $x \in E$ and $0 < r < \diam(E)$ there is an
$L$-Lipschitz map $f \colon \R^n \to \R^m$ such that
\begin{align}
\calH^n(E \cap B(x,r) \cap f(B_n(0,r))) \geq \theta r^n,
\end{align}
where $B_n(0,r)$ denotes the ball in $\R^n$ centred at the origin of
radius $r.$
\end{defn}

In \cite{david1993analysis}, David and Semmes proved many
characterizations of uniformly rectifiable sets. To state a characterization in terms of tangents, we need some
notation. For a set $E \subseteq \R^m$, a point $x \in E$ and a radius
$0 < r < \diam(E),$ define
\begin{align}\label{e:bbeta}
b\beta_E(x,r) = \frac{1}{r}\inf_L \left( \sup_{y \in E \cap B(x,r)}
\dist(y,L) + \sup_{y \in L \cap B(x,r)} \dist(y,E)  \right),
\end{align}
where the infimum is taken over all $n$-dimensional affine subspaces $L$
in $\R^m.$ Thus, $b\beta_E$ measures how well $E$ can be
approximated by an $n$-dimensional affine subspace in a ball, with respect to
Hausdorff distance; how ``flat'' the set is. We will use it together with the following condition which quantifies an \textit{event} occurring at \textit{most} scales $r > 0$ and locations $x \in E.$

\begin{defn}\label{d:Carleson-set-intro}
Let $E \subseteq \R^{m}$ be Ahlfors $n$-regular. A non-negative Borel
measure $\mu$ on $E \times (0,\diam(E))$ is called a \textit{Carleson
measure} if there exists $C > 0$ such that $\mu(B(x,r) \times (0,r))
\leq C r^n$ for all $x \in E$ and $0 < r < \diam(E).$ The smallest such
$C$ is called the \textit{Carleson norm} of $\mu$. A set $\mathscr{A} \subseteq E
\times (0,\diam(E)$ is a \textit{Carleson set} if
$\mathds{1}_\mathscr{A}d\calH^n(x)\tfrac{dr}{r}$ is a Carleson measure.
\end{defn}

Note, if $\mathscr{A} = E \times (0,\diam(E)),$ then for $\mu =
\mathds{1}_\mathscr{A}d\calH^n\tfrac{dr}{r}$ we have $\mu(B(x,r) \times (0,r)) =
\infty$ for all $x$ and $r.$ Thus, roughly speaking, $\mathscr{A}$ is a Carleson
set if set of scales $r$ and locations $x$ such that $(x,r) \in \mathscr{A}$ is
small.

\begin{defn}\label{d:BWGL-intro}
Let $E \subseteq \R^m$ be Ahlfors $n$-regular. We say $E$ satisfies the
\textit{Bi-lateral Weak Geometric Lemma} (BWGL) if for each $\ve > 0$
the set $\{(x,r) \in E \times (0,\diam(E)) : b\beta_E(x,r) > \ve\}$ is a
Carleson set.
\end{defn}

The BWGL condition is a quantitative version of having a \textit{rotating tangent} i.e. a tangent that can depend on the scale. Marstrand
\cite{marstrand1961hausdorff} and Mattila \cite{mattila1975hausdorff}
showed that the existence of a rotating tangent plane at
$\calH^n$-a.e. point (together with a mild dimensional assumption)
ensures rectifiability. The following result of David and Semmes characterizes uniform rectifiability in terms of the BWGL.


\begin{thm}\label{t:BWGL-UR-E}
Let $E \subseteq \R^m$ be Ahlfors $n$-regular. Then $E$ is uniformly
$n$-rectifiable if and only if it satisfies the \emph{BWGL}.
\end{thm}

With the recent interest in studying analysis on metric spaces,
questions regarding the rectifiability of such spaces naturally arise.
The definition of rectifiability in metric spaces differs slightly from
the Euclidean one above. A metric space $X$ is \textit{$n$-rectifiable} if it
can be covered, up to a set of $\calH^n$ measure zero, by countably many
Lipschitz images of \textit{subsets} of $\R^n.$ One of the first to
consider rectifiability in general metric spaces was Kirchheim
\cite{kirchheim1994rectifiable}, who proved that Lipschitz functions
$f\colon A\subset \R^n \to X$ are differentiable (in a suitable metric
sense) almost everywhere, and consequently derived a local structure
theory of rectifiable metric spaces. Using Kirchheim's metric Rademacher
theorem, it is possible to define a notion of an approximate tangent
spaces in terms of local approximation, in terms of the
Gromov--Hausdorff distance, to an $n$-dimensional Banach space. The
Gromov--Hausdorff distance is an \emph{intrinsic} property of the metric
space and hence these tangent spaces do not require additional structure
of the ambient metric space. Recently the first named author proved the
converse direction, that is, the existence of Gromov--Hausdorff
approximate tangent spaces implies rectifiability
\cite{bate2022characterising}. Since the Gromov--Hausdorff distance is
an isometric invariant of metric spaces, this result generalizes those
of Marstrand and Mattila.

Correct notions of quantitative rectifiability in metric spaces have
been sought for some time. With the notable exception of G. C. David's \cite{david2016bi}, for arbitrary metric spaces the story so far
has been entirely focussed on 1-dimensional spaces. The papers
\cite{hahlomaa2005menger,hahlomaa2008menger,schul2007ahlfors,david2020sharp} prove results analogous to the Travelling Salesman Theorem of Jones for general metric spaces. In \cite{hahlomaa2008menger, schul2007ahlfors}
the authors must assume Ahlfors regularity of the spaces while in
\cite{hahlomaa2005menger, david2020sharp} no such assumption is needed.
Instead of measuring flatness via local approximations by lines,
flatness is measured using Menger curvature. See also \cite{lerman2009high,lerman2011high} for higher dimensional sets in Hilbert space. 

For specific metric space or by assuming to ambient structure, more
progress has been made. In the Heisenberg group and for more general
Carnot groups, Travelling Salesman type theorems have been proven in
\cite{chousionis2019traveling,ferrari2007geometric,juillet2010counterexample,li2016traveling,li2016upper} with a complete solution given by S. Li in \cite{li2019stratified}. In these spaces a version of uniform rectifiability defined in terms of
big pieces of \textit{intrinsic Lipschitz graphs} seems to yield
positive results, see for example
\cite{chousionis2019intrinsic,chousionis2022strong,chousionis2023strong,fassler2020semmes,rigot2019quantitative}.
The correct notion of uniform rectifiability in parabolic space was
introduced in \cite{hofmann2003existence,hofmann2004caloric}. See
\cite{bortz2023corona} and the reference therein for results in this
direction. A Travelling Salesman Theorem for subsets of Hilbert space
was proven in \cite{schul2007subsets}, an analogue for higher
dimensional subsets can be found in \cite{hyde2022d}. A combination of
the papers \cite{badger2020subsets,badger2022subsets} provides a
Travelling Salesman Theorem for subsets of Banach spaces. See also
\cite{edelen2018effective}, who prove higher dimensional analogues for
measures in Hilbert and Banach spaces.

\bigskip

\subsection{New results}

In this paper we prove a version of Theorem \ref{t:BWGL-UR-E} for general metric spaces (Theorem \ref{t:UR-BWGL}). As in the case of going from Euclidean rectifiability to rectifiability in metric spaces, we must slightly modify the definition of UR and we must change the way in which we measure local flatness. To go from Euclidean UR to UR in metric spaces we again look at Lipschitz images of \textit{subsets} of $\R^n$ (see Definition \ref{d:UR-intro-metric}). A significant difference is that our definition of BWGL uses the Gromov-Hausdorff distance to measure closeness to an $n$-dimensional Banach space. That is, one replaces $b\beta$ in Definition \ref{d:BWGL-intro} with $\bilat_X$, a suitable Gromov-Hausdorff version of $b\beta$. See Definition \ref{d:bilat} for the definition of $\bilat_X$ and Definition \ref{d:BWGL} for the explicit definition of BWGL in metric spaces. This intrinsic notion of flatness allows us to rid ourselves of requiring ambient structure in the metric space. 

\begin{thmx}\label{t:UR-BWGL}
	Let $X$ be an Ahlfors $n$-regular metric space. Then $X$ is \emph{UR} if and only if $X$ satisfies the \emph{BWGL}. 
\end{thmx}

One could be tempted to prove Theorem \ref{t:UR-BWGL} by quantifying the techniques from \cite{bate2022characterising}. Indeed, this was our first attempt. In short, this fails because the Carleson condition in BWGL is not
strong enough to construct the H\"older mapping in
\cite{bate2022characterising} (we discuss this a little more in Section
\ref{s:Reif-intro} below). An alternate approach to prove Theorem \ref{t:UR-BWGL} would be to translate the classical proof of Theorem \ref{t:BWGL-UR-E} to metric spaces. However, the known proofs rely significantly on either singular integral techniques, Lipschitz graphs or the ambient Euclidean structure.

In pursuit of Theorem \ref{t:UR-BWGL}, we pass via two more conditions inspired by Euclidean uniform rectifiability. The first of these is related to the \textit{Corona Decomposition}, introduced by Semmes and David-Semmes for subsets of Euclidean space in \cite{semmes1990analysis,david1991singular}. We introduce the \textit{Corona Decomposition by normed spaces} which is a ``Gromov-Hausdorff" version of the Euclidean corona decomposition that can be defined for any metric space.
As in the Euclidean case, this serves as a key bridge between UR and the BWGL.  We also rely on a \textit{Corona Decomposition by Ahlfors regular Reifenberg flat spaces} and notions of quantitative differentiation of Lipschitz functions on metric spaces.  Theorem \ref{t:equiv} connects all these notions and others (which we define momentarily).

\begin{thmx}\label{t:equiv}
	Let $X$ be an Ahlfors $n$-regular metric space. Then, the following conditions are equivalent: 
	\begin{enumerate}
		\item $X$ is UR. 
		\item $X$ satisfies a Carlson condition for the $\gamma$-coefficients.
		\item $X$ satisfies the \emph{BWGL}.
		\item $X$ admits a corona decomposition by Ahlfors regular Reifenberg flat metrics on $\R^n.$ 
		\item $X$ admits a corona decomposition by normed spaces.  
	\end{enumerate}
\end{thmx}

Crucial to the proof of Theorem \ref{t:equiv} is the analysis of Reifenberg flat metric spaces, which we return to below, see Theorem \ref{t:Reif-intro}.

\subsection{Uniformly $n$-rectifiable (UR)}

A metric variant of Definition \ref{d:UR-intro} is the following.
\begin{defn}\label{d:UR-intro-metric}
	A metric space $X$  is said to be \textit{uniformly $n$-rectifiable} (UR) if it is \textit{Ahlfors $n$-regular} (see Definition \ref{d:ADR}) and has \textit{Big Pieces of Lipschitz Images (BPLI) of $\R^n$} i.e. there exist constant $\theta, L \geq 0$ such that for each $x \in X$ and $0 < r < \diam(E)$ there is a set $F\subset B_{n}(0,r)$ and  an $L$-Lipschitz map $f \colon F \to X$ such that 
	\begin{align}
		\calH^n(X \cap B(x,r) \cap f(F)) \geq \theta r^n,
	\end{align}
	where $B_n(0,r)$ denotes the ball in $\R^n$ centered at the origin of radius $r.$ 
\end{defn}

\begin{defn}\label{d:BPBI} 
	A  metric space $X$  is said to have  \textit{Big Pieces of Bi-Lipschitz Images (BPBI) of $\R^n$}
	if it is \textit{Ahlfors $n$-regular} there exist constant $\theta, L \geq 0$ such that for each $x \in X$ and $0 < r < \diam(E)$ there is a set $F\subset B_{n}(0,r)$ and  an $L$-bi-Lipschitz map $f \colon F \to X$ such that 
	\begin{align}
		\calH^n(X \cap B(x,r) \cap f(F)) \geq \theta r^n,
	\end{align}
	where $B_n(0,r)$ denotes the ball in $\R^n$ centered at the origin of radius $r.$ 
\end{defn}

	The above list of equivalences can be further extended using \cite{schul2009bi}. Specifically, in \cite[Corollary 1.2]{schul2009bi} it is shown that UR is equivalent to having BPBI. It follows from the results in Section \ref{s:stability} that BPBI is equivalent to having $\text{BP}^2$(BI), that is, sets which have big of sets which have BPBI (see \cite{krandel2022iterating} for a precise definition of $\text{BP}^j$ and related results). Indeed, if $X$ has $\text{BP}^2$(BI) then it has a Carleson condition for the $\gamma$-coefficients (Proposition \ref{l:stability}) and is therefore UR by Theorem \ref{t:equiv}. Thus, we can add the following to the list in Theorem \ref{t:equiv}.
	\begin{enumerate}[font=\itshape]
		\setcounter{enumi}{5}
		\item $X$ has $\text{BP}^j$(BI) for some $j \geq 1.$ 
		\item $X$ has $\text{BP}^j$(BI) for all $j \geq 1.$ 
	\end{enumerate}

\subsection{The $\gamma$-coefficients and the BWGL}
Let $X$ be an Ahlfors $n$-regular metric space. For $\ve > 0$ we say $X$ satisfies the \textit{Bi-lateral Weak Geometric Lemma with parameter $\ve$} (BWGL$(\ve)$) if the set $\{(x,r) \in E \times (0,\diam(E)) : \bilat_X(x,r) > \ve\}$ is a Carleson set (see Definition \ref{d:Carleson-set}, the metric analogue of Definition \ref{d:Carleson-set-intro}). The coefficient $\bilat_X$ is a suitable Gromov-Hausdorff version of $b\beta$ (see Definition \ref{d:bilat}, the metric analogue of \eqref{e:bbeta}). In particular, $\bilat_X(x,r)$ measures how close (in Gromov-Hausdorff distance) the ball $B(x,r) \subseteq X$ is to a ball of the same radius in a $n$-dimensional Banach space. We emphasize that the Banach space, in particular the norm, can change depending on both $r$ and $x$. We say $X$ satisfies the \textit{Bi-lateral Weak Geometric Lemma} (BWGL) if for each $\ve > 0$ it satisfies the BWGL$(\ve)$. In fact, in both Theorem \ref{t:UR-BWGL} and Theorem \ref{t:equiv}, the proof shows that if $\ve$ is sufficiently small then the condition BWGL may be replaced by the condition BWGL$(\ve)$. Explicitly, one may add to Theorem \ref{t:equiv} the following. 
\begin{enumerate}[font=\itshape]
	\setcounter{enumi}{7}
	\item $X$ satisfies the BWGL$(\ve)$ for sufficiently small $\ve > 0.$ 
\end{enumerate}

 In order to quantify Rademacher's Theorem, it is natural to approximate a Lipschitz map $f \colon \R^n \to \R^n$ on a ball $B(x,r)$ by an affine function $A \colon \R^n \to \R^n,$ see \cite{dorronsoro1985characterization}. David and Semmes generalized this quantitative differentiation to $f \colon E \to \R^n$ where $E$ is an Ahlfors $n$-regular set in $\R^m$. This generalization was called the WALA (Weak Approximation of Lipschitz functions by Affine functions) condition, see \cite[Definition I.2.47]{david1993analysis}, and is stated in terms of Carleson sets.

For a Lipschitz $f \colon X \to (\R^n,\|\cdot\|),$ we introduce the coefficient $\gamma_{X,f,\|\cdot\|}^K(x,r)$ which in addition to measuring how bi-laterally flat $B_X(x,r)$ is (in the sense of $\bilat_X(x,r)$), also measures how well $f$ can be approximated by a $K$-Lipschitz affine function on the best approximating tangent space to $B_X(x,r)$ (Definition \ref{d:gamma}). A Carleson condition for the $\gamma$-coefficients is a type of quantitative differentiation for arbitrary metric spaces, see Definition \ref{d:carleson-gamma}. We will be more precise about what we mean by this in Section \ref{s:affine-approx}.

At first glance the Carleson condition for the $\gamma$-coefficients is stronger than the BWGL. Indeed, for every ball $\gamma_{X,f}^K \geq \bilat_X$. Theorem \ref{t:equiv} says that the Carleson conditions for these respective quantities are equivalent. 

An $L^p$ version of the $\gamma$ coefficients has been defined recently in \cite{azzam2023quantitative} where they prove quantitative differentiability estimates for Sobolev functions on UR subsets of $\R^n.$ 

\subsection{Corona Decompositions}

The corona decomposition for subsets of Euclidean space was introduced in \cite{semmes1990analysis,david1991singular}. Very roughly speaking, for a set $E \subseteq \R^n$ this means that there is a subset of bad pairs $(x,r) \in E \cap B(z,R) \times (0,R)$ satisfying a Carleson condition such that if we remove the corresponding balls from $E \cap B(z,R)$, then the remainder can be well-approximated by a Lipschitz graph. Here, well-approximated at a point $x$ means approximation up to the smallest $r$ such that $B(x,r)$ is disjoint from a fixed dilation of balls in the Carleson set. We make use of two analogous definitions using the Gromov-Hausdorff distance to measure approximations. In one case the approximating spaces are Reifenberg flat metrics on $\R^n$ (see below). In the other case, they are $n$-dimensional Banach spaces. See Section \ref{s:Reif}. While these conditions look quite different, it turns out they are equivalent by Theorem \ref{t:equiv}.

We mention that corona decompositions have been introduced in various specific metric spaces. In the context of Heisenberg and Carnot groups such objects have been considered in \cite{di2022extensions} and \cite{naor2022foliated}, respectively. In parabolic space, see \cite{bortz2022coronizations, bortz2023corona}. The graphs used in each of these contexts differ from those used in the Euclidean setting, but both rely on some ambient structure to define the approximations. As for our other conditions, our use of the Gromov-Hausdorff distance avoids us relying on such structure.

 \subsection{Reifenberg flat metric spaces}\label{s:Reif-intro}

\begin{defn}\label{d:RF}
	A metric space $X$ is said to be $(\delta,n)$-Reifenberg flat if $\bilat_X(x,r) \leq \delta$ for all $x \in X$ and $0 < r < \infty.$ At various points we may simply write $\text{RF}$. 
\end{defn}

Reifenberg \cite{reifenberg1960solution} first studied these sets in Euclidean space (defined as in Definition \ref{d:RF} except with $\bilat$ replaced by $b\beta$), where he shows that such sets are locally bi-H\"older equivalent to a ball in $\R^n.$ A vast generalization of this result was given more recently by David and Toro in \cite{david2012reifenberg}, to allow for sets with holes. Reifenberg flat metric spaces have been studied in \cite{cheeger1997structure} and \cite{david1999reifenberg}. In \cite{cheeger1997structure} it is shown that Reifenberg flat metric spaces are locally bi-H\"older equivalent to a ball in $\R^n$, generalising the result of Reifenberg to metric spaces. In \cite{david1999reifenberg}, the authors construct interesting embeddings of such spaces into lower dimensional Euclidean spaces. In all instances above, the definition of Reifenberg flat is slightly different to that given in Definition \ref{d:RF}. Namely, they assume the approximating tangent spaces are Euclidean. 

In Theorem \ref{t:Reif} we prove a Reifenberg type theorem for metric spaces that lies somewhere between that of Cheeger-Colding \cite{cheeger1997structure} and David-Toro \cite{david2012reifenberg}. We do not assume bi-lateral approximations at all scales as in \cite{cheeger1997structure}, but we do not have as general assumptions as those contain in \cite{david2012reifenberg}. Our result differs fundamentally in the fact that we allow non-Euclidean tangents, namely $n$-dimensional Banach spaces, which may vary at different scales and locations. Roughly speaking, we show that if $X$ is a metric space and for each $x \in X$ there exists a scale $0 < r_x < \infty$ (with the compatibility between the scales of nearby points) such that $\xi_X(x,r) \leq \delta$ for all $r_x < r < \infty$ then we can find a metric $\rho$ on $\R^n$, which is bi-H\"older equivalent to the Euclidean metric, such that $(\R^n,\rho)$ approximates $X$ up to the scales determined by the $r_x.$ The result in its full generality is rather complicated to state. Theorem \ref{t:Reif-intro} below is a simplified version of Theorem \ref{t:Reif} where the Reifenberg condition is assumed to hold up to some uniform scale across $X$. Theorem \ref{t:Reif} is well-suited to stopping-time constructions. 

\begin{thmx}\label{t:Reif-intro}
	For each $\ve > 0$ there exists $\delta > 0$ such that the following holds. Let $X$ be a metric space that is $(\delta,n)$-Reifenberg flat up to scale $s_0 > 0$, that is, $\bilat_X(x,r) \leq \delta$ for all $x \in X$ and $s_0 < r < \infty.$ Then, for each $x \in X$ and $r > 0$ there exists a metric $\rho$ which is bi-H\"older equivalent to $|\cdot|$, the Euclidean norm on $\R^n$, and a map $\bar{g} \colon B_X(x,r) \to B_\rho(0,2r)$ such that 
	\begin{align}\label{e:coarse-bi-lip}
		 | \rho(\bar{g}(y),\bar{g}(z)) - d_X(y,z)  | \leq \ve (d_X(y,z) + s_0)
	\end{align}
	for all $y,z \in B_X(x,r)$ and 
	\begin{align}\label{e:last}
		\sup_{u \in B_\rho(0,(1-\ve)r)} \inf_{\bar{g}(B_X(x,r))}\rho(u,\cdot) \leq \ve s_0. 
	\end{align}
	If additionally $X$ is Ahlfors $n$-regular, then $(\R^n,\rho)$ is Ahlfors $n$-regular with regularity constant depending on $n$ and the regularity constant for $X$.
\end{thmx} 

Theorem \ref{t:Reif-intro} is related to \cite[Theorem
3.7]{bate2022characterising} which constructs a H\"older map
$\iota\colon B_{\|\cdot\|}(0,r) \to \tilde X$ into a metric space
containing $X$, whose image is mostly contained in $X$. This is achieved
under the assumption that $B(x,r)$ is well approximated by tangent
planes for many $x\in X$ and all scales $r < s_0$.
This final hypothesis is crucial for the construction, which iteratively
constructs the H\"older map, scale by scale.
As alluded to above, this hypothesis does not follow from the BWGL;
Indeed, it is possible that all locations may have many consecutive
scales that are not approximated by tangent planes where the iteration
cannot continue.
To the best of our knowledge, this situation is unworkable.

Instead, we resort to the weaker conclusion from Theorem
\ref{t:Reif-intro} to gain the advantage that we only require
approximations at scales $r > s_0$. This hypothesis can be
obtained from the BWGL, after a decomposition into stopping time regions.
More precisely, we obtain a corona decomposition by Reifenberg flat
metric spaces (see Section \ref{s:final}). In order to prove UR, we show that RF metric spaces have corona
decompositions by normed spaces (see Section \ref{s:Reif}).
By combining these two results, we see that the BWGL implies a corona
decomposition by sets with corona decompositions by normed spaces.
With some more work, UR follows from this conclusion (see Sections
\ref{s:corona-corona} and \ref{s:corona-UR}).

We give a more detailed overview of our proof strategy in Section
\ref{s:outline} below.

\subsection{Other remarks}
The Ribe program (originally formulated by Bourgain
in \cite{bourgain1986metrical}) asks for a dictionary between results in
metric geometry and Banach space geometry, see
\cite{naor2012introduction}
and
\cite{naor2018metric}.
A standard question related to this program (see e.g. Naor's
\cite[Problem 2]{naor2018metric}) is the question of which metric spaces bi-Lipschitz embed into which
others, with specific focus given to Euclidean targets
\cite{heinonen2003geometric,ostrovskii2013metric,semmes1999bilipschitz}.
In Theorem \ref{t:UR-BWGL} and Theorem \ref{t:equiv}   we give a {\it
characterization} of metric spaces which have, up to controlled
distortion, uniformly big pieces which have linear structure, that is,
uniformly big pieces which bi-Lipschitzly map into a Euclidean space of the same dimension.
This characterization is in terms of intrinsic metric quantities.

One naturally wants to ask whether this can be used to map a whole
(Uniformly Rectifiable) metric space $X$ bi-Lipschitzly into Euclidean
space.
Semmes, in his 1999 paper  \cite[Proposition 2.10]{semmes1999bilipschitz},  showed that one may deform the distance on $X$ by an
$A_1$ weight and map the resulting metric space bi-Lipschitzly into
Euclidean space. In fact, his result is proven in the more general
category of \textit{Doubling} metric spaces which have \textit{Big Pieces of Euclidean
Space}. We will not define these concepts here, but rather refer to the
beautiful exposition of \cite{semmes1999bilipschitz}  and note that
 \cite[Section 4]{semmes1999bilipschitz} is  devoted to UR metric spaces. Thus, for an Ahlfors regular space $X$ satisfying BWGL, one has the conclusion of \cite[Proposition 2.10]{semmes1999bilipschitz}.

Another natural question is to what extent can the quantitative nature be salvaged if one does not assume Ahlfors regularity and the dimension
is larger than 1 (we have already discussed the one dimensional case  in
the first part of the introduction).
In the Euclidean case, there are plenty of results in this direction (see e.g. \cite{david2012reifenberg, azzam2015characterization, tolsa2015characterization, azzam2018analyst, edelen2018effective, villa2019higher,azzam2019quantitative,hyde2020analyst}) and a natural question is if these can
be transported to the metric setting.
We note that Theorem \ref{t:Reif-intro} and Theorem \ref{t:Reif} as well
as the stopping time style of the proof of Theorem \ref{t:equiv} suggest that a discrete version of  Theorem \ref{t:equiv}  holds and
gives some hope for further generalizations. See Section \ref{s:example} for an example in this direction. 

The expert will note that we do not have, in Theorem \ref{t:equiv}, a condition on the square summability of some suitable variant of the Jones $\beta$-number i.e. that $\beta^2(x,r) d\calH^n(x)\tfrac{dr}{r}$ is a Carleson measure, as in \cite{david1993analysis}.  We leave this as an open question. Generalising the main result in \cite{azzam2014quantitative} would be a first step towards such a condition. 

\bigskip

\subsection{Structure of the paper and outline of the proof of Theorem \ref{t:equiv}}\label{s:outline}

\textit{We took care in formatting the table of contents so that it may serve the reader as a guide through the proof}. Here, we give a more detailed outline. 

In Section \ref{s:prelim} we discuss general preliminaries needed for the rest of the paper. Section \ref{s:UR} focusses of preliminaries related to the theory of uniform rectifiability. It is there that we introduce the $\bilat$ and $\gamma$ coefficients and study their fundamental properties. The implication (1) implies (2) is proven in Section \ref{s:affine-approx} and is carried out in two stages. We first show that the Carleson condition for the $\gamma$-coefficients is \textit{stable under big pieces} (Proposition \ref{l:stability}) which, roughly speaking, means if $X$ is a metric space, $\calF$ is a family of metric spaces which satisfies the Carleson condition for the $\gamma$-coefficients, and $X$ has big pieces of $\calF$ (see Definition \ref{d:UR-metric}) then $X$ also satisfies the Carleson condition for the $\gamma$-coefficients. The second stage consist of showing that metric spaces which are bi-Lipschitz equivalent to $\R^n$ satisfy a Carleson condition for the $\gamma$-coefficients (see Definition \ref{d:carleson-gamma} and Proposition \ref{l:approx-bi-lip}). If $X$ is a UR metric space, we can show that $X$ has big pieces of bi-Lipschitz images of $\R^n$ (after embedding in some suitable Banach space $\calB$). Thus, by combining the above two stages we complete the proof of this implication.

Section \ref{s:Reif-approx} is independent of the work in Sections \ref{s:UR} and \ref{s:affine-approx}. In Section \ref{s:Reif-approx} we prove (the more general statement of) Theorem \ref{t:Reif-intro} inspired by the scheme described in \cite{cheeger1997structure}. We need to construct the metric $\rho$ and the map $\bar{g}.$ We obtain $\rho$ by first constructing a sequence $M_\ell$ of smooth connected manifolds, $f_\ell \colon M_\ell \to M_{\ell+1}$ of diffeomorphisms and $\phi_\ell$ of semi-metrics on $M_{\ell}$. Pulling the $\phi_\ell$ back to $\R^n$ via the $f_\ell$ and taking a limit gives a semi-metric $\phi_\infty$ on $\R^n$ (Lemma \ref{l:phi-cauchy}). We then obtain the metric $\rho$ from $\phi_\infty$ (Definition \ref{d:rho'}). The $M_\ell$ are constructed in two stages. In the first stage we construct a sequence $W_\ell$ of (possibly) non-connected manifolds (Section \ref{s:disconnected}) and a sequence of smooth mappings $h_\ell \colon W_{\ell} \to W_{\ell+1}$ (Section \ref{s:smooth}). In the second stage we glue the $W_\ell$ using the maps $h_\ell$ to construct the connected manifolds $M_\ell$ (Section \ref{s:M}). The maps $f_\ell$ are also constructed in Section \ref{s:M} and the semi-metrics in Section \ref{s:metric-map}.  In Section \ref{s:quasi} and Section \ref{s:bar-g} respectively we construct the metric $\rho$ and the map $\bar{g}$ as in Theorem \ref{t:Reif-intro}  and prove the desired properties. In Section \ref{s:AR}, we consider the Ahlfors regular case. In Section \ref{s:last}, we show how Theorem \ref{t:Reif-intro} follows from Theorem \ref{t:Reif}.

We then return to proving the rest of Theorem \ref{t:equiv}. The implication (2) implies (3) is immediate from the fact that $\gamma_{X,f}^K \geq \bilat_X$ (see the paragraph before Theorem \ref{t:equiv} or Lemma \ref{l:WALAM-BWGL}). 

In Section \ref{s:Reif} we show that Ahlfors regular Reifenberg flat metric spaces admit a corona decomposition by normed spaces (see Theorem \ref{t:corona-reif} and Definition \ref{d:CDV}). This is carried out in two stages. In stage one we show, using Theorem \ref{t:Reif} and a result of Guy C. David \cite{david2016bi} on bi-Lipschitz decomposition of Lipschitz maps on metric spaces, that Ahlfors regular Reifenberg flat metric spaces are UR (Proposition \ref{p:reif-UR}). A short description of this step is given at the beginning of Section \ref{s:RF-UR}. The implication (1) implies (2) tells us now that our space satisfies a Carleson condition for the $\gamma$-coefficients. In stage two we leverage this fact to construct a corona decomposition for Reifenberg flat sets. This strategy is explained after the statement of Proposition \ref{p:UR+BP-corona}.

The implication (8) implies (4), and hence (3) implies (4), is proven in Section \ref{s:final} (see Proposition \ref{t:final}). The argument is an application of Theorem \ref{t:Reif} inside a stopping-time argument.

The implication (4) implies (5) is based on Corollary \ref{c:BWGL implies CD^2} and the main result in Section \ref{s:corona-corona}. We showed in Section \ref{s:Reif} that Ahlfors regular Reifenberg flat metric spaces admit a corona decomposition by normed spaces. Thus, any space satisfying (4) has a corona decomposition by spaces which admit corona decompositions by normed spaces ($\text{CD}^2(\text{normed spaces})$), see Corollary \ref{c:BWGL implies CD^2}. Then, based on original arguments by David and Semmes, we show that any space satisfying $\text{CD}^2(\text{normed spaces})$ also has a corona decomposition by normed spaces (Proposition \ref{prop:corona-of-corona-1}).

The implication (5) implies (1) is carried out in Section \ref{s:corona-UR}. In fact, we show that (5) implies BPBI which implies immediately UR. The proof follows closely the arguments in \cite[Section 16-17]{david1991singular} which give the analogous implication in Euclidean space. We have included the details for completeness. 

\subsection{Acknowledgments} 

We thank Matthew Badger for comments on a preliminary version of the introduction to the paper.

\newpage

\section{Preliminaries}\label{s:prelim}
\etocsettocstyle{Contents of this section}{}
\etocsettocmargins[1.5em]{.175\linewidth}{.175\linewidth}

\localtableofcontents

\bigskip

\subsection{Notation} 
Throughout the paper $C$ will denote some absolute constant (depending perhaps on $n$) which may change from line to line. If there exists $C \geq 1$ such that $a \leq Cb,$ then we shall write $a \lesssim b.$ If the constant $C$ depends on a parameter $t$, we shall write $a \lesssim_t b.$ We shall write $a \sim b$ if $a \lesssim b$ and $b\lesssim a,$ similarly we define $a \sim_t b.$ We shall not keep track of how the constants depend on $n$. 

Suppose $(X,d)$ is a metric space. Whenever we want to emphasize the metric on $X$ we will write $d_X.$ We denote by $B_X(x,r)$ or $B_d(x,r)$ the closed ball in $X$ centred at $x$ of radius $r$. If the context is clear we may simply write $B(x,r).$ If $B$ is a ball in $X$ we may write $x_B$ to denote its centre and $r_B$ to denote its radius. For sets $E,F \subseteq X$ let 
\begin{align*}
	\text{dist}(E,F) &= \inf\{d(x,y) : x \in E, \ y \in F \}; \\
	\text{diam}(E) &= \sup\{d(x,y) : x,y \in F \}
\end{align*}
The $n$-dimensional Hausdorff measure of $E$, is defined as 
\begin{align}
	\calH^n(E)  = \lim_{\delta \to 0}  \inf \left\{ \sum_i \diam(U_i)^n : A \subseteq \bigcup_i U_i \ \text{and} \ \diam U_i \leq \delta \right\}.
\end{align}
At various point throughout the paper we will write ${\dist}_d$, $\diam_d$ and $\calH^n_d$ when we wish to emphasize the metric with which these quantities are defined. If the metric is induced by a norm $\|\cdot\|,$ we may also write ${\dist}_{\|\cdot\|},$ $\diam_{\|\cdot\|}$ and $\calH^n_{\|\cdot\|}.$ 

Suppose $f \colon X \to Y$ is a map between metric spaces, $L > 0$ and $\alpha \in (0,1)$. We say $f$ is $L$-\textit{Lipschitz} if 
\begin{align}
	d_Y(f(x),f(y)) \leq L d_X(x,y)
\end{align}
for all $x,y \in X$ and $L$-\textit{bi-Lipschitz} if
\begin{align}
	L^{-1}d_X(x,y) \leq d_Y(f(x),f(y)) \leq L d_X(x,y).
\end{align}
for all $x,y \in X.$ If there exists $C \geq 1$ such that 
\[ d_Y(f(x),f(y)) \leq Cd_X(x,y)^\alpha\]
for all $x,y \in X$ then we say $f$ is \textit{H\"older with exponent} $\alpha.$ If there is $C \geq 1$ such that 
\begin{align}\label{e:bi-holder}
	C^{-1}d_X(x,y)^\frac{1}{\alpha} \leq d_Y(f(x),f(y)) \leq Cd_X(x,y)^\alpha.
\end{align}
for all $x,y \in X$ then we say $f$ is \textit{bi-H\"older with exponent $\alpha.$} We say two metric spaces $X$ and $Y$ are \textit{bi-H\"older equivalent with exponent $\alpha$} if there exists a surjective map $f \colon X \to Y$ which is bi-H\"older with exponent $\alpha.$

\bigskip

\subsection{Normed space and Banach-Mazur distance}
Let $Q(n)$ denote the set of all $n$-dimensional normed spaces. Define the \textit{Banach-Mazur distance} between $X,Y \in Q(n)$ by
\[ d_\BM(X,Y) = \inf \left\{  \|T\|_\op \, \|T^{-1}\|_\op  \right\},\] 
where the infimum is taken over all isomorphisms $T$ from $X$ onto $Y$. 

\begin{rem}\label{r:BM}
	It is known by a result of John \cite{john1948extremum} that $d_\BM(X,Y) \leq n$ for all $X,Y \in Q(n).$ Furthermore, if $I(n)$ denotes the space of isometry classes of $Q(n),$ then $(I(n), \log  d_\BM)$ defines a compact metric space.
\end{rem}

For a pair of normed spaces $X$ and $Y$, let $\calB(X,Y)$ denote the space of bounded linear operators from $X$ to $Y$. We equip $\calB(X,Y)$ with the operator norm defined by 
\[ \|T\|_{\calB(X,Y)} = \sup_{\|x\|_X=1} \|T(x)\|_Y.\] 
In the case that $X = Y$ we will write $\calB(X)$ and $\|\cdot\|_{\calB(X)}$ for short. If the context is clear, we will simply write $\|T\|_\op.$ Let $U \subseteq X$ be open. A mapping $f : X \to Y$ is in $C^1(U)$ if its Fr\'echet derivative $Df$ exists and is continuous at each $x \in U.$ We say $f \in C^2(U)$ if $f,Df \in C^1(U)$. For such a map, let $D^2f$ denote Fr\'echet derivate of $Df.$ If $x \in U$, then $Df(x) \in \calB(X,Y)$ and $D^2f(x) \in \calB(X,\calB(X,Y)).$ For $r > 0$ and $m \in \{1,2\},$ we define 
\begin{align}\label{d:norm}
	 \|f\|_{C^m(U),r} = \sup_{x \in U } \left\{ r^{-1}\|f(x)\|_F + \sum_{k=1}^m r^{k-1}\|D^kf(x)\|_\op \right\},
\end{align}
where $D^1f(x) = Df(x).$ 

\bigskip

\subsection{Gromov-Hausdorff approximations} In this section we define the notion of Gromov-Hausdorff approximations and prove some useful properties concerning them.

\begin{defn}\label{d:GHA}
	Let $(X,d_X)$ and $(Y,d_Y)$ be compact metric spaces and $\delta > 0$. We say a map $\phi \colon X \to Y$ is a $\delta$-\textit{Gromov-Hausdorff  approximation} ($\delta$-GHA) if it is a $\delta$-\textit{isometry}, meaning
	\begin{align}
		\sup_{x,y \in X} \abs{ d_X(x,y) - d_Y(\phi(x),\phi(y)) } \leq \delta,
	\end{align}
	and if $\phi(X)$ is $\delta$-\textit{dense in} $Y$, meaning 
	\begin{align}
		\sup_{u \in Y} {\dist}_Y(u,\phi(X)) \leq \delta. 
	\end{align}
\end{defn}

\begin{rem}
	We do not impose that $\phi$ be Lipschitz. Note however that if $\mathcal{N}$ is a well-separated collection of points in $X$ then $\phi|_\mathcal{N}$ is Lipschitz. Under certain assumptions on $X$ and $Y$, we can always $C(n)\delta$-GHA $\vp \colon X \to Y$ which well-approximates $\phi$ (see Lemma \ref{l:Lipschitz-GHA}).
\end{rem}

Suppose $B$ is a ball in some $n$-dimensional normed space and $\phi : X \to B$ is a $\delta$-GHA. While $\phi$ is not necessarily Lipschitz, it is possible to find a nearby $C\delta$-GHA which is Lipschitz. Before proving this we needs some auxiliary results. 

\begin{lem}\label{l:p-lambda}
	Suppose $(Y,\|\cdot\|)$ is a normed space, $z \in Y$ and $\lambda > 0$. Define a map $p_{z,\lambda} : Y \to B_Y(z,\lambda)$ by 
	\[ p_{z,\lambda}(y) = z + \frac{y-z}{\max\{1,\|y-z\| / \lambda\} }. \]
	Then $p_{z,\lambda}$ is 2-Lipschitz and 
	\begin{align}\label{e:p-y}
		\|p_{z,\lambda}(y)- y\| \leq \max\{0,\|y-z\| - \lambda\}
	\end{align}
	for all $y \in Y.$ 
\end{lem}

\begin{proof}
	By translating it suffices to prove the case $z = 0.$ Let us write $p_\lambda = p_{0,\lambda}.$ It is easy to check that $\|p_{\lambda}(y)\| \leq \lambda$ for all $y \in Y$ so $p_{\lambda}$ takes values in $B_Y(0,\lambda)$. We now consider \eqref{e:p-y}. Let $y \in Y$ and, to ease notation, let $s(y) = \max\{1,\|y\|/\lambda\}.$  If $\|y\| \leq \lambda$ then $s(y) = 1$ so that $p_{\lambda}(y) = y$ and \eqref{e:p-y} follows immediately. Suppose instead that $\|y\| > \lambda.$ Then $s(y) = \|y\|/\lambda$ and 
	\[\|p_\lambda(y) - y \| = \bigg\| \frac{\lambda y}{\|y\|} - y \bigg\| =   \big| \lambda - \|y\| \big|  = \|y\| - \lambda. \]
	Finally, we check that $p_\lambda$ is 2-Lipschitz. Let $x,y \in Y$ and suppose without loss of generality that $\|x\| \leq \|y\|.$ If $\|y\| \leq \lambda$ then $p_{\lambda}(x) = x$ and $p_{\lambda}(y) = y$ so there is nothing to show. Suppose instead that $\|y\| > \lambda.$ In this case, we have $1 \leq s(x) \leq \|y\|/\lambda$ and $\|x\| \leq \lambda s(x).$ Combining this with the triangle inequality, we get 
	\begin{align}
		\| p_\lambda(x) - p_\lambda(y) \| &= \bigg\|\frac{x}{s(x)} - \frac{\lambda y}{\|y\|} \bigg\|  =  \frac{1}{s(x)}\bigg\| x - \frac{ \lambda s(x) y}{\|y\|} \bigg\| \leq \| x- y \| + \bigg| 1 - \frac{ \lambda s(x)}{\|y\|} \bigg|  \|y\| \\
		&\leq \| x-y \| + \bigg( 1- \frac{\|x\|}{\|y\|} \bigg) \|y\| \leq 2\|x -y \|,
	\end{align}
	which finishes the proof.
\end{proof}

\begin{lem}\label{l:Lipschitz-extend}
	Let $L \geq 0,$ $X$ be a metric space and $Y$ an $n$-dimensional normed space. Suppose $A \subseteq X$ and $f : A \to Y$ is $L$-Lipschitz, then there exists an $nL$-Lipschitz extension $F : X \to Y$ of $f.$ 
\end{lem}

\begin{proof}
	Let $\ell^n_\infty = (\R^n,\|\cdot\|_\infty).$ By Remark \ref{r:BM}, we have $d_\BM(Y,\ell_\infty^n) \leq n$ so there exists an isomorphism $T : Y \to \ell^n_\infty$ satisfying $\| T \|_{\op} \leq 1$ and $\|T^{-1}\|_{\op} \leq n.$ Let $g = T \circ f : A \to \ell^n_\infty.$ By applying the McShane Extension Theorem \cite[Theorem 6.2]{heinonen2012lectures} to each component of $T \circ f,$ we obtain an $L$-Lipschitz extension $g : X \to \ell^n_\infty$ of $T \circ f.$ It is then easy to check that $T^{-1} \circ g : X \to Y$ is the required $nL$-Lipschitz extension of $f.$ 
\end{proof}

\begin{lem}\label{l:Lipschitz-GHA}
	Let $0 < \delta < 1,$ $(X,d)$ a compact metric space and $(Y,\|\cdot\|)$ an $n$-dimensional normed space. Suppose $B \subseteq Y$ is a ball of radius $r > 0$ and $\phi :X \to B$ is a $\delta r$-GHA. Then, there exists a $20n\delta r$-GHA $\varphi : X \to B$ which is $20n$-Lipschitz such that 
	\begin{align}\label{e:phi-var}
		\sup_{x \in X} \| \phi(x) - \varphi(x)\| \leq 20n\delta r .
	\end{align}
\end{lem}

\begin{proof}
	By translating and scaling we may assume $x_B = 0$ and $r = 1.$ Let $\mathcal{N}$ be a maximal net in $X$ such that $d(x,y) \geq \delta$ for all distinct $x,y \in \mathcal{N}.$ Since $\phi$ is a $\delta$-GHA, if $x,y \in \mathcal{N}$ are distinct then 
	\[ \| \phi(x) - \phi(y) \| \leq \delta + d(x,y) \leq 2d(x,y). \]
	It follows that $\phi|_\mathcal{N}$ is $2$-Lipschitz. By Lemma \ref{l:Lipschitz-extend}, we may extend $\phi|_{\mathcal{N}}$ to a $2n$-Lipschitz map $\tilde{\varphi} : X \to Y$. The map $\tilde{\varphi}$ is almost what we need, however, it may still take values outside $B$. We instead define 
	\[ \varphi = p_{0,1} \circ \tilde{\varphi} : X \to B,\]
	where $p_{0,1}$ is the map from Lemma \ref{l:p-lambda}. Since $p_{0,1}$ is 2-Lipschitz we know $\varphi$ is $4n$-Lipschitz. Next, let us check \eqref{e:phi-var}. Suppose $x \in X$ and choose $x' \in \mathcal{N}$ such that $d(x,x') \leq \delta.$ Using that $\tilde{\varphi}$ is $2n$-Lipschitz and $\tilde{\varphi}(x') = \phi(x')$, we have 
	\begin{align}\label{e:tilde-phi}
		\| \tilde{\varphi}(x) - \phi(x) \| \leq \| \tilde{\varphi}(x) - \tilde{\varphi}(x') \| +  \| \phi(x') - \phi(x) \| \leq (2n + 1)d(x,x') + \delta \leq 4n\delta. 
	\end{align} 
	Since $\phi$ takes values in $B$ we must have $\| \tilde{\varphi}(x) \| \leq 1 + 4n\delta$ (recall that $B$ is centred at the origin with unit radius). With the definition of $\varphi,$ \eqref{e:p-y} implies $\|\varphi(x) - \tilde{\varphi}(x) \| \leq 4n\delta$. Combined with \eqref{e:tilde-phi}, this gives
	\begin{align}
		\|\varphi(x) - \phi(x) \| \leq \| \varphi(x) - \tilde{\varphi}(x)\| + \| \tilde{\varphi}(x) - \phi(x) \| \leq 8n\delta
	\end{align}
	and completes the proof of \eqref{e:phi-var}. It only remains for us to check that $\varphi$ defines a $20n\delta$-GHA. By \eqref{e:phi-var}, the triangle inequality and the fact that $\phi$ is a $\delta$-GHA, if $x,y \in X$ then 
	\begin{align}
		\left| d(x,y) - \| \varphi(x) - \varphi(y) \| \right| &\leq 	\left| d(x,y) - \| \phi(x) - \phi(y) \| \right|  + \| \varphi(x) - \phi(x) \| \\
		&\hspace{2em}+ \| \varphi(y) - \phi(y) \| \leq 20n\delta,
	\end{align}
	so $\varphi$ defines a $20n\delta$-isometry. Furthermore, if $u \in B$ then there exists $x \in X$ such that $\| \phi(x) - u \| \leq \delta.$ By \eqref{e:phi-var} we have $\| \varphi(x) - u \| \leq 9n\delta$ and so $\varphi(X)$ is $9n$-dense in $B$.  
\end{proof}

If one of $X$ or $Y$ is a normed space, $B \subseteq X$ and $B' \subseteq Y$ are balls, and $\phi \colon B \to B'$ is a $\delta$-GHA, the next lemma tells us that $\phi(x_B)$ is mapped close to $x_{B'}.$ Additionally, for another ball $B'' \subseteq B$, by slightly modifying the restriction of $\phi$ to $B''$ we are able to construct a $C\delta$-GHA from $B''$ to a ball of the same radius in $Y$.

\begin{lem}\label{l:GH-scale}
	Let $0 < 2\delta < \alpha < 1.$ Suppose $(X,d_X)$ and $(Y,d_Y)$ are metric spaces and that one of $X$ or $Y$ is a normed space whose metric is induced by the norm. Let $B \subseteq X$ and $B' \subseteq Y$ be balls of common radius $r > 0$ and suppose $\phi \colon B \to B'$ is a $\delta r$-GHA. First, we have 
	\begin{align}\label{e:near-centre}
		\phi(x_{B}) \in 10\delta B'
	\end{align}
	and 
	\begin{align}\label{e:dist-sim}
		d_Y(\phi(x),x_{B'}) \leq d_X(x,x_B) + 11\delta r
	\end{align}
	for all $x \in B.$ 	Additionally, if $z \in B$ is such that $B_X(z,\alpha r) \subseteq B,$ there exists a map $\vp \colon B_X(z,\alpha r) \to B_Y(\phi(z),\alpha r)$ which is a $50\delta r$-GHA such that 
	\begin{align}\label{e:phi-tildephi}
		d_Y(\phi(x),\vp(x)) \leq \delta r
	\end{align}
	for all $x \in B_X(z,\alpha r).$ 
\end{lem}

\begin{proof}
	We will prove the case when $Y$ is a normed space with norm $\|\cdot\|.$ The proof in the other case is very similar and we omit the details. By translating and scaling we may assume $x_{B'} = 0$ and $r = 1.$ We start with \eqref{e:near-centre}. Suppose towards a contradiction that $\| \phi(x_{B}) \| > 10\delta$. Let $z = -\phi(x_B)/\|\phi(x_B)\|$ so that $\| z - \phi(x_B) \| \geq 1 + 10\delta$ and let $x \in B$ such that $\| \phi(x) - z \| \leq \delta.$ Then, 
	\begin{align}
		d(x,x_B) \geq \|\phi(x) - \phi(x_B) \| - \delta \geq \|\phi(x) - z \| + \| z - \phi(x_B) \| - \delta \geq 1 + 9 \delta.
	\end{align}
	However, since $x \in B$ we have $d(x,x_B)  \leq 1,$ which is a contradiction. Equation \eqref{e:dist-sim} follows immediately from \eqref{e:near-centre} and the fact that $\phi$ is a $\delta$-GHA since for any $x \in B$ we have 
	\begin{align}
	\| \phi(x) \| \leq \|\phi(x) - \phi(x_B)\| + 10\delta \leq d(x,x_B) + 11\delta. 
	\end{align}
	Fix now some $z \in B$ such that $B(z,\alpha) \subseteq B$ and let us define $\vp \colon B_X(z,\alpha) \to B_Y(\phi(z),\alpha)$. The map $\phi|_{B_X(z,\alpha)}$ is almost what we need, however it may take values outside $B_Y(\phi(z),\alpha)$. To remedy this we instead define
	\begin{align}
		\vp = p_{\phi(z),\alpha} \circ \phi|_{B_X(z,\alpha)} \to B_Y(\phi(z),\alpha), 
	\end{align}
	where $p_{\phi(z),\alpha} : Y \to \alpha B_Y(\phi(z),\alpha)$ is the map from Lemma \ref{l:p-lambda}. Since $\phi$ is a $\delta$-GHA, if $x \in B_X(z,\alpha)$ then $\| \phi(x) - \phi(z) \|\leq \alpha + \delta$ and \eqref{e:phi-tildephi} follows from this, \eqref{e:p-y} and the definition of $\vp$. It remains to check that $\vp$ is a $50\delta$-GHA. By \eqref{e:phi-tildephi}, if $x,y \in B_X(z,\alpha)$ then 
	\begin{align}
		\left|\|\vp(x)-\vp(y)\| - d(x,y)\right| &\leq \left|\|\phi(x)-\phi(y)\| - d(x,y)\right| + \|\vp(x) - \phi(x) \| \\
		&\hspace{2em}+ \| \vp(y) - \phi(y) \| \leq 3\delta
	\end{align}
	and so $\vp$ defines a $3 \delta$-isometry. Suppose now that $u \in B_Y(\phi(z),\alpha)$. Let $\tilde{u} \in B_Y(\phi(z),\alpha - 2\delta)$ such that $\|u-\tilde{u}\| \leq 2\delta$ (such a point exists for example by Lemma \ref{l:p-lambda}) and choose a point $x \in B$ such that $\|\phi(x) - \tilde{u} \| \leq \delta$. By the triangle inequality and using that $\phi$ is a $\delta$-GHA we have 
	\begin{align}
		d(x,z) \leq \| \phi(x) - \phi(z) \| + \delta \leq \| \phi(x) - \tilde{u} \| + \|\tilde{u} - \phi(z) \|+ \delta \leq \delta + \alpha - 2\delta + \delta = \alpha. 
	\end{align}
	This gives $x \in B_X(z,\alpha).$ Now, after combining the above observations with \eqref{e:phi-tildephi} we get 
	\begin{align}
		\| \vp(x) - u \| \leq \|\vp(x) - \phi(x) \| + \|\phi(x) - \tilde{u} \| + \| \tilde{u} - u \| \leq 3\delta.
	\end{align}
	and conclude that $\vp(B_X(z,\alpha))$ is $3\delta$-dense. 
	This completes the proof. 
\end{proof}

\begin{rem}
	While this result holds for all choices of $0 < 2\delta < \alpha < 1$ it is only really interesting for $\delta$ much smaller than $\alpha.$ 
\end{rem}

\begin{lem}\label{l:GH-comp}
	Let $X,Y$ and $Z$ be compact metric spaces and $0 < \delta,\ve < 1.$ Suppose $f : X \to Y$ is a $\delta$-GHA and $g: Y \to Z$ is an $\ve$-GHA. Then, $g \circ f : X \to Z$ is an $(\delta + 2\ve)$-GHA. 
\end{lem}

\begin{proof}
	Set $h = g\circ f$. Since $f$ is a $\delta$-isometry and $g$ is an $\ve$-isometry, if $x,y \in X$ then
	\begin{align}
		\abs{ d_X(x,y) - d_Z(h(x),h(y) } &\leq \abs{ d_X(x,y) - d_Y(f(x),f(y)) } \\
		&\hspace{2em} + \abs{ d_Y(f(x),f(y)) - d_Z(h(x),h(y)) }  \leq \delta + \ve
	\end{align}
	 and so $h$ define a $(\delta + \ve)$-isometry. Now suppose $z \in Z.$ Since $g(Y)$ is $\ve$-dense in $Z$, there exists $y \in Y$ such that $d_Z(g(y),z) < \ve.$ Since $f(X)$ is $\delta$-dense in $X,$ there exists $x \in X$ such that $d_Y(f(x),y) \leq \delta.$ Using this with the fact that $g$ is an $\ve$-isometry, we have 
	\begin{align}
		d_Z(h(x),z) &\leq d_Z(h(x),g(y)) + d_Z(g(y),z) \leq d_Z(f(x),y) + \ve + d_Z(g(y),z) \\
		&\leq \delta+ 2\ve. 
	\end{align}
	We conclude that $h(X)$ is $(\alpha + 2\ve)$-dense in $Z,$ and this finishes the proof.  
\end{proof}

In the final part of this section we show that, for $\delta$ small enough, if there exists a $\delta r$-GHA between balls of radius $r$ in two norms spaces then actually there exists a global affine map between the spaces which is bi-Lipschitz with small constant i.e. the balls in two normed spaces are close then the spaces are small perturbations of each other.

We need several preliminary lemmas. The first of these is a useful topological observation. A proof in the case $r =1$ and $\|\cdot\| = |\cdot|$ (the Euclidean norm) can be found in \cite[Lemma 7.3]{bate2020purely}. The proof is easily modified to give the more general statement. 

\begin{lem}\label{l:topology}
	Let $0 < \delta < \tfrac{1}{2}$. Suppose $\|\cdot\|$ is a norm on $\R^n$, $B = B_{\|\cdot\|}(x,r)$ is a ball and $f \colon B \to \R^n$ is a continuous function such that $\|f(x) - x \| \leq \delta r$ for all $x \in \partial B$ then $f(B) \supseteq B(x,(1-\delta)r).$ 
\end{lem}

\begin{cor}\label{c:identity-infty}
	Let $\|\cdot\|$ be a norm on $\R^n$ and $H \colon \R^n \to \R^n$ a continuous function. Suppose there exists $R > 0$ such that $H(x) = x$ for all $x \in \R^n$ satisfying $\| x \| > R.$ Then $H$ is surjective. 
\end{cor}

\begin{proof}
	Let $x \in \R^n$ and fix $r > 2\max\{\|x\|,R\}.$ Our hypotheses on $H$ gives $H(y) = y$ for all $y \in \partial B_{\|\cdot\|}(0,r)$. Apply Lemma \ref{l:topology} with $\delta = \tfrac{1}{4}$ implies $B_{\|\cdot\|}(0,\|x\|) \subseteq B_{\|\cdot\|}(0,\tfrac{3r}{4})  \subseteq H(B_{\|\cdot\|}(0,r))$ which proves the result. 
\end{proof}

\begin{cor}\label{c:topology}
	Let $0 < \delta < \tfrac{1}{10}.$ Suppose $\|\cdot\|_1,\|\cdot\|_2$ are two norms on $\R^n$ and $T \colon (\R^n,\|\cdot\|_1) \to (\R^n,\|\cdot\|_2)$ is affine and $(1+\delta)$-bi-Lipschitz. Then, 
	\[ T(B_{\|\cdot\|_1}(x,r)) \supseteq B_{\|\cdot\|_2}(T(x),(1-C\delta) r)\]
	for all $x \in \R^n$ and $r > 0.$ 
\end{cor}

\begin{proof}
	By scale invariance we may assume $r =1.$ Let $p = p_{x,1-\delta} \colon (\R^n,\|\cdot\|_1) \to B_{\|\cdot\|_1}(x,1)$ be the map from Lemma \ref{l:p-lambda} and define $f \coloneqq T \circ p \circ T^{-1}.$ Let $y = T(x)$. Then, $f$ maps $B_{\|\cdot\|_1}(y,1)$ to itself and satisfies $\| f(z) - z \|_2 \leq 4\delta$ for all $z \in B_{\|\cdot\|_2}(y,1)$. Lemma \ref{l:topology} now gives $f(B_{\|\cdot\|_2}(y,r)) \supseteq B_{\|\cdot\|_2}(y,1-4\delta).$ Since $p \circ T^{-1}(B_{\|\cdot\|_2}(y,1)) \subseteq B_{\|\cdot\|_1}(x,1),$ this proves the result. 
\end{proof}

The following is due to Mankiewicz \cite{mankiewicz1972extension}.  

\begin{lem}\label{l:find-isometry}
	Suppose $X$ and $Y$ are normed spaces and $E \subseteq X$ and $F \subseteq Y$ are closed convex sets with non-empty interior. If $f : E \to F$ is a surjective isometry then there exists an affine isometry $T : X \to Y$ such that $f|_E = T|_E.$ 
\end{lem}

We find the affine map discussed above by a compactness argument, thus, it is useful to know what happens if there exist $\delta$-GHA with arbitrary small $\delta$. The following is sufficient for our needs. 

\begin{lem}\label{l:find-isometry2} 
	Let $\|\cdot\|_1,\|\cdot\|_2$ be norms on $\R^n$ and suppose $\phi_i : B_{\|\cdot\|_1}(0,1)\rightarrow B_{\|\cdot\|_2}(0,1)$ is a sequence of $\delta_i$-\emph{GHA} with $\delta_i \to 0.$ There exists a subsequence $\phi_{i_k}$ such that the following holds. For all $\ve > 0$ there is $K \in \N$, depending only on $\ve$, and a linear map $T : \R^n \to \R^n$ such that 
	\begin{align}\label{e:T-isom2}
		\|T(x) - T(y) \|_2 = \|x-y\|_1 \mbox{ for all } x,y \in \R^n 
	\end{align}
	and 
	\[ \sup_{x \in B_{\|\cdot\|_1}(0,1)} \| \phi_{i_k}(x) - T(x) \|_2 \leq \ve \mbox{ for all } k \geq K.  \]
\end{lem}

\begin{proof}
	Denote $B_1 = B_{\|\cdot\|_1}(0,1)$ and $B_2 = B_{\|\cdot\|_2}(0,1)$. By Lemma \ref{l:Lipschitz-GHA}, we may as well assume that each $\phi_i$ is $20n$-Lipschitz. Then, by Arzel\`a-Ascoli, we can find a subsequence $\{\phi_{i_k}\}$, which converge uniformly to a function $f : B_1 \to B_2.$ If we can show that $f$ is the restriction of a linear isometry $T \colon (\R^n,\|\cdot\|_1) \to (\R^n,\|\cdot\|_2)$ then we are done. To see that this holds, first notice that
		\begin{align}
		\abs{ \| f(x) - f(y) \|_2 -\|x-y\|_1} &= \lim_{k \to \infty}  \abs{  \| \phi_{i_k}(x) - \phi_{i_k}(y) \|_2 - \|x-y\|_1 }\\
		&\leq  \lim_{k \to \infty} \delta_{i_k} = 0
	\end{align}
	for all $x,y \in B_1.$ Hence, $f$ is an isometry. Additionally, if $u \in B_2$ then there exist points $x_{i_k} \in B_1$ such that $\| \phi_{i_k}(x_{i_k}) - u \| \leq \delta_{i_k}.$ Since $B_1$ is compact we find a further subsequence $x_{i_{k_j}}$ which converges to a point $x \in B_1.$ Since the sequence $\phi_{i_{k_j}}$ converges uniformly to $f$ on $B_1,$ we have 
	\begin{align}
		\| f(x) - u \| = \lim_{j \to \infty} \| \phi_{i_{k_j}}(x_{i_{k_j}}) - u \| = 0,
	\end{align}
	so $f$ is a surjective isometry. Now, by Lemma \ref{l:find-isometry} there exists an affine map $T : \R^n \to \R^n$ satisfying \eqref{e:T-isom2} and such that $T|_{B_1} = f|_{B_1}$. By \eqref{e:near-centre} we know $\| \phi_{i_k}(0) \| \leq 10 \delta_{i_k}$ which implies $T(0) = f(0) = 0,$ so $T$ is in fact linear. 
\end{proof}

\begin{lem}\label{l:GH-linear}
	For all $0 < \ve < 1$ there exists $0 < \delta < 1$ such that the following holds. Let $\|\cdot\|_1,\|\cdot\|_2$ be two norms on $\R^n$ and suppose $\phi : B_{\|\cdot\|_1}(x_1,r)\rightarrow B_{\|\cdot\|_2}(x_2,r) $ a $\delta r$-\emph{GHA}. Then, there exists an affine map $T : \R^n \rightarrow \R^n$ such that 
	\begin{align}
		(1 - C\ve)\| x-y \|_1 \leq \|T(x) - T(y) \|_2 \leq (1+C\ve)\|x-y\|_1 \mbox{ for all } x,y \in \R^n
	\end{align}
	and 
	\begin{align}
		\sup_{x \in B_{\|\cdot\|_1}(x_1,r)} \|\phi(x) - T(x)\|_2 \leq \ve r. 
	\end{align}
\end{lem}

\begin{proof}
	By translating, scaling and composing $\phi$ with an affine transformation, it is enough find a linear map in the case where $r = 1$ and both balls are centred at the origin. Suppose the lemma is false (with these extra assumption), then there exists $0 < \ve <1$ such that the following holds. For each $i \in \N,$ there are norms $\| \cdot\|_{1,i},\| \cdot\|_{2,i}$ on $\R^n$ and a $\tfrac{1}{i}$-GHA $\phi_i \colon B_1^i \to B_2^i$, where $B_1^i = B_{\| \cdot\|_{1,i}}(0,1)$ and $B_2^i = B_{\| \cdot\|_{2,i}}(0,1)$, such that
	\begin{align}\label{e:>ve}
		\sup_{x \in B_1^i} \|\phi_i(x) - T(x)\|_{2,i} > \ve
	\end{align}
	for all linear maps $T : \R^n \rightarrow \R^n$ satisfying 
	\begin{align}\label{e:bi-lip-T} 
		(1-C\ve)\| x-y \|_{1,i} \leq \|T(x) - T(y) \|_{2,i}\leq (1+C\ve)\|x-y\|_{1,i}  \mbox{ for all } x,y \in \R^n.
	\end{align}
	Recalling that $(I(n),\log d_{\text{BM}})$ is compact (see Remark \ref{r:BM}), we may suppose there are norms $\|\cdot\|_1$ and $\|\cdot\|_2$ such that $(\R^n,\| \cdot\|_{1,i}) \to (\R^n, \| \cdot\|_1)$ and $(\R^n,\| \cdot\|_{2,i}) \to (\R^n,\| \cdot\|_2)$ in $I(n).$ Let $B_1 = B_{\| \cdot\|_1}(0,1)$ and $B_2 = B_{\| \cdot\|_2}(0,1).$ 
	
	We define a sequence of maps $\varphi_i : B_1 \to B_2$ as follows. For $i \in \N$ let
	\[ \delta_1^i = 2 \log d_\BM( (\R^n,\| \cdot\|_1), (\R^n,\| \cdot\|_{1,i})) \mbox{ and } \delta_2^i = 2 \log ( (\R^n,\| \cdot\|_2), (\R^n,\| \cdot\|_{2,i})). \]
	By definition, there exist linear isomorphisms $T_{1,i} : (\R^n,\| \cdot\|_1) \to (\R^n,\| \cdot\|_{1,i})$ and $T_{2,i} : (\R^n,\| \cdot\|_{2,i}) \to (\R^n,\| \cdot\|_2)$ such that   
	\begin{align}\label{e:Tnorm}
		\|T_{j,i}\| \leq 1 \mbox{ and } \|T_{j,i}^{-1}\| \leq e^{\delta_j^i/2} \leq 1 + \delta_j^i, \quad j =1,2.
	\end{align}
	Then, let 
	\[ \varphi_i \coloneqq T_{2,i} \circ \phi_i \circ T_{1,i} \colon B_1 \to B_2. \]
	We claim 
	\begin{align}\label{e:GHA}
		\varphi_i \mbox{ is a } C(\delta_i^1 + \delta_i^2 + \tfrac{1}{i})\mbox{-GHA}. 
	\end{align}	
	Indeed, suppose first that $x,y \in B_1.$ By \eqref{e:Tnorm} and since $\phi_i$ is a $\tfrac{1}{i}$-GHA, we have 
		\begin{align}
		\|\varphi_i(x) - \varphi_i(y)\|_2 &\leq \| \phi_i \circ T_{1,i}(x) - \phi_i \circ T_{1,i}(y) \|_{2,i} \leq \|T_{1,i}(x) - T_{1,i}(y)\|_{1,i} + \tfrac{1}{i}  \\
		&\leq \|x-y\|_1 + \tfrac{1}{i}
	\end{align}
	and 
	\begin{align}
		\|\varphi_i(x) - \varphi_i(y)\|_2 &\geq (1-C\delta_2^i) \| \phi_i \circ T_{1,i}(x) - \phi_i \circ T_{1,i}(y) \|_{2,i} \\
		&\geq (1-C\delta_2^i)\left(\|T_{1,i}(x) - T_{1,i}(y)\|_{1,i} - \tfrac{1}{i}\right) \\
		&\geq (1-C\delta_1^i)(1-C\delta_2^i)\left(\|x-y\|_1 - \tfrac{1}{i}\right) \\
		&\geq \|x-y\|_1 - C\delta_1^i - C\delta_2^i - \tfrac{1}{i}, 
	\end{align}
	where we used the fact that $\|x-y\|_1 \leq 2$ in the final inequality. Since $x,y \in B_1$ were arbitrary, combining the above two estimates we have that $\vp_i$ is a $C(\delta_1^i + \delta_2^i + \tfrac{1}{i})$-isometry.	
	
	Now, let $u \in B_2.$ By \eqref{e:Tnorm} and Corollary \ref{c:topology} we have $T_{2,i}(B_2^i) \supseteq B_{\|\cdot\|_2}(0,1-C\delta_2^i).$ Hence, there exists $z \in B_2^i$ such that $\|T_{2,i}(z) - u\|_2 \leq C \delta_2^i.$ Since $\phi_i$ is a $\tfrac{1}{i}$-GHA, there exists $y \in B_1^i$ such that $\| \phi_i(y) - z \|_{2,i} \leq \tfrac{1}{i}.$ Again, by \eqref{e:Tnorm} and Corollary \ref{c:topology} we have $T_{1,i}(B_1) \supseteq B_{\|\cdot\|_{1,i}}(0,1-C\delta_1^i)$, so there exists $x \in B_1$ such that $\| T_{1,i}(x) - y\|_{1,i} \leq C \delta_1^i.$ Combining the above, with more applications of \eqref{e:Tnorm}, gives
	\begin{align}\label{e:varphi-isom}
		\| \varphi_i(x) - u \|_2 &\leq \|T_{2,i} \circ \phi_i \circ T_{1,i}(x) - T_{2,i}\circ \phi_i(y) \|_2 + \| T_{2,i} \circ \phi_i(y) - T_{2,i}(z) \|_2 \\
		&\hspace{2em} +  \|T_{2,i}(z) - u \|_2 \\ 
		&\leq \| T_{1,i}(x) - y \|_{1,i} + \tfrac{1}{i} + \| \phi_i(y) - z \|_{2,i} + \| T_{2,i}(z) - u \|_2 \\
		&\leq C(\delta_1^i + \delta_2^i + \tfrac{1}{i}). 
	\end{align}
	Since $u \in B_2$ was arbitrary, it follows that $\vp_i(B_1)$ is $C(\delta_1^i + \delta_2^i + \tfrac{1}{i})$-dense in $B_2.$ This now finishes the proof of \eqref{e:GHA}. 
	
	Since each $\varphi_i : B_1 \rightarrow B_2$ is a $C(\delta_1^i + \delta_2^i + \tfrac{1}{i})$-GHA and $\delta_1^i + \delta_2^i + \tfrac{1}{i} \to 0$, we can apply Lemma \ref{l:find-isometry2} to find a subsequence of the $\varphi_{i}$ (which we do not relabel), an integer $K \geq 0$ depending on $\ve,$ and a linear map $\tilde{T} : \R^n \to \R^n$ such that 
	\begin{align}\label{e:T-isom} \|\tilde{T}(x) - \tilde{T}(y)\|_2 = \|x-y\|_1 \mbox{ for all } x,y \in \R^n
	\end{align}
	and 
	\begin{align}\label{e:var-T}
		\sup_{x \in B_1} \| \varphi_{i}(x) - \tilde{T}(x) \|_2 < \tfrac{\ve}{4} \mbox{ for all } i \geq K.
	\end{align}
	Suppose $i \geq K$ is such that 
	\begin{align}\label{e:delta_i}
		C(\delta_1^i + \delta_2^i + \tfrac{1}{i}) \leq \tfrac{\ve}{2}
	\end{align}
	and set $T \coloneqq T_{2,i}^{-1} \circ \tilde{T} \circ T_{1,i}^{-1}.$ Then, $T$ is a linear map from $\R^n$ to itself and satisfies \eqref{e:bi-lip-T} by \eqref{e:Tnorm}, \eqref{e:T-isom} and \eqref{e:delta_i}. Now, let $x \in B_1^i$ and choose $x' \in B_{\|\cdot\|_{1,i}}(0,1-C\delta_1^i)$ such that $\|x - x'\|_{1,i} \leq C\delta_1^i.$ It follows that $T^{-1}(x') \in B_1$ by \eqref{e:Tnorm}. By \eqref{e:bi-lip-T}, \eqref{e:delta_i} and since $\phi_i$ is a $\tfrac{1}{i}$-GHA, we have
	\begin{align}
		\| \phi_i(x) - T(x) \|_{2,i} &\leq \|\phi_i(x') - T(x') \|_{2,i} + \| \phi_i(x') - \phi_i(x) \|_{2,i} + \| T(x') - T(x) \|_{2,i} \\
		&\leq  \|\phi_i(x') - T(x') \|_{2,i}  + C(\delta_1^i + \delta_2^i + \tfrac{1}{i}) \\
		&\leq \|\phi_i(x') - T(x') \|_{2,i} + \tfrac{\ve}{2}. 
	\end{align}
	Since $T^{-1}(x') \in B_1,$ \eqref{e:Tnorm} and \eqref{e:var-T} give
	\begin{align}
		\|\phi_i(x') - T(x') \|_{2,i} &=  \| T_2^{-1} \circ \varphi_i \circ T_1^{-1}(x') - T_2^{-1} \circ \tilde{T} \circ T_1^{-1}(x') \|_{2,i}  \\
		&\leq (1+\delta_2^i) \| \varphi_i \circ T_1^{-1}(x') - \tilde{T} \circ T_1^{-1}(x') \|_2 \leq \tfrac{\ve}{2}. 
	\end{align}  
	Combining the above two estimates we get 
	\[  \|\phi_i(x) -T(x)\|_{2,i} \leq \ve. \]
	Since $x \in B_{1,i}$ was arbitrary this contradict \eqref{e:>ve} and we finish the proof of the lemma. 	
\end{proof}

\bigskip

\subsection{Maps between normed spaces} In this section we include various result concerning mappings between normed spaces which will be important in the next section. 

\begin{lem}\label{l:invertible}
	Let $0 < \ve , r < 1/2$, $m \in \{1,2\}$ and $(Y,\|\cdot\|)$ be an $n$-dimensional normed space. Suppose $H : Y \to Y$ is in $C^m(Y)$, surjective and satisfies $\| H - \emph{Id} \|_{C^m(Y),r} \leq \ve.$ Then $H$ is invertible with $H^{-1} \colon Y \to Y$ in $C^m(Y)$. Moreover,
	\begin{align}\label{e:H^-1}
		\| H^{-1} - \emph{Id} \|_{C^m(Y),r} \lesssim \ve.
	\end{align}
\end{lem}

\begin{proof}
	By surjectivity and the Inverse function Theorem, to verify that $H$ is invertible with inverse $H^{-1} \in C^m(Y),$ it suffices to check that $DH(x)$ is injective (hence invertible) for each $x \in Y.$ To see this, suppose $v,w \in Y$. Then, 
		\begin{align}
	  	\| DH(x)(v) - DH(x)(w)\| &= \| DH(x)(v-w)\|  \\
		&\geq \|v-w\| - \|DH(x)(v-w) - (v-w)\| \\
		&\geq \|v-w\|/2. 
	\end{align}
	It now only remains to prove \eqref{e:H^-1}. The above inequality implies 
	\[\| DH(x)^{-1} \|_{\calB(Y)} \leq 2\]
	for all $x \in Y.$ Thus, $\|D[H^{-1}](z)\|_{\calB(Y)} \leq 2$ for all $z \in Y$ by the Inverse Function Theorem. So, for any $z \in Y,$ the estimate $\| H - \text{Id} \|_{C^m(Y),r} \leq \ve$ then implies
	\begin{align}
		\|H^{-1}(z)  - z \| = \| H^{-1}(z) - H(H^{-1}(z)) \| \leq \ve r
	\end{align}
	and 
	\begin{align}
		\| D[H^{-1}](z) -  \text{Id} \|_{\calB(Y)} &= \| D[H^{-1}](z)  - D[H\circ H^{-1}](z) \|_{\calB(Y)} \\
		& = \| D[H^{-1}](z) - DH(H^{-1}(z)) D[H^{-1}](z) \|_{\calB(Y)} \\
		& \leq \ve \| D[H^{-1}](z) \|_{\calB(Y)}  \leq 2\ve . 
	\end{align}
	Since $z$ was arbitrary, this completes the proof in the case $m =1.$ If $m=2$ we proceed as follows. Define a map $G : Y \to B(Y)$ by setting $G(x) = DH(H^{-1}(x))$. In this way $D[H^{-1}](x) = G(x)^{-1}$ (here we identify $\calB(Y)$ with the set of all $n \times n$ matrices and $G(x)^{-1}$ denotes the inverse matrix of $G(x)$). By the derivative formula for inverse matrices, the chain rule and our estimates for $H$ and $H^{-1}$, we have 
	\begin{align}
		\|D^2[H^{-1}](z)(y)\|_{\calB(E)} &= \|D[H^{-1}](x) \circ DG(z)(y) \circ D[H^{-1}](x) \|_{\calB(E)} \\
		&\lesssim \|[D^2H(H^{-1}(x)) \left(D[H^{-1}](x)(y)\right)\|_{\calB(E)} \\
		&\leq \frac{\ve}{r}\| D[H^{-1}](x)(y) \| \lesssim \frac{\ve}{r}\|y\|
	\end{align} 
	for all $y \in Y.$ It now follows that $\| D^2[H^{-1}](z) \|_{\calB(Y,\calB(Y))} \lesssim \ve/r$ and this finishes the proof.
\end{proof}

\begin{lem}\label{l:modification}
	For each $\eta > 0$ there exists $\ve > 0$ such that the following holds. Let $m \in \{1,2\},$ $r > 0,$ $(Y,\|\cdot\|)$ be an $n$-dimensional normed space and $U_1 \subseteq U_2 \subseteq Y$ be open bounded sets. Suppose that $\dist(U_1,U_2^c) \geq \eta r$ and $H \colon U_2 \to Y$ is in $C^m(U_2)$ and satisfies
	\begin{align}\label{e:H-Id}
		\| H - \emph{Id}  \|_{C^m(U_2),r} \leq \ve. 
	\end{align}
	Then, there exists a map surjective map $\hat{H} : Y \to Y$ in $C^m(Y)$ such that
	\begin{align}
		\hat{H}|_{U_1} &= H|_{U_1}, \label{e:H1} \\
		\hat{H}|_{U_2^c} &= \emph{Id}, \label{e:H2} \\
		\hat{H}(x) &= H(x) \mbox{ whenever } H(x) = x \label{e:H3}.
	\end{align}
	Furthermore, $\hat{H}$ is invertible with inverse $\hat{H}^{-1} : Y \to Y$ in $C^m(Y)$ and 
	\begin{align}
		\| \hat{H} - \emph{Id} \|_{C^m(Y),r} &\lesssim_\eta \ve \mbox{ and } \| \hat{H}^{-1} - \emph{Id} \|_{C^m(Y),r} \lesssim_\eta \ve. \label{e:H4}
	\end{align}
\end{lem}

\begin{proof}
	Let $\eta > 0$ be fixed. Since $\dist(U_1,U_2^c) \geq \eta r,$ there exists a smooth bump function $\varphi : Y \to \R$ such that $\varphi \equiv 1$ on $U_1,$ $\varphi \equiv 0$ on $U_2^c$ and $\| D^k \varphi \|_\op \lesssim_\eta r^{-k}$ for each $k \in \{1,2\}.$ Let $\hat{H} = (H - \mbox{Id}) \varphi + \mbox{Id}.$ Then, $\hat{H} \in C^m(Y)$ and it is simple to check that \eqref{e:H1}, \eqref{e:H2} and \eqref{e:H3} hold with this definition. If $x \in Y$ then by \eqref{e:H-Id}, we have 
	\begin{align}
		\| \hat{H}(x) - x \| = \| (H(x) - x) \varphi(x)\| \leq \| H(x) - x \| \, |\varphi(x)| \lesssim \ve r
	\end{align}
	and 
	\begin{align}
		\| D\hat{H}(x)  - I_n \|_\op &= \| D[\hat{H} - \mbox{Id}](x) \|_\op  \\
		&\leq \| D[H-\mbox{Id}](x)  \|_\op |\varphi(x)| + \| H(x) - x \| \, \| D  \varphi(x) \|_\op   \lesssim_\eta \ve,
	\end{align}
	where $I_n$ is the identity matrix on $\R^n$. If $m=1$ this proves the first estimate in \eqref{e:H4}. If $m=2$ and $y \in Y$ is such that $\|y\| =1,$ then we have the further estimate 
	\begin{align}
		\| D^2[\hat{H}-\text{Id}](x)(y) \|_{\calB(Y)} &\leq \| D^2[H-\text{Id}](x)(y)\|_{\calB(Y)} |\varphi(x)| \\
		&\hspace{2em} + 2\|D[H-\text{Id}](x)\|_{\calB(Y)}|D\varphi(x)(y)| \\
		&\hspace{4em}+ \|H(x) - x\|  \|D^2\varphi(x)(y)\|_{\calB(Y)} \lesssim_\eta  \ve r^{-1}
	\end{align}
	so that $\|D^2[\hat{H}-\text{Id}](x)\|_{\op} \lesssim_\eta \ve r^{-1}.$ This now proves the first half \eqref{e:H4} in this case. Since $\hat{H}$ is continuous and equal to the identity outside a large ball (by \eqref{e:H2}), it is surjective by Corollary \ref{c:identity-infty}. Then, supposing $\ve$ is small enough with respect to $\eta,$ the first part of \eqref{e:H4} and Lemma \ref{l:invertible} imply that $\hat{H}$ is invertible with inverse $\hat{H}^{-1} : Y \to Y$ in $C^m(Y)$ satisfying the second half of \eqref{e:H4}.
\end{proof}

\begin{lem}\label{l:composition}
	Let $(Y,\|\cdot\|)$ be a normed space and fix parameters $0 < \ve < 1$, $r > 0$, $k \geq 1$ and $m \in \{1,2\}.$ For each $1 \leq i \leq k$ let $U_i \subseteq Y$ be an open set, $f_i : U_i \to Y$ a map in $C^m(U_i)$, $g_i \colon Y \to Y$ a $(1+\ve)$-bi-Lipschitz affine map, and suppose that $\| f_i - g_i \|_{C^m(U_i),r} \leq \ve$. If $U \subseteq Y$ is open and $f_k \circ \dots \circ f_1$ is defined on $U$ then  
	\begin{align}\label{e:comp}
		\|f_k \circ \cdots \circ f_1 - g_k \circ \cdots g_1 \|_{C^m(U),r} \lesssim_k \ve. 
	\end{align}
\end{lem}

\begin{proof}
	We prove the result by induction on $k$. The base case $k =1$ holds by assumption. Now, fix some $k \geq 1$ for which the result holds and consider the case $k+1.$ Let $F_k = f_k \circ \cdots \circ  f_1$ and $G_k = g_j \circ \dots \circ g_1.$ We proceed with \eqref{e:comp}. Fix $U \subseteq Y$ on which $f_{k+1} \circ F_k$ is well-defined and let $x \in U.$ Using the estimate \begin{align}\label{e:f-g}
		\| f_{k+1} - g_{k+1} \|_{C^m(U_i),r} \leq \ve
	\end{align} 
	with \eqref{e:comp} for $k$ and the fact that $g_{k+1}$ is $(1+\ve)$-bi-Lipschitz, we get 
	\begin{align}
		\| f_{k+1}(F_k(x)) - g_{k+1}(G_k(x))\| &\leq \| f_{k+1}(F_k(x)) - g_{k+1}(F_k(x)) \| \\
		&\hspace{2em} + \|g_{k+1}(F_k(x)) - g_{k+1}(G_k(x))\| \lesssim \ve r.
	\end{align}
	Since $g_{k+1}$ and $G_k$ are affine, there exists linear maps $L_1,L_2 \colon Y \to Y$ such that $Dg_{k+1}(y) = L_1$ and $DG_k(y) = L_1$ for all $y \in Y.$ Since each $g_i$ is $(1+\ve)$-bi-Lipschitz, we have that $L_1$ is $(1+\ve)$-bi-Lipschitz and both $G_k$ and $L_2$ are $(1+C\ve)$-bi-Lipschitz for some $C$ depending on $k$. Since $G_k$ is $(1+C\ve)$-bi-Lipschitz and $\|F_k - G_k\|_{C^m(U),r} \lesssim_k \ve$ by induction, it follows that $\|DF_k(x)\|_\op \lesssim_k \ve$. Using the above with \eqref{e:comp} for $k$ and \eqref{e:f-g}, we have 
	\begin{align}
		\| D[f_{k+1} \circ F_k](x) - L_1L_2 \|_{\calB(Y)} &\leq \|Df_{k+1}(F_k(x)) DF_k(x) - L_1 DF_k(x) \|_{\calB(Y)} \\
		&\hspace{2em} + \|L_1 DF_k(x) - L_1L_2\|_{\calB(Y)} \\
		&\leq \ve \| DF_k(x) \|_{\calB(Y)} + (1+\ve)\|DF_k(x) - L_2 \|_{\calB(Y)} \\
		&\lesssim_k \ve. 
	\end{align}
	If $m=1$ this completes the proof. Suppose then that $m =2.$ Since $DG_k(y) = L_1$ and $Dg_{k+1}(y) = L_2$ for all $y \in Y$, we have $D^2G_k(y) = D^2g_{k+1}(y) = 0$ for all $y \in Y.$ Using \eqref{e:comp} for $k$ with \eqref{e:f-g} now implies $\| D^2f_{k+1}(y)(z)\|_{\calB(Y)} \leq \tfrac{\ve}{r} \|z\|$ and $\| D^2F_{k}(y)(z)\|_{\calB(Y)} \lesssim_k \tfrac{\ve}{r}\|z\|$ for all $y,z \in Y.$ Hence, 
	\begin{align}
		\| D^2[f_{k+1}\circ F_k](x)(y) \|_{\calB(Y)} &\leq \|D^2f_{k+1}(F_k(x)) (DF_k(x)(y)) DF_k(x) \|_{\calB(Y)} \\
		&\hspace{4em} + \| Df_{k+1}(F_k(x)) D^2F_k(x)(y) \|_{\calB(Y)} \\
		&\lesssim_k \frac{\ve}{r}\| DF_k(x)(y) \| \,\|DF_k(x)\|_{\calB(Y)} \\
		&\hspace{2em} + \frac{\ve}{r} \|Df_{k+1}(F_k(x))\|_{\calB(Y)} \|y\| \lesssim \frac{\ve}{r} \|y\|.
	\end{align}
	This implies $\|D^2[f_{k+1} \circ F_k](x)\|_{\calB(Y,\calB(Y))} \lesssim_k \tfrac{\ve}{r}$ and completes the proof. 
\end{proof}

\begin{lem}\label{l:C^2-comp}
	Let $\ve,r > 0$ and $\Lambda\geq 1$. Suppose $f : \R^n \to \R^n$ is in $C^1(\R^n)$ such that $\| f - \emph{Id} \|_{C^1(\R^n),r}  \leq \ve$ and $g : \R^n \to \R^n$ is invertible such that $g,g^{-1} \in C^2(\R^n)$ and 
	\begin{align}\label{e:g-control}
		\sup_{x \in \R^n} \left(\|Dg(x)\| + \|D[g^{-1}](x)\| + r\|D^2g(x) \| + r \|D^2[g^{-1}](x)\| \right)\leq \Lambda.
	\end{align}
	Then, 
	\begin{align}
		\| g \circ f \circ g^{-1} - \emph{Id} \|_{C^1(\R^n),r} \lesssim_\Lambda \ve
	\end{align}
\end{lem}

\begin{proof}
	Fix $x \in \R^n.$ Applying the Mean Value Theorem with \eqref{e:g-control} and the estimate on $\|f-\text{Id}\|_{C^1(\R^n),r}$, we first have 
	\begin{align}
		\| g \circ f \circ g^{-1}(x) - x \| &= \| g \circ f \circ g^{-1}(x) - g \circ g^{-1}(x) \| \\
		& \leq \Lambda \|f \circ g^{-1}(x) - g^{-1}(x) \| \leq L\ve r. 
	\end{align}
	To estimate the derivatives, we first note that $\| D[g \circ f \circ g^{-1}](x)  - I_n \|$ is at most 
	\begin{align}
		 &\| D[g](f(g^{-1}(x))) - D[g](g^{-1}(x) \| \,  \| D[f](g^{-1}(x))\| \, \| D[g^{-1}](x) \| \\
		&\hspace{2em} + \|D[g](g^{-1}(x))  D[f](g^{-1}(x))  D[g^{-1}](x) -I_n\|
	\end{align}
	By the Mean Value Theorem, \eqref{e:g-control} and the estimate on $\|f-\text{Id}\|_{C^1(\R^n),r}$, the first term is at most 
	\begin{align}
		\Lambda^3r^{-1} \| f(g^{-1}(x)) - g^{-1}(x) \| \leq \ve \Lambda^3. 
	\end{align}
	Since $D[g](g^{-1}(x)) D[g^{-1}](x) = I_n,$ the second term is as most 
	\begin{align}
		\| D[g](g^{-1}(x)) ( D[f](g^{-1}(x)) - \text{Id}) D[g^{-1}](x) \| \leq \ve \Lambda^2,
	\end{align}
	which completes the proof. 
\end{proof}

\begin{lem}\label{l:linear-C^0}
	Let $\|\cdot\|$ be a norm on $\R^n$, $z \in \R^n$ and $\ve,r > 0.$ Let $H : \R^n \to \R^n$ be an affine function such that $\| H(x) - x \| \leq \ve r$ for all $x \in B_{\|\cdot\|}(z,r)$. Then, 
	\begin{align}\label{e:linear1}
		\| DH \cdot v  - v \| \lesssim \ve \| v \| 
	\end{align}
	for all $v \in \R^n.$ 
\end{lem}

\begin{proof}
	Without loss of generality we shall assume $z = 0$ and $r =1.$ Since $H$ is affine it is of the form $H = y + L$ for some linear map $L : \R^n \to \R^n$ and some point $y \in \R^n.$ Evaluating $H$ at the origin and using \eqref{e:linear1} gives $\| y \|_k \leq \ve .$ Then, for any $v \in \R^n,$ we have 
	\begin{align}
		\| DH \cdot v - v \|_k &=  \|v\| \bigg\| L\left(\frac{v}{\|v\|}\right) - \frac{v}{\|v\|} \bigg\|_k \leq  \|v\| \left( \bigg\| H\left(\frac{ v}{\|v\|}\right) - \frac{v}{\|v\|} \bigg\|_k + \|y\|_k \right) \\
		&\leq 2\ve \|v\|
	\end{align}
\end{proof}

\bigskip

\subsection{Cyclic maps} As in \cite{gallier2012parametric} and \cite{violo2021functional}, it will be convenient for us to introduce the following compatibility condition for functions defined on normed spaces. 

\begin{defn}\label{d:cyclic}
	Let $X,Y,Z$ be normed space and $r > 0.$ Suppose $f : X \to Y, g : Y \to Z$ and $h: X \to Z$ are bijective. We say the maps $f,g,h$ are $r$-\textit{cyclic} if the following holds. Set
	\begin{align}\label{e:def-I} I_{2,1} = f, \ I_{1,2} = f^{-1},  \  I_{3,2} = g, \ I_{2,3} = g^{-1}, \  I_{3,1} =h, \ I_{1,3} = h^{-1}
\end{align}
	and denote $B_1 = B_X(0,r), \ B_2 = B_Y(0,r)$ and $B_3 = B_Z(0,r).$ For each $\sigma \in S_3$ (the symmetric group) we have that
	\begin{align}
		\begin{split}\label{e:cyclic-hyp}
			&\mbox{ for all } x \in  B_{\sigma(1)} \mbox{ such that } I_{\sigma(2),\sigma(1)}(x) \in B_{\sigma(2)} \mbox{ and } I_{\sigma(3),\sigma(2)} \circ I_{\sigma(2),\sigma(1)}(x) \in  B_{\sigma(3)} \\ 
			&\mbox{ we have } I_{\sigma(3),\sigma(1)}(x)   = I_{\sigma(3),\sigma(2)} \circ I_{\sigma(2),\sigma(1)}(x).
		\end{split}
	\end{align}
\end{defn}

\begin{rem}\label{r:necessary}
	The reason for introducing such a notion will become more clear in later sections. For now, observe the following. Suppose $(U_i,\vp_i),$ $i=1,2,3$, are coordinate patches in a manifold $M$ satisfying $\vp_i(U_i) = B_i$. If $I_{i,j} = \vp_i \circ \vp_j^{-1}$ are their transition maps, then \eqref{e:cyclic-hyp} must hold. In Section \ref{s:disconnected} we will construct a manifold from transition maps satisfy the condition in \eqref{e:cyclic-hyp}. 
\end{rem}

\begin{rem}\label{r:cyc}
	It turns out we do not need to check \eqref{e:cyclic-hyp} for every $\sigma \in S_3.$ Indeed, suppose $\sigma \in S_3$ satisfies \eqref{e:cyclic-hyp}. If $\sigma' \in S_3$ is such that $\sigma'(1) = \sigma(3), \sigma'(2) = \sigma(2)$ and $\sigma'(3) = \sigma(1)$ then \eqref{e:cyclic-hyp} (with $\sigma'$) also holds. In particular, one only needs to consider the permutations $(1,2,3), (3,1,2)$ and $(1,3,2).$ 
\end{rem}

The following lemma gives us a way to check the cyclic condition. Roughly speaking it says that if we know \eqref{e:cyclic-hyp} is satisfied for at least one permutation and approximately satisfied for all remaining permutation, each with parameter $r$, then in fact \eqref{e:cyclic-hyp} is satisfies for all permutation with a slightly smaller parameter. 

\begin{lem}\label{l:check-cyc}
	
	Let $X,Y,Z$ be normed spaces and $r > 0.$ Suppose $f : X \to Y, g : Y \to Z$ and $h: X \to Z$ are bijective. Let $I_{i,j}$ and $B_i$ as in Definition \ref{d:cyclic}. Additionally, let $\|\cdot\|_1, \, \|\cdot\|_2$ and $\|\cdot\|_3$ denote the norms on $X,Y$ and $Z$ respectively. Let $L \geq 1$ and $\ve > 0$ such that $L\ve < 1$ and assume that each $I_{i,j}$ is $L$-bi-Lipschitz and for each $\sigma \in S_3$ and each $x \in  B_{\sigma(1)}$ that 
	\begin{align}\label{e:almost-cyc} \| I_{\sigma(3),\sigma(2)}\circ I_{\sigma(2),\sigma(1)}(x) - I_{\sigma(3),\sigma(1)}(x) \|_{\sigma(3)} \leq \ve r.
	\end{align}
	Moreover, suppose \eqref{e:cyclic-hyp} is satisfied for $\sigma = (1,2,3).$ Then the maps $I_{1,2},I_{2,3}$ and $I_{1,3}$ are $(1 - L\ve)r$-cyclic. 
\end{lem}

\begin{proof}
	As mentioned in Remark \ref{r:cyc}, we only need to check \eqref{e:cyclic-hyp} for the permutations $(1,2,3),(3,1,2)$ and $(1,3,2).$ The case $(1,2,3)$ holds by hypothesis. Consider next $(3,1,2).$ Let $\lambda = (1-L\ve)$ and let $x \in \lambda B_3$ such that $y = I_{1,3}(x) \in \lambda B_1$ and $I_{2,1}(y) = I_{2,1}(I_{1,3}(x)) \in \lambda B_2.$ Using that $I_{3,2}$ is $L$-bi-Lipschitz with \eqref{e:almost-cyc}, we have 
	\begin{align}
		\| I_{3,2}(I_{2,1}(y)) - x\|_{3} \leq L \| I_{2,1}(y) - I_{2,3}(x) \|_2 = L \| I_{2,1}(I_{1,3}(x)) - I_{2,3}(x) \|_3 \leq L\ve r. 
	\end{align}
		By our choice of $\lambda,$ this implies $I_{2,3}(I_{2,1}(y)) \in B_3.$ Using \eqref{e:cyclic-hyp} in the case of $(1,2,3)$ then gives 
	\begin{align}
		I_{2,3}(x) = I_{2,3}(I_{3,1}(y)) = I_{2,3}(I_{3,2}(I_{2,1}(y))) = I_{2,1}(I_{1,3}(x))). 
	\end{align}
	The proof for $(1,3,2)$ is almost identical, we omit the details. 
\end{proof}

\bigskip

\subsection{Christ-David cubes and Carleson conditions}

At various points in the paper It will be convenient for us to work with a version of ``dyadic cubes'' tailored to $X$. These are the so-called Christ-David cubes, which were first introduced by David \cite{david1988morceaux} and generalized in \cite{christ1990b} and \cite{hytonen2012non}. The following is a combination of the formulation in \cite{christ1990b} and \cite{hytonen2012non}.

\begin{lem}\label{cubes}
	Let $X$ be a doubling metric space and $X_k$ be a sequence of maximal $\varrho^k$-separated nets, where $\varrho = 1/1000$ and let $c_0 = 1/500.$ Then, for each $k \in \Z$, there is a collection $\mathcal{D}_k$ of ``cubes'' such that the following conditions hold.
	\begin{enumerate}
		\item For each $k \in \Z, \ X = \bigcup_{Q \in \calD_k}Q.$
		\item If $Q_1,Q_2 \in \calD = \bigcup_{k}\calD_k$ and $Q_1 \cap Q_2 \not= \emptyset,$ then $Q_1 \subseteq Q_2$ or $Q_2 \subseteq Q_1.$ 
		\item For $Q \in \calD_k$ set $\ell(Q) = 5\varrho^k.$ Then there is $x_Q \in X_k$ such that
		\begin{align*}
			B_X(x_Q,c_0\ell(Q)) \subseteq Q \subseteq B_X(x_Q , \ell(Q)). 
		\end{align*}
		\item If $X$ is Ahlfors $n$-regular, then there exists $C \geq 1$ such that $\calH^n( \{ x \in Q : \dist(x,X \setminus Q) \leq \eta \varrho^k \} ) \lesssim \eta^\frac{1}{C} \ell(Q)^d$ for all $Q \in \calD$ and $\eta > 0.$
	\end{enumerate}
\end{lem}

Given a collection of cubes $\calD$ and $Q  \in \calD,$ define 
\begin{align}\label{d:subcubes}
	\calD(Q) = \{R \in \calD : R \subseteq Q\}.
\end{align}
Let $\text{Parent}(Q)$ be the unique cube $R \in \calD$ such that $R \supseteq Q$ and $\ell(Q) = \varrho \ell(Q).$ Let $\text{Child}(Q)$ denote the collection of cubes $R \subseteq Q$ such that $Q = \text{Parent}(R).$ Let $\text{Sibling}(Q)$ be the collection of cubes $R$ such that $\text{Parent}(R) = \text{Parent}(Q).$ 

\begin{rem}\label{r:contained-cubes}
	If $A > 2$ and $Q,R \in \calD$ are such that $Q \subseteq R$ it is easy to see that $AB_Q \subseteq AB_R.$ For a proof see \cite[Lemma 2.10]{azzam2018analyst}.
\end{rem}

\begin{defn}\label{d:packing}
	Let $X$ be an Ahlfors $n$-regular metric space and $\calD$ a system of Christ-David cubes for $X$. A collection of cubes $\calB \subseteq \calD$ satisfies a \textit{Carleson packing condition} if there exists $C > 0$ such that
	\begin{align}
		\sum_{\substack{Q \in \calB \\ Q \subseteq Q_0}} \ell(Q)^n \leq C\ell(Q_0)^n
	\end{align}
	for all $Q_0 \in \calD.$ 
\end{defn}

The Carleson packing condition is a natural measure of smallness. If $\calB$ satisfies a Carleson packing condition then roughly speaking there are not too many cubes in $\calB.$ The Carleson packing condition is a discretized version of the usual Carleson type conditions defined below. 

\begin{defn}\label{d:Carleson-set}
	Let $X$ be an Ahlfors $n$-regular metric space. A non-negative Borel measure $\mu$ on $X \times (0,\diam(X))$ is called a \textit{Carleson measure} if there exists $C > 0$ such that $\mu(B(x,r) \times (0,r)) \leq C r^n$ for all $x \in X$ and $0 < r < \diam(X).$ The smallest such $C$ is called the \textit{Carleson norm} of $\mu$. A set $\mathscr{A} \subseteq X \times (0,\diam(X))$ is a \textit{Carleson set} if $\mathds{1}_\mathscr{A}d\calH^n(x)\tfrac{dr}{r}$ is a Carleson measure.  
\end{defn}

It is easy to show the following, which allows us to pass between the conditions and discrete Carleson conditions. 

\begin{lem}\label{l:set-packing}
	Let $K > 1$ and $X$ be Ahlfors $(C_0,n)$-regular. Let $f$ a real valued function on $X \times (0,\diam(X))$ and suppose there exists $C \geq 1$ such that $r f(x,r) \leq Csf(y,s)$ for all $x,y \in X$ and $0 \leq r \leq s < \diam(X)$ which satisfy $B(x,r) \subseteq B(y,s).$ Let $\calD$ be a system of dyadic cubes for $X$. For each $Q \in \calD,$ set $f(Q) \coloneqq f(x_Q,K\ell(Q)).$ Then, $\{(x,r) \in X \times (0,\diam(X)) : f(x,r) > \ve \}$ is a Carleson set for each $\ve > 0$ if and only if $\{Q \in \calD : f(Q) > \ve\}$ satisfies a Carleson packing condition of each $\ve > 0.$ The Carleson norms depending quantitatively on each other and $C_0.$ 
\end{lem}

Finally, we define what we mean by stopping-time region. 

\begin{defn}\label{StoppingTime}
	Suppose $X$ is a metric space and $\calD$ is a system of Christ-David cubes on $X$. A collection $S \subseteq \calD$ is called a \textit{stopping-time region} if the following conditions hold.
	\begin{enumerate}
		\item There is a \textit{top cube} $Q(S) \in S$ such that $Q(S)$ contains all cubes in $S$. 
		\item If $Q \in S$ and $Q \subseteq R \subseteq Q(S),$ then $R \in S$.
		\item If $Q \in S$, then all siblings of $Q$ are also in $S$. 
	\end{enumerate}
	Let $\min(S)$ denote the \textit{minimal cubes} of $S$ i.e. those cubes in $S$ whose children are not in $S$. 
\end{defn}

\newpage

\section{UR and BWGL in metric spaces and the $\gamma$-coefficients}\label{s:UR}
\etocsettocstyle{Contents of this section}{}
\etocsettocmargins[1.5em]{.175\linewidth}{.175\linewidth}

\localtableofcontents

\bigskip

\subsection{Main definitions}

Although we have explained roughly each of the conditions (1), (2) and (3) in the introduction, we shall define them explicitly here. We shall also prove some useful properties concerning them. 

\begin{defn}\label{d:ADR}
	A metric space $X$ is said to be \textit{Ahlfors $n$-regular} if there exists a constant $C_0 \geq 1$ such that for each $x \in X$ and $0 < r < \diam(X)$ we have 
	\begin{align}
		C_0^{-1} r^n \leq \mathcal{H}^n(B_X(x,r)) \leq  C_0 r^n.
	\end{align}
	If only the first inequality holds we say $X$ is \textit{Ahlfors lower $n$-regular}. If only the second inequality holds we say $X$ is \textit{Ahlfors upper $n$-regular}. We call $C_0$ the \textit{regularity constant}. If the context is clear we may simply write Ahlfors regular, Ahlfors lower regular or Ahlfors upper regular. 
\end{defn}

\begin{defn}\label{d:UR-metric}
	Let $X$ be an Ahlfors $n$-regular metric space and $\calF$ a family of Ahlfors $n$-regular metric spaces with uniform regularity constant. We say that $X$ has \textit{big pieces of} $\calF$ if there exists a constant $\theta > 0$ such that for each $x \in X$ and $0 < r < \diam(X)$ there exists $F \in \calF$ which satisfies 
	\begin{align}
		\calH^n(F \cap B_X(x,r))  \geq \theta r^n. 
	\end{align}
\end{defn}

Both the BWGL and the Carleson condition for the $\gamma$-coefficients are stated in terms of Carleson sets, see Definition \ref{d:Carleson-set}. The relevant coefficient for defining the BWGL is $\bilat_X$, which is a Gromov-Hausdorff version of $b\beta$, measuring local approximations by $n$-dimensional Banach spaces. 

\begin{defn}\label{d:bilat}
	Let $(X,d)$ be a metric space, $x \in X$ and $0 < r < \diam(X).$ For a norm $\| \cdot \|$ on $\R^n$, let $\Phi(x,r,\|\cdot\|)$ denote the collection of all mappings of the form $\vp : B(x,r) \rightarrow B_{\|\cdot\|}(0,r).$ For $\vp \in \Phi(x,r,\|\cdot\|)$, define
	\begin{align}
		\unilat_X(x,r,\|\cdot\|,\vp) &=  \frac{1}{r} \sup_{y,z \in B(x,r)}\left| \, d(y,z) - \|\vp(y) - \vp(z) \| \, \right| ;\\
		\eta_X(x,r,\|\cdot\|,\vp) &= \frac{1}{r}\sup_{u\in B_{\|\cdot\|}(0,r)} {\dist}_{\|\cdot\|}(u,\vp(B(x,r)));\\
		\bilat_X(x,r,\|\cdot\|,\varphi) &= \unilat_X(x,r,\vp,\|\cdot\|) + \eta_X(x,r,\vp,\|\cdot\|), 
	\end{align}
	and
	\[\bilat_X(x,r) = \inf_{\|\cdot\|} \inf_{\vp \in \Phi(x,r,\|\cdot\|)} \bilat_{X}(x,r,\|\cdot\|,\vp).\]
\end{defn}

\begin{rem}
	If $X \subseteq \R^m$ and $\|\cdot\|$ is the Euclidean norm on $\R^n$ then the coefficient $\unilat_X$ behaves in some sense like $\beta^2$, where $\beta$ is the Jones $\beta$-number, as defined in \cite{jones1990rectifiable}. This behaviour was suggested in \cite{david1999reifenberg} and a rigorous statement was proven in \cite{violo2022remark}. See also \cite{badger2020subsets} and \cite{edelen2018effective}.
\end{rem}

\begin{defn}\label{d:BWGL}
	Let $X$ be an Ahlfors $n$-regular metric space. For $\ve > 0$ we say $X$ satisfies the \textit{Bi-lateral Weak Geometric Lemma with parameter $\ve$} (BWGL$(\ve)$) if the set $\{(x,r) \in E \times (0,\diam(E)) : \bilat_X(x,r) > \ve\}$ is a Carleson set. We say $X$ satisfies the \textit{Bi-lateral Weak Geometric Lemma} (BWGL) if for every $\ve > 0$ it satisfies the BWGL$(\ve)$. 
\end{defn}

To define the Carleson condition for the $\gamma$-coefficients we of course need to define the $\gamma$-coefficients. For a metric space $X,$ a ball $B_X(x,r)$, a norm $\|\cdot\|$ and a Lipschitz $f \colon X \to (\R^n,\|\cdot\|),$ we introduce the coefficient $\gamma_{X,f,\|\cdot\|}^K(x,r)$ which as well as measuring how bi-laterally flat $B_X(x,r)$ is (in the sense of $\bilat_X(x,r)$), also measures how well $f$ is approximated by $K$-Lipschitz ``affine functions on $B_X(x,r)$". Since $X$ is an arbitrary metric space, it doesn't make sense to talk about affine functions. The coefficient really measures how well $f$ is well-approximated by affine functions defined on the best approximating tangent space to $B_X(x,r)$ and is defined below.

\begin{defn}\label{d:gamma}
	Let $(X,d)$ be a metric space, $x \in X$ and $0 < r < \diam(X).$ Let $\| \cdot \|_0$ and $\|\cdot\|$ be norms on $\R^n$. For $K \geq 1$ and mappings $f :X \rightarrow (\R^n,\|\cdot\|_0)$ and $\vp \colon B(x,r) \to (\R^n,\|\cdot\|)$, define
	\[\Omega_{X,f,\|\cdot\|_0}^K(x,r,\|\cdot\|,\vp) = \frac{1}{r} \inf_A \sup_{y \in B(x,r)} \| f(y) - A(\vp(y))\|_0,\] 
	where the infimum is taken over all affine mappings $A : \R^n \rightarrow \R^n$ with $\text{Lip}(A) \leq K$ when viewed as a function from $(\R^n,\|\cdot\|)$ to $(\R^n,\|\cdot\|_0)$. Then, define
	\[\gamma_{X,f,\|\cdot\|_0}^K(x,r) = \inf_{\|\cdot\|} \inf_{\vp \in \Phi(x,r,\|\cdot\|)} [\bilat_X(x,r,\|\cdot\|,\varphi) + \Omega_{X,f,\|\cdot\|_0}^K(x,r,\|\cdot\|,\vp)].\]	
\end{defn}

The Carleson condition for the $\gamma$-coefficients is a sort of quantitative differentiation condition for Lipschitz functions on a space $X$. It states that, at most scales and locations, the space $X$ is approximately bi-laterally flat and Lipschitz functions are well-approximated by affine functions with uniform gradient on $X$. 

\begin{defn}\label{d:carleson-gamma}
	Let $X$ be an Ahlfors $n$-regular metric space. We say $X$ satisfies a \textit{Carleson condition for the $\gamma$-coefficients} if there exists $K \geq 1$ and for each $\ve > 0$ a constant $C = C(\ve)$ such that if $\|\cdot\|_0$ is a norm and $f \colon X \to ( \R^n,\|\cdot\|_0)$ is 1-Lipschitz, then 
	\begin{align}\label{d:gamma-large}
		\{(x,r) \in \diam(X) \times (0,\diam(X)) : \gamma_{X,f}^K(x,r) > \ve\}
	\end{align}
	is Carleson set with Carleson norm $C$.
\end{defn}

The following is immediate from the definitions. 

\begin{lem}\label{l:WALAM-BWGL}
	Let $X$ be a metric space, $x \in X$ and $0 < r < \diam(X)$. Let $K \geq 1,$ $\|\cdot\|_0$ a norm on $\R^n$ and $f \colon X \to (\R^n,\|\cdot\|_0)$ a 1-Lipschitz function. Then, $\gamma_{X,f,\|\cdot\|_0}^K(x,r) \geq \bilat_X(x,r).$ In particular, a Carleson condition for the $\gamma$-coefficients implies the \emph{BWGL}. 
\end{lem}

\begin{rem}
	In \cite{david1993analysis}, David and Semmes introduced the \textit{Weak Approximation of Lipschitz functions by Affine functions} (WALA) for subsets of Euclidean space. Roughly speaking, $X \subseteq \R^m$ satisfies the WALA condition if for any given 1-Lipschitz function $f \colon X \to \R^m$, and for most $(x,r) \in X \times \diam(X)$ there exists an affine function defined on $\R^n$ which well-approximates $f$ in $L^\infty(B_X(x,r))$. By most, we mean the complement of these pairs is a Carleson set. While it is tempting to think of the Carleson condition for the $\gamma$-coefficients as a metric version of the WALA ($f$ is well-approximated by affine functions at most scales and locations), this is not really the case. It is more accurately a metric version of BWGL + WALA. We need to include the BWGL condition so that the geometry of $X$ is well-represented in the tangent space. Roughly speaking, this allows us to view the affine function on the tangent space as being an affine function on $X$. The Carleson condition on the $\gamma$-coefficients is seemingly much stronger than the WALA. As we will see from Theorem \ref{t:equiv}, a Carleson condition on the $\gamma$-coefficients is equivalent to UR. In Euclidean space, it is still an open problem to determine whether or not WALA implies UR.
\end{rem}




\bigskip

\subsection{Preliminaries with the $\gamma$-coefficient} Let us start by proving some basic properties regarding the coefficients $\gamma_{X,f,\|\cdot\|_0}^K.$


\begin{lem}
	Let $X$ be a metric space, $x \in X$ and $0 < r < \diam(X)$. If $\|\cdot\|$ is a norm on $\R^n$ and $\vp \in \Phi(x,r,\|\cdot\|)$ then $\| \vp(x) \| \leq 10 \bilat_X(x,r,\|\cdot\|,\vp) r.$ 
\end{lem}

\begin{proof}
	The proof very similar to \eqref{e:near-centre}, we omit the details. 
\end{proof}

\begin{lem}
	Let $K \geq 1$, $X$ be a metric space, $\|\cdot\|_0$ a norm and $f:X \rightarrow (\R^n,\|\cdot\|_0)$ be $1$-Lipschitz. Then, 
	 
	\begin{align}\label{e:u1}
		\bilat_X(x,r) \leq \gamma_{X,f,\|\cdot\|_0}^K(x,r) \lesssim 1
	\end{align}
	for all $x \in X$ and $0 < r< \diam(X).$ Furthermore, if $y \in X$ and $ 0 < s < \diam(X)$ are such that $B(x,r) \subseteq B(y,s),$ then 
	\begin{align}\label{e:u1'}
		\bilat_X(x,r) \lesssim \frac{s}{r} \bilat_X(y,s)
	\end{align}
	and 
	\begin{align}\label{e:u2}
		\gamma_{X,f,\|\cdot\|_0}^K(x,r) \lesssim_K \frac{s}{r} \gamma_{X,f,\|\cdot\|_0}^K(y,s). 
	\end{align}
\end{lem}

\begin{proof}
	Let $x \in X$ and $0 < r < \diam(X).$ We start with \eqref{e:u1}. The first inequality is immediate from Lemma \ref{l:WALAM-BWGL}. For the second inequality, let $|\cdot|$ denote the Euclidean norm on $\R^n,$ let $\vp \in \Phi(x,r,|\cdot|)$ such that $\vp(y) = 0$ for all $z \in B(x,r)$ and let $A$ be a translation of $\R^n$ such that $A(0)= f(x).$ Notice that $A$ is an affine map and 1-Lipschitz from $\R^n$ to itself. With these choices, it is easy to see that $\unilat_X(x,r,|\cdot|,\vp) \leq 2$ and $\eta_X(x,r,|\cdot|,\vp) \leq 1.$ Since $f$ is 1-Lipschitz, $f(z) \in B_{\|\cdot\|}(A(0),r)$ for all $z \in B(x,r).$ Thus, $\Omega_{X,f,\|\cdot\|}(x,r,|\cdot|,\vp) \leq 1$ and \eqref{e:u1} follows. 
	
	Now let $y \in X$ and $0 < s < \diam(X)$ such that $B(x,r) \subseteq B(y,s).$ We prove \eqref{e:u1'}. Let $\ve > 0$, $\| \cdot \| \in \mathcal{N}$ and $\phi \in \Phi(y,s,\|\cdot\|)$ such that 
	\begin{align}\label{e:bilat-opt}
		\unilat_X(y,s,\|\cdot\|,\phi) + \eta_X(y,s,\|\cdot\|,\phi) \leq \bilat_X(y,s) + \ve. 
	\end{align}
	Recall the properties of $p_{z,\lambda}$ from Lemma \ref{l:p-lambda} and let $\tilde{\phi} = p_{\phi(x),r} \circ \phi.$ Then $\tilde{\phi}$ maps $B(x,r)$ to $B_{\|\cdot\|}(\phi(x),r)$ and if $z \in B(x,r)$ then 
	\begin{align}
		\begin{split}\label{e:near-phi}
		\| \tilde{\phi}(z) - \phi(z) \| &\leq \max\{0,\|\phi(z) - \phi(x)\| - r\} \\
		&\leq \max\{0,d(z,x) + s\unilat_X(y,s,\|\cdot\|,\phi) - r \} \\
		&\leq s \unilat_X(y,s,\|\cdot\|,\phi). 
		\end{split}
	\end{align}	
	Let us estimate $\bilat_X(x,r),$ starting with $\unilat_X(x,r,\|\cdot\|,\tilde{\phi}).$ By the triangle inequality, \eqref{e:near-phi} and the definition of $\unilat_X$, we have 
	\begin{align}
		\begin{split}\label{e:beta-est}
		r\unilat_X(x,r,\|\cdot\|_,\tilde{\phi}) &= \sup_{w,z \in B(x,r)} 	| d(w,z) - \| \tilde{\phi}(w) - \tilde{\phi}(z) \| \,| \\
		&\leq 	\sup_{w,z \in B(x,r)} | d(w,z) - \| \phi(w) - \phi(z) \| | + 2s \unilat_X(y,s,\|\cdot\|,\phi) \\
		&\leq 3s \unilat_X(y,s,\|\cdot\|,\phi).
	\end{split}
	\end{align}
	To estimate $\eta_X(x,r,\|\cdot\|,\tilde{\phi}),$ let us fix some $u \in B_{\|\cdot\|}(\phi(x),r).$ Then, let $\tilde{r} = r - s \bilat_X(y,s,\|\cdot\|,\phi) $ and set  \[\tilde{u} = p_{\phi(x),\tilde{r}}(u) \in B_{\|\cdot\|}(\phi(x),\tilde{r}).\]
	Observe that
	\begin{align}
		\| \tilde{u} \| \leq \| \tilde{u} - \phi(x) \| + \|\phi(x)\| \leq \tilde{r} + d(x,y) + s \unilat_X(y,s,\|\cdot\|,\phi) \leq r + (s-r) = s, 
	\end{align}
	so that $\tilde{u} \in B_{\|\cdot\|}(0,s).$ Then, by definition, there exists a point $z \in B(y,s)$ such that 
	\[\| \tilde{u} - \phi(z) \|\leq s \eta_X(y,s,\|\cdot\|,\phi).\]
	In fact we have $z \in B(x,r)$ since
	\begin{align}
		d(x,z) &\leq \| \phi(x) - \phi(z) \| + s \unilat_X(y,s,\|\cdot\|,\phi) \\
		&\leq \| \phi(x) - \tilde{u}\| + \|\tilde{u} - \phi(z) \| + s \unilat_X(y,s,\|\cdot\|,\phi) \\
		&\leq \tilde{r} + s \eta_X(y,s,\|\cdot\|,\phi) + s \unilat_X(y,s,\|\cdot\|,\phi) = r. 
	\end{align}
	From this inequality we also get $\phi(y) \in B(\phi(x),r)$, hence, $\tilde{\phi}(y) = p_{\phi(x),r}(\phi(y)) = \phi(y).$ This then gives 
	\begin{align}
		\| u - \tilde{\phi}(y) \| &\leq \| u - \tilde{u} \| + \| \tilde{u} - \phi(y) \| \leq s\bilat_X(y,s,\|\cdot\|,\phi) + s \eta_X(y,s,\|\cdot\|,\phi) \\
		&\leq 2s\bilat_X(y,s,\|\cdot\|,\phi). 
	\end{align}
	Since $u \in B_{\|\cdot\|}(\phi(x),r)$ is arbitrary, this implies 
	\begin{align}\label{e:eta-est}
		r \eta_X(x,r,\|\cdot\|,\tilde{\phi}) \leq  2s \bilat_X(y,s,\|\cdot\|,\phi).
	\end{align}
	Since $\ve$ was arbitrary, combining \eqref{e:beta-est} and \eqref{e:eta-est} finishes the proof of \eqref{e:u1'}. 
	
	We now prove \eqref{e:u2}. Let $\tau > 0$, $\| \cdot \| \in \mathcal{N}$ and $\phi \in \Phi(y,s,\|\cdot\|)$ such that 
	\begin{align}
		\begin{split}\label{e:gamma-opt'}
		\unilat_X(y,s,\|\cdot\|,\phi) &+ \eta_X(y,s,\|\cdot\|,\phi) + \Omega_{X,f,\|\cdot\|_0}^K(y,s,\|\cdot\|,\phi)\\
		&\leq \gamma_{X,f,\|\cdot\|_0}^K(y,s) + \ve. 
		\end{split}
	\end{align}
	Define $\tilde{\phi}$ as below \eqref{e:bilat-opt} so that the estimates \eqref{e:near-phi}, \eqref{e:beta-est} and \eqref{e:eta-est} still hold. It only remains to estimate $\Omega_{X,f,\|\cdot\|_0}^K(x,r,\|\cdot\|,\tilde{\phi}).$ For this, let $\delta > 0$ and $A \colon (\R^n,\|\cdot\|) \to (\R^n,\|\cdot\|_0)$ be affine and $K$-Lipschitz such that 
	\begin{align}\label{e:omega-opt}
		\frac{1}{s} \sup_{z \in B(y,s)} \| f(z) - A(\phi(z)) \|_0 \leq \Omega_{X,f,\|\cdot\|_0}^K(y,s,\|\cdot\|,\phi) + \delta. 
	\end{align}
	Using \eqref{e:near-phi}, \eqref{e:omega-opt} and the fact that $A$ is $K$-Lipschitz, we have 
	\begin{align}
		\sup_{z \in B(x,r)} \| f(z)  - A(\tilde{\phi}(z)) \|_0 &\leq \sup_{z \in B(x,r)} \| f(z)  - A(\phi(z)) \|_0 + \| A(\phi(z)) - A(\tilde{\phi}(z)) \|_0 \\
		&\leq s[\Omega_{X,f,\|\cdot\|_0}^K(y,s,\|\cdot\|,\phi) + \delta] + Ks\unilat_X(y,s,\|\cdot\|,\phi).
	\end{align}
	Since $\delta$ is arbitrary, this implies 
	\begin{align}\label{e:omega-est}
		r \Omega_{X,f,\|\cdot\|_0}^K(x,r,\|\cdot\|,\tilde{\phi}) \leq s\Omega_{X,f,\|\cdot\|_0}^K(y,s,\|\cdot\|,\phi) + Ks\unilat_X(y,s,\|\cdot\|,\phi).
	\end{align}
	Then, since $\tau$ was arbitrary, this, \eqref{e:gamma-opt'}, \eqref{e:beta-est} and \eqref{e:eta-est} imply \eqref{e:u2}. 
\end{proof}

\begin{defn}
	Let $(Z,d)$ be a metric space and suppose $X,Y \subseteq Z$. For $x \in X$ and $0 < r < \diam(X)$ define 
	\[I_{X,Y}(x,r) = \frac{1}{r} \sup_{\substack{y \in E\cap B(x,r)\\ \dist(y,Y) \leq r}} \dist(x,Y),\]
	where we set the supremum over the empty set to be zero. 
\end{defn}

\begin{lem}\label{l:beta-error}
	Let $(Z,d)$ be a metric space and suppose $X_1,X_2 \subseteq Z$ are such that $X_1 \cap X_2 \neq \emptyset.$ Let $x \in X_1 \cap X_2$, $0 < r < \min\{\diam(X_1),\diam(X_2)\}$ and suppose $f : X \rightarrow (\R^n,\|\cdot\|_0)$ is 1-Lipschitz. Write $f_1$ and $f_2$ for the restrictions of $f$ to $X_1$ and $X_2$, respectively. Then 
	\[ \gamma_{X_1,f_1,\|\cdot\|_0}^K(x,r) \lesssim_K \gamma_{X_2,f_2,\|\cdot\|_0}^K(x,3r) + I_{X_1,X_2}(x,r) + I_{X_2,X_1}(x,3r). \]
\end{lem}

\begin{proof}
	Let $0 < \ve < 1$ and let $\|\cdot\|$ be a seminorm on $\R^n$ and $\vp_2 \in \Phi(x,3r,\|\cdot\|)$ such that
	\begin{align}
		\begin{split}\label{e:infrho}
		\Omega_{X_2,f_2,\|\cdot\|_0}^K(x,3r,\|\cdot\|,\vp_2) &+ \unilat_{X_2}(x,3r,\|\cdot\|,\vp_2) + \eta_{X_2}(x,3r,\|\cdot\|,\vp_2) \\
		&\leq \gamma_{X_2,f_2,\|\cdot\|_0}^K(x,3r) + \ve.
		\end{split}
	\end{align}
	For brevity, in the remainder of the proof we will write $\Omega_2 \coloneqq \Omega_{X_2,f_2,\|\cdot\|_0}^K(x,3r,\|\cdot\|,\vp_2)$, $\unilat_{2} \coloneqq \unilat_{X_2}(x,3r,\|\cdot\|,\vp_2)$ and $\eta_{2} \coloneqq \eta_{X_2}(x,3r,\|\cdot\|,\vp)$. Let us observe for later that 
	\begin{align}\label{e:near-centre-GHA} 
		\| \vp_2(x) \| \leq 30r (\unilat_2 + \eta_2)
	\end{align}
	by \eqref{e:near-centre}. We start by constructing a map $\vp_1 : X_1 \cap B(x,r) \to B_{\|\cdot\|}(0,r)$ as follows. For each $p \in X_1 \cap B(x,r)$, let $\tilde{p} \in X_2$ be a point in $X_2$ such that 
	\begin{align}\label{e:p-tildep}
		d(p,\tilde{p}) \leq \dist(p,X_2) + \ve r \leq r I_{X_1,X_2}(x,r) + \ve r \leq 2r,
	\end{align}
	so that $\tilde{p} \in X_2 \cap B(x,3r).$ Using \eqref{e:near-centre-GHA} and \eqref{e:p-tildep} with the fact that $p \in X_1 \cap B(x,r)$ we have 
	\begin{align}\label{e:normA'}
		\|\vp_2(\tilde{p})\| &= \|\vp(x) - \vp(\tilde{p})\| +Cr(\unilat_2 + \eta_2) \leq d(x,\tilde{p}) + Cr( \unilat_{2} +\eta_2) \\
		&\leq d(x,p) + d(p,\tilde{p}) + Cr( \unilat_{2} +\eta_2) \\
		&\leq r + rI_{X_1,X_2}(x,r)  + Cr (\unilat_{2} +\eta_2) + \ve r.
	\end{align}
	By Lemma \ref{l:p-lambda} we can find a point $\vp_1(p) \in B_{\|\cdot\|}(0,r)$ such that 
	\begin{align}\label{e:A-A'}
		\|\vp_1(p) - \vp_2(\tilde{p})\| \leq rI_{X_1,X_2}(x,r)  + Cr( \unilat_{2} +\eta_2) + \ve r.
	\end{align}

	Let us estimate $\unilat_{X_1}(x,r,\|\cdot\|,\vp_1).$ Let $y,z \in X_1 \cap B(x,r).$ Since $\tilde{y},\tilde{z} \in X_2 \cap B(x,3r)$ we know by the definition of $\unilat_2$ that $ \| \vp_2(\tilde{y}) - \vp_2(\tilde{z})\|  - d(\tilde{y},\tilde{z}) \lesssim \unilat_2.$ Using this with \eqref{e:p-tildep} and \eqref{e:A-A'}, we have
	\begin{align}
		\| \vp_1(y) - \vp_1(z) \| - d(y,z) &\leq \| \vp_1(y) - \vp_2(\tilde{y}) \| + \| \vp_2(\tilde{y}) - \vp_2(\tilde{z})\| \\
		&\hspace{2em} +\|\vp_2(\tilde{z}) - \vp_1(z) \| - d(\tilde{y},\tilde{z}) + d(y,\tilde{y}) + d(z,\tilde{z}) \\
		&\lesssim r I_{X_1,X_2}(x,r) + r\unilat_{2} +r\eta_2 + \ve r.
	\end{align}
	Similarly, 
	\begin{align}
		d(y,z) - \| \vp_1(y) - \vp_1(z) \| &\leq r I_{X_1,X_2}(x,r) +  r\unilat_{2} +r\eta_2 + \ve r. 
	\end{align}
	Combining the above two inequality, we have 
	\begin{align}\label{e:beta-bound}
		\unilat_{X_1}(x,r,\|\cdot\|,\vp_1)\lesssim \unilat_{2} +\eta_2 + I_{X_1,X_2}(x,r) + \ve.
	\end{align}

	Now we estimate $\eta_{X_1}(x,r,\|\cdot\|,\vp_1).$ Let $u \in B_{\|\cdot\|}(0,r)$ and let $\tilde{u} = \lambda u$ where $0 < \lambda \leq 1$ is maximal such that 
	\begin{align}\label{e:tildeu}
		\|\tilde{u}\| \leq r - r(33\unilat_{2} + 33\eta_{2} + 3I_{X_2,X_1}(x,3r)) -\ve r.
	\end{align}
	This choice of $\tilde{u}$ will become more clear shortly, observe for the moment that 
	\begin{align}\label{e:u-tildeu}
		\|u - \tilde{u}\| \leq r(33\unilat_{2} + 33\eta_{2} + 3I_{X_2,X_1}(x,3r)) +\ve r.
	\end{align}
	By definition of $\eta_2$, there exists $q \in X_2 \cap B(x,3r)$ such that $\| \tilde{u} - \vp_2(q) \| \leq 3r \eta_{2}.$ Let $p \in X_1$ be a point in $X_1$ such that $d(q,p) \leq \dist(q,X_1) + \ve r$ and let $\tilde{p}$ be as in \eqref{e:p-tildep}. Observe that
	\begin{align}
		d(p,x) &\leq d(p,q) + d(q,x) \leq \dist(q,X_1) + \ve r + \| \vp_2(q) - \vp_2(x) \| + 3r\unilat_{2} \\
		&\leq 3r I_{X_2,X_1}(x,3r) + \ve r + \| \tilde{u} \| + \| \tilde{u} - \vp_2(q) \| + \|\vp_2(x)\| + 3r\unilat_{2} \\
		&\leq 3r I_{X_2,X_1}(x,3r) + \ve r + \| \tilde{u} \| + 33r \eta_{2} + 33r\unilat_{2} \leq r,
	\end{align}
	so that $p \in X_1 \cap B(x,r).$ This is why we chose $\tilde{u}$ as in \eqref{e:tildeu}. Notice also, since $\dist(p,X_2) \leq d(p,q),$ we have
	\begin{align}
		d(q,\tilde{p}) &\leq d(q,p) + d(p,\tilde{p}) \leq \dist(q,p) + \dist(p.X_2) + \ve r \leq 2d(q,p) + \ve r\\
		&\leq 2\dist(q,X_1) + 3 \ve r \leq 6r I_{X_2,X_1}(x,3r) + 3\ve r. 
	\end{align}
	Using this estimate with the definition of $\unilat_2$, the definition of $q$ after \eqref{e:u-tildeu} and \eqref{e:A-A'}, it follows that 
	\begin{align}
		\|\tilde{u} - \vp_1(p)\| &\leq \| \tilde{u} - \vp_2(q) \| + \|\vp_2(q) - \vp_2(\tilde{p}) \| + \|\vp_2(\tilde{p}) - \vp_1(p) \| \\
		&\lesssim r( \unilat_{2} + \eta_{2} + I_{X_2,X_1}(x,3r) + \ve ).
	\end{align}
	With \eqref{e:u-tildeu}, this implies
	\begin{align}\label{e:eta-bound}
		\eta_{X_1}(x,r,\|\cdot\|,\vp_1) \lesssim \unilat_{2} + \eta_{2} + I_{X_2,X_1}(x,3r) + \ve.
	\end{align}
	
	Finally, we can estimate $\Omega_{X_1,f_1,\|\cdot\|_0}^K(x,r,\|\cdot\|,\vp_1).$ Let $A : \R^n \rightarrow \R^n$ be an affine function such that $\|A(x) - A(y)\|_0 \leq K\|x-y\|$ for all $x,y \in \R^n$ and 
	\[ \frac{1}{3r} \sup_{y \in X_2 \cap B(x,3r)} \|f_2(x) - A(\vp_2(x))\|_0 \leq \Omega_2 + \ve.\]
	Let $y \in X_1 \cap B(x,r)$ and let $\tilde{y} \in X_2 \cap B(x,3r)$ be as in \eqref{e:p-tildep}. By \eqref{e:p-tildep}, \eqref{e:A-A'} and recalling that $f_1,f_2$ are the restrictions of $f$ to $X_1,X_2$, respectively, we have 
	\begin{align}
		\|f_1(y) - A(\vp_1(y))\|_0 &\leq \|f(y) - f(\tilde{y})\|_0 + \|f_2(\tilde{y}) - A(\vp_2(\tilde{y}))\|_0 \\
		&\hspace{2em} + \|A(\vp_2(\tilde{y})) - A(\vp_1(y))\|_0 \\
		&\leq d(y,\tilde{y}) + 3r(\Omega_2 + \ve) + K\|\vp_2(\tilde{y}) - \vp_1(y)\| \\
		&\lesssim_K  r(\unilat_2 + \Omega_2 + I_{X_1,X_2}(x,r) + \ve).
	\end{align}
	Thus, 
	\begin{align}\label{e:omega-bound}
		\Omega_{f_1,X_1}^K(x,r,\vp_1) \lesssim_K \unilat_2 + \Omega_2 + I_{X_1,X_2}(x,r) + \ve. 
	\end{align}
	Combining this with \eqref{e:infrho}, \eqref{e:beta-bound}, \eqref{e:eta-bound} and taking $\ve \to 0$ proves the lemma. 
\end{proof}

\begin{rem}
	It is easy to show that $r I_{X,Y}(x,r) \leq s I_{X,Y}(y,s)$ any $x,y \in X$ and $0 < r \leq s < \diam(X)$ such that $B(x,r) \subseteq B(y,s).$ In particular, the above estimate may be written as 
	\begin{align}\label{e:modified-error} \hspace{1em} \gamma_{X_1,f_1,\|\cdot\|_0}^K(x,r) \lesssim_K \gamma_{X_2,f_2,\|\cdot\|_0}^K(x,3r) + I_{X_1,X_2}(x,3r) + I_{X_2,X_1}(x,3r). 
	\end{align}
	Although slightly weaker, this estimate will ease the notation in Section \ref{s:affine-approx}.
\end{rem}

\newpage

\section{UR implies BWGL and affine approximations}\label{s:affine-approx}
\etocsettocstyle{Contents of this section}{}
\etocsettocmargins{.01\linewidth}{.01\linewidth}

\localtableofcontents

\bigskip

\subsection{Main result} Our goal in this section is to prove the following.

\begin{thm}\label{t:approx-UR}
	Let $X$ be an Ahlfors $n$-regular metric space. If $X$ is UR then $X$ satisfies a Carleson condition for the $\gamma$-coefficients. 
\end{thm}

The proof of Theorem \ref{t:approx-UR} goes roughly as follows. Suppose $\mathscr{Y}$ is a collection of Ahlfors $n$-regular sets with uniform regularity constant and suppose each $Y \in \mathscr{Y}$ satisfies the Carleson condition for the $\gamma$-coefficients with a uniform constant $K$ and uniform Carleson norms. We first show that \textit{if} $X$ has big pieces of elements of $\mathscr{Y}$ (in the sense of Definition \ref{d:UR-metric}) \textit{then} $X$ itself satisfies a Carleson condition for the $\gamma$-coefficients. That is, the Carleson condition for the $\gamma$-coefficients is \textit{stable} under the big pieces functor. This result is inspired by the stability result for BWGL first shown for subsets of Euclidean space and the Heisenberg group by Rigot \cite{rigot2019quantitative}. The results in \cite{rigot2019quantitative} were inspired by the stability result for the WGL (a one-sided version of the BWGL) in Euclidean spaces in \cite[Part IV]{david1993analysis}. 

If $X$ is uniformly $n$-rectifiable then it has big pieces of spaces which are bi-Lipschitz equivalent to $\R^n$ (after embedding in a suitable Banach space). The second part of the proof of Theorem \ref{t:approx-UR} is showing that these spaces satisfy a Carleson condition for the $\gamma$-coefficients with uniform $K$ and Carleson norms.

\bigskip

\subsection{Stability of Carleson condition for $\gamma$ under big pieces}\label{s:stability} The goal now is to prove the following proposition, which states that the Carleson condition for the $\gamma$-coefficients is stable under big pieces. 

\begin{prop}\label{l:stability}
	Let $Z$ be a metric space, $C_1,K \geq 1$, and let 
	\begin{align}
		\calF = \{Y \subseteq Z : \ &Y \mbox{ is Ahlfors $n$-regular with regularity constant $C_1$ and satisfies a } \\
		&\hspace{2em}\mbox{ Carelson  condition for the $\gamma$-coefficients with constant} \\
		&\hspace{2em} \mbox{ $K$ and Carleson norms $\{C(\ve)\}_{\ve > 0}$}\}.
	\end{align} 
If $X \subseteq Z$ is Ahlfors $n$-regular with regularity constant $C_0$ and has big pieces of $\calF$ then $X$ satisfies a Carleson condition for the $\gamma$-coefficients with constant $K$. The family of Carleson norms depends only on $C_0,C_1,\{C(\ve)\}, n$ and the big pieces constant $\theta.$ 
\end{prop}

To prove Proposition \ref{l:stability} it will be useful to work with Christ-David cubes (Lemma \ref{cubes}). Below we formulate the BWGL and the Carleson condition for the $\gamma$-coefficients in terms of Christ-David cubes.

\begin{defn}\label{d:notation-cubes}
	Suppose $(X,d)$ is a doubling metric space and $\mathcal{D}$ is a system of Christ-David cubes on $X$. For $A \geq 1,$ a cube $Q \in \mathcal{D}$, a norm $\|\cdot\|$ on $\R^n$ and a map $\vp \in \Phi(x_Q,A\ell(Q),\|\cdot\|)$, define
	\begin{align}
		\unilat_X(AQ,\|\cdot\|,\vp) &\coloneqq \unilat_X(x_Q,A\ell(Q),\|\cdot\|,\vp); \\
		\eta_X(AQ,\|\cdot\|,\vp) &\coloneqq \eta_X(x_Q,A\ell(Q),\|\cdot\|,\vp); \\
		\bilat_X(AQ) &\coloneqq \bilat_X(x_Q,A\ell(Q)). 
	\end{align}
	If $\|\cdot\|_0$ is another norm and $f : X \to (\R^n,\|\cdot\|_0)$ is 1-Lipschitz, we also define 
	\begin{align}
		\Omega_{X,f,\|\cdot\|_0}^K(AQ,\|\cdot\|,\vp)&\coloneqq\Omega_{X,f,\|\cdot\|_0}^K(x_Q,A\ell(Q),\|\cdot\|,\vp); \\
		\gamma_{X,f,\|\cdot\|_0}^K(AQ) &\coloneqq \gamma_{X,f,\|\cdot\|_0}^K(x_Q,A\ell(Q)).		
	\end{align}
	If there are further metric spaces $Y,Z$ such that $X,Y \subseteq Z,$ define 
	\begin{align}
		I_{X,Y}(AQ) = I_{X,Y}(x_Q,A\ell(Q)). 
	\end{align}
\end{defn}

As an immediate consequence of Lemma \ref{l:set-packing}, \eqref{e:u2} and \eqref{e:u1'} we obtain the following discretized formulations of the BWGL and the Carleson condition for the $\gamma$-coefficients. The definition of Carleson packing condition be found in Definition \ref{d:packing}.

\begin{lem}\label{l:cubes-BWGL}
	Let $A > 1$, $X$ be an Ahlfors $n$-regular metric space and $\calD$ a system of Christ-David cubes on $X$. Then, $X$ satisfies the \emph{BWGL} if and only if for $\ve > 0$ the set $\{Q \in \calD : \bilat_{X}(AQ) > \ve\}$ satisfies a Carleson packing condition.
\end{lem}

\begin{lem}\label{l:cubes-wala}
	Let $A > 1,$ $X$ be an Ahlfors $n$-regular metric space and $\calD$ a system of Christ-David cubes on $X$. Then, $X$ satisfies a Carleson condition for the $\gamma$-coefficients if and only if there exists $K \geq 1$ and for each $\ve > 0$ a constant $C = C(\ve)$ such that if $\|\cdot\|_0$ is a norm and $f : X \to (\R^n,\|\cdot\|_0)$ is 1-Lipschitz then $\{Q \in \calD : \gamma_{X,f,\|\cdot\|_0}^K(AQ) > \ve\}$ satisfies a Carleson packing condition with Carleson norm $C$.	
\end{lem}

It will be via Lemma \ref{l:cubes-wala} that we prove Proposition \ref{l:stability}. We need to collect some more results before beginning the proof. 

\begin{lem}\label{l:BO}
	Let $(X,d)$ be an Ahlfors $(C_0,n)$-regular metric space and $s > 0$. Let $a > 0$, $A \geq 1$ and suppose $\mathcal{N} \subseteq X$ is an $as$-separated net with $\diam(\mathcal{N}) \leq As$. Then, $|\mathcal{N}| \lesssim_{C_0,n,A,a} 1.$
\end{lem}

\begin{proof}
	Write $\mathcal{N} = \{z_i\}_{i \in I}.$ For each $i \in \mathcal{N},$ let $B_i = B(z_i, as/4).$ Pick some $i_0 \in I$ and let $B = B(z_{i_0},2As).$ Then, $B_i \subseteq B$ for each $i \in I$ and the result now follows by Ahlfors regularity since  
	\begin{align}
		| \mathcal{N} |  C_0^{-1} (as)^n \leq \sum_{i \in I} \calH^n(B_i) \leq \calH^n(B) \leq C_0 (2As)^n. 
	\end{align}
\end{proof}

\begin{lem}\label{l:Carleson-I}
	Suppose $(Z,d)$ is a metric space and $X,Y \subseteq Z$ are Ahlfors $n$-regular. Let $C_0$ be the regularity constant for $X$ and  $\mathcal{D}$ a system of Christ-David cubes for $X.$ Then, $\{Q \in \calD : I_{X,Y}(3Q) > \ve\}$ satisfies a Carleson packing condition for each $\ve > 0$ with Carleson norm depending only on $C_0,n$ and $\ve.$ 
\end{lem}

\begin{proof}
	Let $\ve > 0$ and $Q_0 \in \mathcal{D}$. Without loss of generality we may suppose $\ve < 1.$ For each $Q \subseteq Q_0$ such that $I_{X,Y}(3Q) > \ve$ there exists a point $y_Q \in X \cap 3B_Q$ such 
	\begin{align}\label{e:y_Q}
		3\ve \ell(Q) < \dist(y_Q,Y) \leq 3\ell(Q).
	\end{align}
	Choose $k = k(Q)$ to be the smallest positive integer such that $5\varrho^k \ve^{-1}\leq \ell(Q)$ and let $T(Q) \in \mathcal{D}_{k}$ such that $y_Q \in T(Q)$. In particular, $3B_Q \cap T(Q) \neq \emptyset.$ Let 
	\[C_* = 1+  \tfrac{3}{\ve\varrho}.\]
	By maximality and using the fact that $\ve < 1,$  we have 
	\begin{align}\label{e:ell(T)} 
		\ell(T(Q)) \leq \frac{\ell(T(Q))}{\ve} \leq \ell(Q) < \frac{\ell(T(Q))}{\ve \varrho} \leq C_* \ell(T(Q)).
	\end{align}
	Since $3B_Q \cap T(Q) \neq \emptyset,$ the first two inequalities imply 
	\begin{align}\label{e:Tin}
		T(Q) \subseteq 6B_Q. 
	\end{align}
	Furthermore, using that $y_Q \in T(Q)$ with \eqref{e:y_Q} and the second and third inequalities in \eqref{e:ell(T)}, we have 
	\begin{align}
		\begin{split}\label{e:good}
		2\ell(T(Q)) &\leq \dist(y_Q,Y) - \ell(T(Q)) \leq \dist(x_{T(Q)},Y) \\
		&\leq \dist(y_Q,Y) + \ell(T(Q)) \leq C_* \ell(T(Q)).
		\end{split}
	\end{align}
	We claim 
	\begin{align}\label{e:T(Q)}
		|\{Q \in \mathcal{D} : T(Q) = T\}| \lesssim_{C_0,n} 1. 
	\end{align}
	Indeed, suppose $T \in \calD$. By \eqref{e:ell(T)}, if $Q,Q' \in \calD$ are distinct cubes such that $T(Q) = T(Q') = T$, then $\ell(Q) = \ell(Q') = 5\varrho^{k_0}$ for some $k_0 \in \Z.$ Recall from Lemma \ref{cubes} that the centred of the cubes in $\calD_k$ are chosen from a $\varrho^k$-separated net, for such cubes we have $d(x_Q,x_{Q'}) \geq \varrho^{k_0}$. Furthermore, since $3B_Q \cap T \neq \emptyset$ and $3B_{Q'} \cap T \neq \emptyset$, the triangle inequality and the second inequality in \eqref{e:ell(T)} implies $d(x_Q,x_{Q'}) \lesssim \varrho^{k_0}$. Applying Lemma \ref{l:BO} now gives \eqref{e:T(Q)}. 
	
	Let $\calG$ be the set of $T \in \mathcal{D}$ such that $T \subseteq 6B_{Q_0}$ and $2\ell(T) \leq \dist(x_{T},Y) \leq C_* \ell(T)$. If $Q \subseteq Q_0$ then $6B_Q \subseteq 6B_{Q_0}$ by Remark \ref{r:contained-cubes}. Thus, it follows from \eqref{e:Tin} and \eqref{e:good} that $T(Q) \in \calG$ for each $Q$ such that $I_{X,Y}(3Q) > \ve.$  This, the third inequality in \eqref{e:ell(T)} and \eqref{e:T(Q)} imply
	\begin{align}\label{e:suffices}
		\sum \{ \ell(Q)^n : Q \subseteq Q_0, \ I_{X,Y}(3Q) > \ve\} &\lesssim_{\ve} \sum \{ \ell(T(Q))^n : Q \subseteq Q_0, \ I_{X,Y}(3Q) > \ve\} \\
		&\lesssim_{C_0,n} \sum \{ \ell(T)^n : T \in \calG, \ T \subseteq 6B_{Q_0}\}.
	\end{align}
	We observe now that $\calG$ has bounded overlap. Indeed, let $p \in X$ and let $\calG(p) = \{ T \in \calG : p \in T\}$. Since $\dist(\cdot,Y)$ is 1-Lipschitz, if $T \in \calG(p)$ then 
	\begin{align}
		\ell(T) \leq \dist(x_T,Y) - d(x_T,p) \leq \dist(p,Y) \leq  \dist(x_T,Y) + d(x_T,p) \leq 2C_* \ell(T)
	\end{align} 
	This, and the dependency of $C_*$ on $\ve$, imply $| \{k \in \Z : \calG(p) \cap \calD_k \neq \emptyset\} | \lesssim_{\ve} 1.$ Since $\calD_k$ is disjoint for each $k \in \Z$ we have $|\calG(p) \cap \calD_k| \leq 1$. Thus, $|\calG(p)| \lesssim_{\ve} 1$ as required. 
	
	We use bounded overlap and Ahlfors regularity to conclude the lemma since 
	\begin{align}
		\sum \{ \ell(T)^n : T \in \calG, \ T \subseteq 6B_{Q_0}\}  &\lesssim_{C_0}  \sum \{ \calH^n(T) : T \in \calG, \ T \subseteq 6B_{Q_0}\} \\
		&\lesssim_{\ve} \calH^n(X \cap 6B_{Q_0}) \lesssim_{C_0} \ell(Q_0)^n.
	\end{align}
\end{proof}

By Lemma \ref{l:set-packing}, we obtain the following corollary. 

\begin{cor}\label{c:distance}
	Let $Z$ be a metric space and suppose $X,Y\subseteq Z$ are Ahlfors $n$-regular. Let $C_0$ be the regularity constant for $X$. For all $\ve> 0$ the set $\{ (x,r) \in X \cap \R^+ : I_{X,Y}(x,r) > \ve\}$ is a Carleson set with Carleson norm depending only on $C_0,n$ and $\ve.$
\end{cor}

To conclude the proof of Proposition \ref{l:stability} we require the following.

\begin{lem}[{\cite[Lemma IV.1.12]{david1993analysis}}]\label{l:David-Semmes}
	Let $X$ be an Ahlfors $n$-regular metric space and $\mathcal{D}$ a system of dyadic cubes for $X$. Let $\alpha : \mathcal{D} \to [0,\infty)$ be given and suppose there are $N > 0$ and $\eta > 0$ such that 
	\[ \mathcal{H}^n \left( \left\{x \in R : \sum_{\substack{Q \subseteq R \\ x \in Q  }} \alpha(Q) \leq N \right\} \right) \geq \eta \ell(R)^n.\] 
	for all $R \in \mathcal{D}.$ Then, 
	\[ \sum_{Q \subseteq R} \alpha(Q) \ell(Q)^n \lesssim_{N, \eta} \ell(R)^n\]
	for all $R \in \mathcal{D}.$ 
\end{lem}

\begin{proof}[Proof of Proposition \ref{l:stability}]
	Let $\mathcal{D}$ be a system of Christ-David cubes for $X$ and $\ve > 0$. Suppose $\|\cdot\|$ is a norm and $f\colon X \to (\R^n,\|\cdot\|)$ is a 1-Lipschitz mapping. Set $\calB = \{Q \in \mathcal{D} : \gamma_{X,f,\|\cdot\|}^K(3Q) > \ve\}.$ By Lemma \ref{l:cubes-wala} and Lemma \ref{l:David-Semmes}, it suffices to find constants $N,\eta > 0$ depending only on $C_0,C_1,C(\ve),n$ and $\theta$ such that
	\begin{align}\label{e:large-good}
		\calH^n\left(\left\{x \in R : \sum_{\substack{Q \subseteq R \\ x \in Q}} \chi_{\calB}(Q) \leq N	\right\}\right) \geq \eta \ell(R)^n
	\end{align}
	for all $R \in \calD.$ 
	
	Fix a cube $R \in \calD$ and recall the constant $c_0$ from Lemma \ref{cubes}. By assumption we can find an Ahlfors $n$-regular metric space $Y = Y_R\subseteq Z$ with regularity constant $C_1$ satisfying a Carleson condition for the $\gamma$-coefficients with constant $K$ and Carleson norms $\{C(\ve)\}_{\ve>0}$ and such that 
	\begin{align}\label{e:E_R large}
		\calH^n(R \cap Y) \geq \calH^n(X \cap Y \cap c_0B_R) \geq \theta (c_0\ell(R))^n.
	\end{align}
	For $\delta > 0$ let 
	\begin{align}
		\mathscr{B}^1_\delta &= \{(x,t) \in X \times \R^+ : \gamma_{X,f,\|\cdot\|}^K(x,t) > \delta \}; \\ 
		\mathscr{B}^2_\delta&= \{(x,t) \in X \times \R^+ : I_{X,Y}(x,t) > \delta\}, \\
		\mathscr{B}_\delta^3 &= \{(x,t) \in Y \times \R^+ : \gamma_{Y,f,\|\cdot\|}^K(x,t) > \delta\}, \\
		\mathscr{B}_\delta^4 &= \{(x,t) \in Y \times \R^+ : I_{Y,X}(x,t) > \delta\}.
	\end{align}
	Let $C_2$ (resp. $C_3$) be the implicit constant in \eqref{e:u2} (resp. \eqref{e:modified-error}). Then, let 
	\[\ve_1 = \frac{3\ve}{4C_2}, \  \ve_2 = \frac{\ve_1}{3C_3} \mbox{ and } \ve_3 = \frac{\ve_2}{2C_2}.\] 
	Suppose $x \in R \cap Y$ and $Q \in \calB$ is such that $x \in Q \subseteq R$. Observe first that $3B_Q \subseteq B(x,4\ell(Q))$, hence,  $(x,4\ell(Q)) \in \mathscr{B}_{\ve_1}^1$ by \eqref{e:u2}. Now, applying \eqref{e:modified-error} we find $(x,12\ell(Q)) \in \mathscr{B}^i_{\ve_2}$ for some $i \in \{2,3,4\}.$ Finally, if $(x,12\ell(Q)) \in \mathscr{B}_{\ve_2}^i$ and $12\ell(Q) \leq r < 24\ell(Q)$ then $(x,r) \in \mathscr{B}_{\ve_3}^i$ by \eqref{e:u2}. Combining these observations we get  
	\begin{align}
		\sum_{\substack{Q \subseteq R \\ x \in Q}} \chi_{\calB}(Q) &\leq \sum_{\substack{Q \subseteq R \\ x \in Q}} \chi_{\mathscr{B}_{\ve_1}^1}(x,4\ell(Q)) \leq  \sum_{\substack{Q \subseteq R \\ x \in Q}} \sum_{i=2}^{4} \chi_{\mathscr{B}_{\ve_2}^i}(x,12\ell(Q)) \\
		&\leq \sum_{\substack{Q \subseteq R \\ x \in Q}} \int_{12\ell(Q)}^{24\ell(Q)} \sum_{i=2}^{4} \chi_{\mathscr{B}_{\ve_3}^i}(x,r) \,  \frac{dr}{12\ell(Q)} \leq 2 \sum_{i=2}^4 \int_0^{24\ell(R)}  \chi_{\mathscr{B}_{\ve_3}^i}(x,r) \, \frac{dr}{r}
	\end{align}
	Fix some arbitrary $z \in R \cap Y.$ Then, $R \cap Y \subseteq B_Y(z,24\ell(Q)).$ Now, using the above estimate with Corollary \ref{c:distance} along with the fact that $Y$ satisfies a Carleson condition for the $\gamma$-coefficients, we have 
	\begin{align}\label{e:R_N}
		\int_{R \cap Y} \sum_{\substack{Q \subseteq R \\ x \in Q}} \chi_{\mathcal{B}}(Q) \, d\calH^n(x) &\leq 2  \sum_{i=2}^4  \int_{B_Y(z,24\ell(R))} \int_0^{24\ell(R)}\chi_{\mathscr{B}_{\ve_3}^i}(x,r) \, \frac{dr}{r}d\mathcal{H}^n(x) \\
		&\lesssim_{C_0,C_1,C(\ve_3),\ve,n} \ell(R)^n.
	\end{align}
	Thus, for $N \geq 1$ large enough depending on each of $C_0,C_1,C(\ve_3),\ve,n$ and $\theta$, we have 
	\[\calH^n\left(\left\{x \in R \cap Y : \sum_{\substack{Q \subseteq R \\ x \in Q}} \chi_{\mathcal{B}}(Q) > N\right\}\right) \leq \frac{\theta(c_0\ell(R))^n}{2}.\] 
	If we set $\eta = \theta(c_0\ell(R))^n/2$, then \eqref{e:large-good} now follows from the above inequality and \eqref{e:E_R large}. 
\end{proof}

\bigskip

\subsection{UR implies big pieces of bi-Lipschitz images of $\R^n$ in Banach spaces}

In this section we prove the following. 

\begin{lem}\label{l:big-pieces}
	Suppose $X$ is UR with constants $C_0,L$ and $\theta$. Then there exists $L'$, depending on $L,n$ and $\theta$, a Banach space $\calB$ and an isometric embedding $j \colon X \to \calB$ such that $j(X)$ has big pieces of 
	\begin{align}
		\calF = \{ g(\R^n) : g \colon \R^n \to \calB \mbox{ is } L'\mbox{-bi-Lipschitz}\}
	\end{align}
	in $\calB.$	The big pieces constant depends only on $\theta$. 
\end{lem}

For the proof of Lemma \ref{l:big-pieces} require the following result.

\begin{lem}\label{l:extension}
	Let $\calB$ be a Banach space, $A \subseteq \R^n$, $L \geq 1$ and suppose $f \colon A \to \calB$ is $L$-bi-Lipschitz onto its image. Then there exist constants $m$ and $L'$, depending only on $L$ and $n,$ such that the following holds. Let $\calB' = \calB \times \R^m$ and equip $\calB'$ with the metric $d_{\calB'} = \sqrt{d_\calB^2 + d_{\R^m}^2}.$ Then, identifying $\calB$ with a subset of $\calB'$, there exists an extension $g \colon \R^n \to \calB'$ of $f$ such that $g$ is $L'$-bi-Lipschitz onto its image. 
\end{lem}

This result is proved in the case $\calB  = \R^d$ in \cite[Proposition 17.4]{david1991singular} (note, the roles of $d$ and $n$ are reversed there). In this case we simply have $\calB' = \R^{d + m}$ with $d_{\calB'}$ the Euclidean metric. The essential properties of $\R^d$ (the codomain) used in the proof of \cite[Proposition 17.4]{david1991singular} are the triangle inequality and the fact that $L$-Lipschitz functions from subsets of $\R^n$ into $\R^d$ can be extended to $L'$-Lipschitz functions from $\R^n$ into $\R^d$ with $L'$ depending only on $L$ and $n$. The analogous extension property for Lipschitz functions taking values in Banach spaces was proven in \cite{johnson1986extensions}. Thus, the exact same proof gives Lemma \ref{l:extension}.

\begin{proof}[Proof of Lemma \ref{l:big-pieces}]
	Let $m$ be the constant from Lemma \ref{l:extension}, let $i \colon X \to \ell^\infty(X)$ be the Kuratowski embedding and set $\calB = \ell^\infty(X) \times \R^m$. From now on we will identify $\ell^\infty(X)$ with a subset of $\calB$ and $i$ with an isometric embedding of $X$ in $\calB$ in the obvious way.  
	
	Let $Y = i(X)$ and let us show that $Y$ has big pieces of $\calF$ in $\calB.$ Fix $x \in Y$ and $0 < r < \diam(Y)$. Notice that $Y$ is UR since $X$ is UR and $i$ is an isometry. In particular, by \cite[Corollary 1.2]{schul2009bi}, it has BPBI in the sense of Definition \ref{d:BPBI} with constants $L_*$ (depending on $L,n$ and $\theta$) and $\theta_*$ (depending only on $\theta$). Thus, there exists a set $A \subseteq \R^n$ and an $L_*$-bi-Lipschitz map $f \colon A \to Y \subseteq \ell^\infty(X)$ such that 
	\begin{align}\label{e:big'}
		\calH^n(f(A) \cap B_Y(x,r)) \geq \theta_* r^n. 
	\end{align}
	By Lemma \ref{l:extension} we can find an $L'$-bi-Lipschitz extension $g \colon \R^n \to \calB$ of $f$ with $L'$ depending only on $L_*$ and $n.$ It now follows that \eqref{e:big'} that 
	\begin{align}
		\calH^n(g(\R^n) \cap B_Y(x,r)) \geq \calH^n(f(A) \cap B_{Y}(x,r))  \geq \theta_* r^n
	\end{align}
	as required. 
\end{proof}

\bigskip

\subsection{A Carleson condition for $\gamma$ on bi-Lipschitz images of $\R^n$}

The main goal of this section is to prove the following.

\begin{prop}\label{l:approx-bi-lip}
	Let $(\Sigma,d)$ be a metric space and suppose there exists $L \geq 1$ and an $L$-bi-Lipschitz map $g : \R^n \rightarrow \Sigma$ such that $\Sigma = g(\R^n).$ Then $\Sigma$ satisfies a Carleson condition for the $\gamma$ coefficients with constant $K = K(L,n).$ The Carleson norm for the set in \eqref{d:gamma-large} depends only on $\ve,L$ and $n.$   
\end{prop}

For the proof of Proposition \ref{l:approx-bi-lip} we begin by setting some notation and recalling some important results. For a cubes $I$ in $\R^n$ let $x_I$ denote its centre and $\ell(I)$ denote its side-length. Let $\Delta$ be the set of dyadic cubes in $\R^n.$ Let $\mathcal{S}\mathcal{N}$ be the collection of all semi-norms on $\R^n$, $\mathcal{N}$ the collection of all norms on $\R^n$ and $\mathcal{N}_L$ the collection of norms on $\R^n$ such that 
\[ L^{-1}\|x\| \leq |x| \leq L \|x\|  \mbox{ for all } x \in \R^n,\] 
where $|\cdot|$ denotes the Euclidean norm on $\R^n.$ The following is a consequence of Dorronsoro's Theorem \cite{dorronsoro1985characterization}.

\begin{thm}\label{t:dorr}
	For each $L \geq 1$ there exists $K = K(L,n)$ such that the following holds. Suppose $\|\cdot\|_0 \in \mathcal{N}$, $f : \R^n \rightarrow (\R^n,\|\cdot\|_0)$ is $L$-Lipschitz and $\ve > 0.$ For each $J \in \Delta,$ we have 
	\begin{align}
		\sum \{ \ell(I)^n : I \in \Delta, \ I \subseteq J, \ \Omega_{\R^n,f,\|\cdot\|_0}^K(x_I,3\ell(I),|\cdot|,\emph{Id}) > \ve\} \lesssim_{\ve,L,n} \ell(J)^n.
	\end{align}
\end{thm}

\begin{rem}
	Dorronsoro's Theorem implies Theorem \ref{t:dorr} immediately for $L$-Lipschitz functions $g : \R^n \to \R^n$ and the coefficients $\Omega_{\R^n,g,|\cdot|}^K(\cdot,\cdot,|\cdot|,\text{Id})$. In the general case we proceed as follows. Let $T \colon (\R^n,\|\cdot\|_0) \to \R^n$ be a $\sqrt{n}$-bi-Lipschitz linear isomorphism (whose existence is guaranteed by Remark \ref{r:BM}) and set $g \coloneqq T \circ f.$ It is easy to show \[\Omega_{\R^n,f,\|\cdot\|_0}^{Kn^\frac{1}{2}} \lesssim_n \Omega_{\R^n,g,|\cdot|}^K\]
	so that Theorem \ref{t:dorr} with $f$ and $\|\cdot\|_0$ follows from the same result with $g$ and $|\cdot|.$ 
\end{rem}

For a metric space $(X,d),$ a Lipschitz function $f : \R^n \rightarrow X$ and a cube $I$ in $\R^n$ (not necessarily dyadic), define
\[ \text{md}_f(I) = \frac{1}{\ell(I)} \inf_{\|\cdot\| \in \mathcal{S} \mathcal{N}} \sup_{x,y \in I} \abs{d(f(x),f(y)) -\|x-y\| }\]
and 
\[ \text{md}_f^L(I) = \frac{1}{\ell(I)} \inf_{\|\cdot\| \in \mathcal{N}_L} \sup_{x,y \in I} \abs{d(f(x),f(y)) -\|x-y\| }\]
The following is proved in \cite{azzam2014quantitative}.
\begin{thm}\label{t:metric-diff}
	Let $L \geq 1,$ $f : \R^n \rightarrow X$ be $L$-Lipschitz function and $\ve > 0$. Then, for each $J \in \Delta$, 
	\[ \sum \{ \ell(I)^n : I \in \Delta, \ I \subseteq J, \ \emph{md}_f(3I) > \ve \} \lesssim_{d,\ve,L} \ell(J)^n. \]
\end{thm}

We record the following corollary.

\begin{cor}\label{c:AS}
	Let $L \geq 1,$ $f : \R^n \rightarrow X$ be an $L$-bi-Lipschitz function and $\ve > 0$. Then, for each $J \in \Delta$, 
	\[ \sum \{ \ell(I)^n : I \in \Delta, \ I \subseteq J, \ \emph{md}^{2L}_f(3I) > \ve \} \lesssim_{d,\ve,L} \ell(J)^n. \]
\end{cor}

\begin{proof}
	Without loss of generality we may suppose $ 0 < \ve < (4L)^{-1}$ (larger values for $\ve$ are immediate from this case). Suppose $I \in \Delta$ is such that $\text{md}_f(3I) \leq \ve.$ It suffices to show $\text{md}_f^{2L}(3I) \leq \ve$. To see that this holds, let $0 < \delta < (4L)^{-1}$ and $\| \cdot \| \in \mathcal{N}$ such that 
	\begin{align}\label{e:mdf}
		\sup_{x,y \in 3I} \abs{d(f(x),f(y)) -\|x-y\| } \leq (\text{md}_f(3I) + \delta) \ell(3I) \leq (\ve + \delta) \ell(3I).
	\end{align}
	Since $\delta$ is arbitrary, if we show $\| \cdot \| \in \mathcal{N}_{2L}$ then we are done. Let $z \in \R^n$ and choose $x,y \in 3I$ such that $|x - y| = \ell(3I)$ and $z = \lambda (y-x)$ for some $\lambda > 0.$ By \eqref{e:mdf}, our choice of $\ve$ and $\delta$, and the fact that $f$ is $L$-bi-Lipschitz, we have 
	\begin{align}
		L^{-1}|z| &= L \lambda |y-x| \leq \lambda d(f(x),f(y)) \leq \lambda \|x -y \| + (\ve + \delta)\lambda \ell(3I) \\
		&= \| z \| + (2L)^{-1}|z|. 
	\end{align} 
	After rearranging, the above inequality implies $|z| \leq 2L \| z \|.$ For the reverse inequality, we use \eqref{e:mdf} again to get 
	\begin{align}
		\|z\| = \lambda \|y-x \| \leq \lambda d(f(x),f(y)) + 2\ve \lambda \ell(3I) \leq L \lambda |y-x| + 2 \ve \lambda\ell(3I) \leq 2L |z|.  
	\end{align}
	Thus, $\| \cdot \| \in \mathcal{N}_{2L}$. 
\end{proof}

For the remainder of the section we fix a metric space $(\Sigma,d)$ and suppose there exists an $L$-bi-Lipschitz function $g : \R^n \to \Sigma$ such that $\Sigma = g(\R^n).$ Let $\calD$ be a system of Christ-David cubes for $\Sigma.$

\begin{lem}\label{l:I_Q}
	For each $Q \in \calD$ there exists $I_Q \in \Delta$ such that $\ell(I_Q) \sim_L \ell(Q)$ and $3B_Q \subseteq g(3I_Q).$ For each $I \in \Delta$, we have
	\begin{align}\label{e:I_Q}
		|\{Q \in \calD : I_Q = I\}| \lesssim_{L,n} 1.
	\end{align}
\end{lem}

\begin{proof}
	Let $k_0 \in \Z$ be such that $2^{k_0-1} \leq L < 2^{k_0}.$ For $Q \in \calD$ let $k(Q) \in \Z$ such that $2^{k(Q)-1} \leq \ell(Q) < 2^{k(Q)}.$ Let $I_Q$ be the unique cube such that $g^{-1}(x_Q) \in I_Q$ and $\ell(I_Q) = 2^{k_0 + k(Q)}.$ Clearly $\ell(I_Q) \sim_L \ell(Q).$ Let $x \in 3B_Q$ and let $y \in \R^n$ such that $x = g(y).$ Then, since $g$ is $L$-bi-Lipschitz, 
	\begin{align}
		| y - x_{I_Q} | &\leq  | g^{-1}(x) - g^{-1}(x_Q) | + |g^{-1}(x_Q) - x_{I_Q}| \\
		&\leq 3L \ell(Q) < 3 \cdot 2^{k^0+k(Q)} = 3\ell(I_Q).
	\end{align}
	It is easy to check $B(x_{I_Q},3\ell(I_Q)) \subseteq 3I_Q$ so that $y \in 3I_Q$ and $3B_Q \subseteq 3I_Q.$ We finish the proof by checking \eqref{e:I_Q}. 
	
	Let $I \in \Delta$ and suppose $Q,Q' \in \calD$ are such that $I_Q = I_{Q'} = I$. Then, $\ell(Q) = \ell(Q') = 5\varrho^k$ for some $k \in \Z$. Since $\{c_0B_Q : Q \in \calD_k\}$ are disjoint this implies $d(x_Q,x_{Q'}) \gtrsim \ell(Q).$ Since $g(x_Q) \in I,$ we have $|g(x_Q)-x_{I}| \lesssim_n \ell(I) \lesssim_L \varrho^k.$ Thus, $d(x_Q,g^{-1}(x_I)) \lesssim_{n,L} \varrho^k.$ Similarly, $d(x_{Q'},g^{-1}(x_I)) \lesssim_{n,L} \varrho^k$. Thus, $d(x_{Q},x_{Q'}) \lesssim_{n,L} \varrho^k.$ We can now apply Lemma \ref{l:BO} to conclude \eqref{e:I_Q}. 
\end{proof}

\begin{lem}\label{l:Omega-Sigma}
	Suppose $\|\cdot\|_0$ is a norm and $f \colon \Sigma \to (\R^n,\|\cdot\|_0)$ is a map. Set $h \coloneqq f \circ g.$ Then, for $Q \in \calD$ and $K \geq 1,$  
	\[ \gamma_{\Sigma,f,\|\cdot\|_0}^{K}(3Q) \lesssim_{K,L} \Omega_{\R^n,h,\|\cdot\|_0}^{\frac{K}{2L}}(x_{I_Q},3\ell(I_Q),|\cdot|,\emph{Id}) + \emph{md}_g^{2L}(3I_Q), \]
	with $I_Q$ as in Lemma \ref{l:I_Q}.
\end{lem}

\begin{proof}
	Let $\delta > 0$ and suppose $\|\cdot\| \in \mathcal{N}_{2L}$ is a norm on $\R^n$ such that 
	\begin{align}\label{e:AS}
		\sup_{x,y \in 3I_Q} \abs{ d(g(x),g(y)) - \|x-y\| } \leq \ell(3I_Q) ( \text{md}_g^{2L}(3I_Q) + \delta).
	\end{align}
	Let $y_Q = g^{-1}(x_Q)$ and set $\vp = p_{y_Q,3\ell(Q)} \circ g^{-1},$ where $p_{z,\lambda}$ is the map from Lemma \ref{l:p-lambda}. Then $\vp$ maps $\Sigma \cap 3B_Q$ to $B_{\|\cdot\|}(g^{-1}(x_Q),3\ell(Q))$ and  
	\begin{align}\label{e:vp-g}
		\|\vp(x) - g^{-1}(x)\| \leq \max\{0, \|g^{-1}(x) - g^{-1}(x_Q)\| - 3\ell(Q)\}.
	\end{align}	
	We begin by estimating $\unilat_\Sigma(3Q,\|\cdot\|,\vp)$. If $x \in \Sigma \cap 3B_Q \subseteq g(3I_Q)$,  then \eqref{e:AS} and \eqref{e:vp-g} imply
	\begin{align}\label{e:tildeA}
		\| \vp(x) - g^{-1}(x) \| &\leq \max\{0 ,  \|g^{-1}(x) - g^{-1}(x_Q)\| - 3\ell(Q)\} \\
		&\leq \max\{0, d(x,y) +  \ell(3I_Q)({\md}^{2L}_g(3I_Q) + \delta) - 3\ell(Q)\} \\
		&\leq \ell(3I_Q)({\md}^{2L}_g(3I_Q) + \delta).
	\end{align}
	By \eqref{e:AS} and \eqref{e:tildeA}, if $x,y \in \Sigma \cap 3B_Q$ then
	\begin{align}
		\abs{d(x,y) - \|\vp(x) - \vp(y)\| } &\leq \big|d(x,y) - \|g^{-1}(x) - g^{-1}(y)\|\big| + \| \vp(x) - g^{-1}(x) \|  \\
		&\hspace{2em} + \| \vp(y) - g^{-1}(y) \|  \\
		&\leq 3\ell(3I_Q)({\md}^{2L}_g(3I_Q) + \delta).
	\end{align}
	Taking the supremum over all $x,y \in \Sigma \cap 3B_Q$ and recalling that $\ell(I_Q) \sim \ell(Q),$ we have 
	\begin{align}\label{e:beta-sigma}
		\unilat_\Sigma(3Q,\|\cdot\|,\vp) \lesssim_L {\md}^{2L}_g(3I_Q) + \delta.
	\end{align}
	
	Now for $\eta_\Sigma(3Q,\|\cdot\|,\vp).$ Let $u \in B_{\|\cdot\|}(y_Q,3\ell(Q))$, $\lambda = \ell(3I_Q)({\md}^{2L}_g(3I_Q) + \delta)$ and choose a point $\tilde{u}  \in B_{\|\cdot\|}(y_Q,3\ell(Q) - 2\lambda)$ such that $\| u -\tilde{u} \| \leq 2\lambda.$ Such a point can be found by Lemma \ref{l:p-lambda}. By \eqref{e:AS} we have $x = g(\tilde{u}) \in \Sigma \cap 3B_Q$ since 
	\[d(x_Q,x) = d(g(y_Q),g(\tilde{u})) \leq \| y_Q - \tilde{u} \| + \lambda \leq 3\ell(Q) -\lambda.\]
	This, along with \eqref{e:tildeA}, gives
	\begin{align}
		{\dist}_{\|\cdot\|}(u,\vp(\Sigma \cap 3B_Q)) &\leq {\dist}_{\|\cdot\|}(u,\vp(x)) \leq \| u- g^{-1}(x) \| + \|g^{-1}(x) - \vp(x) \| \\
		&= \|u - \tilde{u}\| +  \|g^{-1}(x) - \vp(x) \| \leq 3\lambda.
	\end{align}
	Since $u \in B_{\|\cdot\|}(y_Q,3\ell(Q))$ was arbitrary, this implies
	\begin{align}\label{e:eta-sigma}
		\eta_{\Sigma}(3Q,\|\cdot\|,\vp) \lesssim_L {\md}^{2L}_g(3I_Q) + \delta.
	\end{align}
	
	It remains to estimate $\Omega_{\Sigma,f,\|\cdot\|_0}^K(3Q,\|\cdot\|,\vp).$ Recall that $h = f \circ g.$ Let $A : \R^n \rightarrow (\R^n,\|\cdot\|_0)$ be an affine function such that $\text{Lip}(A) \leq \tfrac{K}{2L}$ and 
	\[\sup_{y \in 3I_Q} \|h(y) - A(y)\|_0 \leq \ell(3I_Q)[\Omega_{\R^n,h,\|\cdot\|_0}^{\frac{K}{2L}}(x_{I_Q},3\ell(I_Q),|\cdot|,\text{Id})+ \delta].\]
	Since $\|\cdot\| \in \mathcal{N}_{2L},$ we have $\text{Lip}(A) \leq K$ when viewed as a function from $(\R^n,\|\cdot\|)$. Let $x \in 3B_Q$ and set $y = g^{-1}(x) \in 3I_Q.$ By the definition of $A$ and \eqref{e:tildeA} we have
	\begin{align}
		\|f(x) - A(\vp(x))\|_0 &\leq \|f(x) - A(g^{-1}(x))\|_0 + \|A(g^{-1}(x)) - A(\vp(x)) \|_0 \\
		&\leq \|h(y) - A(y) \|_0 + K \|g^{-1}(y) - \vp(y)\| \\
		&\lesssim_{K} \ell(3I_Q)[\Omega_{\R^n,h,\|\cdot\|_0}^{\frac{K}{2L}}(x_{I_Q},3\ell(I_Q),|\cdot|,\text{Id}) + \text{md}^{2L}_g(3I_Q) + 2\delta].
	\end{align} 
	Taking the supremum over all $x \in 3B_Q$ and recalling that $\ell(I_Q) \sim_L \ell(Q)$ we have 
	\begin{align} \Omega_{\Sigma,f,\|\cdot\|_0}^K(3Q,\|\cdot\|,\vp) \lesssim \Omega_{\R^n,h,\|\cdot\|_0}^{\frac{K}{2L}}(x_{I_Q},3\ell(I_Q),|\cdot|,\text{Id}) + {\md}^{2L}_g(3I_Q) + 2 \delta.
	\end{align}
	Combining this with \eqref{e:beta-sigma} and \eqref{e:eta-sigma} then taking $\delta  \to 0$ finishes the proof of the lemma. 
	
\end{proof}

\begin{lem}\label{l:packing-bi-lip}
	There exists $K = K(L,n) \geq 1$ and for each $\ve > 0$ a constant $C = C(\ve)$ such that if $\|\cdot\|_0$ is a norm and $f : \Sigma \to (\R^n,\|\cdot\|_0)$ is 1-Lipschitz then $\{Q \in \calD : \gamma_{X,f,\|\cdot\|_0}^K(3Q) > \ve\}$ satisfies a Carleson packing condition with Carleson norm $C.$ 
\end{lem}

\begin{proof}
	Let $h = f \circ g$ so that $h$ is $L$-Lipschitz from $\R^n$ from $(\R^n,\|\cdot\|_0).$ Let $\tilde{K} = \tilde{K}(L,n)$ be the constant from Theorem \ref{t:dorr} and set $K = 2\tilde{K}L.$ By Lemma \ref{l:Omega-Sigma} there exists a constant $C = C(K,L) \geq 1$ such that 
	\begin{align}
		 \{Q \in \calD : \gamma_{\Sigma,f}^{K}(3Q) > \ve\} &\subseteq \{ Q \in \calD : \Omega_{\R^n,h,\|\cdot\|_0}^{\tilde{K}}(x_{I_Q},3\ell(I_Q),|\cdot|,\text{Id}) > \ve/C\} \\
	&\hspace{2em} \cup \{Q \in \calD :  \text{md}_g^{2L}(3I_Q) > \ve/C\}.
\end{align}
	Then, by Theorem \ref{t:dorr}, Corollary \ref{c:AS} and Lemma \ref{l:I_Q}, we have 
	\begin{align}
		&\sum_{\substack{Q \in \calD \\ Q \subseteq Q_0 }} \{ \ell(Q)^n : \gamma_{\Sigma,f,\|\cdot\|_0}^{K}(3Q) > \ve\} \\
		&\hspace{2em} \lesssim_{L,n} \sum_{\substack{I \in \Delta \\ I \subseteq I_{Q_0}}} \{ \ell(I)^n : \Omega_{\R^n,h,\|\cdot\|_0}^{\tilde{K}}(x_I,3\ell(I),|\cdot|,\text{Id}) > \ve/C\}  \\
		&\hspace{4em}+ \sum_{\substack{I \in \Delta \\ I \subseteq I_{Q_0}}} \{ \ell(I)^n : \text{md}^{2L}_g(3I) > \ve/C\} \lesssim_{L,\ve,n} \ell(I_{Q_0})^n \lesssim_{L,n} \ell(Q_0)^n. 
	\end{align}
\end{proof}

We conclude the proof of Proposition \ref{l:approx-bi-lip} (hence, the proof of Theorem \ref{t:approx-UR}) by combining Lemma \ref{l:cubes-wala} with Lemma \ref{l:packing-bi-lip}.

\newpage

\section{Metric Reifenberg Flat (RF) parametrizations with varying tangents}\label{s:Reif-approx}


\etocsettocstyle{Contents of this section}{}
\etocsettocmargins{.01\linewidth}{.01\linewidth}

\localtableofcontents

\bigskip

\subsection{Main result}

	In this section we prove our generalized Reifenberg theorem. The argument is inspired by the scheme of \cite{cheeger1997structure} and its presentation in \cite{violo2021functional}. The following is our main result.

\begin{thm}\label{t:Reif}
	Let $(X,d)$ a metric space, $x_0 \in X$ and $\delta,\ve > 0$.  For each $\ell \geq 0$, let $r_\ell = 10^{-\ell},$ $\{x_{j,\ell}\}_{j \in J_\ell}$ a collection of points in $B_X(x_0,1)$ and $\{ \| \cdot \|_{j,\ell}\}_{j \in J_\ell}$ a collection of norms on $\R^n.$ For $j \in J_\ell$ let 
	\[B^{j,\ell} = B_X(x_{j,\ell},r_\ell) \mbox{ and } B_{j,\ell} = B_{\|\cdot\|_{j,\ell}}(0,r_\ell).\] 
	Suppose each of the following conditions is satisfied. 
	\begin{enumerate}
		\item At the top scale $J_0 = \{j_0\}$ and $x_{j_0,0} = x_0.$
		\item For each $\ell \geq 0$ and $i,j \in J_\ell$ with $i \neq j,$ we have 
		\begin{align}\label{e:thm12}
			d_X(x_{i,\ell},x_{j,\ell}) \geq r_\ell.
		\end{align}
		\item For each $\ell \geq 1$ and $j \in J_\ell$ there exists $i \in J_{\ell-1}$ such that
		\begin{align}\label{e:thm13} 
			x_{j,\ell} \in 2B^{i,\ell-1}.
		\end{align}			
		\item  For each $\ell \geq 0$ and $j \in J_\ell$ there exist maps $\away_{j,\ell} : 100B_{j,\ell} \to 100B^{j,\ell}$ and $\toward_{j,\ell} : 100B^{j,\ell} \to 100B_{j,\ell}$ each of which is a $\delta r_\ell$-GHA and such that
		\begin{align}\label{e:centre-pos}
			\alpha_{j,\ell}(0) = x_{j,\ell} \mbox{ and } \beta_{j,\ell}(x_{j,\ell}) = 0;
		\end{align}
		\begin{align}\label{e:almost-id}
			\| \toward_{j,\ell} \circ \away_{j,\ell} - \emph{Id} \|_{C^0} \leq \delta r_\ell \mbox{ and }  \| \away_{j,\ell} \circ \toward_{j,\ell} - \emph{Id} \|_{C^0} \leq \delta r_\ell. 
		\end{align}
	\end{enumerate}  

	If $\delta$ is chosen small enough, depending on $\ve$ and $n$, there exists a metric $\rho$ on $\R^n$ which is 
	\begin{align}\label{e:Reif-Holder}
		\mbox{bi-H\"older equivalent to $\|\cdot\|_{j_0,0}$ with exponent $1-\ve$},
\end{align}
	 as in \eqref{e:bi-holder}, such that 
	 \begin{align}\label{e:Reif-Reif}
	 	\mbox{$(\R^n,\rho)$ is $(\ve,n)$-Reifenberg flat}
	 \end{align}
  and 
	\begin{align}\label{e:Reif1}
		| \rho(x,y) - \|x-y\|_{j_0,0} | \leq \ve \mbox{ for all $x,y \in \R^n.$ }
	\end{align} 
	
	Furthermore, there are maps $g : 6B_{j_0,0} \to 6B^{j_0,0}$ and $\bar{g} \colon 6B^{j_0,0} \to 6B_{j_0,0}$ satisfying the following conditions. Suppose $\ell \geq 0$, $j \in J_\ell$ and let $x,y \in 6B^{j,\ell}$ and $w,z \in 6B_{j_0,\ell}$ such that $g(w),g(z) \in 6B^{j,\ell}.$ Then,
	\begin{align}
		d_X(g(\bar{g}(x)),x) \leq \ve r_\ell   \label{e:Reif2.0},\\
		\rho(\bar{g}(g(w)),w) \leq \ve r_\ell  \label{e:Reif2.5}, \\
		| d_X(g(w),g(z)) - \rho(w,z) | \leq \ve r_\ell \label{e:Reif3.0}, \\
		|\rho(\bar{g}(x),\bar{g}(y)) - d_X(x,y) | \leq \ve r_\ell. \label{e:Reif3.5}
\end{align}
	Additionally, 
	\begin{align}\label{e:regularity}
		\mbox{ if $X$ is Ahlfors $n$-regular, then $(\R^n,\rho)$ is Ahlfors $n$-regular. }
	\end{align}
	In \eqref{e:regularity}, the regularity constant depends on $n$ and the regularity constant for $X$.
\end{thm}

After some preparation in Sections \ref{s:transitions} - \ref{s:M}, the proof of Theorem \ref{t:Reif} begins in Section \ref{s:metric-map} and concludes in Section \ref{s:AR}. 

\begin{rem}
	Conditions (1)-(3) impose a sort of tree like structure on our collection of balls $B^{j,\ell},$ with $B^{j_0,0}$ forming the root of the tree. These conditions appear naturally in stopping-time arguments, as we shall see in Section \ref{s:final}. Condition (4) is really saying that $\bilat_X(100B^{j,\ell}) \leq \delta/100$ (recall the definition from Definition \ref{d:bilat}). We state this condition with two mapping because it will be useful for us to have a GHA defined on the whole of the ball in $X$ and a GHA defined on the whole of the tangent ball. Equation \eqref{e:almost-id} is just saying that the $\alpha$ is roughly the inverse of $\beta$ and vice-versa. It is easy to construct the maps $\alpha$ and $\beta$ satisfying these conditions if one knows that $\bilat_X(100B^{j,\ell}) \leq \delta/100$ for each ball. Thus, one should interpret the hypotheses of Theorem \ref{t:Reif} as saying that $X$ is $(\delta,n)$-Reifenberg flat in a tree of balls (outside of this tree we assume nothing). 
	
	With this structure, we construct a metric $\rho$ which is $(\ve,n)$-Reifenberg flat and bi-H\"older equivalent to $\|\cdot\|_{j_0,0}$. Moreover, \eqref{e:Reif1} says that $\rho$ is only a small perturbation of $\|\cdot\|_{j_0,0}$ relative to the top scale. The maps $g$ and $\bar{g}$ act as a kind of multi-scale Gromov-Hausdorff approximation between $X$ and $(\R^n,\rho)$. In particular, the existence of these maps tells us that balls in the tree are bi-laterally well-approximated by $(\R^n,\rho).$

\end{rem}

\begin{rem}
	Combining \eqref{e:Reif1} - \eqref{e:Reif3.5} we also have  
	\begin{align}\label{e:g-near-centre}
		d_X(g(0),x) \leq 10\ve  \mbox{ and } \rho(\bar{g}(x),0) \leq 10\ve .
	\end{align}
	Indeed, suppose towards a contradiction that $\rho(\bar{g}(x),0) \geq 5\ve.$ By \eqref{e:Reif1} this implies $\| \bar{g}(x) \|_{j_0,0} \geq 4\ve.$ Let $z = -6\bar{g}(x)/\|\bar{g}(x)\|_{j_0,0}$ so that $\| \bar{g}(x) - z \| \geq 6 + 4\ve.$ Combining \eqref{e:Reif1}, \eqref{e:Reif2.0} and \eqref{e:Reif3.0}, we now have $\rho(\bar{g}(x),z) \geq 6+3\ve$ and 
	\begin{align}
		d_X(x,g(z)) \geq d_X(g(\bar{g}(x)),g(z)) - \ve \geq \rho(\bar{g}(x),z) - 2\ve \geq 6+\ve.
	\end{align}
	Since $g$ takes values in $6B^{j_0,0},$ this is a contradiction and proves the second inequality in \eqref{e:g-near-centre} with the slightly better constant $5\ve.$ To prove the first inequality in \eqref{e:g-near-centre}, we combine this improved estimate with \eqref{e:Reif2.5} and \eqref{e:Reif3.5} to get 
	\begin{align}
		d_X(g(0),x) \leq \rho(\bar{g}(g(0)),\bar{g}(x)) + \ve \leq \rho(\bar{g}(g(0)),0) + \rho(\bar{g}(x),0) + \ve \leq 10\ve. 
	\end{align}
\end{rem}

\begin{rem}
	For our application it is desirable that $(\R^n,\rho)$ is linearly locally contractible (see Definition \ref{d:LLC}). It is possible to show this directly via the construction however we chose to prove more generally that all $(\delta,n)$-Reifenberg flat sets are linearly locally contractible for $\delta$ small enough (see Proposition \ref{p:Reif-LLC}). We actually use Theorem \ref{t:Reif} to prove this fact. 
\end{rem}

We prove the conclusions of Theorem \ref{t:Reif} with $\ve$ replaced by $C\ve$ for some $C$ depending only on $n$, which is sufficient. The conclusions of Theorem \ref{t:Reif} are stronger as $\ve$ becomes smaller, so we restrict our attention to $\ve$ sufficiently small. Let 
\[0 < \ve \leq \delta < 1/100.\] 
Throughout the remainder of this section we let $(X,d)$ be a fixed metric space, $\{x_{j,\ell}\}_{\ell \geq 0,j \in J_\ell}$ be a collection of points in $B_X(0,1)$, $\{\|\cdot\|_{j,\ell}\}_{\ell \geq 0,j \in J_\ell}$ be a collection of norms on $\R^n,$ and we assume conditions (1)-(4) of Theorem \ref{t:Reif} hold with parameter $\delta.$   

\bigskip

\subsection{Constructing smooth transition maps}\label{s:transitions}

One can think of $\toward_{j,\ell} \colon 100B^{j,\ell} \to 100B_{j,\ell}$ as a rough coordinate patch in $X$. If $\toward_{i,\ell},\toward_{j,\ell}$ are nearby coordinate patches (in the sense that $15B^{i,\ell} \cap 15B^{j,\ell} \neq \emptyset$), their `transition map' is given by $\toward_{j,\ell} \circ \away_{i,\ell}$ (Lemma \ref{l:alpha-beta-circ}). As mentioned in Section \ref{s:outline}, we prove Theorem \ref{t:Reif} by first constructing a sequence of approximating smooth manifolds for $X$. Keeping this in mind, Proposition \ref{p:coherent} below produces, for each nearby pair $i,j \in J_\ell$, a smooth perturbation $\tilde{I}_{i,j,\ell} \colon \R^n_{j,\ell} \to \R_{i,\ell}^n$ of the map $\toward_{j} \circ \away_i$ satisfying a necessary cyclic condition (see Remark \ref{r:necessary}).  A suitable restriction of these maps will be the transition maps for the manifolds we construct in Section \ref{s:disconnected}.


It may be useful to recall from \eqref{d:norm} that $\|\cdot\|_{C^m(U),r}$ denotes the uniform norm of the first $m$ Fr\'echet derivatives.

\begin{prop}\label{p:coherent}
	There exists constants $\ve_1 = \ve_1(n)$ and $\delta_ 1 = \delta_1(\ve)$ such that the following holds. Suppose $0 < \ve < \ve_1$ and $0 < \delta < \delta_1$. For $\ell \geq 0$, let 
	\begin{align}\label{e:F_ell}
		\calF_\ell = \{(i,j) \in J_\ell \times J_\ell : 15B^{i,\ell} \cap 15B^{j,\ell} \neq \emptyset\}.
	\end{align}
	Then, for each $(i,j) \in \calF_\ell$ there exists a smooth diffeomorphism $\tilde{I}_{i,j,\ell} \colon \R_{j,\ell}^n \to \R_{i,\ell}^n$ and an affine map $T_{i,j,\ell} : \R^n_{j,\ell} \to \R^n_{i,\ell}$ satisfying the following conditions. First, both $\tilde{I}_{i,j,\ell}$ and $T_{i,j,\ell}$ are $(1+C\ve)$-bi-Lipschitz and 
	\begin{align}\label{e:prop'}
		\|\tilde{I}_{i,j,\ell}  -T_{i,j,\ell}\|_{C^2(\R^n_{j,\ell}),r_\ell} \lesssim \ve.
	\end{align}
	Second, we have 
	\begin{align}\label{e:prop}
		\|\tilde{I}_{i,j,\ell}(x)  - \toward_{i,\ell} \circ \away_{j,\ell}(x) \|_{i,\ell} \lesssim \ve r_\ell
	\end{align}
	for all $x \in 45B_{j,\ell}.$ Finally, for any triple $(i,j),(j,k),(i,k) \in \calF_\ell,$ the maps $\tilde{I}_{i,j,\ell},\tilde{I}_{j,k,\ell}$ and $\tilde{I}_{i,k,\ell}$ are $9r_\ell$-cyclic in the sense of Definition \ref{d:cyclic}. 
\end{prop}

Fix $\ell \geq 0$ for the remainder of this subsection. When the context is clear, we will write $J = J_\ell$ and $\calF = \calF_\ell$. For $j \in J$, we will also write 
\[ x_{j} = x_{j,\ell}, \ \|\cdot\|_{j} = \| \cdot\|_{j,\ell}, \ B^{j} = B^{j,\ell}, \  B_{j} = B_{j,\ell} \mbox{ and } \R^n_j = (\R^n,\|\cdot\|_j).\] 

\begin{lem}\label{l:alpha-beta-circ}
	Let $(i,j) \in \calF$. Then, $\toward_i \circ \away_j \circ \toward_j \circ \away_i$ is well-defined on $45B_i$ and
	\begin{align}
		\| \toward_i \circ \away_j \circ \toward_j \circ \away_i(x) - x \|_i \lesssim \delta r_\ell
	\end{align}
	for all $x \in 45B_{i}.$ 
\end{lem}

\begin{proof}
	Let $x \in 45B_i.$ First, by \eqref{e:dist-sim} we have $d(\away_i(x),x_i) \leq 45r_\ell + 11\delta r_\ell \leq 46r_\ell$ (recalling that $\delta < 1/100$). In particular, $\away_i(x) \in 46B^i \subseteq \text{dom}(\toward_j)$. It now follows from \eqref{e:almost-id} that $\away_j ( \toward_j ( \away_i(x) ))\in 47B^i \subseteq \text{dom}(\toward_i)$. Thus, $\toward_i \circ \away_j \circ \toward_j \circ \away_i$ is well-defined. Then, using \eqref{e:almost-id} with the fact that $\away_j,\toward_i$ are $\delta r_\ell$-GHA, we get
	\begin{align}
		\| \toward_i ( \away_j ( \toward_j (\away_i(x)))) - x \|_i &\leq \|  \toward_i ( \away_j ( \toward_j (\away_i(x)))) - \toward_i(\away_i(x)) \|_i + \| \toward_i(\away_i(x)) - x \|_i \\
		&\leq \| \away_j(\toward_j(\away_i(x))) - \away_i(x) \|_i + \delta r_\ell \lesssim \delta r_\ell. 
	\end{align}
\end{proof}

\begin{lem}\label{l:T-alpha}
	There exists $\delta_1 = \delta_1(\ve)$ such that the following holds. Suppose $0 < \delta \leq \delta_1.$ Then, for each $(i,j) \in \calF$ there exists an affine $(1 + C\ve)$-bi-Lipschitz map $T_{i,j} : \R^n_i \to \R^n_j$ such that 
	\begin{align}\label{e:T-alpha_i}
		\| T_{i,j}(x) - \toward_i \circ \away_j(x) \|_i \lesssim \ve r_\ell
	\end{align}
	for all $x \in 45B_j$. Additionally, if we let $T_{j,i} = T_{i,j}^{-1},$ then 
	\begin{align}\label{e:T-alpha_j}
		\| T_{j,i}(x) - \toward_j \circ \away_i(x) \|_j \lesssim \ve r_\ell 
	\end{align}
	for all $x \in 45B_i.$ 
\end{lem}

\begin{proof}
	Let $(i,j) \in \calF.$ We start by constructing $T_{i,j}.$ Using the final part of Lemma \ref{l:GH-scale}, there exists a $C\delta r_\ell$-GHA $\phi_j: 80B_{j} \to 80B^j$ such that $d(\away_j(x) , \phi_j(x)) \leq C\delta r_\ell$ for all $x \in 80B_{j}$ (recall that $\away_j(0) = x_j$ by \eqref{e:centre-pos}). Letting $z = \toward_i(x_j)$ and applying Lemma \ref{l:GH-scale} again gives a $C\delta r_\ell$-GHA $\phi_i : 80B^j \to B_{\|\cdot\|_i}(z,80r_\ell)$ such that $\| \toward_i(x) - \phi_i(x)\|_i \leq C\delta r_\ell $ for all $x \in 80B^j.$ By Lemma \ref{l:GH-comp}, 
	\[ \phi_{i,j} \coloneqq \phi_i \circ \phi_j : 80B_{j} \to B_{\|\cdot\|_i}(z,80r_\ell)\]
	defines a $C\delta r_\ell$-GHA. By Lemma \ref{l:GH-linear}, as long as $\delta$ is chosen small enough depending on $\ve,$ there exists an affine $(1+C\ve)$-bi-Lipschitz map $T_{i,j} : \R^n_j \to \R^n_i$ such that 
	\[ \| T_{i,j}(x) - \phi_{i,j}(x) \|_i \leq \ve r_\ell \]
	for all $x \in 80B_{j}.$ 
	
	We now consider \eqref{e:T-alpha_i}, which we prove for a slightly larger ball. Suppose $x \in 79B_{j}.$ By \eqref{e:dist-sim}, $\away_j(x) \in 80B^j$ and by construction $\phi_j(x) \in 80B^j.$ Taking $\delta$ small enough depending on $\ve$, and using the relation between $\away_j,\phi_j$ and $\toward_i,\phi_i$ on $80B^j,$ we have 
	\begin{align}
		\| T_{i,j}(x) - \toward_i \circ \away_j(x)\|_i &\leq \| T_{i,j}(x) - \phi_{i,j}(x) \|_i + \| \phi_i \circ \phi_j(x) - \toward_i \circ \phi_j(x) \|_i \\
		&\hspace{2em}+ \| \toward_i \circ \phi_j(x) - \toward_i \circ \away_j(x) \|_i \\
		&\leq \ve r_\ell + C\delta r_\ell + d(\phi_j(x),\away_j(x)) + C\delta r_\ell \lesssim \ve r_\ell. 
	\end{align}
	
	We move on to \eqref{e:T-alpha_j}. Suppose $x \in 45B_{i}.$ Using \eqref{e:dist-sim} again, we have $\away_i(x) \in 46B^i \subseteq 76B^j$, where the final inclusion follows from the triangle inequality since $15B^i \cap 15B^j \neq \emptyset$. Another application of \eqref{e:dist-sim} then gives $\toward_j \circ \away_i(x) \in 77B^{j} \subseteq 79B^j.$ Finally, by Lemma \ref{l:alpha-beta-circ}, \eqref{e:T-alpha_i} and since $T_{j,i}$ is $(1+C\ve)$-bi-Lipschitz, we get
	\begin{align}
		\| T_{j,i}(x) - \toward_j \circ \away_i(x) \|_j &\lesssim \| x - T_{i,j} \circ \toward_j \circ \away_i(x) \|_i \\
		&\leq \| x - \toward_i\circ \away_j \circ \toward_j \circ \away_i(x) \|_i \\
		&\hspace{2em} + \|  \toward_i\circ \away_j \circ \toward_j \circ \away_i(x) -  T_{i,j} \circ \toward_j \circ \away_i(x) \|_i \lesssim \ve r_\ell. 
	\end{align}
\end{proof}

\begin{rem}
	For the remainder of this subsection we will assume $\delta < \delta_1$ so that the maps $T_{i,j}$ are well defined. 
\end{rem}

\begin{lem}\label{l:T-cyclic}
	Let $(i,j), (j,k) , (i,k) \in \calF$ and $\id \colon \R^n \to \R^n$ be the identity map. Then,
	\begin{align}
		\| T_{k,i} \circ T_{i,j} \circ T_{j,k} - \id \|_{C^2(12B_k),r_\ell} \lesssim \ve. 
	\end{align}
\end{lem}

\begin{proof}
	Let $H \coloneqq T_{k,i} \circ T_{i,j} \circ T_{j,k}$ and suppose $x \in 12B_{k}.$ Since $T_{k,i}$ is $(1+C\ve)$-bi-Lipschitz, the same is true of its linear part, $L_{k,i}.$ Hence, 
	\begin{align}
		\| T_{k,i} \circ T_{i,j} \circ T_{j,k}(x) - x \|_k &= \| L_{k,i} \left( T_{i,j} \circ T_{j,k}(x) - T_{i,k}(x) \right) \|_k\\
		&\lesssim \| T_{i,j} \circ T_{j,k}(x) - T_{i,k}(x) \|_i .
	\end{align}
	Applying \eqref{e:dist-sim} twice we have $\toward_j \circ \away_k(x) \in 45B_{j}$ so we can apply Lemma \ref{l:T-alpha} to the map $T_{i,j}$ at this point. Using also that $T_{i,j}$ is bi-Lipschitz, and applying Lemma \ref{l:T-alpha} to the maps $T_{j,k}$ and $T_{i,k}$ at $x$, we get  
	\begin{align}
		\| T_{i,j} \circ T_{j,k}(x) - T_{i,k}(x) \|_i &\leq \|T_{i,j} (T_{j,k}(x)) - T_{i,j}(\toward_j \circ \away_k(x)) \|_i  \\
		&\hspace{2em} + \| T_{i,j}(\toward_j \circ \away_k(x)) - \toward_i \circ \away_k(x)\|_i \\
		&\hspace{4em} + \| \toward_i \circ \away_k(x) - T_{i,k}(x) \|_i \\
		&\lesssim \| \toward_i \circ \away_j \circ \toward_j \circ \away_k(x) - \toward_i \circ \away_k(x) \|_i + C\ve r_\ell.
	\end{align} 
	Now, using \eqref{e:almost-id}, the fact that $\toward_i$ is a $\delta r_\ell$-GHA and $\delta \leq \ve$, we have  
	\begin{align}
		\| \toward_i \circ \away_j \circ \toward_j \circ \away_k(x) - \toward_i \circ \away_k(x) \|_i &\leq d(\away_j \circ \toward_j \circ \away_k(x) , \away_k(x) ) + C\ve r_\ell\\
		&\leq C\delta r_\ell + C\ve r_\ell \lesssim \ve r_\ell. 
	\end{align}
	Combining the above three estimates gives 
	\begin{align}\label{e:H-x}
		\|H(x) - x \|_k \lesssim \ve r_\ell.
	\end{align}
	Since $H$ is affine, this and Lemma \ref{l:linear-C^0} now imply 
	\begin{align}\label{e:DH-I}
		\|DH(x) \cdot y -y \|_k \lesssim \ve \|y\| 
	\end{align}
	for all $y \in \R^n.$ The estimates on the second derivatives are trivial since $H$ is affine and $\id$ is linear so that $D^2H$ and $D^2\id$ are zero. This observation together with \eqref{e:H-x} and \eqref{e:DH-I} completes the proof. 
\end{proof}

If $0 < \lambda < 12$ and $\ve \ll \lambda$ then Lemma \ref{l:T-cyclic} says the maps $T_{i,j}$ are almost $\lambda r_\ell$-cyclic in the sense of Definition \ref{d:cyclic}. Suppose there exist maps $I_{i,j}$ which well approximate the $T_{i,j}$ (so that the $I_{i,j}$ are also almost $\lambda r_\ell$-cyclic), the following lemma finds smooth perturbations $\hat{I}_{i,j}$ of the $I_{i,j}$ which are exactly $\lambda' r_\ell$-cyclic for some parameter $\lambda'$ which is slightly smaller than $\lambda$. The $\hat{I}_{i,j}$ still well approximate the $T_{i,j}$ so it is possible to apply the lemma again. We will prove Proposition \ref{p:coherent} by carefully iterating this lemma.

\begin{lem}\label{l:single-mod}
	The exists a constant $C_1 \geq 1$ depending only on $n$ and for each $C_2 \geq 1$ and $N \geq 1$ a constant $\ve_0 = \ve_0(C_2,N) > 0$ such that the following holds. Suppose $\ve \leq \ve_0$, let $1/N < \lambda \leq 10,$ $(i,j), (j,k),(i,k) \in \calF$ and suppose for each distinct pair $a,b \in \{ i,j,k \}$ that there is a smooth diffeomorphism $I_{a,b} : \R^n_b \to \R^n_a$ such that $I_{a,b}^{-1} = I_{b,a}$ and 
	\begin{align}\label{e:I-T}
		\| I_{a,b} - T_{a,b} \|_{C^2(\R^n_b),r_\ell} \leq C_2\ve. 
	\end{align}
	Then there exists a smooth diffeomorphism $\hat{I}_{k,j} : \R^n_j \to \R^n_k$ such that 
	\begin{align}
		&\hspace{1em}\hat{I}_{k,j} = I_{k,j} \mbox{ outside } I_{j,i}(\lambda B_i); \label{e:I1} \\
		&\hspace{1em}\hat{I}_{k,j}(I_{j,i}(x)) = I_{k,i}(x) \mbox{ for all } x \in \lambda B_i \mbox{ such that } I_{k,j}(I_{j,i}(x)) = I_{k,i}(x);  \label{e:I2}  \\
		&\hspace{1em}\hat{I}_{k,j}, I_{j,i},I_{k,i} \mbox{ are } ( \lambda - \tfrac{1}{N})r_\ell\mbox{-cyclic}. \label{e:I3} 
	\end{align}
	Furthermore, letting $\hat{I}_{j,k} = \hat{I}^{-1}_{k,j},$ we have 
	\begin{align}\label{e:I4}
		\| I_{j,k} - \hat{I}_{j,k} \|_{C^2(\R^n_k),r_\ell} \leq C_1C_2 \ve \mbox{ and } \| I_{k,j} - \hat{I}_{k,j} \|_{C^2(\R^n_j),r_\ell} \leq C_1C_2 \ve. 
	\end{align}
\end{lem}

\begin{proof}
	Let $\lambda' = \lambda - \tfrac{1}{2N}$, $\lambda'' = \lambda - \tfrac{1}{N}$ and set
	\begin{align}
		U_1 &= I_{k,j} \circ I_{j,i}\left(B_i(0,\lambda'r_\ell)\right) \cap B_k(0,\lambda'r_\ell); \label{d:U_1}\\
		U_2 &=  I_{k,j} \circ I_{j,i}\left(B_i(0,\lambda r_\ell)\right) \cap B_k(0,\lambda r_\ell). \label{d:U_2}
	\end{align}
	Suppose $(a,b) = (i,j)$ or $(a,b) = (j,i).$ Since $U_2 \subseteq 12B_k$, we can combine Lemma \ref{l:composition} with Lemma \ref{l:T-cyclic} and \eqref{e:I-T} to get,
	\begin{align}
		\begin{split}\label{e:Icomp}
		\| I_{k,a} \circ I_{a,b} \circ I_{b,k} - \mbox{Id} \|_{C^2(U_2),r_\ell} &\leq \| I_{k,a} \circ I_{a,b} \circ I_{b,k} - T_{k,a} \circ T_{a,b} \circ T_{b,k} \|_{C^2(U_2),r_\ell} \\
		&\hspace{2em} + \| T_{k,a} \circ T_{a,b} \circ T_{b,k} - \id
		\|_{C^2(U_2),r_\ell} \\
		&\lesssim C_2\ve .
	\end{split}
	\end{align}
	
	If $U_1 = \emptyset$, it turns out the functions $I_{k,j},I_{j,i}$ and $I_{k,i}$ are already $\lambda''r_\ell$-cyclic and we can just take $\hat{I}_{a,b} = I_{a,b}$ for each distinct $a,b \in \{i,j,k\}$. To see that the maps are $\lambda''r_\ell$-cyclic it suffices to check, by Remark \ref{r:cyc}, that \eqref{e:cyclic-hyp} holds for the permutations $(i,j,k),(k,i,j)$ and $(i,k,j).$ We will show \eqref{e:cyclic-hyp} is vacuous in each case i.e. there are no points satisfying the hypotheses of \eqref{e:cyclic-hyp}. For $(i,j,k)$ we notice immediately, using the definition of $U_1$ in \eqref{d:U_1} and the fact that $\lambda'' < \lambda'$, that there is no $x \in \lambda'' B_i$ such that $I_{j,i}(x) \in \lambda'' B_j$ and $I_{k,j}(I_{j,i}(x)) \in \lambda'' B_k.$ For $(k,i,j),$ suppose there exists $x \in \lambda'' B_k$ such that $y = I_{i,k}(x) \in \lambda''B_i$ and $z = I_{j,i}(y) \in \lambda''B_j.$ Since $x \in \lambda'' B_k,$ \eqref{e:Icomp} implies $I_{k,j}(I_{j,i}(y)) = I_{k,j}(I_{j,i}(I_{i,k}(x)))  \in \lambda' B_k$ as long as $\ve$ is small enough depending on $C_2.$ This gives $U_1 \neq \emptyset$ which is a contradiction. Finally, for $(i,k,j),$ suppose there exists $x \in \lambda''B_i$ such that $w = I_{k,i}(x) \in \lambda''B_k$ and $I_{j,k}(I_{k,i}(x)) \in \lambda'' B_j.$ Since $w \in \lambda'' B_k$, \eqref{e:Icomp} implies $I_{k,j}(I_{j,i}(x)) = I_{k,j}(I_{j,i}(I_{i,k}(w))) \in \lambda' B_k$ so that again $U_1 \neq \emptyset,$ a contradiction.
	
	Suppose instead that $U_1 \neq \emptyset.$ We'd like to apply Lemma \ref{l:modification} to the map $H = I_{k,i} \circ I_{i,j} \circ I_{j,k}$ and the sets $U_1$ and $U_2.$ Equation \eqref{e:H-Id} follows immediately from \eqref{e:Icomp}. We will now verify that
	\begin{align}\label{e:separation}
		\dist(U_1,U_2^c) \geq \frac{r_\ell}{10N}.
	\end{align}
	Let $w \in U_1$ and $z \in U_2^c.$ If $z \not\in B_k(0,\lambda r_\ell)$ then $\| w- z\|_k \geq r_\ell/2N \geq r_\ell/10N$ since $U_1 \subseteq B_k(0,\lambda'r_\ell)$. Suppose instead that $z \in B_k(0,\lambda r_\ell).$ Let $x \in B_{i}(0,\lambda' r_\ell)$ such that $w = I_{k,j}(I_{j,i}(x))$ and $y \in \R^n$ such that $z = I_{k,j}(I_{j,i}(y)).$ It must be that $y \notin B_i(0,\lambda r_\ell)$, otherwise $z \in U_2$ by the definition \eqref{d:U_2}. Thus, $\| x - y \|_i \geq r_\ell/2N$. Applying \eqref{e:I-T}, using that each $T_{ab}$ is $(1+C\ve)$-bi-Lipschitz, and choosing $\ve$ small enough depending on $C_2$ and $N$, we have 
	\begin{align}
		\| w-z \|_k &= \| I_{k,j}(I_{j,i}(x))  - I_{k,j}(I_{j,i}(y)) \|_k \\
		&\geq \| T_{k,j}(T_{j,i}(x))  - T_{k,j}(T_{j,i}(y)) \|_k - CC_2\ve r_\ell  \\
		&\geq \frac{\| x-y\|_i}{2} - CC_2\ve r_\ell \geq \frac{r_\ell}{10N}. 
	\end{align}
	In either case $\|w-z\|_k \geq r_\ell/10N$. Since $w,z$ were arbitrary this completes the proof of \eqref{e:separation} 
	
	Choosing again $\ve$ small enough depending on $C_2$ and $N$, we can apply Lemma \ref{l:modification}. Let $\hat{H}$ be the resulting modification of $H$ satisfying \eqref{e:H1}-\eqref{e:H4}. Define
	\[ \hat{I}_{k,j} = \hat{H} \circ I_{k,j} \mbox{ and } \hat{I}_{j,k} = \hat{I}_{k,j}^{-1}. \]
	Let us check \eqref{e:I1}-\eqref{e:I4}, beginning with \eqref{e:I1}. By definition of $U_2$ (see \eqref{d:U_2}) if $x \in I_{j,i}(\lambda B_i)^c$ then $I_{k,j}(x) \in U_2^c.$ Equation \eqref{e:I1} now follows from \eqref{e:H1} and the definition of $\hat{I}_{k,j}$ above. For \eqref{e:I2}, suppose that $y \in \lambda B_i$ is such that $I_{k,j}(I_{j,i}(y)) = I_{k,i}(y).$ Applying $H$ to both sides, and recalling that $I_{a,b}^{-1} = I_{b,a},$ we get $H(I_{k,i}(y)) = I_{k,i}(y)$ so that $\hat{H}(I_{k,i}(y)) = I_{k,i}(y)$ by \eqref{e:H3}. This gives  
	\begin{align}
		I_{k,i}(y) = \hat{H} \circ I_{k,i}(y) = \hat{H} \circ I_{k,j} \circ I_{j,i}(y) = \hat{I}_{k,j} \circ I_{j,i}.
	\end{align}
	Equation \eqref{e:I4} follows immediately from \eqref{e:H4}, Lemma \ref{l:invertible} and Lemma \ref{l:composition} as long as $C_1$ is chosen sufficiently large.  For \eqref{e:I3}, let $\lambda''' = \lambda - \tfrac{3}{4N}$ and let $z \in B_i(0,\lambda'''r_\ell)$ such that $I_{j,i}(z) \in B_j(0,\lambda''' r_\ell)$ and $\hat{I}_{k,j}(I_{j,i}(z)) \in B_k(0,\lambda''' r_\ell).$ By \eqref{e:I4}, $I_{k,j}(I_{j,i}(z)) \in B_k(0,\lambda'r_\ell)$ so in fact $I_{k,j}(I_{j,i}(z)) \in U_1.$ Then, \eqref{e:H2} implies  
	\begin{align}
		\hat{I}_{k,j}(I_{j,i}(z)) = \hat{H}(I_{k,j}(I_{j,i}(z))) = H(I_{k,j}(I_{j,i}(z))) = I_{k,i}(z).
	\end{align} 
	We now get \eqref{e:I3} from Lemma \ref{l:check-cyc} since $\lambda''' - \tfrac{1}{4N} = \lambda - \tfrac{1}{N}.$ 
\end{proof}

Before iterating Lemma \ref{l:single-mod}, we need one more result and a little more notation. 

\begin{lem}\label{l:families}
	There exists $\bigconst_0 = \bigconst_0(n)$ and a partition of $J_\ell$ into families $\{J^m_\ell\}_{m=1}^{\bigconst_0}$ such that 
	\begin{align}\label{e:large-sep}
		d(x_i,x_j) \geq 100r_\ell
	\end{align}
	for all distinct $i,j \in J^m_\ell$. For $i \in J_\ell,$ let $m(i)$ be such that $i \in J_\ell^{m(i)}.$ If $(i,j),(i,k) \in \calF$ then $m(i) \neq m(j)$ and $m(j) \neq m(k).$ 
\end{lem}

\begin{proof}
	We start by proving the following claim. Let $\bigconst \geq 1$ and suppose we have defined sets $J_\ell^1, \dots, J_\ell^{\bigconst}$ satisfying \eqref{e:large-sep} and 
	\begin{align}\label{e:not-full}
		J_\ell \setminus \bigcup_{m=1}^{\bigconst} J_\ell^m \neq\emptyset.
	\end{align}
	Then, there is $\bigconst_0 = \bigconst_0(n)$ such that 
	\begin{align}\label{e:upper-m}
		\bigconst \leq \bigconst_0.
	\end{align}
	To prove the claim, pick some $j \in J_\ell \setminus \bigcup_{m=1}^{\bigconst} J_\ell^m.$ By maximality, for each $1 \leq m \leq \bigconst$ there exists $j_m \in J_\ell^m$ such that $d(x_j,x_{j_m}) \leq 100r_\ell,$ in particular, $x_{j_m} \in 100B^j.$ Let $z_{j_m} = \toward_j(x_{j_m}) \in 100B_j$. Since $\toward_j$ is a $\delta r_\ell$-GHA, Theorem \ref{t:Reif} (2) implies 
	\[ \|z_{j_m} - z_{j_{m'}} \|_j \geq r_\ell/2 \]
	for all $1 \leq m,m' \leq \bigconst.$ A standard volume argument in $\R^n_j$ then implies \eqref{e:upper-m}. 
	
	We return now to the construction of the partition, which we define iteratively. First, let $J_\ell^1$ be a maximal collection of indices in $J_\ell$ such that \eqref{e:large-sep} holds for all distinct $i,j \in J_\ell^1$. Let $\bigconst \geq 1$ and suppose we have defined sets $J_\ell^1, \dots, J_\ell^{\bigconst}$ each satisfying \eqref{e:large-sep}. If \eqref{e:not-full} holds then we let $J_\ell^{\bigconst+1}$ be a maximal collection of indices in $J_\ell \setminus \bigcup_{m=1}^{\bigconst} J_\ell^m$ such that \eqref{e:large-sep} holds for all distinct $i,j \in J_\ell^{\bigconst+1}$. Otherwise, we stop. Equation \eqref{e:upper-m} guarantees the process terminates with the correct bound on the numbers of sets in our partition. 
\end{proof}

\begin{rem}
	We will write $J^m$ for $J_\ell^m$ when the context is clear. 
\end{rem}

\begin{rem}\label{r:ve_0}
	Let 
	\begin{align}
		N_1 \coloneqq {\bigconst_0 \choose 3} \quad \mbox{ and } \quad \ve_1 \coloneqq \sup\{\ve_0(pC_1^p,N_0) : p \in \N, \ 1 \leq p \leq N_0\}
	\end{align} 
	with $\bigconst_0$ as in Lemma \ref{l:families} and $C_1,\ve_0$ as in Lemma \ref{l:single-mod}. Tracking the dependencies, it follows that $N_1 = N_1(n)$ and $\ve_1 = \ve_1(n).$ 
\end{rem}

Consider a finite sequence $\{D_p\}_{p=0}^{N_1}$ of 3-tuples defined as follows. First, let $D_0 = (3,2,1).$ Then, if $D_p = (a,b,c)$ has been defined for some $p \geq0,$ let 
\begin{align}
	D_{p+1}= 
	\begin{cases}
		(a,b,c+1) & \mbox{ if }  c < b-1;\\
		(a,b+1,1) & \mbox { if } c = b-1 \mbox{ and } b < a-1; \\
		(a+1,2,1) & \mbox{ if }  c = b-1 \mbox{ and } b = a-1. 
	\end{cases}
\end{align}

Now we can iterate Lemma \ref{l:single-mod}. 

\begin{lem}\label{l:I^p}
	Suppose $0 < \ve \leq \ve_1$. Then, for each $0 \leq p \leq N_1$ and $(i,j) \in \calF$ there are smooth diffeomorphisms $I_{i,j}^p : \R^n_j \to \R^n_i$ such that $(I_{i,j}^p)^{-1} = I_{j,i}^p$ and 
	\begin{align}\label{e:I'-T'}
		\| I_{i,j}^p - T_{i,j} \|_{C^2(\R^n_j),r_\ell} \leq pC_1^p \ve. 
	\end{align}
	Furthermore, suppose $(i,j),(j,k),(i,k) \in \calF$ are such that $m(i) < m(j) < m(k)$. If $(m(k),m(j),m(i)) = D_{p'}$ for some $1 \leq p' \leq p  <N_1$ then $I_{i,j}^p,I_{j,k}^p,I_{i,k}^p$ are $\lambda_p r_\ell$-cyclic, where $\lambda_p = 10 - \tfrac{p}{N_1}.$
\end{lem}

\begin{proof}
	We prove the lemma by induction. For each $(j,k) \in \calF$ let 
	\[I_{j,k}^0 = T_{j,k}.\]
	The case $p = 0$ is then immediate. For the inductive step, suppose that we have established the result for some $0 \leq p < N_1.$ Let 
	\[D_{p+1} = (a_3,a_2,a_1).\]
	We will produce the maps $I_{j,k}^{p+1}$ by modifying the maps $I_{j,k}^p.$ However, we will only modify a map $I_{j,k}^{p}$ if $m(j) = a_2$, $m(k) = a_3$ and there exists $i \in J$ with $m(i) = a_1$ such that $(i,j),(i,k) \in \calF.$ Otherwise we simply set $I_{j,k}^{p+1}  = I_{j,k}^p.$ Let us describe how we perform the modification in this first case. Fix indices $i,j,k \in J$ as above. 
	By \eqref{e:I'-T'} in the case $p$, the maps $I_{i,j}^p,I_{j,k}^p,I_{i,k}^p$ satisfy the hypotheses of Lemma \ref{l:single-mod} with $C_2 = pC_1^p$ and $\lambda = \lambda_p.$ By assumption we have $\ve \leq \ve_1$ which by Remark \ref{r:ve_0} implies $\ve \leq \ve_1(C_2,N_1).$ We can now apply Lemma \ref{l:single-mod} with the constant $C_2$ and $\lambda$ as above to get maps $I_{j,k}^{p+1}$ and $I_{k,j}^{p+1}$. \\
	
	By construction, for each $(i,j) \in \calF$ we have $(I_{i,j}^{p+1})^{-1} =I_{j,i}^{p+1}$ and, by \eqref{e:I4}, \eqref{e:I'-T'} and the triangle inequality, 
	\begin{align}
		\begin{split}\label{e:hyp-ell}
		\| I_{i,j}^{p+1} - I^p_{i,j} \|_{C^2(\R^n_j),r_\ell} &\leq pC_1^{p+1}\ve ; \\
		\| I_{i,j}^{p+1} - T_{i,j} \|_{C^2(\R^n_j),r_\ell} &\leq (p+1)C_1^{p+1} \ve.
	\end{split}
	\end{align}
	That is, \eqref{e:I'-T'} holds in the case $p+1.$ We are left to check the cyclic condition. Let $(i,j),(j,k),(i,k) \in \calF$ such that $m(i) < m(j) < m(k)$ and $(m(k),m(j),m(i)) = D_{p'}$ for some $1 \leq p' \leq p+1.$ Since we only modified maps $I_{j,k}^p$ and $I_{k,j}^p$ with $m(j) = a_2$ and $m(k) = a_3$ we only need to consider $p'$ with $D_{p'} = (a_3,a_2,b_1)$ for some $1 \leq b_1 \leq a_1.$ The remaining cyclic relations are true by hypothesis. If $b_1 = a_1$ then $I_{i,j}^{p+1},I_{j,k}^{p+1},I_{i,k}^{p+1}$ are $\lambda_{p+1} r_\ell$-cyclic by Lemma \ref{l:single-mod}. 
	
	It only remains to consider the case $b_1 < a_1.$ Let $\lambda = \lambda_p - 1/(2{N_1})$ and suppose $x \in \lambda B_i$ is such that $I_{j,i}^{p+1}(x) \in \lambda B_j$ and $I_{k,j}^{p+1}(I_{j,i}^{p+1}(x)) \in \lambda B_k.$ Since $\lambda <  \lambda_p$ and $I_{j,i}^{p+1} = I_{j,i}^p$ (we didn't modify these maps at this stage), we have 
	\begin{align}\label{e:Il1}
		x \in \lambda_p B_i \mbox{ and } I_{j,i}^p(x) \in \lambda_p B_j.
	\end{align}
	Additionally, if $\ve$ small enough depending on $C_1$ and ${N_1},$ \eqref{e:hyp-ell} gives 
	\begin{align}\label{e:Il2}
		I_{k,j}^p(I_{j,i}^p(x)) = I_{k,j}^p(I_{j,i}^{p+1}(x))\in \lambda_p B_k.
	\end{align}
	The inductive hypothesis implies $I_{k,j}^p(I_{j,i}^{p}(x))  = I_{k,i}^p(x),$ so, if we can show
	\begin{align}\label{e:ell+1-cy}
		I_{k,j}^{p+1}(I_{j,i}^{p}(x)) = I_{k,j}^p(I_{j,i}^{p}(x)),
	\end{align}
	then using again that $I_{j,i}^{p+1} = I_{j,i}^p$ and $I_{i,k}^{p+1} = I_{i,k}^p$ (we also didn't modify this map), we have
	\begin{align}
		I_{k,j}^{p+1}(I_{j,i}^{p+1}(x)) = I_{k,j}^{p+1}(I_{j,i}^{p}(x)) = I_{k,j}^p(I_{j,i}^{p}(x)) = I_{k,i}^p(x) = I_{k,i}^{p+1}(x). 
	\end{align}
	The full cyclic relation follows from this and Lemma \ref{l:check-cyc}. 
	
	Thus, our goal for the rest of the proof is to show \eqref{e:ell+1-cy}. By construction, the map $I_{k,j}^{p+1}$ differs from $I_{k,j}^{p}$ only if there exists $i' \in J$ with $m(i') = a_1$ such that $(i',j),(i',k) \in \calF.$ We restrict our attention to this case. Let $y = I_{j,i}^{p}(x)$. If $y \in I_{j,i'}^p(\lambda B_{i'})^c \subseteq I_{j,i'}^p(\lambda_{p+1} B_{i'})^c,$ then $I_{k,j}^{p+1}(y) = I_{k,j}^p(y)$ by \eqref{e:I1}. This gives exactly \eqref{e:ell+1-cy}. Suppose instead that there exists $z \in \lambda B_{i'}$ such that $I_{j,i'}^p(z) = y = I_{j,i}^p(x).$ Since $\lambda < \lambda_{p},$ we have $z \in \lambda_{p} B_{i'}$. By \eqref{e:Il1} and \eqref{e:Il2} we also have $I_{j,i'}(z) = y \in \lambda_{p} B_j$ and $I_{k,j}^p(I_{j,i'}(z)) \in \lambda_p B_k.$ It follows from the induction hypothesis that
	\begin{align}
		I_{k,j}^p(I_{j,i'}^p(z)) = I_{k,i'}^p(z). 
	\end{align}
	Then, by \eqref{e:I2},
	\begin{align}
		I_{k,j}^{p+1}(I_{j,i}^p(x)) = I_{k,j}^{p+1}(I_{j,i'}^p(z)) = I_{k,j}^p(I_{j,i'}(z)) = I_{k,j}^p(I_{j,i}(x)).
	\end{align}
	This establishes \eqref{e:ell+1-cy} and closes the induction.

\end{proof}

\begin{proof}[Proof of Proposition \ref{p:coherent}]
	Let $\delta_1$ be as in Lemma \ref{l:T-alpha} and $\ve_1$ be as in Remark \ref{r:ve_0}. For $(i,j) \in \calF_\ell$ we set $\tilde{I}_{i,j,\ell} = I_{i,j,\ell}^{N_1}.$ Equation \eqref{e:prop'} and the cyclic properties are immediate from Lemma \ref{l:I^p}, using the fact that $N_1 = N_1(n).$ Equation \eqref{e:prop} follows from \eqref{e:prop'} and \eqref{e:T-alpha_i}.  
\end{proof}

\bigskip

\subsection{Construction of the (possibly) non-connected manifolds $W_\ell$}\label{s:disconnected}
In this subsection we construct a sequence of Finsler manifolds $W_\ell$ whose transition maps will be a suitable restriction of those constructed in the previous section. Let $\ve_1, \delta_1 > 0$ be the constants appearing in Proposition \ref{p:coherent} and suppose for the remainder of this subsection that $0 < \ve < \ve_1$ and $0 < \delta < \delta_1$. Additionally, fix some $\ell \geq 0.$ As in the previous section we will write $\calF = \calF_\ell$ when the context is clear. 

For $(i,j) \in \calF,$ let  
\begin{align}\label{e:defn-I}
	\Omega_{i,j} = \tilde{I}^{-1}_{i,j}(8B_i) \cap 8B_j, \mbox{ and } I_{i,j} = \tilde{I}_{i,j}|_{\Omega_{i,j}},
\end{align}
where $\tilde{I}_{i,j} = \tilde{I}_{i,j,\ell}$ is the map from Proposition \ref{p:coherent}. Since $\tilde{I}_{i,j}^{-1}$ is an open mapping (since it is a diffeomorphism) and both $8B_i$ and $8B_j$ are open in $\R^n$, we have that $\Omega_{i,j}$ is open in $\R^n.$ 

\begin{defn}
	For $j \in J_\ell$ let $\text{in}_j \colon 8B_j \to \bigsqcup_{j \in J_\ell} 8B_j$ be the natural injection. Define a relation $\sim$ on $\bigsqcup_{j \in J_\ell} 8B_j$ by declaring $x \sim y$ if $x = y$ or if there exists $i,j \in J_\ell$ with $i \neq j$ and $x' \in 8B_i, y' \in 8B_j$ such that $\text{in}_i(x') = x, \ \text{in}_j(y') =y$ and $y' = I_{j,i}(x').$ 
\end{defn}

\begin{lem}
	The relation $\sim$ is an equivalence relation.
\end{lem}

\begin{proof}
	Reflexivity and symmetry are immediate, so we only need to verify transitivity. Suppose $x,y,z \in \bigsqcup_{j \in J_\ell} 8B_j$ are such that $x \sim y$ and $y \sim z.$ Without loss of generality, we may assume $x,y,z$ are distinct. Then, there exists distinct indices $i,j,k \in J_\ell$ and points $x' \in 8B_i, \ y' \in 8B_j, \ z' \in  8B_k$ such that $x = \text{in}_i(x'), \ y = \text{in}_j(y')$ and $z = \text{in}_k(z').$ By definition, we have 
	\[ y' = I_{j,i}(x') \in 8B_j \mbox{ and } z' = I_{k,j}(y') = I_{k,j} \circ I_{j,i}(x') \in 8B^i.\]
	By Proposition \ref{p:coherent}, the maps $\tilde{I}_{i,j},\tilde{I}_{j,k},\tilde{I}_{i,k}$ are $9r_\ell$-cyclic, so that $z = \tilde{I}_{k,i}(x)$. Since $z' \in 8B^i$ and $x' \in 8B^i$ this implies $x \in \Omega_{k,i}$, hence, $z' = I_{k,i}(x').$  By definition this means $x \sim z,$ as required.
\end{proof}

Let 
\begin{align}\label{e:W_ell}
	W_\ell = \bigsqcup_{j \in J_\ell} 8B_j / \sim .
\end{align}
Let $Q : \bigsqcup_{j \in J_\ell} 8B_j \to W_\ell$ be the quotient map induced by the equivalence relation $\sim$ and for each $j \in J_\ell$ let 
\[ q_j \coloneqq Q \circ \text{in}_j \colon 8B_j \to W_\ell. \] 
Let $\tau_\ell$ be the coarsest topology on $W_\ell$ such that each $q_j$ is continuous i.e. 
\[ \tau_\ell = \{ V \subseteq W_\ell : q_j^{-1}(V) \mbox{ is an open subset of } \R^n \mbox{ for all } j \in J_\ell\}.\]
By definition $\text{in}_j$ is injective and by construction $Q$ is injective on $\text{in}_j(8B_j)$. Thus, $q_j$ is injective and hence invertible. For $j \in J_\ell$ and $0 < \lambda \leq 8,$ let 
\begin{align}\label{e:U}
	\lambda U^j = q_j(\lambda B_j) \mbox{ and } \varphi_j = q_j^{-1} : 8U^j \to 8B_j.
\end{align}


\begin{figure}

	\tikzset{every picture/.style={line width=0.75pt}} 
	
	\begin{tikzpicture}[x=0.75pt,y=0.75pt,yscale=-1,xscale=1]
		
		\draw[name path=ML]   (112.67,330.67) -- (105.34,348.34) -- (87.67,355.67) -- (69.99,348.34) -- (62.67,330.67) -- (69.99,312.99) -- (87.67,305.67) -- (105.34,312.99) -- cycle ;
		\draw[name path=MR]   (452,327.09) -- (446.05,341.25) -- (432.39,348.27) -- (417.42,344.87) -- (408.13,332.63) -- (408.88,317.3) -- (419.32,306.03) -- (434.55,304.1) -- (447.46,312.42) -- cycle ;
		\draw    (117.67,334) .. controls (167.83,384.16) and (357.16,365.71) .. (399.75,336.23) ;
		\draw [shift={(401,335.33)}, rotate = 145.13] [color={rgb, 255:red, 0; green, 0; blue, 0 }  ][line width=0.75]    (10.93,-4.9) .. controls (6.95,-2.3) and (3.31,-0.67) .. (0,0) .. controls (3.31,0.67) and (6.95,2.3) .. (10.93,4.9)   ;
		\draw    (114.33,299.33) .. controls (146.67,259.73) and (360.01,269.14) .. (403.73,297.79) ;
		\draw [shift={(114.33,299.33)}, rotate = 315.8] [color={rgb, 255:red, 0; green, 0; blue, 0 }  ][line width=0.75]    (10.93,-4.9) .. controls (6.95,-2.3) and (3.31,-0.67) .. (0,0) .. controls (3.31,0.67) and (6.95,2.3) .. (10.93,4.9)   ;
		\draw    (75,481.67) .. controls (67.08,474.41) and (30.41,422.72) .. (68.49,392.58) ;
		\draw [shift={(69.67,391.67)}, rotate = 140.13] [color={rgb, 255:red, 0; green, 0; blue, 0 }  ][line width=0.75]    (10.93,-4.9) .. controls (6.95,-2.3) and (3.31,-0.67) .. (0,0) .. controls (3.31,0.67) and (6.95,2.3) .. (10.93,4.9)   ;
		\draw    (418.1,494.3) .. controls (455.89,463.78) and (441.36,402.19) .. (421,392.34) ;
		\draw [shift={(416.33,495.67)}, rotate = 320.13] [color={rgb, 255:red, 0; green, 0; blue, 0 }  ][line width=0.75]    (10.93,-4.9) .. controls (6.95,-2.3) and (3.31,-0.67) .. (0,0) .. controls (3.31,0.67) and (6.95,2.3) .. (10.93,4.9)   ;
		\draw [name path=BR]  (284.33,449.67) -- (344.33,457.01) -- (395,479.67) -- (393,498.34) -- (380.33,514.34) -- (371,518.67) -- (354.33,511.34) -- (320.33,519.34) -- (304.33,502.01) -- (295.67,519.34) -- (258.33,523.34) -- (232.33,516.01) -- (211.67,506.67) -- (216.33,486.01) -- (212.33,473.34) -- (235.67,462.67) -- (249,451.34) -- cycle ;
		\draw[name path=BL]   (162.33,454.34) -- (222.33,461.67) -- (273,484.34) -- (271,503.01) -- (258.33,519.01) -- (249,523.34) -- (232.33,516.01) -- (198.33,524.01) -- (182.33,506.67) -- (173.67,524.01) -- (136.33,528.01) -- (110.33,520.67) -- (89.67,511.34) -- (94.33,490.67) -- (90.33,478.01) -- (113.67,467.34) -- (127,456.01) -- cycle ;

		\draw[ fill={rgb, 255:red, 74; green, 144; blue, 226 }  ,fill opacity=1 , name intersections={of=BL and BR}] (intersection-1) -- (273,484.34)  -- (271,503.01) -- (258.33,519.01) -- (intersection-2) -- (intersection-3) -- (211.67,506.67) -- (216.33,486.01) -- (212.33,473.34) -- cycle ;
		
		\draw[name path=TL]  (173.67,40.34) -- (233.67,47.67) -- (284.33,70.34) -- (282.33,89.01) -- (269.67,105.01) -- (260.33,109.34) -- (248.33,119.67) -- (211,125.67) -- (200.33,129.67) -- (183,130.34) -- (147.67,114.01) -- (121.67,106.67) -- (101,97.34) -- (97,76.34) -- (101.67,64.01) -- (120.33,47.01) -- (138.33,42.01) -- cycle ;
		\draw[name path=TR]   (286.33,35.67) -- (346.33,43.01) -- (397,65.67) -- (395,84.34) -- (382.33,100.34) -- (373,104.67) -- (361,115.01) -- (323.67,121.01) -- (313,125.01) -- (295.67,125.67) -- (260.33,109.34) -- (234.33,102.01) -- (213.67,92.67) -- (209.67,71.67) -- (214.33,59.34) -- (233,42.34) -- (251,37.34) -- cycle ;

		\draw[fill={rgb, 255:red, 74; green, 144; blue, 226 }  ,fill opacity=1 , name intersections={of=TR and TL}] (intersection-2) -- (233.67,47.67) -- (284.33,70.34) -- (282.33,89.01) -- (269.67,105.01) -- (intersection-1) -- (234.33,102.01) -- (213.67,92.67) -- (209.67,71.67) -- (214.33,59.34) -- cycle;

		\draw    (77.5,291) .. controls (52.75,261.3) and (55.44,133.59) .. (86.55,99.02) ;
		\draw [shift={(87.5,98)}, rotate = 129.12] [color={rgb, 255:red, 0; green, 0; blue, 0 }  ][line width=0.75]    (10.93,-3.29) .. controls (6.95,-1.4) and (3.31,-0.3) .. (0,0) .. controls (3.31,0.3) and (6.95,1.4) .. (10.93,3.29)   ;
		\draw    (400.5,102) .. controls (440.08,141.6) and (454.17,253.73) .. (419.55,287.99) ;
		\draw [shift={(418.5,289)}, rotate = 312.68] [color={rgb, 255:red, 0; green, 0; blue, 0 }  ][line width=0.75]    (10.93,-3.29) .. controls (6.95,-1.4) and (3.31,-0.3) .. (0,0) .. controls (3.31,0.3) and (6.95,1.4) .. (10.93,3.29)   ;
		
		\draw  [name path = LL, color={rgb, 255:red, 74; green, 144; blue, 226 }  ,draw opacity=1 ][line width=0.75] [line join = round][line cap = round] (98.65,310) .. controls (94.63,312) and (94.91,318.34) .. (94,322) .. controls (89.16,341.36) and (101,340.48) .. (101,350.25) ;
		
		\draw[name intersections={of=ML and LL} , fill={rgb, 255:red, 74; green, 144; blue, 226 }]  (intersection-1) -- (105.34,348.34)-- (112.67,330.67)-- (105.34,312.99) --(intersection-2)(98.65,310) .. controls (94.63,312) and (94.91,318.34) .. (94,322) .. controls (89.16,341.36) and (101,340.48) .. (101,350.25) ;

		\draw[name path = RL, color={rgb, 255:red, 74; green, 144; blue, 226 }  ,draw opacity=1 ][line width=0.75] [line join = round][line cap = round] (420,305.8) ..  controls (423.27,314.53) and (426.88,319.06) .. (426,327) .. controls (425.37,332.67) and (417,337.42) .. (414.5,341) ;

		\draw[name intersections={of=MR and RL}, fill={rgb, 255:red, 74; green, 144; blue, 226 } ] (intersection-1) -- (408.13,332.63) -- (408.88,317.3) -- (419.32,306.03) -- (intersection-2) (420,305.8) ..  controls (423.27,314.53) and (426.88,319.06) .. (426,327) .. controls (425.37,332.67) and (417,337.42) .. (414.5,341) ;

		\draw (113.33,307.4) node [anchor=north west][inner sep=0.75pt]    {$\Omega _{i,j}$};
		\draw (380.33,307.4) node [anchor=north west][inner sep=0.75pt]    {$\Omega _{j,i}$};
		\draw (402,360.07) node [anchor=north west][inner sep=0.75pt]    {$8B_{i,\ell } \subset \mathbb{R}^{n}$};
		\draw (45.67,358.4) node [anchor=north west][inner sep=0.75pt]    {$8B_{j,\ell } \subset \mathbb{R}^{n}$};
		\draw (210,342.07) node [anchor=north west][inner sep=0.75pt]    {$( \beta _{j} \circ \alpha _{i})^{-1}$};
		\draw (229.33,278.07) node [anchor=north west][inner sep=0.75pt]    {$I_{i,j}$};
		\draw (240.67,310.35) node [anchor=north west][inner sep=0.75pt]   [align=left] {\eqref{e:prop}};
		\draw (400,30.74) node [anchor=north west][inner sep=0.75pt]    {$8U^{i,\ell }{} \subset W_{\ell}$};
		\draw (238.63,311.35) node [anchor=north west][inner sep=0.75pt]  [rotate=-90.17]  {$\approx $};
		\draw (33,30.74) node [anchor=north west][inner sep=0.75pt]    {$8U^{j,\ell} \subset W_{\ell }$};
		\draw (75.33,172.41) node [anchor=north west][inner sep=0.75pt]    {$\varphi _{j}^{-1}$};
		\draw (398.67,180.41) node [anchor=north west][inner sep=0.75pt]    {$\varphi _{i}$};
		\draw (31.67,521.4) node [anchor=north west][inner sep=0.75pt]    {$8B^{j,\ell } \subset X$};
		\draw (373,522.07) node [anchor=north west][inner sep=0.75pt]    {$8B^{i,\ell } \subset X$};
		\draw (58.67,421.41) node [anchor=north west][inner sep=0.75pt]    {$\beta _{j}$};
		\draw (416,427.41) node [anchor=north west][inner sep=0.75pt]    {$\alpha _{i}$};

	\end{tikzpicture}
	
	\caption{The manifold $W_\ell$}
	\label{f:figure-for-5.2}
\end{figure}
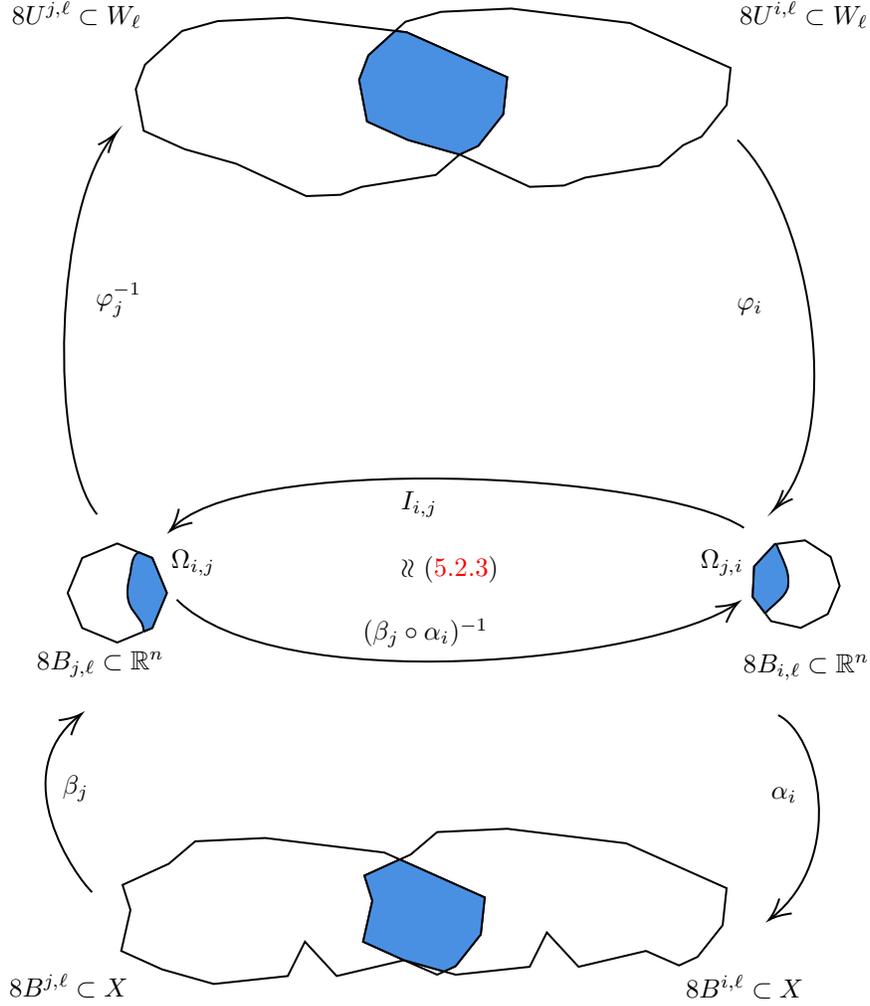

\begin{rem}\label{l:W_0}
	Recalling from Theorem \ref{t:Reif} (1) that $J_0 = \{j_0\},$ we simply have $W_0 = 8B_{j_0}$ and $q_{j_0} = \text{Id}.$  
\end{rem}

\begin{rem}\label{r:transition'}
	It follows by construction that if $i,j \in J_\ell$ and $x = \vp_j(w)$ for some $w \in 8U^i \cap 8U^j$ then $x \in \Omega_{j,i}.$ Furthermore, for each $x \in \Omega_{i,j}$ we have
	\begin{align}\label{r:transition}
		I_{i,j}(x) = q_i^{-1}(q_j(x)) = \varphi_i(\varphi_j^{-1}(x)).
	\end{align}
\end{rem}

\begin{lem}
	The space $(W_\ell,\tau_\ell)$ is a topological manifold i.e. a locally Euclidean Hausdorff space. The collection $\{(8U^j,\varphi_j) : j \in J\}$ is a smooth atlas for $W_\ell$. 
\end{lem}

\begin{proof}
	To prove $W_\ell$ is locally Euclidean, it suffices to check $8U^j$ is open in $W_\ell$ and $q_j : 8B_j \to 8U^j$ is a homeomorphism for each $j \in J_\ell$. This follows if we show $q_j(V) \in \tau_\ell$ for each open $V \subseteq 8B_j$, which is equivalent to showing that $q_i^{-1}(q_j(V))$ is open in $\R^n$ for all $i \in J$. Fix such an open $V$ and $i \in J.$ If $i=j,$ then $q_j^{-1}(q_j(V)) = V$ which is open by assumption. If $i\neq j,$ by \eqref{e:defn-I} and \eqref{r:transition} we have 
	\begin{align}
		q_i^{-1}(q_j(V)) = I_{i,j}(V) = \tilde{I}_{i,j}(V \cap \Omega_{i,j}). 
	\end{align}
	Since $V$ and $\Omega_{i,j}$ are open in $\R^n,$ and $\tilde{I}_{i,j}$ is an open mapping (since it is a diffeomorphism from $\R^n$ onto itself), we have that $q_i^{-1}(q_j(V))$ is open.
	
	Now, let us check that $W_\ell$ is Hausdorff. Suppose $x,y \in W_\ell.$ If $x,y \in 8U^{j}$ for some $j \in J_\ell$ then it is easy to find disjoint open sets $V_x \ni x$, $V_y \ni y$ in $W_\ell$ by using the Hausdorff property of $8B_{j}.$ Suppose then that there exist $i,j \in J_\ell$ such that $x \in 8U^{i} \setminus 8U^j$ and $y \in 8U^j \setminus 8U^i.$ Since $8U^i$ and $8U^i \setminus \overline{8U^i}$ are open in $W_\ell$, if $y \not \in \partial(8U^i)$ then again there exist open sets $V_x \ni x$, $V_y \ni y$ in $W_\ell$, so we may support further that $y \in 8U^j \cap \partial(8U^i).$  In this case, let $\tilde{x} = \psi_i(x), \ \tilde{y} = \psi_j(y),$ 
	\[ \eta = \max\{ {\dist}_i(\tilde{x},\partial(8B_i)) , {\dist}_i(\tilde{y},\partial(8B_i)) \} /4,\]
	and set 
	\begin{align}
		V_x = \psi_i^{-1}(B_{\|\cdot\|_i}(\tilde{x},\eta)) \quad \mbox{ and } \quad V_y = \psi_j^{-1}(B_{\|\cdot\|_j}(\tilde{y},\eta)). 
	\end{align}
	It is clear that $V_x \ni x$, $V_y \ni$ and that both sets are open in $W_\ell.$ It only remains to check that they are disjoint. Suppose towards a contradiction that there exists $z \in V_x \cap V_y.$ By Remark \ref{r:transition'} we have $\psi_j(z) \in B_{\|\cdot\|_j}(\tilde{y},\eta) \cap  \Omega_{i,j}$ and $\tilde{I}_{i,j}(\psi_j(z))  = I_{i,j}(\psi_j(z)) \in B_{\|\cdot\|_i}(\tilde{x},\eta)$. Since $\tilde{I}_{i,j}$ is $(1+C\ve)$-bi-Lipschitz, this implies 
	\begin{align}\label{e:close-boundary}
		\| \tilde{x} - \tilde{I}_{i,j}(\tilde{y}) \|_i &\leq \eta + \| \tilde{I}_{i,j}(\psi_j(z)) -  \tilde{I}_{i,j}(\tilde{y}) \|_i \\
		&\leq \eta + (1+C\ve) \| \psi_j(z) - \tilde{y} \|_j \leq 3\eta \leq 3\dist(\tilde{x},\partial(8B_i))/4.
	\end{align}
	However, since $y \in 8U^j \cap \partial(8U^i)$ and $\psi_j$ is continuous we have $\tilde{y} \in 8B_j \cap \partial(\Omega_{i,j}).$ Moreover, since $\tilde{I}_{i,j}$ is continuous, we have $\tilde{I}_{i,j}(\tilde{y}) \in \partial(8B_i)$ so that $ \| \tilde{x} - \tilde{I}_{i,j}(\tilde{y}) \|_i \geq \dist(\tilde{x},\partial(8B_i)).$ This contradicts \eqref{e:close-boundary}. 
	
	The fact that $\{(8U^j,\varphi_j) : j \in J\}$ is a smooth atlas for $W_\ell$ is a smooth atlas for $W_\ell$ follows immediately from \eqref{r:transition} and the fact that $I_{i,j}$ is smooth by Proposition \ref{p:coherent}. 
\end{proof}

From now on we view $W_\ell$ as equipped with the above smooth structure. In particular, if $x \in W_\ell$ then the tangent space $T_xW_\ell$ is well-defined. We equip $W_\ell$ with a Finsler metric by defining, for each $x \in W_\ell,$ a norm $\|\cdot\|_{W_\ell,x}$ on $T_xW_\ell$ as follows. Let $\{\theta_{j,\ell}\}_{j \in J_\ell}$ be a partition of unity subordinate to $\{8U^j\}_{j \in J_\ell}.$ Then, for each $v \in T_xW_\ell,$ let 
\begin{align}
	\|v\|_{W_{\ell},x} = \sum_{j \in J_\ell} \theta_{j,\ell}(x)\| D\vp_j(x) \cdot v\|_j. 
\end{align}
We equip $T_xW_\ell$ with this norm so that the following is true.

\begin{lem}\label{l:Dpsi}
	Let $j \in J_\ell$ and $x \in 8U^j.$ The map $D\varphi_j(x) \colon T_xW_\ell \to T_{\varphi_j(x)}\R_j^n$ is $(1+C\ve)$-bi-Lipschitz. 
\end{lem}

\begin{proof}
	Let $v \in T_x  W_{\ell}.$ Suppose $i \in J_\ell$ is such that $\theta_{i,\ell}(x) \neq 0$. Since $\supp(\theta_{i,\ell}) \subseteq 8U^j$ we have $x \in 8U^i \cap 8U^j$ so that $\vp_i(x) = I_{i,j}(\vp_j(x))$ by \eqref{r:transition}. Then, 
	\begin{align}
		D\varphi_j(x) = D[I_{i,j} \circ \varphi_j](x) = DI_{i,j}(\varphi_j(x)) \circ D\varphi_j(x). 
	\end{align}
	Since $I_{i,j}$ is $(1+C\ve)$-bi-Lipschitz, this gives 
	\[(1+C\ve)^{-1} \|D\varphi_{j}(x)\cdot v \|_j \leq \| D\varphi_i(x) \cdot v\|_i \leq (1+C\ve) \|D\varphi_{j}(x)\cdot v \|_j.\]
	Since $\sum_{j \in J_\ell}\theta_{j,\ell}(x) = 1,$ this implies 
	\[(1+C\ve)^{-1}\|v\|_{W_{\ell},x} \leq  \|D\varphi_j(x) \cdot v\|_j \leq (1+C\ve)\|v\|_{W_\ell,x}, \]
	which completes the proof. 
\end{proof}

\bigskip

\subsection{Smooth maps between manifolds}\label{s:smooth} In this section we construct a sequence of smooth mappings $f_\ell \colon E_\ell \to W_{\ell+1}$ for some suitable subset $E_\ell \subseteq W_{\ell}.$ Let's give a rough description of how this is done. Suppose $j \in J_{\ell+1}$. Similar to Lemma \ref{l:T-alpha}, the flatness condition implies the tangent spaces corresponding to the indices $j$ and $i(j)$ (as in Theorem \ref{t:Reif} (3)) are roughly compatible i.e. there exists a smooth $(1+C\ve)$-bi-Lipschitz map $K_{i(j),j}$ between them which well-approximates the rough transition maps given by $X$ (Lemma \ref{l:K-alpha}). Pushing $8U^{j,\ell+1} \subseteq W_{\ell+1}$ through the map
\begin{align}\label{e:discuss}
	 \vp^{-1}_{i(j),\ell}  \circ K_{i(j),j} \circ\vp_{j,\ell+1}
\end{align}
we find a corresponding region $C^{j,\ell+1}_0$ in $W_{\ell}$. We contract each of these regions slightly and take $E_\ell$ to be their union (see \eqref{e:h_ell1}). Let $h_{j,\ell+1} \colon C^{j,\ell+1}_0 \to 8U^{j,\ell+1}$ be the inverse of the mapping in \eqref{e:discuss}. For nearby patches $C^{i,\ell+1}_0$ and $C^{j,\ell+1}_0$, these mappings are small perturbations of each other (Lemma \ref{l:psi-almost}). A modification procedure similar to that in the previous section then produces a smooth mapping $h_\ell \colon E_\ell \to W_{\ell+1}$ so that, locally, $h_\ell$ is small perturbation of the preliminary map $h_{j,\ell+1}$ (Lemma \ref{l:partial-h}, see in particular (3)).

Before stating the main proposition, we state some preliminary results and notation. For each $\ell \geq 0$ and $j \in J_{\ell+1}$ choose some $i(j) \in J_\ell$ as in Theorem \ref{t:Reif} (3). We recall that this means 
\begin{align}\label{e:i(j)}
	x_{j,\ell+1} \in 2B^{i(j),\ell}.
\end{align}
The first result is a slight modification of Lemma \ref{l:T-alpha}.

\begin{lem}\label{l:K-alpha}
	Let $\ell \geq 0$ and $j \in J_{\ell+1}$. Then there exists an affine $(1 + C\ve)$-bi-Lipschitz map $K_{i(j),j} : \R^n_j \to \R^n_{i(j)}$ such that 
	\begin{align}\label{e:K1}
		\| K_{i(j),j} - \toward_{i(j)} \circ \away_j(x) \|_{i(j)} \lesssim \ve r_\ell
	\end{align}
	for all $x \in 45B_j$. Additionally, if we let $K_{j,i(j)} = K_{i(j),j}^{-1},$ then 
	\begin{align}\label{e:K2}
		\| K_{j,i(j)}(x) - \toward_j \circ \away_{i(j)}(x) \|_j \lesssim \ve r_{\ell}
	\end{align}
	for all $x \in \toward_{i(j)}(99B^j).$ 
\end{lem}

\begin{proof}
	The proof is very similar to Lemma \ref{l:T-alpha}. By Theorem \ref{t:Reif} (3) we have $100B^j \subseteq 100B^{i(j)}.$ Let $z = \toward_{i(j)}(x_j).$ By the final part of Lemma \ref{l:GH-scale}, there exists a $C\delta r_\ell$-GHA $\phi_{i(j)} : 100B^j \to B_{\|\cdot\|_{i(j)}}(z,100r_{\ell+1})$ such that $\|\toward_{i(j)}(x) - \phi_{i(j)}(x)\|_{i(j)} \leq C\delta r_\ell$ for all $x \in 100B^j.$ By Lemma \ref{l:GH-comp}, 
	\[ \phi_{i(j),j} \coloneqq \phi_{i(j)} \circ \away_j : 100B_{j} \to B_{\|\cdot\|_{i(j)}}(z,100r_{\ell+1})\]
	defines a $C\delta r_\ell$-GHA. By Lemma \ref{l:GH-linear}, as long as $\delta$ is chosen small enough depending on $\ve,$ there exists an affine $(1+C\ve)$-bi-Lipschitz map $K_{i(j),j} : \R^n_j \to \R^n_{i(j)}$ such that 
	\[ \| K_{i(j),j}(x) - \phi_{i(j),j}(x) \|_i \leq \ve r_\ell \]
	for all $x \in 100B_{j}.$ 
	
	We prove \eqref{e:K1} for a slightly larger ball. Suppose $x \in 100B_{j}.$ By definition, $\away_j(x) \in 100B^j$. Taking $\delta$ small enough depending on $\ve$, and using the above relation between $\toward_{i(j)},\phi_{i(j)}$ on $100B^j,$ we have 
	\begin{align}
		\| K_{i(j),j}(x) - \toward_{i(j)} \circ \away_j(x)\|_{i(j)} &\leq \| K_{i(j),j}(x) - \phi_{i(j),j}(x) \|_{i(j)} \\
		&\hspace{2em}+ \| \phi_{i(j)} \circ \away_j(x) - \toward_{i(j)} \circ \away_j(x) \|_{i(j)} \\
		&\leq \ve r_\ell + C\delta r_\ell \lesssim \ve r_\ell. 
	\end{align}
	Let us take a look at \eqref{e:K2}. Suppose $x \in \toward_j(99B^j).$ By \eqref{e:almost-id} we have $\away_{i(j)}(x) \in 100B^j.$ By definition, we have $\toward_j \circ \away_{i(j)}(x) \in 100B_{j}.$ Then, using Lemma \ref{l:alpha-beta-circ}, \eqref{e:K1} and since $K_{j,i(j)}$ is $(1+C\ve)$-bi-Lipschitz, the left-hand side of \eqref{e:K2} is at most 
	\begin{align}
		&\| x - K_{i(j),j} \circ \toward_j \circ \away_{i(j)}(x)\|_{i(j)} \\
		&\hspace{4em} \leq \| x - \toward_{i(j)}\circ \away_j \circ \toward_j \circ \away_{i(j)}(x) \|_{i(j)} \\
		&\hspace{6em}+ \|  \toward_{i(j)}\circ \away_j \circ \toward_j \circ \away_{i(j)}(x) -  K_{{i(j)},j} \circ \toward_j \circ \away_{i(j)}(x) \|_{i(j)} \lesssim \ve r_\ell.
	\end{align}
\end{proof}


\begin{figure}

	\tikzset{every picture/.style={line width=0.75pt}} 
	
	\begin{tikzpicture}[x=0.75pt,y=0.75pt,yscale=-1,xscale=1]
		
		\draw [shift={(15,-15)}]  (173,329.5) -- (164.99,348.79) -- (149.31,363.65) -- (129.55,370.7) -- (110.24,368.3) -- (95.79,357.02) -- (89.53,339.43) -- (92.88,319.57) -- (105.08,301.98) -- (123.33,290.7) -- (143.45,288.3) -- (160.84,295.35) -- (171.5,310.21) -- cycle ;
		\draw    (133,142.5) .. controls (93.6,171.02) and (110.29,230.08) .. (139.92,257.29) ;
		\draw [shift={(141,258.5)}, rotate = 227.29] [color={rgb, 255:red, 0; green, 0; blue, 0 }  ][line width=0.75]    (10.93,-4.9) .. controls (6.95,-2.3) and (3.31,-0.67) .. (0,0) .. controls (3.31,0.67) and (6.95,2.3) .. (10.93,4.9)   ;
		\draw    (498,297) .. controls (537.6,267.3) and (547.8,212.61) .. (514.04-12,169.79-10) ;
		\draw [shift={(513-12,168.5-10)}, rotate = 50.86] [color={rgb, 255:red, 0; green, 0; blue, 0 }  ][line width=0.75]    (10.93,-4.9) .. controls (6.95,-2.3) and (3.31,-0.67) .. (0,0) .. controls (3.31,0.67) and (6.95,2.3) .. (10.93,4.9)   ;
		\draw    (205,325) .. controls (240.64,359.16) and (403.2,377.63) .. (449-10.63,350.34-8) ;
		\draw [shift={(451-10,349.5-8)}, rotate = 145.36] [color={rgb, 255:red, 0; green, 0; blue, 0 }  ][line width=0.75]    (10.93,-4.9) .. controls (6.95,-2.3) and (3.31,-0.67) .. (0,0) .. controls (3.31,0.67) and (6.95,2.3) .. (10.93,4.9)   ;
		\draw  [shift={(0,10)}]  (438,304) .. controls (294.45,285.69) and (174.42,215.91) .. (174.95,165.04) ;
		\draw [shift={(174.75,173.5)}, rotate = 85.37] [color={rgb, 255:red, 0; green, 0; blue, 0 }  ][line width=0.75]    (10.93,-4.9) .. controls (6.95,-2.3) and (3.31,-0.67) .. (0,0) .. controls (3.31,0.67) and (6.95,2.3) .. (10.93,4.9)   ;
		\draw    (210,133) .. controls (249.8,103.15) and (362.86,79.74) .. (436.89,126.79) ;
		\draw [shift={(438,127.5)}, rotate = 212.97] [color={rgb, 255:red, 0; green, 0; blue, 0 }  ][line width=0.75]    (10.93,-4.9) .. controls (6.95,-2.3) and (3.31,-0.67) .. (0,0) .. controls (3.31,0.67) and (6.95,2.3) .. (10.93,4.9)   ;
		\draw  [shift={(15,-15)}] [fill={rgb, 255:red, 74; green, 144; blue, 226 }  ,fill opacity=1 ] (155.32,322.22) -- (149.13,332.04) -- (138.92,338.11) -- (128.59,338.11) -- (122.08,332.04) -- (121.89,322.22) -- (128.09,312.39) -- (138.3,306.32) -- (148.63,306.32) -- (155.13,312.39) -- cycle ;
		\draw [shift={(0,119.93-94.93)}] [fill={rgb, 255:red, 74; green, 144; blue, 226 }  ,fill opacity=1 ] (188,113) -- (180.26,124.17) -- (168.21,131.07) -- (156.47,131.07) -- (149.51,124.17) -- (150,113) -- (157.74,101.83) -- (169.79,94.93) -- (181.53,94.93) -- (188.49,101.83) -- cycle ;
		\draw  [fill={rgb, 255:red, 74; green, 144; blue, 226 }  ,fill opacity=1 ] (492,138) -- (485.43,149.17) -- (474.12,156.07) -- (462.37,156.07) -- (454.69,149.17) -- (454,138) -- (460.57,126.83) -- (471.88,119.93) -- (483.63,119.93) -- (491.31,126.83) -- cycle ;
		\draw [shift={(-8,8)}] [fill={rgb, 255:red, 74; green, 144; blue, 226 }  ,fill opacity=1 ] (502,317) -- (494.07,328.76) -- (481.52,336.02) -- (469.16,336.02) -- (461.71,328.76) -- (462,317) -- (469.93,305.24) -- (482.48,297.98) -- (494.84,297.98) -- (502.29,305.24) -- cycle ;

		\draw (126,90.4) node [anchor=north west][inner sep=0.75pt]    {\eqref{e:C} $8C^{j,\ell +1} \subset W_{\ell}$};
		\draw (100,370.4) node [anchor=north west][inner sep=0.75pt]    {$8B_{i( j) ,\ell } \subset \mathbb{R}^{n}$};
		\draw (455,90.4) node [anchor=north west][inner sep=0.75pt]    {$8U^{j,\ell +1} \subset W_{\ell +1}$};
		\draw (455,339.42+18) node [anchor=north west][inner sep=0.75pt]    {$8B_{j,\ell +1} \subset \mathbb{R}^{n}$};
		\draw (321,68.4) node [anchor=north west][inner sep=0.75pt]    {$h_{j,\ell +1}$ \eqref{e:h}};
		\draw (69,183.4) node [anchor=north west][inner sep=0.75pt]    {$\varphi _{i( j) ,\ell }$};
		\draw (270,238.4) node [anchor=north west][inner sep=0.75pt]    {$\Phi _{j ,\ell +1}$ \eqref{e:psi}};
		\draw (286,371.4) node [anchor=north west][inner sep=0.75pt]    {$K_{j,i( j)}$ \eqref{e:K1}};
		\draw (485,219.4) node [anchor=north west][inner sep=0.75pt]    {$\varphi _{j,\ell +1}^{-1}$};

	\end{tikzpicture}
	
	\caption{The preliminary map $h_{j,\ell+1}$}
	\label{f:figure-for-5.4a}
\end{figure}
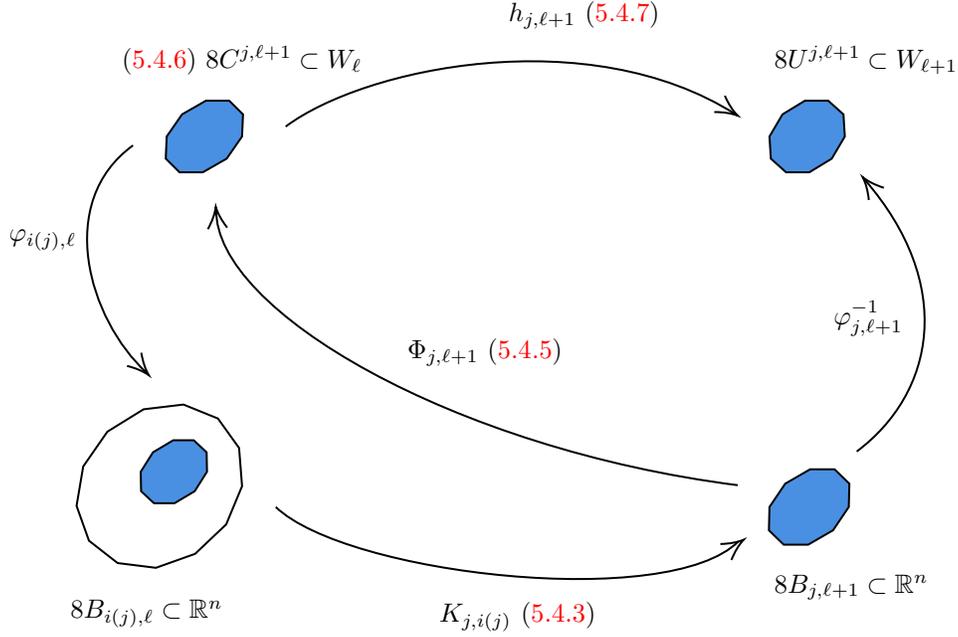

For each $\ell \geq 0$ and $j \in J_{\ell+1}$, let 
\begin{align}\label{e:psi}
	\Phi_{j,\ell+1}  = \varphi_{i(j),\ell}^{-1} \circ K_{i(j),j}|_{8B_{j,\ell+1}},
\end{align}
where $\varphi_{i(j),\ell}$ is the chart map from the previous section. Let $N_0 = N_0(n)$ be as in Lemma \ref{l:families}. For each integer $0 \leq m < 8N_0,$ let $\lambda_m = 8 - \tfrac{m}{N_0}$ and define 
\begin{align}\label{e:C}
	C_m^{j,\ell+1} = \Phi_{j,\ell+1}(\lambda_m B_{j,\ell+1}) \subseteq W_\ell.
\end{align}
Also, define the map
\begin{align}\label{e:h}
	h_{j,\ell+1} \coloneqq \varphi_{j,\ell+1}^{-1} \circ \Phi_{j,\ell+1}^{-1}|_{C_0^j} : C_0^{j,\ell+1} \to W_{\ell+1}.
\end{align}
and note that
\[ h_{j,\ell+1}(C_0^j) = 8U^{j,\ell+1} \]
by construction. See Figure \ref{f:figure-for-5.4a}. Finally, define the set 
\begin{align}\label{e:h_ell1} 
	E_\ell \coloneqq  \bigcup_{j \in J_{\ell+1}} C_{N_0}^{j,\ell+1}.
\end{align}


We can now state the main proposition of this section. See Figure \ref{f:figure-for-5.4b}. 

\begin{figure}[h!]

	\tikzset{every picture/.style={line width=0.75pt}} 
	
	\begin{tikzpicture}[x=0.75pt,y=0.75pt,yscale=-1,xscale=1]
		
		\draw[shift={(-30,+40)}]    (275.5,144) .. controls (305.1,114.3) and (366.84,153.21) .. (407.29,124.88) ;
		\draw [shift={(408.5-30,40+124)}, rotate = 147.13] [color={rgb, 255:red, 0; green, 0; blue, 0 }  ][line width=0.75]    (10.93,-3.29) .. controls (6.95,-1.4) and (3.31,-0.3) .. (0,0) .. controls (3.31,0.3) and (6.95,1.4) .. (10.93,3.29)   ;
		
		\draw   (137.79,181) -- (130.25,194.44) -- (116,200) -- (103.38,194.44) -- (99.79,181) -- (107.33,167.56) -- (121.58,162) -- (134.2,167.56) -- cycle ;
		\draw  [fill={rgb, 255:red, 74; green, 144; blue, 226 }  ,fill opacity=1 ] (159.66,200) -- (157.84,211.42) -- (149.28,218.48) -- (137.27,218.48) -- (126.39,211.42) -- (120.79,200) -- (122.61,188.58) -- (131.17,181.52) -- (143.18,181.52) -- (154.06,188.58) -- cycle ;
		\draw   (182.78,191.93) -- (178.87,203.98) -- (168.62,211.42) -- (155.95,211.42) -- (145.7,203.98) -- (141.78,191.93) -- (145.7,179.88) -- (155.95,172.43) -- (168.62,172.43) -- (178.87,179.88) -- cycle ;
		\draw   (205.51,183.84) -- (197.39,196.81) -- (182.48,200.25) -- (169.5,192.13) -- (166.07,177.21) -- (174.18,164.24) -- (189.1,160.8) -- (202.08,168.92) -- cycle ;
		\draw   (195.06,158.4) -- (187.86,176.33) -- (172.65,187.18) -- (158.33,184.59) -- (153.3,170.08) -- (160.51,152.14) -- (175.72,141.3) -- (190.03,143.89) -- cycle ;

		\draw[shift={(-50,+30)}]   [line width=0.75] [line join = round][line cap = round] (514,133) .. controls (516,128.97) and (514.51,126.97) .. (515,125) .. controls (517.54,114.83) and (542.04,97.01) .. (553,99) .. controls (564.62,101.11) and (564.37,112.87) .. (566,121) .. controls (568.37,132.85) and (576.73,143.68) .. (571,158) .. controls (566.8,168.5) and (560.08,168.15) .. (549,169) .. controls (548.61,169.03) and (544.97,177.27) .. (544,178) .. controls (536.53,183.6) and (527.07,179.25) .. (517.28,181.96) .. controls (516,184) and (514.67,186.1) .. (507.53,189.05) .. controls (494.09,191.16) and (485.9,186.41) .. (485,181) .. controls (484.75,179.53) and (483.92,172.32) .. (483,172) .. controls (478.88,171.31) and (475.82,172.09) .. (472,171) .. controls (466.69,169.48) and (457.65,157.41) .. (459,152) .. controls (461.93,140.3) and (484.12,119.12) .. (497,132) .. controls (501.41,136.41) and (499.95,144.95) .. (503,148) .. controls (507.26,152.26) and (518.42,144.31) .. (519,142) .. controls (520.9,134.42) and (514,137.65) .. (514,133) ;

		\draw [shift={(-50,+30)}] [fill={rgb, 255:red, 74; green, 144; blue, 226 }  ,fill opacity=1 ]  (521,170.5) -- (517.28,181.96) .. controls (516,184) and (514.67,186.1) .. (507.53,189.05) .. controls (494.09,191.16) and (485.9,186.41) .. (485,181) .. controls (484.75,179.53) and (483.92,172.32) .. (483,172) -- (485.72,159.04) -- (495.47,151.95) -- (507.53,151.95) -- (517.28,159.04) -- cycle ;

		\draw (71,129.4) node [anchor=north west][inner sep=0.75pt]    {\eqref{e:h_ell1} $E_{\ell } \subset W_{\ell }$};
		\draw (83,201.4) node [anchor=north west][inner sep=0.75pt]    {$C^{1,\ell+1}_{N_0}$};
		\draw (113,223.4) node [anchor=north west][inner sep=0.75pt]    {$C^{2,\ell+1}_{N_0}$};
		\draw (154,218.4) node [anchor=north west][inner sep=0.75pt]    {$C^{3,\ell+1}_{N_0}$};
		\draw (211-10,180.4+10) node [anchor=north west][inner sep=0.75pt]    {$C^{4,\ell+1}_{N_0}$};
		\draw (323-30,138.4) node [anchor=north west][inner sep=0.75pt]    {$h_{\ell }$ \eqref{e:h_ell2}};
		\draw (380,129.4) node [anchor=north west][inner sep=0.75pt]    {$W_{\ell+1}^* \subset W_{\ell +1}$};
		\draw (380,145.4) node [anchor=north west][inner sep=0.75pt] 
		{\eqref{e:h_ell3''}};
		\draw (410,223.4) node [anchor=north west][inner sep=0.75pt]    {$7U^{2,\ell}$};
		
	\end{tikzpicture}
	
	\caption{The map $h_\ell$}
	\label{f:figure-for-5.4b}
\end{figure}

\begin{prop}\label{p:h_ell}
	For each $\ell \geq 0$ there exists a diffeomorphism $h_\ell \colon E_\ell \to W_{\ell+1}$ such that $Dh_\ell(x) \colon T_xW_\ell \to T_{h_\ell(x)}W_{\ell+1}$
	is $(1+C\ve)$-bi-Lipschitz for each $x \in E_\ell.$ If $x \in C_{N_0}^{j,\ell+1}$ for some $j \in J_{\ell+1}$, then 
	\begin{align}\label{e:h_ell2}
		h_\ell(x) = \varphi_{j,\ell+1}^{-1} \circ \hat{H}_j \circ \Phi_{j,\ell+1}^{-1}(x) = \varphi_{j,\ell+1}^{-1} \circ \hat{H}_j \circ K_{j,i(j)} \circ \varphi_{i(j),\ell}(x)
	\end{align}
	for some diffeomorphism $\hat{H}_j \colon \R^n_j \to \R^n_j$ satisfying 
	\begin{align}\label{e:h_ell3}
		\|\hat{H}_j - \emph{Id}\|_{C^1,r_\ell} \lesssim \ve.
	\end{align}
	If $j \in J_{\ell+1}$ and $A \in \N$ such that $1 \leq A \leq 6$ then
	\begin{align}\label{e:h_ell3'}
		AU^{j,\ell+1} \subseteq h_\ell\left(C^{j,\ell+1}_{(7-A)N_0}\right) \subseteq h_\ell(C^{j,\ell+1}_{N_0}) \subseteq h_\ell(E_\ell \cap 3U^{i(j)}).
	\end{align}
	In particular,
	\begin{align}\label{e:h_ell3''}
		\bigcup_{j \in J_{\ell+1}} 6U^{j,\ell+1} \subseteq W_{\ell+1}^* \coloneqq  h_{\ell}(E_\ell).
	\end{align}
	Finally, we have 
	\begin{align}\label{e:h_ell4}
		E_{\ell+1} \subseteq \bigcup_{j \in J_{\ell+1}} 3U^{j,\ell+1}.
	\end{align}
\end{prop}

Fix $\ell \geq 0$ for the remainder of the subsection. When it is clear we will drop the subscript $\ell$ from our notation. Before proving the above proposition, we need some preliminary results.


\begin{lem}\label{l:C-U}
	For each $1 \leq m \leq N_0$ and $j,k \in J_{\ell+1}^m$ we have $C_0^j \cap C_0^k = \emptyset$ and $C^j_0 \subseteq 3U^{i(j)}.$ 
\end{lem}

\begin{proof}
	Fix  $1 \leq  m \leq N_0$ and $j,k \in J_{\ell+1}^m$ and suppose towards a contradiction that there exists $x \in C^j_0 \cap C^k_0.$ Let $y \in 8B_j$ and $z \in 8B_k$ such that $ x = \Phi_j(y) =\Phi_k(z)$. Recalling the definition of $\Phi$ from \eqref{e:psi}, this implies $\vp_{i(j)}^{-1} \circ K_{i(j),j}(y) = \vp_{i(k)}^{-1} \circ K_{i(k),k}(z).$ Applying $\vp_{i(j)}$ to both sides and recalling \eqref{r:transition}, we get 
	\begin{align}
		K_{i(j),j}(y) = I_{i(j),i(k)} \circ K_{i(k),k}. 
	\end{align}
	Using this with \eqref{e:almost-id}, Lemma \ref{l:T-alpha}, Lemma \ref{l:K-alpha}, and the fact that $\beta_{i(j)}$ is a $\delta r_\ell$-GHA and $I_{i(k)}$ is bi-Lipschitz, it is not difficult to show that
	\begin{align}
		d(\away_j(y),\away_k(z)) &\leq d(\away_j(y),\away_{i(k)} \circ \toward_{i(k)} \circ \away_k(z)) + d(\away_{i(k)} \circ \toward_{i(k)} \circ \away_k(z) , \away_k(z)) \\
		&\leq \| \toward_{i(j)} \circ \away_j(y) - \toward_{i(j)} \circ \away_{i(k)} \circ \toward_{i(k)} \circ \away_k(z) \|_{i(j)} + \delta r_\ell \\
		&\leq \| K_{i(j)} - I_{i(j),i(k)} \circ K_{i(k),k} \|_{i(j)} + C\ve r_\ell \lesssim \ve r_\ell
	\end{align}
	Since $\away_j(y) \in 9B^j$ and $\away_k(z) \in 9B^k$ by \eqref{e:dist-sim}, it follows from above estimate and the triangle inequality that $d(x_j,x_k) \leq 20r_\ell.$ This contradicts the separation condition \eqref{e:large-sep}. 
	
	Let us now show that $C^j_0 \subseteq 3U^{i(j)}.$ Recall from \eqref{e:psi} and \eqref{e:C} that $C_0^j = \varphi_{i(j)}^{-1} \circ K_{i(j),j} (8B_{j}).$ For $x \in 8B_j$ we have by Theorem \ref{t:Reif} (3), Lemma \ref{l:K-alpha}, \eqref{e:dist-sim} and  \eqref{e:centre-pos} that 
	\begin{align}
		\|K_{i(j),j}(x)\|_{i(j)} &\leq \| \toward_{i(j)} \circ \away_j(x) - \toward_{i(j)}(x_{i(j)}) \|_{i(j)} + C\ve r_\ell \leq d(\away_j(x),x_{i(j)}) +C\ve r_\ell \\
		&\leq d(\away_j(x),x_j) + d(x_j,x_{i(j)}) + C\ve r_\ell \leq 9r_{\ell+1} + 2r_\ell +  C \ve r_\ell \leq 3r_\ell. 
	\end{align}
	Thus, $K_{i(j),j}(x) \in 3B_{i(j)}.$ By \eqref{e:U} this implies $\varphi_{i(j)}^{-1} \circ K_{i(j),j} (x) \in 3U^{i(j)}$. Since $x$ was arbitrary, we conclude $C_0^j \subseteq 3U^{i(j)},$ as required.  
\end{proof}

\begin{lem}\label{l:whereinX'}
	Let $i,j \in J_\ell$ and suppose $w \in 8U^i \cap 8U^j.$ Then, 
	\begin{align}\label{e:where1'}
		d_X(\away_i \circ \varphi_i(w),\away_j \circ \varphi_j(w)) \lesssim \ve r_\ell. 
	\end{align}
\end{lem} 

\begin{proof}
	Recalling \eqref{r:transition} we have $\varphi_j(w) = I_{j,i}(\varphi_i(w)) = \tilde{I}_{j,i}(\varphi_i(w)) \in 8B_j.$ By \eqref{e:almost-id} and Proposition \ref{p:coherent} we have $d(\away_j(\toward_j(x)),x) \leq \delta r_\ell \leq \ve r_\ell$ for all $x \in 10B^j$ and $\| \tilde{I}_{j,i}(x) - \toward_j\circ\away_i(x)\|_j \lesssim \ve r_\ell$ for all $x \in 8B_i.$ Since $\varphi_i(w) \in 8B_i,$ combing the above we get  
	\begin{align}
		d(\away_i(\varphi_i(w)),\away_j(\varphi_j(w))) &= d( \away_i(\varphi_i(w)),\away_j(\tilde{I}_{j,i}(\varphi_i(w)))) \\
		&\leq d(  \away_i(\varphi_i(w)),\away_j(\toward_j(\away_i(\varphi_i(w))))) +C\ve r_\ell \lesssim \ve r_\ell. 
	\end{align}
\end{proof}

\begin{lem}\label{l:inside}
	Let $j,k \in J_{\ell+1}$ and set $V_m = \Phi_{j}^{-1}(C_m^{j} \cap C_m^{k})$ for $1 \leq m \leq N_0.$ Suppose $x \in V_m$ for some $m \geq 2$ and suppose $y \in 8B_{j,\ell+1}$ is such that $\|x-y\|_{j} \leq r_\ell/(4N_0).$ Then, $y \in V_{m-1}$
\end{lem}

\begin{proof}
	It is enough to show $\dist(z,V_m) > r_\ell/(4N_0)$ for all $z \in 8B_{j,\ell+1} \setminus V_{m-1}.$ Pick such a $z$ and let $w \in V_m$ such that $\|w-z\|_{j} \leq 2\dist(z,V_m).$ Furthermore, choose a point $p \in \partial V_{m-1}$ on the line segment connecting $w$ with $z$. We have 
	\begin{align}
		{\dist}(p,V_m) \leq \| p - w \| =  \| w -z \| - \| p- z \| \leq 2{\dist}(z,V_m)
	\end{align}
	so in particular it suffices to prove $\dist(p,V_m) > r_\ell/(2N_0).$ This is our goal for the remainder of the proof. First notice, since $\Phi_{j}$ and $\Phi_{k}$ are homeomorphisms, we have 
	\[ \partial V_{m-1} \subseteq \partial(\lambda_{m-1}B_{j}) \cup \Phi_{j}^{-1} \circ \Phi_{k}(\partial(\lambda_{m-1}B_{k}))\]
	and
	\begin{align}\label{e:V_m}
		V_{m} \subseteq \lambda_m B_{j} \cap \Phi_{j}^{-1} \circ \Phi_{k}(\lambda_{m}B_{k}).
	\end{align}
	If $p \in \partial(\lambda_{m-1}B_{j}),$ then immediately $\dist(p,V_m) \geq r_\ell/{N_0}.$ Suppose instead that $p \in \Phi_{j}^{-1} \circ \Phi_{k}(\partial(\lambda_{m-1}B_{k}))$ and let $q  \in \lambda_{m-1}B_{k}$ such that $p = \Phi_{j}^{-1}(\Phi_{k}(q)) \in \lambda_{m-1}B_{k}$. Since $q  \in \lambda_{m-1}B_{k}$ we have $\dist(q,\lambda_{m}B_{k}) \geq r_\ell/{N_0}$. If $z' \in \Phi_k^{-1}(C^j_0 \cap C^k_0)$ it follows from the definition in \eqref{e:psi} and \eqref{r:transition} that  
	\[\Phi_{j}^{-1} (\Phi_{k}(z')) = K_{j,i(j)} \circ \tilde{I}_{i(j),i(k)} \circ K_{i(k),k}(z').\] 
	By \eqref{e:V_m} and since by Lemma \ref{l:T-alpha} and Lemma \ref{l:K-alpha} each of the $K_{a,b}$ and $\tilde{I}_{a,b}$ are $(1+C\ve)$-bi-Lipschitz, we then have 
	\begin{align}
		\dist(p,V_m) &\geq  \dist(\Phi_{j}^{-1}(\Phi_{k}(q)) , \Phi_{j}^{-1}(\Phi_{k}(\lambda_{m}B_{k}))) \\
	&\geq (1+C\ve)^{-1} \dist(q,\lambda_mB_k) \geq r_\ell/(2N_0)
	\end{align}
	for $\ve$ small enough. This finishes the proof of the lemma. 
\end{proof}

\begin{lem}\label{l:almost-transition}
	Let $j,k \in J_{\ell+1}$ and suppose there exists $x \in \emph{dom}(\Phi_{k}^{-1}\circ\Phi_{j}) \subseteq 8B_{j}.$ Then, $\tilde{I}_{k,j}$ is well-defined and 
	\begin{align}\label{e:almost-transition}
		\| \Phi_{k}^{-1}(\Phi_{j}(x)) - \tilde{I}_{k,j}(x) \|_{k} \lesssim  \ve r_{\ell}. 
	\end{align}
\end{lem}

\begin{proof}
	To prove that $\tilde{I}_{k,j}$ is well-defined, it suffices to show $d(x_{j},x_{k}) \leq 18r_{\ell+1}$ (in this case $15B^{j} \cap 15B^{k} \neq \emptyset$ and so $(j,k) \in \calF$, recall the definition in \eqref{e:F_ell}). Since $x \in \text{dom}(\Phi_{k}^{-1}\circ\Phi_{j})$, there exists $y \in 8B_k$ such that $\Phi_j(x) = \Phi_k(y).$ By \eqref{e:dist-sim} we have 
	\begin{align}
		d(x_j,x_k) &\leq d(x_j,\away_j(x)) + d(\away_{j}(x),\away_{k}(y))  + d(\away_{k}(y),x_k) \\
		&\leq d(\away_{j}(x),\away_{k}(y))  + 17r_{\ell+1}, 
	\end{align}
	so we only need to show $d(\away_{j}(x),\away_{k}(y)) \leq r_{\ell+1}.$ To see that this, we apply \eqref{e:almost-id}, Lemma \ref{l:K-alpha}, Lemma \ref{l:whereinX'} and \eqref{e:psi} to get 
	\begin{align}
		d(\away_j(x),\away_k(y)) &\leq d(\away_{i(j)} \circ \toward_{i(j)} \circ \away_j(x) , \away_{i(k)} \circ \toward_{i(k)} \circ \away_k(x)) + C\ve r_\ell \\
		&\leq d_X(\away_{i(j)}  \circ K_{i(j),j}(x) , \away_{i(k)} \circ K_{i(k),k}(y) ) + C\ve r_\ell \\
		&= d_X(\away_{i(j)} \circ \varphi_{i(j)} \circ \Phi_j(x)  ,  \away_{i(k)} \circ \varphi_{i(k)} \circ \Phi_k(y))  + C\ve r_\ell \lesssim \ve r_\ell. 
	\end{align}
	Taking $\ve$ small enough we get the desired estimate. We move onto proving \eqref{e:almost-transition}. Using the definition in \eqref{e:psi} and applying \eqref{r:transition}, we have 
	\[\Phi_{k}^{-1} (\Phi_{j}(x)) = K_{k,i(k)} \circ \tilde{I}_{i(k),i(j)} \circ K_{i(j),j}(x).\]
	Then, by combination of \eqref{e:almost-id}, Lemma \ref{l:T-alpha}, Lemma \ref{l:K-alpha} and Proposition \ref{p:coherent}, the left-hand side of  \eqref{e:almost-transition} is at most
	\begin{align}
		\begin{split}\label{e:Phi-I}
		 &\| K_{k,i(k)} \circ \tilde{I}_{i(k),i(j)} \circ K_{i(j),j}(x) - \tilde{I}_{k,j}(x) \|_{k} \\
		&\hspace{1em}\leq \| \toward_{k} \circ \away_{i(k)} \circ \toward_{i(k)} \circ \away_{i(j)} \circ \toward_{i(j)} \circ \away_{j}(x) - \tilde{I}_{k,j}(x) \|_{k} + C\ve r_\ell\\
		&\hspace{1em}\leq \| \toward_{k} \circ \away_{j}(x) - \tilde{I}_{k,j}(x) \|_{k} + C\ve r_{\ell+1} \lesssim \ve r_\ell \leq 18r_\ell.
		\end{split}
	\end{align}
\end{proof}

For the following lemma, recall the definition of $h_j$ from \eqref{e:h}. 

\begin{lem}\label{l:psi-almost}
	Let $j,k \in J_{\ell+1}$, $1 \leq m \leq {N_0}$ and suppose $C_{m-1}^j \cap C_{m-1}^k \neq \emptyset$. Then, 
	\begin{align}\label{e:psi-almost0}
		h_{j}(C_{m-2}^{j}\cap C_{m-2}^{k}) \subseteq h_{k}(C_0^{k}).
	\end{align}
	In particular, the map $\Phi_{k}^{-1} \circ h_{k}^{-1} \circ h_{j} \circ \Phi_{k}$ is well-defined on $\Phi^{-1}_{k}(C_{m-2}^{j} \cap C_{m-2}^{k})$. Furthermore, we have 
	\begin{align}\label{e:psi-almost}
		\| \Phi_{k}^{-1} \circ h_{k}^{-1} \circ h_{j} \circ \Phi_{k} - \id \|_{C^1(\Phi^{-1}_{k}(C_{m-2}^{j} \cap C_{m-2}^{k})),r_\ell} \lesssim \ve. 
	\end{align}
\end{lem}

\begin{proof}
	We start with \eqref{e:psi-almost0}. Let $x \in C_{m-2}^{j} \cap C_{m-2}^{k}$ and let $y = \Phi^{-1}_{j}(x).$ By definition $y \in \text{dom}(\Phi_k^{-1} \circ \Phi_j)$ and so $\tilde{I}_{k,j}$ is well-defined by Lemma \ref{l:almost-transition}. Recall the definition of $\Omega_{k,j}$ and $I_{k,j}$ from \eqref{e:defn-I}. Since $\Phi_k^{-1} \circ \Phi_j(y) \in \lambda_{m-2}B_k$, \eqref{e:almost-transition} implies $\tilde{I}_{k,j}(y) \in 8B_{k}.$ In particular $y\in \Omega_{k,j}$  and $I_{k,j}(y) = \tilde{I}_{k,j}(y).$ Now let $w = I_{k,j}(y) \in 8B_k$ and $v = \Phi_j(w) \in C^k_0.$ Then, using \eqref{r:transition} with \eqref{e:h} and the above definitions, we get 
	\begin{align}
		h_j(x) = \varphi^{-1}_j \circ \Phi_j^{-1}(x) =  \varphi_k^{-1}(I_{k,j}(y)) =  \varphi_k^{-1}(w) =	h_k(v) \in h_k(C_0^k).
	\end{align} 
	
	Let us prove \eqref{e:psi-almost}. Let $U \coloneqq \Phi^{-1}_{k}(C_{m-2}^{j} \cap C_{m-2}^{k})$. Unravelling the definition \eqref{e:psi} and \eqref{e:h} while making use of \eqref{r:transition} we have
	\begin{align}
		\Phi_{k}^{-1} \circ h_{k}^{-1} \circ h_{j} \circ \Phi_{k}|_U  &= I_{k,j}  \circ K_{j,i(j)} \circ I_{i(j),i(k)} \circ K_{i(k),k}|_U.
	\end{align}
	Using this with Lemma \ref{l:composition} and \eqref{e:prop'}, the left-hand side of \eqref{e:psi-almost} is at most   
	\begin{align}
		&\| \Phi_{k}^{-1} \circ h_{k}^{-1} \circ h_{j} \circ \Phi_{k} - T_{k,j}  \circ K_{j,i(j)} \circ T_{i(j),i(k)} \circ K_{i(k),k}\|_{C^1(U),r_\ell}  \\
		&\hspace{2em} +  \| T_{k,j}  \circ K_{j,i(j)} \circ T_{i(j),i(k)} \circ K_{i(k),k} - \text{Id} \|_{C^1(U),r_\ell} \\
		&\hspace{1em}\leq C\ve + \| T_{k,j}  \circ K_{j,i(j)} \circ T_{i(j),i(k)} \circ K_{i(k),k} - \text{Id} \|_{C^1(U),r_\ell}.
	\end{align}
	Applying \eqref{e:Phi-I} with Lemma \ref{l:composition} and \eqref{e:prop'}, we have $\| T_{k,j}  \circ K_{j,i(j)} \circ T_{i(j),i(k)} \circ K_{i(k),k} - x \|_{k} \lesssim \ve r_\ell$ for all $x \in U.$ Then, after applying Lemma \ref{l:linear-C^0}, we obtain $\| T_{k,j}  \circ K_{j,i(j)} \circ T_{i(j),i(k)} \circ K_{i(k),k} - \text{Id} \|_{C^1(U),r_\ell} \lesssim \ve.$ This finishes the proof of the lemma. 
\end{proof}

\begin{lem}\label{l:partial-h}
	There exists a sequence of maps $g_m : \bigcup_{k=1}^m \bigcup_{j \in J_{\ell+1}^k} C_m^j \to W_{\ell+1},$ $1 \leq m \leq N_0,$ such that the following holds: 
	\begin{enumerate}
		\item If $1 \leq m < m' \leq {N_0}$ and $i \in J_{\ell+1}^{m}, \ j \in J_{\ell+1}^{m'}$ are such that $C^i_{m'-1} \cap C^j_{m'-1} \neq \emptyset,$ then 
		\begin{align}
			g_{m}(C^i_{m'-1} \cap C^j_{m'-1}) \subseteq h_j(C_0^j) = B^{j,\ell+1}.
		\end{align}
		\item  If $1 \leq m < m' \leq {N_0}$ and $i \in J_{\ell+1}^{m}, \ j \in J_{\ell+1}^{m'}$ are such that $C^i_{m'-1} \cap C^j_{m'-1} \neq \emptyset,$ then 
		\begin{align}
			\| \Phi_j^{-1} \circ  h^{-1}_j \circ g_m \circ \Phi_j - \emph{Id} \|_{C^1(\Phi_j^{-1}(C^i_{m'-1} \cap C^j_{m'-1})),r_{j,k+1}} \lesssim_m \ve. 
		\end{align}
		\item If $1 \leq m \leq N_0$ and $j \in J^m_{\ell+1}$ then exists a diffeomorphism $\hat{H}_j \colon \R^n_j \to \R^n_j$ such that $\|\hat{H}_j - \emph{Id}\|_{C^1(\R^n),r_\ell} \lesssim \ve$ and 
		\[ g_{m'}(x) =  \varphi_j^{-1} \circ \hat{H}_j \circ \Phi_j^{-1}(x)\]
		for all $m < m' \leq N_0$ and $x \in C^j_{m'}.$
	\end{enumerate}
\end{lem}

\begin{proof}
	We construct the maps by induction. First, define $g_1 \colon \bigcup_{j \in J^1} C_1^j \to W_{\ell+1}$ by setting 
	\[g_1(x) = h_j(x) \mbox{ whenver } x \in C_1^j.\] 
	This is well-defined since $C^i_1 \cap C^j_1 =\emptyset$ for all $i,j \in J^1$ such that $i \neq j$ by Lemma \ref{l:C-U}. Conditions (1) and (2) follow immediately from Lemma \ref{l:psi-almost}. Condition (3) follows if we set $\hat{H}_j = \text{Id}$ for each $j \in J^1.$ This completes the base case. 
	
	Assume now that we have defined the maps up to some integer $1 \leq m \leq N_0-1$ and let us describe how to construct $g_{m+1}.$ For $j \in J^{m+1}$, set 
	\begin{align}
		H_j &= \Phi_j^{-1} \circ h_j^{-1} \circ g_m \circ \Phi_j, \label{e:H_j} \\
		U_{1,j} &= \Phi_j^{-1}\left( \left( \bigcup_{k=1}^m \bigcup_{i \in J^k} C^i_{m+1} \right)\cap C^j_{m+1} \right); \label{d:U1} \\
		U_{2,j} &= \Phi_j^{-1}\left( \left( \bigcup_{k=1}^m \bigcup_{i \in J^k} C^i_{m} \right)\cap C^j_{m} \right). \label{d:U2}
	\end{align}
	We want to apply Lemma \ref{l:modification} with the map $H_j$ and the sets $U_{1,j}$ and $U_{2,j}.$ Clearly, $U_{1,j} \subseteq U_{2,j}.$ Combining conditions (1) and (2) for the map $g_m$ implies $H_j$ is well-defined on $U_{2,j}$ and satisfies 
	\begin{align}\label{e:H-lemma}
		\|H_j - \text{Id}\|_{C^1(U_{2,j}),r_{\ell}} \lesssim \ve.
	\end{align}
	It only remains to check 
	\begin{align}\label{e:distU}
		\dist(U_{1,j},U_{2,j}^c) \gtrsim \frac{r_\ell}{N_0}.  
	\end{align}
	To see that \eqref{e:distU} holds, let $x \in U_{1,j}$ and $y \in U_{2,j}^c.$ By definition $x \in \lambda_{m+1}B_{j,\ell+1}.$ If $y \not\in \lambda_m B_{j,\ell+1}$ then immediately we get $\|x-y\|_j \geq r_\ell/N_0.$ Suppose instead $y \in  \lambda_m B_{j,\ell+1}$ so that $\Phi_j(y) \in C^j_m.$  By \eqref{d:U1}, since $x \in U_{1,j}$ there exists $i \in \bigcup_{k=1}^m J^k$ such that $\Phi_j(x) \in C^i_{m+1} \cap C_{m+1}^j$ i.e. $x \in \Phi_j(C_{m+1}^i \cap C_{m+1}^j)$. By \eqref{d:U2}, since $y \in U_{2,j}^c$ and $\Phi_j(y) \in C^j_m$ it must be that $\Phi_j(y) \not\in C_m^i$. Hence, $y \not\in \Phi_j(C_m^i \cap C_m^j).$ Lemma \ref{l:inside} now implies $\|y-x\|_j > r_\ell/(4N_0)$. This completes the proof of \eqref{e:distU}. 
	
	Now, let $\hat{H}_j : \R^n \to \R^n$ be the map obtained from Lemma \ref{l:modification} by modifying $H_j.$ Let us recall here that by \eqref{e:H4} and \eqref{e:H-lemma} we have 
	\begin{align}\label{e:hat-H}
		\|\hat{H}_i - \text{Id}\|_{C^1(\R^n_i),r_{\ell+1}} \lesssim \ve.
	\end{align}
	Now, define
	\begin{align}\label{e:tilde-h}
		\tilde{h}_j = h_j \circ \Phi_j \circ \hat{H}_j \circ \Phi_j^{-1} : C_0^j \to W_{\ell+1}
	\end{align}
	and let
	\begin{align}\label{e:g_m}
		g_{m+1} = 	
		\begin{cases}
			g_m & \mbox{ on } \bigcup_{\ell=1}^m \bigcup_{i \in J^\ell} C_{m+1}^i \\
			\tilde{h}_j & \mbox { on } C_{m+1}^j  \mbox{ for each } j \in J^{m+1}.
		\end{cases}
	\end{align}
	Using \eqref{e:H1} with \eqref{e:H_j}, if $z \in \Phi_j(U_{1,j})$ then $h_j \circ \Phi_j \circ \hat{H}_j \circ \Phi_j^{-1}(z) = g_m(z)$ so that $g_{m+1}$ is well-defined.  Now that we have defined $g_{m+1},$ we are left to check conditions (1) -- (3). \\
	
	\noindent\textbf{Condition (1):} Suppose $m +1 < m' \leq N_0$ and $i \in J^{m+1}, j \in J^{m'}$ are such that $C^i_{m'-1} \cap C^j_{m'-1} \neq \emptyset$ and fix some $x \in C^i_{m'-1} \cap C^j_{m'-1}$. Since $m+1 < m'$ we have $m+2 \leq m'$ so that $m +1 \geq m'-1.$ This gives $x \in C_{m+1}^i$  and \eqref{e:g_m} now implies
	\begin{align}\label{e:g=h}
		g_{m+1}(x) = \tilde{h}_i(x).
	\end{align}
	By Lemma \ref{l:psi-almost} and \eqref{e:g=h}, it suffices to show 
	\begin{align}\label{e:tildehin}
		\tilde{h}_i(x) \in h_i(C_{m'-2}^i \cap C_{m'-2}^j).
	\end{align}
	To see \eqref{e:tildehin}, we first observe by \eqref{e:C} and \eqref{e:H-lemma} that $\Phi^{-1}_i(x) \in \lambda_{m'-1}B_i$ and	
	\begin{align}\label{e:H-phi}
		\| \hat{H}_i(\Phi_i^{-1}(x)) - \Phi_i^{-1}(x) \|_i \lesssim \ve r_\ell.
	\end{align}
	Hence, $\hat{H}_i(\Phi^{-1}_i(x)) \in \lambda_{m'-2} B_i$, so that 
	\begin{align}\label{e:Ci}
		\Phi_i\circ \hat{H}_i \circ \Phi_i^{-1}(x) \in C_{m'-2}^i
	\end{align}
	by \eqref{e:C}. Since $x \in C_{m'-1}^i \cap C_{m'-1}^j,$ we have $\Phi_i^{-1}(x) \in \text{dom}(\Phi_j^{-1} \circ \Phi_i).$ Combining this with Lemma \ref{l:inside} and \eqref{e:H-phi} gives $\hat{H}_i\circ \Phi_i^{-1}(x) \in \text{dom}(\Phi_j^{-1} \circ \Phi_i)$. Lemma \ref{l:almost-transition} now implies $\Phi_j^{-1} \circ \Phi_i \circ \hat{H}_i \circ \Phi_i^{-1}(x) \in \lambda_{m'-2}B_j$ so that 
	\begin{align}\label{e:Cj}
		\Phi_i \circ \hat{H}_i \circ \Phi_i^{-1}(x) \in C_{m'-2}^j.
	\end{align}
	Equation \eqref{e:tildehin} follows from \eqref{e:Ci}, \eqref{e:Cj}, and the definition of $\tilde{h}_i$ in \eqref{e:tilde-h}. \\
	
	\noindent\textbf{Condition (2):} Suppose $m +1 < m' \leq N_0$ and $i \in J^{m+1}, j \in J^{m'}$ are such that $C^i_{m'-1} \cap C^j_{m'-1} \neq \emptyset$. By \eqref{e:Cj}, if $y \in \Phi_j^{-1}(C^i_{m'-1} \cap C^j_{m'-1})$, then $ \Phi_i \circ \hat{H}_i \circ \Phi^{-1}_i \circ \Phi_j(y) \in \text{dom}(\Phi_j^{-1}).$ Furthermore, by the same reason as for \eqref{e:g=h}, we have 
	\[g_{m+1}(\Phi_j(y)) = \tilde{h}_i(\Phi_j(y)).\]
	Combining this with \eqref{e:tilde-h} gives
	\begin{align}
		\begin{split}\label{e:long}
			\Phi_j^{-1} \circ h_j^{-1} \circ g_{m+1} \circ \Phi_j(y)&= \Phi_j^{-1} \circ h_j^{-1} \circ h_i \circ \Phi_i \circ \hat{H}_i \circ \Phi^{-1}_i \circ \Phi_j(y) \\
			&= \Phi_j^{-1} \circ h_j^{-1} \circ h_i \circ \Phi_j \circ \Phi^{-1}_j  \circ \Phi_i \circ \hat{H}_i \circ \Phi^{-1}_i \circ \Phi_j(y) \\
			&= f \circ g \circ \hat{H}_i \circ g^{-1}(y), 
		\end{split}
	\end{align}
	where, $f = \Phi_j^{-1} \circ h_j^{-1} \circ h_i \circ \Phi_j$ and $g =  \Phi^{-1}_j  \circ \Phi_i = K_{j,i(j)} \circ I_{i(j),i(i)} \circ K_{i(i),i}.$ By Lemma \ref{l:psi-almost}, we have
	\begin{align}\label{e:f-I}
		\|f - \text{Id}\|_{C^1(\Phi_j^{-1}(C^i_{m'-1} \cap C^j_{m'-1})),r_{\ell+1}} \lesssim \ve.
	\end{align}
	Clearly $g$ is invertible with $g,g^{-1} \in C^2(\R^n)$. Furthermore, using \eqref{e:prop'} along with the fact that $T_{i(j),i(i)}$, $K_{j,i(j)}$, $ K_{i(i),i}$ are affine and $(1+C\ve)$-bi-Lipschitz, we have 
	\begin{align}
		\sup_{x \in \R^n} \left(\|Dg(x)\| + \|D[g^{-1}](x)\| + r_{\ell+1}\|D^2g(x) \| + r_{\ell+1} \|D^2[g^{-1}](x)\| \right)\lesssim 1. 
	\end{align}
	Then, Lemma \ref{l:C^2-comp} and \eqref{e:H-lemma} give
	\begin{align}\label{e:comp-I}
		\| g \circ \tilde{H}_i \circ g^{-1} - \text{Id} \|_{C^1(\R^n_j),r_{\ell+1}} \lesssim \ve.
	\end{align}
	Condition (2) follows from \eqref{e:long}, \eqref{e:f-I}, \eqref{e:comp-I} and Lemma \ref{l:composition}. \\
	
	\noindent\textbf{Condition (3):} Suppose $1 \leq m +1 < m' \leq N_0$ and $j \in J^m.$ Iterating \eqref{e:g_m}, gives $g_{m'} |_{C^j_{m'}} = g_m|_{C^j_{m'}} = \tilde{h}_j|_{C^j_{m'}}.$ Condition (3) now follows from \eqref{e:h}, \eqref{e:tilde-h},  and \eqref{e:H-lemma}.

\end{proof}

\begin{proof}[Proof of Proposition \ref{p:h_ell}]
	Recall the definition of $E_\ell$ from \eqref{e:h_ell1}. We set
	\[ h_\ell \coloneqq g_{N_0} \colon E_\ell \to W_{\ell+1}. \]
	Both \eqref{e:h_ell2} and \eqref{e:h_ell3} are then immediate from Lemma \ref{l:partial-h} (3). For the first inclusion in \eqref{e:h_ell3'}, we notice that for each $j \in J_{\ell+1}$ and $A \in \N$ such that $1 \leq A \leq 6$, by Lemma \ref{l:topology} and \eqref{e:h_ell3}, that $AB_j \subseteq \hat{H}_j((A+1)B_j).$ Using this with \eqref{e:U}, \eqref{e:C} and \eqref{e:h_ell2}, gives 
	\begin{align}
		AU^j &= \varphi_j^{-1}(AB_j) \subseteq \varphi_j^{-1} \circ \hat{H}_j((A+1)B_j) \\
		&= \varphi_j^{-1} \circ \hat{H}_j \circ \Phi_j^{-1}\left(C^j_{(7-A)N_0}\right) = h_\ell\left(C^j_{(7-A)N_0}\right). 
	\end{align}
	The second inclusion in \eqref{e:h_ell3'} is immediate from \eqref{e:C}. The final inclusion in \eqref{e:h_ell3'} follows from the definition of $E_\ell$ and Lemma \ref{l:C-U}. Equation \eqref{e:h_ell3''} is immediate from \eqref{e:h_ell3'}. For \eqref{e:h_ell4}, suppose $x \in E_{\ell+1}.$ By definition there exists $1 \leq m \leq N_0$ and $j \in J^m_{\ell+2}$ such that $x \in C^{j,\ell+2}_{N_0}.$ If $i(j) \in J_{\ell+1}$ is as in \eqref{e:i(j)}, it follows from Lemma \ref{l:C-U} that $x \in C^{j,\ell+2}_{N_0} \subseteq 3U^{i(j)} \subseteq 6U^{i(j)}$ and we are done. 
\end{proof}

\bigskip

\subsection{Construction of the connected manifolds $M_\ell$}\label{s:M}
In this section we construct a sequence of connected Finsler manifolds $M_\ell$ by filling between the connected components of $W_\ell$ using charts from $W_1,\dots,W_{\ell-1}.$ We define the $M_\ell$ (as smooth manifolds) inductively, together with topologies $\tau_\ell$ and diffeomorphisms $p_\ell \colon W_{\ell}^* \to p_\ell(W_\ell^*) \subseteq M_{\ell}$, where $W_\ell^*$ is as in \eqref{e:h_ell3''}. Note that we have only defined $W_{\ell}^*$ for $\ell \geq 1$ so let us declare now that
\begin{align}\label{e:W_0^*}
	W_0^* \coloneqq W_0 = 8B_{j_0},
\end{align}
where the last equality follows from Remark \ref{l:W_0}. After constructing the $M_\ell$ as smooth manifolds, we equip them with a suitable Finsler metric (see after Corollary \ref{c:connect}).

We being by letting $\tau_0$ be the usual topology on $\R^n$, 
\begin{align}\label{e:M_0}
	M_0 = \R^n_{j_0} \mbox{ and } p_{0} \coloneqq \text{Id} : W_0^* \to M_0 
\end{align}
Equipping $M_0$ with $\tau_0,$ it is clear that $(M_0,\tau_0)$ is a smooth connected manifold and $p_{0}$ is a diffeomorphism. Let $\ell \geq 0$ and suppose we have constructed $M_\ell,$ $\tau_\ell$ and $p_\ell : W_{\ell}^* \to p_\ell(W_{\ell}^*).$ Let us define $M_{\ell+1}, \tau_{\ell+1}$ and $p_{\ell+1},$ starting with $M_{\ell+1}.$ Recall the definition of $E_\ell$ from \eqref{e:h_ell1}. Define open subsets $E'_\ell \subseteq E''_\ell \subseteq E_\ell$ by 
\begin{align}
	E_\ell' \coloneqq \bigcup_{j \in J_{\ell+1}} C_{N_0+2}^{j,\ell+1} \quad \mbox{ and } \quad E_\ell'' \coloneqq \bigcup_{j \in J_{\ell+1}} C_{N_0+1}^{j,\ell+1}
\end{align}
and let
\begin{align}\label{e:F_ell'}
	G_{\ell} = p_{\ell}(E_\ell) \mbox{ and } G_{\ell}' = p_\ell(E_{\ell}').
\end{align}
It follows from Proposition \ref{p:h_ell} that $E_\ell \subseteq W_\ell^*$ so that $G_\ell$ and $G_{\ell}'$ are well-defined. Consider a punctured version of $M_{\ell},$ 
\begin{align}\label{e:M_2^*}
	M_{\ell}^* = M_{\ell} \setminus \overline{G_{\ell}'}.
\end{align}
Similar to Section \ref{s:disconnected}, the manifold $M_{\ell+1}$ will be define by taking the quotient of $M_{\ell}^* \sqcup W_{\ell+1}^*$ with respect to some equivalence relation $\sim$, which we now define. Let 
\[\text{in}_1 \colon M_\ell^* \to M_\ell^*\sqcup W_{\ell+1}^* \quad \mbox{ and } \quad \text{in}_2 \colon W_{\ell+1}^* \to M_\ell^*\sqcup W_{\ell+1}^*\]
be the natural injections. Suppose $x,y \in M_\ell^* \sqcup W_{\ell+1}^*$. We declare $x \sim y$ if either  
\begin{enumerate}[label=(\roman*)]
	\item $x = y;$
	\item there are $x' \in M_\ell^* \cap G_\ell$ and $y' \in W_{\ell+1}^*$ such that $x = \text{in}_1(x'), \ y = \text{in}_2(y')$ and $y' = h_\ell \circ p_{\ell}^{-1}(x');$ 
	\item there are $x' \in W_{\ell+1}^*$ and $y' \in M_\ell^* \cap G_\ell$ such that $x = \text{in}_2(x'), \ y = \text{in}_1(y')$ and $x' = h_\ell \circ p_{\ell}^{-1}(y').$ 
\end{enumerate} 
It is easy to check that the relation $\sim$ defines an equivalence relation on $M_{\ell}^* \sqcup W_{\ell+1}^*$. Now let


\[ M_{\ell+1} = \left( M_{\ell}^* \sqcup W_{\ell+1}^*\right) / \sim  \]
and 
\[\Pi : M_\ell^*\sqcup W_{\ell+1}^* \to M_{\ell+1}\]
be the quotient map induced by the equivalence relation $\sim$. Then, define 
\[ \pi_{\ell+1} \coloneqq \Pi \circ \text{in}_1 \colon M_\ell^* \to M_{\ell+1}, \  p_{\ell+1} \coloneqq \Pi \circ \text{in}_2 \colon W_{\ell+1}^* \to M_{\ell+1}.\] 
Since $\text{in}_1$ (resp. $\text{in}_2$) is injective and $\Pi$ is injective when restricted to $\text{in}(M_{\ell}^*)$ (resp. $\text{in}_2(W_{\ell+1}^*)$), we have that $\pi_{\ell+1}$ (resp. $p_{\ell+1}$) is invertible. Let $\tau_{\ell+1}$ be the coarsest topology on $M_{\ell+1}$ such that $\pi_{\ell+1}$ and $p_{\ell+1}$ are continuous i.e. 
\begin{align}
	 \tau_{\ell+1} = \{ U \subseteq M_{\ell+1} \colon &\pi_{\ell+1}^{-1}(U) \mbox{ is an open subset of } M_\ell^*, \\
&p_{\ell+1}^{-1}(U) \mbox{ is an open subset of } W_{\ell+1}^*\}.
\end{align}
\begin{rem}\label{r:trans}
	It follows from construction that if $x \in M_{\ell}^*$ and $y \in W_{\ell+1}^*$ are such that $\pi_{\ell+1}(x) = p_{\ell+1}(y)$ then 
	\begin{align}
		p_{\ell+1}^{-1}(\pi_{\ell+1}(x)) = h_{\ell}(p_{\ell}^{-1}(x)) \quad \mbox{ and } \quad  \pi_{\ell+1}^{-1}(p_{\ell+1}(y)) = p_{\ell}(h_{\ell}^{-1}(y)). 
	\end{align}
\end{rem}

\begin{lem}\label{l:smooth}
	The space $(M_{\ell+1},\tau_{\ell+1})$ is a smooth manifold and the maps $\pi_{\ell+1}$ and $\ p_{\ell+1}$ are diffeomorphisms onto their image. 
\end{lem}

\begin{proof}
	We first check that $M_{\ell+1}$ is a topological manifold. Both $M_{\ell}^*$ and $W_{\ell+1}^*$ are smooth manifolds (being open subsets of $M_\ell$ and $W_{\ell+1},$ respectively). Let $\{(V_1^i,\sigma_1^i)\}_{i \in I_1}$ be a smooth atlas for $M_{\ell}^*$ and $\{(V_2^i,\sigma_2^i)\}_{i \in I_2}$ a smooth atlas for $W_{\ell+1}^*.$ For each $i_1 \in I_1$ and $i_2 \in I_2$ the sets $\pi_{\ell+1}(V_1^i)$, $p_{\ell+1}(V_2^{i_2})$ are open in $M_{\ell+1}$ and by construction 
	\[ M_{\ell+1} \subseteq \bigcup_{i \in I_1} \pi_{\ell+1}(V_1^i) \cup \bigcup_{i \in I_2} p_{\ell+1}(V_2^i).\]
	
	To see that $M_{\ell+1}$ is locally Euclidean it suffices to show 
	\[\pi_{\ell+1} : M_{\ell}^* \to \pi_{\ell+1}(M_\ell^*) \mbox{ and } p_{\ell+1} \colon W_{\ell+1}^* \to p_{\ell+1}(W_{\ell+1}^*) \mbox{ are homeomorphisms.} \]
	We will show all the details in the case of $\pi_{\ell+1},$ the arguments for $p_{\ell+1}$ are similar and we omit them. Let $U \subseteq M_{\ell}^*$ open. It suffices to show $\pi_{\ell+1}(U) \in \tau_{\ell+1}$. Clearly, $\pi_{\ell+1}^{-1}(\pi_{\ell+1}(U))  = U$ is open in $M_{\ell}^*$ by definition. By Remark \ref{r:trans} we have 
	\[ p_{\ell+1}^{-1}(\pi_{\ell+1}(U)) = h_{\ell}(p_{\ell}^{-1}(U)).\]
	By Proposition \ref{p:h_ell} and the induction hypothesis we have that $h_\ell \colon E_{\ell} \to W_{\ell+1}^*$ and $p_{\ell} : W_{\ell}^* \to p_\ell(W_\ell^*)$ are homeomorphism. Thus, $p_{\ell+1}^{-1}(\pi_{\ell+1}(U))$ is open. Since both $\pi_{\ell+1}^{-1}(\pi_{\ell+1}(U))$ and $p_{\ell+1}^{-1}(\pi_{\ell+1}(U))$ are open, we conclude by definition that $\pi_{\ell+1}(U) \in \tau_{\ell+1}.$ 
	
	Let us now check that $M_{\ell+1}$ is Hausdorff. Let $x,y \in M_{\ell+1}.$ If $x,y \in \pi_{\ell+1}(M_{\ell}^*)$ or $x,y \in p_{\ell+1}(W_{\ell+1}^*)$ then we can find disjoint open sets $V_x \ni x$, $V_y \ni y$ in $M_{\ell+1}$ by using the Hausdorff properties in $M_{\ell}^*$ and $W_{\ell+1}^*$, respectively. Suppose then that 
	\[x \in \pi_{\ell+1}(M_{\ell^*}) \setminus p_{\ell+1}(W_{\ell+1}^*) \mbox{ and } y \in p_{\ell+1}(W_{\ell+1}^*) \setminus \pi_{\ell+1}(M_{\ell}^*).\]
	We claim that 
	\begin{align}\label{e:Hausdorff}
		x \in M_{\ell+1} \setminus \overline{p_{\ell+1}(h_\ell(E_{\ell}''))} \mbox{ and } y \in p_{\ell+1}(h_\ell(E_{\ell}'')).
	\end{align}
	Observing that $E_\ell''$ is open in $W_{\ell}$ and $h_\ell, \ p_{\ell+1}$ are homeomorphism, the sets in \eqref{e:Hausdorff} are disjoint open sets in $M_{\ell+1}$. Thus, \eqref{e:Hausdorff} is sufficient to prove $M_{\ell+1}$ is Hausdorff. The first part of \eqref{e:Hausdorff} follow since 
	\[ \overline{p_{\ell+1}(h_\ell(E_{\ell}''))} = p_{\ell+1}(h_{\ell}(\overline{E_\ell''})) \subseteq p_{\ell+1}(h_\ell(E_\ell)) = p_{\ell+1}(W_{\ell+1}). \]
	To see that the second part holds, suppose towards a contradiction that $y = p_{\ell+1}(h_\ell(z))$ for some $z \in E_{\ell} \setminus E_{\ell}'' \subseteq E_\ell \setminus \overline{E_{\ell}'}.$ By \eqref{e:F_ell'} and \eqref{e:M_2^*}, we have $p_{\ell}(z) \in M_{\ell}^*.$ Notice also that 
	\begin{align}
		h_{\ell} \circ p_{\ell}^{-1}(p_\ell(z)) = h_{\ell}(z) = p_{\ell+1}^{-1}(y).
	\end{align}
	It follows from the definition of $\sim$ that 
	\[ y = \pi_{\ell+1}(p_{\ell}(z)) \in \pi_{\ell+1}(M_{\ell}^*),\]
	which is a contradiction. This concludes the proof that $M_{\ell+1}$ is a topological manifold. 
	
	We now show that $M_{\ell+1}$ is smooth by showing the atlas
	\begin{align}
		\calA = \{(\pi_{\ell+1}(V_1^i) , \sigma_1^i \circ \pi_{\ell+1}^{-1}) : i \in I_1\} \cup \{(p_{\ell+1}(V_2^i) , \sigma_2^i \circ p_{\ell+1}^{-1}) : i \in I_2\}
	\end{align}
	is a smooth. Consider charts $(B,\phi),(B',\phi') \in \calA$ with $B \cap B' \neq \emptyset.$ If $B = \pi_{\ell+1}(V_1^i)$ and $B' = \pi_{\ell+1}(V_1^{i'})$ for some $i,i' \in I_1$, the associated transition maps are smooth using by the smoothness of transition maps for $M_{\ell}^*$. Similarly, if $B = p_{\ell+1}(V_2^i)$ and $B' = p_{\ell+1}(V_2^{i'})$ for some $i,i' \in I_2,$ the associated transition maps are smooth by using the smooth of transition maps for $W_{\ell+1}^*.$ Suppose then that $B = \pi_{\ell+1}(V_1^i)$ and $B = p_{\ell+1}(V_2^{i'})$ for some $i \in I_1$ and $i' \in I_2.$ By assumption the maps $p_{\ell}, \sigma_1^i$ and $\sigma_2^{i'}$ are smooth. By Proposition \ref{p:h_ell} the same is true of $h_\ell.$ Thus, we have transition maps
	\begin{align}
		\sigma_2^{i'} \circ p_{\ell+1}^{-1} \circ \pi_{\ell+1} \circ \sigma_1^i  |_{\sigma_1^i(V_1^i)} = 	\sigma_2^{i'} \circ h_\ell \circ p_{\ell}^{-1} \circ \sigma_1^i  |_{\sigma_1(V_1^i)} 
	\end{align}
	and
	\begin{align}
		\sigma_1^i \circ \pi_{\ell+1}^{-1} \circ p_{\ell+1} \circ \sigma_2^{i'} |_{\sigma_2^{i'}(V_2^{i'})}  = \sigma_1^i \circ p_{\ell} \circ h_{\ell}^{-1} \circ \sigma_2^{i'} |_{\sigma_2^{i'}(V_2^{i'})}  
	\end{align}
	which are all smooth. It now follows that $\calA$ is a smooth atlas. 
\end{proof}

\begin{lem}\label{l:f-homeo}
	The map $f_\ell \colon M_{\ell} \to M_{\ell+1}$ defined by 
	\begin{align}\label{e:def-f}
		f_{\ell}(x) = 	
		\begin{cases}
			p_{\ell+1} \circ h_{\ell} \circ p_{\ell}^{-1}(x) & \mbox{ if } x \in G_{\ell} \\
			\pi_{\ell+1}(x) & \mbox{ if } x \in M_{\ell}^*
		\end{cases}
	\end{align}
	is a well-defined diffeomorphism and satisfies
	\begin{align}\label{e:f(G)}
		f(G_\ell) = p_{\ell+1}(W_{\ell+1}^*).
	\end{align} 
\end{lem}

\begin{proof}
	It is immediate from the definition of $\sim$ that if $x \in M_{\ell}^* \cap G_{\ell}$ then $\text{in}_1(x) \sim \text{in}_2(h_\ell\circ p_\ell^{-1}(x)).$ Recalling the definition of $\Pi$, $\pi_{\ell+1}$ and $p_{\ell+1},$ this implies $\pi_{\ell+1}(x) = p_{\ell+1} \circ h_{\ell} \circ p_{\ell}^{-1}(x) $ so that $f_{\ell}$ is well-defined. It follows from \eqref{e:h_ell3''}, \eqref{e:F_ell'} and \eqref{e:def-f} that \[f_\ell(G_\ell) = p_{\ell+1}(W_{\ell+1}^*).\]
	It only remains to check that $f_\ell$ is a diffeomorphism. Since each of the maps $\pi_{\ell+1},p_{\ell+1},h_{\ell}$ and $p_{\ell}$ are diffeomorphisms on their images, $f_{\ell}$ is a local diffeomorphism. To prove $f_\ell$ is a global diffeomorphism we just need to show that it is bijective. To see that $f_\ell$ is injective, suppose $x,y \in M_{\ell}$ are distinct. If $x,y \in G_{\ell}$ or $x,y \in M_\ell^*$ then $f_\ell(x) \neq f_\ell(y)$ since $p_{\ell+1} \circ h_{\ell} \circ p_{\ell}^{-1}$ and $\pi_{\ell+1}$ are injective. Thus, it suffices to consider the case $x \in G_\ell \setminus M_{\ell}^* = \bar{G_\ell'}$ and $y\in M_\ell^* \setminus G_\ell.$ In this case, if $f_\ell(x) = f_\ell(y)$ then by Remark \ref{r:trans} we have 
	\begin{align}
		x = p_\ell \circ h_\ell^{-1} \circ h_\ell \circ p_\ell^{-1}(x) = \pi_{\ell+1}^{-1} \circ p_{\ell+1} \circ h_\ell \circ p_\ell^{-1}(x) = y, 
	\end{align}
	a contradiction. To see that $f_{\ell}$ is surjective, we have by \eqref{e:def-f} and \eqref{e:f(G)} that
	\[ f_\ell(M_\ell) \supseteq f_\ell(G_\ell) \cup f_\ell(M_\ell^*) = p_{\ell+1}(W_{\ell+1}^*) \cup \pi_{\ell+1}(M_\ell^*) = M_{\ell+1}.\]
\end{proof}

By assumption $M_\ell$ is connected and by Lemma \ref{l:f-homeo}, $M_{\ell+1}$ is homeomorphic to $M_{\ell}.$ In particular, we obtain the following which closes the inductive step. 

\begin{cor}\label{c:connect}
	The space $M_{\ell+1}$ is connected. 
\end{cor}

From now on we will view $M_{\ell}$ as equipped with the above smooth structure. In particular, if $x \in M_{\ell}$, the tangent space $T_{x}M_{\ell}$ is well-defined. We equip $M_\ell$ with a Finsler metric by defining, for each $x \in M_\ell,$ a norm $\|\cdot\|_{M_\ell,x}$ on $T_xM_\ell$ as follows. If $\ell = 0$ and $x \in M_0$ we simply let 
\begin{align}\label{e:norm-0}
	\|\cdot\|_{M_0,x} = \|\cdot\|_{j_0}.
\end{align}
For $\ell \geq 1,$ let $\{\theta_1^\ell,\theta_2^\ell\}$ be a partition of unity subordinate to $\{p_{\ell}(W_{\ell}^*),\pi_{\ell}(M_{\ell-1}^*)\}.$ Then, for $x \in M_\ell$ and $v \in T_xM_\ell,$ define 
\begin{align}\label{d:F}
	\hspace{2.15em}\|v\|_{M_\ell,x} \coloneqq \theta_1^\ell(x) \|D[p_\ell^{-1}](x)\cdot v\|_{W_\ell,p_\ell^{-1}(x)} + \theta_2^\ell(x) \|D[\pi_\ell^{-1}](x) \cdot v\|_{M_{\ell-1},\pi_\ell^{-1}(x)}.
\end{align}
By Lemma \ref{l:smooth} both $\pi_{\ell}$ and $p_{\ell}$ are diffeomorphisms so the metric is well-defined. With this metric we have the following. 

\begin{lem}\label{l:p-bi-lip}
	Let $\ell \geq 0.$ For each $x \in W_\ell^*,$ the map $Dp_\ell(x) \colon T_x W_{\ell} \to T_{p_\ell(x)}M_\ell$ is $(1+C\ve)$-bi-Lipschitz. Similarly, for each $x \in M_{\ell-1}^*,$ the map $D\pi_\ell(x) \colon T_x M_{\ell-1} \to T_{\pi_\ell(x)} M_\ell$ is $(1+C\ve)$-bi-Lipschitz. 
\end{lem}

\begin{proof}
	Let $x \in W_\ell^*$ and $v \in T_x W_\ell.$ If $p_\ell(x) \not\in \pi_\ell(M_{\ell-1}^*)$ then $\theta_2^\ell(p_\ell(x)) = 0$ and $\theta_1^\ell(p_\ell(x)) = 1.$ Hence $\|Dp_\ell(x) \cdot v\|_{M_\ell,p_\ell(x)} = \|v\|_{W_\ell,x}$ and there is nothing to show. Thus, we may suppose $p_\ell(x) \in \pi_\ell(M_{\ell-1}^*).$ Let $y \in M_{\ell-1}^*$ such that $p_\ell(x) = \pi_\ell(y).$ Then, by \eqref{d:F}, 
	\begin{align}\label{e:p-bi-lip}
		\|Dp_\ell(x)\cdot v\|_{M_\ell,p_\ell(x)} &= \theta_1^\ell(p_\ell(x))\|v\|_{W_\ell,x} \\
		&\hspace{2em} + \theta_2^\ell(p_\ell(x))\|D[\pi_\ell^{-1} \circ p_\ell](x) \cdot v \|_{M_{\ell-1},y}. 
	\end{align}
	Let us focus on the second term. From Remark \ref{r:trans}  we have that $\pi_\ell^{-1} \circ p_\ell(x) = p_{\ell-1} \circ h_{\ell-1}^{-1}(x),$ hence, 
	\begin{align}\label{e:here}
		D[\pi_\ell^{-1} \circ p_\ell](x) = D[p_{\ell-1}\circ h_{\ell-1}^{-1}](x)  = Dp_{\ell-1}(h_{\ell-1}^{-1}(x)) \circ D[h_{\ell-1}^{-1}](x). 
	\end{align} 
	
	We claim that 
	\[p_{\ell-1} (h_{\ell-1}^{-1}(x)) \not \in \pi_{\ell-1}(M_{\ell-2}^*).\] 
	Indeed, suppose towards a contradiction that there exists $y \in M_{\ell-2}^*$ such that $p_{\ell-1}(h_{\ell-1}^{-1}(x)) = \pi_{\ell-1}(y).$ By Remark \ref{r:trans} and the definition of $W_{\ell}^*$ in \eqref{e:h_ell3''}, it follows that 
	\begin{align}
		h_{\ell-2} \circ p_{\ell-2}^{-1}(y) = p_{\ell-1}^{-1} \circ \pi_{\ell-1}(y) = h_{\ell-1}^{-1}(x) \in E_{\ell-1}.
	\end{align}
	In particular, by \eqref{e:h_ell1} and Lemma \ref{l:C-U}, there exists an index $j \in J_{\ell}$ such that $h_{\ell-2} \circ p_{\ell-2}^{-1}(y) \in C^{j,\ell}_{N_0} \subseteq 3U^{i(j)},$ with $i(j) \in J_{\ell-1}$ as in \eqref{e:i(j)}. Then, \eqref{e:h_ell1} and the first inclusion in \eqref{e:h_ell3'} imply 
	\[p_{\ell-2}^{-1}(y) \in C^{i(j),\ell-1}_{4N_0} \subseteq \overline{E_{\ell-2}'},\]
	which gives $y \in \overline{G_{\ell-2}'}$ by \eqref{e:F_ell'}. This contradicts the definition of $M_{\ell-2}^*$ in \eqref{e:M_2^*}, and proves the claim.
	
	The claim gives that $\theta_{2}^{\ell-1}(p_{\ell-1} \circ h_{\ell-1}^{-1}(x)) = 0.$ From this and the definition of $\|\cdot\|_{M_{\ell-1},\cdot}$ in \eqref{d:F}, it follows that $Dp_{\ell-1}(h_{\ell-1}^{-1}(x))$ is an isometry. Returning to \eqref{e:here}, we now have 
	\begin{align}
		\|D[\pi_\ell^{-1} \circ p_\ell](x) \cdot v \|_{M_{\ell-1},y} = \|D[h_{\ell-1}^{-1}](x)\cdot v\|_{W_{\ell-1},h_{\ell-1}^{-1}(x)}.
	\end{align}
	Applying Proposition \ref{p:h_ell} gives  
	\begin{align}
		(1+C\ve)^{-1}\|v\|_{W_\ell,x} \leq \|D[\pi_\ell^{-1} \circ p_\ell](x) \cdot v \|_{M_{\ell-1},y} \leq (1+C\ve)\|v\|_{W_\ell,x}.
	\end{align}
	The fact that $Dp_\ell(x)$ is $(1+C\ve)$-bi-Lipschitz follows by plugging this estimate into \eqref{e:p-bi-lip} and using the fact that $\theta_1^\ell + \theta_2^\ell =1$. The proof for $D\pi_\ell(x)$ is virtually identical, we omit the details. 
\end{proof}

Since $M_\ell$ is a connected manifold, it is path connected. In particular, we can define the path metric. For $x,y \in M_{\ell},$ let $\Gamma_{x,y,\ell}$ be the collection of all differentiable curves $\gamma \colon [0,1] \to M_{\ell}$ such that $\gamma(0) = x$ and $\gamma(1) = y.$ Let $d_{\ell}$ be the path metric on $M_\ell$ defined by 
\begin{align}\label{e:path-metric}
	d_{\ell}(x,y) = \inf_{\gamma \in \Gamma_{x,y,\ell}} \int_0^1 \| \gamma'(t) \|_{M_\ell,\gamma(t)} \, dt
\end{align}

\begin{lem}\label{l:f-bilip1}
	If $x \in M_\ell$ and $v \in T_xM_\ell$ then 
	\begin{align}\label{e:f-bi-lip}
		(1+C\ve)^{-1}\|v\|_{M_\ell,x} \leq \|Df_\ell(x) \cdot v\|_{M_{\ell+1},f_\ell(x)} \leq (1+C\ve)\|v\|_{M_\ell,x}.
	\end{align}
	In particular, $f_\ell \colon M_\ell \to M_{\ell+1}$ is bi-Lipschitz. Assume further that $f_\ell(x) \in M_{\ell+1} \setminus \overline{p_{\ell+1}(W_{\ell+1}^*)}$. Then, 
	\begin{align}\label{e:f-isom}
		\|Df_\ell(x) \cdot v\|_{M_{\ell+1},f_\ell(x)} = \|v\|_{M_\ell,x}
	\end{align}
\end{lem}


\begin{proof}
	Fix a pair $x \in M_\ell$ and $v \in T_xM_\ell.$ We start with \eqref{e:f-bi-lip}. Suppose first $x \in G_\ell$ so that $f_\ell(x) = p_{\ell+1} \circ h_\ell \circ p_\ell^{-1}(x)$ by \eqref{e:def-f}. By Proposition \ref{p:h_ell}, Lemma \ref{l:p-bi-lip} and \eqref{d:F} we have  
	\begin{align}
		\|Df_\ell(x)\cdot v\|_{M_{\ell+1},f_\ell(x)} &\leq (1+C\ve)\|D[ h_{\ell} \circ p_{\ell}^{-1}](x) \cdot v \|_{W_{\ell+1},h_{\ell} \circ p_{\ell}^{-1}(x)} \\
		&\leq (1+C\ve) \|D[p_\ell^{-1}](x) \cdot v\|_{W_\ell,p_\ell^{-1}(x)} \leq (1+C\ve) \|v\|_{M_\ell,x}.
	\end{align}
	The argument for the lower estimate is similar. If instead $x \in M_\ell^*,$ so that $f_\ell(x) = \pi_{\ell+1}(x)$ by \eqref{e:def-f}, we can apply Lemma \ref{l:p-bi-lip} to $D\pi_{\ell+1}$ which gives \eqref{e:f-bi-lip} immediately.
	
	We now prove \eqref{e:f-isom}. Suppose $f_\ell(x) \in M_{\ell+1} \setminus \overline{p_{\ell+1}(W_{\ell+1}^*)}$ and let $V$ be an open neighbourhood of $x$ such that $f_\ell(V) \subseteq M_{\ell+1} \setminus \overline{p_{\ell+1}(W_{\ell+1}^*)}.$ By \eqref{e:f(G)} we have $f_\ell(G_\ell) = p_{\ell+1}(W_{\ell+1}^*).$ It follows that $V \subseteq M_{\ell}^*$ and $f_{\ell}|_V = \pi_{\ell+1}|_V$. Thus, 
	\begin{align}\label{e:identity}
		I_n = D[\pi_{\ell+1}^{-1} \circ f_\ell](x) = D[\pi_{\ell+1}^{-1}](f_\ell(x))Df_{\ell}(x).
	\end{align} 
	Recall the definitions of $\theta_i^{\ell+1}$ from above \eqref{d:F}. Since $\theta^{\ell+1}_1$ is supported in $p_{\ell+1}(W_{\ell+1}^*)$ and $\{\theta_1^{\ell+1},\theta^{\ell+1}_2\}$ is a partition of unity, we have $\theta_1^{\ell+1}(f_\ell(x)) = 0$ and $\theta_2^{\ell+1}(f_\ell(x)) = 1.$ In particular, by \eqref{d:F} and \eqref{e:identity}, we have 
	\begin{align}
		\| Df_\ell(x) \cdot v \|_{M_{\ell+1},f_\ell(x)} = \| D[\pi_{\ell+1}^{-1}](f_\ell(x)) Df_\ell(x) \cdot v\|_{M_{\ell},\pi_{\ell+1}^{-1}(f_\ell(x))} = \| v \|_{M_\ell,x}. 
	\end{align}
\end{proof}

We introduce an important collection of charts on $M_{\ell}$ (not covering $M_\ell$) which roughly speaking correspond to the location of $W_{\ell}$ in $M_{\ell}.$ Recall the definitions of $\lambda U^{j,\ell}$ and $\vp_{j,\ell}$ from Section \ref{s:disconnected}. For $0 \leq \lambda \leq 6$, $\ell \geq 0$ and $j \in J_\ell$ let 
\begin{align}\label{d:D}
	\lambda D^j = p_\ell(\lambda U^j).
\end{align}
Furthermore, let 
\begin{align}\label{d:psi}
	\psi_j \coloneqq \varphi_j \circ p_\ell^{-1} \colon 6D^j \to 6B_j.
\end{align}
We prove some basic properties concerning the maps $\psi_j$ and the sets $\lambda D^j.$ The first lemma tells us that the $I_{i,j}$ act as transitions maps these coordinate charts in $M_\ell.$ 

\begin{lem}\label{l:tran-psi}
	Let $\ell \geq 1,$ $i,j \in J_\ell$ and suppose $x \in 6B_i$ is such that $\psi_i^{-1}(x) \in 6D^j.$ Then $x \in \Omega_{j,i}$ and
	\[ I_{j,i}(x) =  \psi_j(\psi_i^{-1}(x)). \]
\end{lem}

\begin{proof}
	Since $\psi_i^{-1}(x) \in 6D^j$, it follows from \eqref{d:D} and \eqref{d:psi} that $\vp_i^{-1}(x) \in 6U^j.$ In particular $x \in \Omega_{j,i}$ and $\vp_j(\vp_i^{-1}(x)) = I_{j,i}(x)$ by Remark \ref{r:transition'}. Using this and again \eqref{d:psi}, we get 
	\begin{align}
		\psi_j(\psi_i(x)) = \vp_j \circ p_\ell^{-1} \circ p_\ell \circ \vp_i^{-1}(x) = \vp_j(\vp_i^{-1}(x)) = I_{j,i}(x).
	\end{align}
\end{proof}

\begin{lem}\label{l:f-good}
	Let $\ell \geq 0$, $j \in J_{\ell+1}$ and suppose $x \in M_\ell$ is such that $f_\ell(x) \in 6D^j.$ Then, 
	\begin{align}
		f_\ell(x) = \psi_j^{-1}(x) \circ \hat{H}_j \circ K_{j,i(j)} \circ \psi_{i(j)}(x)
	\end{align}
	with $i(j)$ as in \eqref{e:i(j)}, $K_{i(j),j}$ as in Lemma \ref{l:K-alpha} and $\hat{H}_j$ as in Proposition \ref{p:h_ell}.
\end{lem}

\begin{proof}
	By \eqref{d:D} and \eqref{e:h_ell3''} we have $f_\ell(x) \in 6D^j = p_{\ell+1}(6U^j) \subseteq p_{\ell+1}(W_{\ell+1}^*)$. Equation \eqref{e:f(G)} then implies $x \in G_\ell$ so that $f_\ell(x) = p_{\ell+1} \circ h_\ell \circ p_{\ell}^{-1}(x).$ Since $h_\ell(p_\ell^{-1}(x)) = p_{\ell+1}^{-1}(f_\ell(x)) \in 6U^j,$ It follows from the first inclusion in \eqref{e:h_ell3'} that there exists $p_\ell^{-1}(x) \in C_{N_0}^j$. Now, using \eqref{e:h_ell2} and \eqref{d:psi}, we get 
	\begin{align}
		f_\ell(x) = p_{\ell+1} \circ \vp_j^{-1} \circ \hat{H}_j \circ K_{j,i(j)} \circ \vp_{i(j)} \circ p_\ell^{-1}(x) = \psi_j^{-1} \circ \hat{H}_j \circ K_{j,i(j)} \circ \psi_{i(j)}(x)
	\end{align}
\end{proof}

If $j \in J_{\ell+1}$ and $i(j) \in J_\ell$ is as in \eqref{e:i(j)}, it follows from the triangle inequality that $6B^j \subseteq 3B^{i(j)}.$ The follow lemma tells us that the corresponding charts in $M_\ell$ satisfy a similar nesting property (up to pushing $3D^{i(j)}$ forward by $f_\ell$).

\begin{lem}\label{l:inclusion}
	If $\ell \geq 0$ and $j \in J_{\ell+1}$ then $6D^j \subseteq f_\ell(3D^{i(j)})$. Moreover, 
	\begin{align}\label{e:nested}
		\bigcup_{j \in J_{\ell+1}} 6D^j \subseteq p_{\ell+1}(W_{\ell+1}^*) \subseteq f_\ell\left( \bigcup_{j \in J_\ell} 3D^j\right) \subseteq f_\ell(p_\ell(W_\ell^*)).
	\end{align}
\end{lem}

\begin{proof}
	Let $\ell \geq 0$ and $j \in J_{\ell+1}.$ Recalling the definitions of $G_\ell$ and $f_\ell$ from \eqref{e:F_ell'} and \eqref{e:def-f}, respectively, and applying \eqref{d:D} and \eqref{e:h_ell3'} we have 
	\begin{align}
		6D^j  &= p_{\ell+1}(6U^j) \subseteq p_{\ell+1} \circ h_\ell(E_\ell \cap 3U^{i(j)}) = p_{\ell+1} \circ h_\ell \circ p_\ell^{-1} \circ p_\ell(E_\ell \cap 3U^{i(j)}) \\
		&= p_{\ell+1} \circ h_\ell \circ p_\ell^{-1}(G_\ell \cap 3D^{i(j)}) =  f_\ell(G_\ell \cap 3D^{i(j)}) \subseteq f_\ell(3D^{i(j)}). 
	\end{align}
	We now check \eqref{e:nested}. Recall from \eqref{e:f(G)} that $f_\ell(G_\ell) = p_{\ell+1}(W_{\ell+1}^*).$ Using this with \eqref{e:h_ell3''} and \eqref{e:h_ell4} we get 
	\begin{align}
		\bigcup_{j \in J_{\ell+1}} 6D^j & = p_{\ell+1}\left(\bigcup_{j \in J_{\ell+1}}6U^j\right) \subseteq p_{\ell+1}(W_{\ell+1}^*) = f_\ell(G_\ell) = f_{\ell}(p_\ell(E_\ell))  \\
		&\subseteq f_{\ell}\left(p_\ell\left(\bigcup_{j \in J_\ell} 3U^j\right)\right)  = f_{\ell}\left(\bigcup_{j \in J_\ell} 3D^j\right). 
	\end{align}
	This take care of the first two inequality in \eqref{e:nested}. The final inequality follows from \eqref{e:h_ell3''} since 
	\[ f_{\ell}\left(\bigcup_{j \in J_\ell} 3D^j\right) = f_{\ell}\left(p_\ell\left(\bigcup_{j \in J_\ell} 3U^j\right)\right) \subseteq f_\ell(p_\ell(W_\ell^*)). \]
\end{proof}

The following lemma will be useful in showing that each $\psi_j$ is $(1+C\ve)$-bi-Lipschitz.

\begin{lem}\label{l:lower-bound}
	Let $\ell \geq 0,$ $j \in J_\ell$ and suppose $\gamma \colon [t_1,t_2] \to 6D^j$ is a smooth curve. Then, 
	\begin{align}
		\int_{t_1}^{t_2} \|\gamma'(t)\|_{M_\ell,\gamma(t)} \, dt \geq (1+C\ve)^{-1} \| \psi_j(\gamma(t_2)) - \psi_j(\gamma(t_1)) \|_j . 
	\end{align}
\end{lem}

\begin{proof}
	Since $\gamma(t) \in 6D^j$ for all $t \in [t_1,t_2]$ we have that $\psi_j \circ \gamma \colon [0,1] \to 6B_j$ defines a smooth path from $\psi_j(\gamma(t_1))$ to $\psi_j(\gamma(t_2)).$ Using this with Lemma \ref{l:Dpsi} and Lemma \ref{l:p-bi-lip} gives
	\begin{align}
		\| \psi_j(\gamma(t_2)) - \psi_j(\gamma(t_1)) \|_j &= \int_{t_1}^{t_2} \| (\psi_j \circ \gamma)'(t) \|_j \, dt \\
		&= \int_{t_1}^{t_2} \| D\vp_j(p_\ell^{-1}(\gamma(t))) D[p_\ell^{-1}](\gamma(t))\cdot\gamma'(t) \|_j \, dt \\
		&\leq (1+C\ve) \int_{t_1}^{t_2} \| \gamma'(t) \|_{M_{\ell},\gamma(t)} \, dt.
	\end{align}
\end{proof}

\begin{lem}\label{l:psi-bi-lip}
	For each $\ell \geq 0$ and $j \in J_\ell,$ the map $\psi_j \colon 6D^j \to 6B_j$ is $(1+C\ve)$-bi-Lipschitz when $6D^j$ is equipped with $d_\ell$.
\end{lem}

\begin{proof}	
	We start by showing the following. \\
	
	\noindent\textbf{Claim 1:} For each $\ell \geq 0$, $j \in J_\ell$ and $x,y \in 6D^j$ we have  
	\begin{align}
		d_\ell(x,y) \leq (1+C\ve)  \|\psi_j(x) - \psi_j(y) \|_j.
	\end{align}
	
	Indeed, fix some $\ell \geq 0$, $j \in J_\ell$, $x,y \in 6D^j$ and let $\tilde{\gamma} \colon [0,1] \to 6B_j$ defined by $\tilde{\gamma}(t) = \psi_j(x) + t(\psi_j(y) - \psi_j(x))$ be the straight line connecting $\psi_j(x)$ with $\psi_j(y).$ Define a curve $\gamma \colon [0,1] \to M_\ell$ by 
	\[\gamma  = \psi_j^{-1} \circ \tilde{\gamma} = p_\ell \circ \vp_j^{-1} \circ \tilde{\gamma}. \]
	Applying Lemma \ref{l:Dpsi} and Lemma \ref{l:p-bi-lip} we get 
	\begin{align}
		d_\ell(x,y) &\leq \int_0^1 \| \gamma'(t) \|_{M_\ell,\gamma(t)} \, dt \leq (1+C\ve) \int_0^1 \| \tilde{\gamma}'(t)  \|_j \, dt \\
		&=  (1+C\ve)\| \psi_j(y) - \psi_j(x) \|_j.
	\end{align}
	
	For the reverse estimate, we start by showing a weaker statement. \\
	
	\noindent\textbf{Claim 2:} For each $\ell \geq 0$, $j \in J_\ell$ and $x,y \in 3D^j$ we have 
	\begin{align}\label{e:lip}
		\|\psi_j(x) - \psi_j(y) \|_j \leq (1+C\ve)d_\ell(x,y).
	\end{align}
	
	Indeed, fix some $\ell \geq 0$, $j \in J_\ell$ and $x,y \in 3D^j.$ If $x=y$ then $\psi_j(x) = \psi_j(y)$ and there is nothing to show so we may suppose $x$ and $y$ are distinct. Now, let $\gamma \colon [0,1] \to M_\ell$ a smooth curve with $\gamma(0) = x, \ \gamma(1) = y$ such that  
	\begin{align}\label{e:gamma-opt}
		\int_0^1 \|\gamma'(t)\|_{M_\ell,\gamma(t)} \, dt \leq (1+\ve)d_\ell(x,y).
	\end{align}
	If $\gamma(t) \in 6D^j$ for all $t \in [0,1]$ then the estimate \eqref{e:lip} follows immediately from \eqref{e:gamma-opt} and Lemma \ref{l:lower-bound}. Suppose instead that there exists $t \in [0,1]$ such that $\gamma(t) \in \partial(6D^j).$ In this case we can find $t_1,t_2 \in [0,1]$ such that 
	\begin{align}\label{e:boundary}
		\gamma(t_1),\gamma(t_2) \in \partial((6-\ve)D^j) \subseteq 6D^j
	\end{align}
	and 
	\[\gamma(t) \in (6-\ve)D^j \subseteq 6D^j \mbox{ for all } 0 \leq t < t_1 \mbox{ and } t_2 < t \leq 1.\]
	Recalling $\gamma(0) = x \in 3D^j$ and using \eqref{e:boundary}, we have $\psi_j(\gamma(0)) \in 3B_j$ and $\psi_j(\gamma(t_1)) \in \partial((6-\ve)B_j)$. This implies $\| \psi_j(\gamma(t_1)) - \psi_j(\gamma(0)) \|_j \geq (3-\ve)r_\ell.$ Similarly, $\| \psi_j(\gamma(1)) - \psi_j(\gamma(t_2)) \|_j \geq (3-\ve)r_\ell.$ Using these estimates along with Lemma \ref{l:lower-bound} and \eqref{e:gamma-opt} gives 
	\begin{align}
		\begin{split}\label{e:d-r}
			(1+\ve)d_\ell(x,y) &\geq \int_0^{t_1} \|\gamma'(t)\|_{M_\ell,\gamma(t)} \, dt + 	\int_{t_2}^1 \|\gamma'(t)\|_{M_\ell,\gamma(t)} \, dt \\
			&\geq \frac{1}{1+C\ve}\left(\| \psi_j(\gamma(t_1)) - \psi_j(\gamma(0)) \|_j  + \| \psi_j(\gamma(1)) - \psi_j(\gamma(t_2)) \|_j \right)\\
			&\geq \frac{6r_\ell}{1+C\ve}.
		\end{split}
	\end{align} 
	Since $x,y \in 3D^j$ we have $\psi_j(x),\psi_j(y) \in 3B_j$ so that $\|\psi_j(x) - \psi_j(y) \|_j \leq 6r_\ell.$ This observation and \eqref{e:d-r} finish the proof of Claim 2 since 
	\begin{align}
		\|\psi_j(x) - \psi_j(y) \| \leq 6r_\ell \leq (1+C\ve)d_\ell(x,y).
	\end{align}

	We use Claim 2 to prove the following, which completes the proof of the lemma. \\
	
	\noindent\textbf{Claim 3:} For each $\ell \geq 0$, $j \in J_\ell$ and $x,y \in 6D^j$ we have 
	\begin{align}\label{e:lip1}
		\|\psi_j(x) - \psi_j(y) \|_j \leq (1+C\ve)d_\ell(x,y).
	\end{align}
	The case $\ell = 0$ is immediate from construction, but let us outline where to find the relevant information. Recall from Theorem \ref{t:Reif} that $J_0 = \{j_0\}.$ By \eqref{e:norm-0} and \eqref{e:path-metric} the path metric $d_0$ is just the metric induced by $\|\cdot\|_{j_0}.$ By Remark \ref{l:W_0}, \eqref{e:M_0} and \eqref{d:psi} we have $\psi_{j_0} = \text{Id}$. Combining these facts gives \eqref{e:lip1}.
	
	Suppose instead that $\ell \geq 1$ and fix some $j \in J_\ell$ and $x,y \in 6D^j.$ By Lemma \ref{l:inclusion} we have $f_{\ell-1}^{-1}(x),f_{\ell-1}^{-1}(y) \in 3D^{i(j)}$ so that, by Lemma \ref{l:f-bilip1} and Claim 2, 
	\begin{align}\label{e:lip2}
		\hspace{1em}\| \psi_{i(j)}(f_{\ell-1}^{-1}(x)) - \psi_{i(j)}(f_{\ell-1}^{-1}(y)) \|_{i(j)} &\leq (1+C\ve)d_{\ell-1}(f_{\ell-1}(x),f_{\ell-1}(y)) \\
		&\leq (1+C\ve)d_{\ell}(x,y). 
	\end{align}
	By Lemma \ref{l:f-good}, 
	\begin{align}
		&\| \psi_{i(j)}(f_{\ell-1}^{-1}(x)) - \psi_{i(j)}(f_{\ell-1}^{-1}(y)) \|_{i(j)} \\
		&\hspace{2em}=  \| K_{i(j),j} \circ \hat{H}_{j}^{-1} \circ \psi_j(x) - K_{i(j),j} \circ \hat{H}_{j}^{-1} \circ \psi_j(y)\|_{i(j)},
	\end{align}
	with $K_{i(j),j}$ as in Lemma \ref{l:K-alpha} and $\hat{H}_j$ as in Proposition \ref{p:h_ell}. Since both $K_{i(j),j}$ and $\hat{H}_j$ are $(1+C\ve)$-bi-Lipschitz, this implies
	\begin{align}\label{e:lip3}
		\| \psi_j(x) - \psi_j(y)\|_j \leq (1+C\ve)	\| \psi_{i(j)}(f_{\ell-1}^{-1}(x)) - \psi_{i(j)}(f_{\ell-1}^{-1}(y)) \|_{i(j)} .
	\end{align}	
	Equations \eqref{e:lip2} and \eqref{e:lip3} finish the proof of Claim 3 and hence the proof of the lemma.

\end{proof} 

\begin{lem}\label{l:ball-D}
	Let $C_*$ be the constant appearing Lemma \ref{l:lower-bound} and let $0 < \lambda ,\mu < 6$ be real numbers such that $\nu \coloneqq \lambda + (1+C_*\ve)\mu < 6.$ Let $\ell \geq 0,\ j \in J_\ell$ and $x \in \lambda D^j.$ If $y \in M_\ell$ is such that $d_\ell(x,y) < \mu r_\ell$ then 
	\begin{align}
		y\in \nu D^j.
	\end{align}
\end{lem}

\begin{proof}
	It suffices to prove the contrapositive, that is, if $y \in M_\ell \setminus \nu D^j$ then $d_\ell(x,y) \geq \mu r_\ell.$ Fix some $y \in  M_\ell \setminus \nu D^j.$ Let $\eta > 0$ and a choose smooth path $\gamma \colon [0,1] \to M_{\ell}$ with $\gamma(0) = x$, $\gamma(1) = y$ and such that 
	\begin{align}
		\int_0^1 \|\gamma'(t)\|_{M_{\ell},\gamma(t)} \, dt \leq d_\ell(x,y) + \eta. 
	\end{align}
	Since $y \not\in \nu D^j$ there exists $t \in [0,1]$ such that $\gamma(t) \in \partial(\nu D^j).$ Let $t_0$ be the infimum over all such $t.$  In this way we have  $\gamma(t_0) \in \partial(\nu B_j) \mbox{ and } \gamma(t) \in \nu D^j  \mbox{ for all } 0 \leq t < t_0.$ By the definition of $\psi_j$ and $\nu D^j$, we have $x \in \lambda B_j$ and $\psi_j(\gamma(t_0)) \in \partial(\nu B_j)$ so that 
	\begin{align}
		\| \psi_j(x) - \psi_j(\gamma(t_0)) \|_j \geq (1+C_*\ve) \mu r_\ell. 
	\end{align}
	Applying Lemma \ref{l:lower-bound} gives 
	\begin{align}
		\mu r_\ell \leq (1+C_*\ve)^{-1}	\| \psi_j(x) - \psi_j(\gamma(t_0)) \|_j \leq 	\int_0^1 \|\gamma'(t)\|_{M_{\ell},\gamma(t)} \, dt \leq d_\ell(x,y) + \eta. 
	\end{align}
	Since $\eta$ was arbitrary, this implies $d_\ell(x,y) \geq \mu r_\ell$ and finishes the proof of the lemma. 
\end{proof}

In the next section we will need to track the trajectory of points $x \in \R^n$ under the maps $f_\ell$ from \eqref{e:def-f}. To do this we define a sequence of maps $F_\ell \colon \R^n \to M_\ell$ inductively by setting
\begin{align}\label{e:map-F}
	F_0 = \text{Id} \mbox{ and } F_\ell = f_{\ell-1} \circ F_{\ell-1} \mbox{ for } \ell \geq 1. 
\end{align}

\begin{lem}\label{l:first}
	Let $k \geq 0$ and let $x \in \R^n$ such that $F_k(x) \in M_k \setminus \bigcup_{j \in J_k} 5D^j.$ Then, 
	\[{\dist}_\ell\left(F_\ell(x), \bigcup_{j \in J_{\ell}} 6D^j\right) \geq r_{k}/2\]
	for all $\ell \geq k+1,$ where $\dist_\ell$ is the distance in $M_\ell$ with respect to the path metric $d_\ell.$ 
\end{lem} 

\begin{proof}
	Let $B_{k} = B(F_k(x),r_{k}/2)$ (the ball taken with respect to $d_{k}$) and, for $\ell \geq k+1$, let $B_\ell = f_{\ell-1} \circ \dots \circ f_{k+1}(B_{k}).$ Since $F_k(x) \not\in \bigcup_{j \in J_k} 5D^j$, we have by Lemma \ref{l:ball-D} that $d_k(F_k(x),\bigcup_{j \in J_k} 4D^j) \geq (1+C\ve)^{-1}r_k.$ It follows from the definition of $B_k$ that 
	\begin{align}
		\overline{B_k} \cap \bigcup_{j \in J_{k}} 4D^j = \emptyset. 
	\end{align}
	Thus, for $\ell \geq k+1,$ this and \eqref{e:nested} imply 
	\begin{align}\label{e:isom}
		\overline{B_\ell} \cap \bigcup_{j \in J_\ell} 6D^j \subseteq \overline{B_{\ell}} \cap p_{\ell}(W_\ell^*) = \emptyset
	\end{align}
	so that 
	\begin{align}\label{e:disjoint-2}
		{\dist}_\ell\left(F_\ell(x),\bigcup_{j \in J_{\ell}} 6D^j\right) \geq {\dist}_{\ell}\left(F_\ell(x),\partial(B_\ell)\right).
	\end{align}
	Let us estimate ${\dist}_{\ell}\left(F_\ell(x),\partial(B_\ell)\right)$ from below. Suppose $\gamma \colon [0,1] \to B_\ell$ is a smooth curve connecting $F_\ell(x)$ with $\partial(B_\ell)$. It follows that \[\tilde{\gamma} \coloneqq f_k^{-1} \circ \cdots \circ f_{\ell-1}^{-1} \circ \gamma \colon [0,1] \to B_k\]
	is a smooth curve from $F_k(x)$ to $\partial(B_k).$ By \eqref{e:isom}, for each $k \leq m \leq \ell -1$ and each $t \in [0,1]$ we have 
	\begin{align}
		f_m^{-1} \circ \cdots \circ f_{\ell-1}^{-1}(\gamma(t)) \in B_m \subseteq M_m \setminus p_{m}(W_m^*). 
	\end{align}
	This and \eqref{e:f-isom} imply
	\[ \| \tilde{\gamma}'(t) \|_{M_k,\tilde{\gamma}(t)} = \| \gamma'(t) \|_{M_\ell,\gamma(t)} \mbox{ for all } t \in [0,1]. \]
	Hence, 
	\begin{align}
		\int_0^1 \| \gamma'(t) \|_{M_\ell,\gamma(t)} \, dt = \int_0^1 \| \tilde{\gamma}'(t) \|_{M_k,\tilde{\gamma}(t)} \, dt \geq {\dist}_k(F_k(x),\partial(B_k)) = r_k/2. 
	\end{align} 
	Taking the infimum over all such paths implies ${\dist}_{\ell}\left(F_\ell(x),\partial(B_\ell)\right) \geq r_k/2$ which, when coupled with \eqref{e:disjoint-2}, finishes the proof of the lemma. 
\end{proof}

\begin{lem}\label{l:outside}
	Let $x,y \in \R^n$ and suppose $k \geq 0$ is the largest integer such $F_k(x),F_k(y) \in 6D^j$ for some $j \in J_k.$ If $d_k(F_k(x),F_k(y)) \leq r_{k+1}/2$ then 
	\begin{align}\label{e:outside1}
		F_{k+1}(x),F_{k+1}(y) \not\in \bigcup_{j \in J_{k+1}} 5D^j
	\end{align}
	and
	\begin{align}\label{e:outisde2}
		{\dist}_\ell\left(F_\ell(x), \bigcup_{j \in J_{\ell}} 6D^j\right), \  {\dist}_\ell\left(F_\ell(y), \bigcup_{j \in J_{\ell}} 6D^j\right) \geq r_{k+1}/2
	\end{align}
	for all $\ell \geq k+2,$ where $\dist_\ell$ is the distance in $M_\ell$ with respect to the path metric $d_\ell.$ 
\end{lem}

\begin{proof}
	The inequalities in \eqref{e:outisde2} follow immediately from Lemma \ref{l:first} and \eqref{e:outside1}. Let us see how to prove \eqref{e:outside1}. By symmetry we need only show that $F_{k+1}(x) \not\in \bigcup_{j \in J_{k+1}} 5D^j.$ Suppose towards a contradiction that there exists $j \in J_{\ell+1}$ such that $F_{k+1}(x) \in 5D^j$. By Lemma \ref{l:f-bilip1} we have 
	\begin{align}
		d_{k+1}(F_{k+1}(x),F_{k+1}(y)) \leq (1+C\ve)d_k(F_k(x),F_k(y)) \leq 3r_{k+1}/4. 
	\end{align} 
	Since $F_{k+1}(x) \in 5D^j,$ Lemma \ref{l:ball-D} implies $F_{k+1}(y) \in 6D^j$ which contradicts the maximality of $k$ and proves \eqref{e:outside1}. 
\end{proof}

\bigskip

\subsection{Construction of the metric $\rho$ and the map $g$}\label{s:metric-map}
We turn our attention now to proving Theorem \ref{t:Reif}.

We obtain $g$ as a limit of maps $g_\ell$ defined as follows. Recall the definition of $F_\ell$ from \eqref{e:map-F}. For $x \in \R^n$ and $\ell \geq 0$ we will denote 
\[ x_\ell = F_\ell(x).\]
For each $x \in 50B_{j_0}$, let
\begin{align}\label{e:ell(x)}
	\ell(x) =	\sup\left\{\ell \geq 0 : x_\ell \in \bigcup_{j \in J_\ell}6D^j\right\},
\end{align}
where we set the supremum over the empty set to be zero. Set $j_0(x) = j_0$ and for each $0 \leq \ell \leq \ell(x)$ choose an arbitrary index $j_\ell(x)$ such that $x_\ell \in 6D^{j_\ell(x)}.$ Then, for $\ell \geq 0$, set
\begin{align}\label{e:g-bar}
	g_\ell(x) = 
	\begin{cases}
		\away_{j_\ell(x)} \circ \psi_{j_\ell(x)}\circ F_\ell(x) &\mbox{ if } \ell \leq \ell(x); \\
		\away_{j_{\ell(x)}(x)}  \circ \psi_{j_{\ell(x)}(x)} \circ F_{\ell(x)}(x)  &\mbox{ if } \ell > \ell(x).	
	\end{cases}
\end{align}

See Figure \ref{f:figure-for-5.6}. 

\begin{figure}	
	\tikzset{every picture/.style={line width=0.75pt}} 
	
	\begin{tikzpicture}[x=0.75pt,y=0.75pt,yscale=-0.85,xscale=0.85]
		
		\draw   (23.7,65) -- (58,65) -- (43.3,92) -- (9,92) -- cycle ;

		\draw [color={rgb, 255:red, 0; green, 0; blue, 0 }  ,draw opacity=1 ][line width=0.75] [line join = round][line cap = round]   (99.05,71.68) .. controls (104.11,59.56) and (119.51,68.61) .. (125.36,70.06) .. controls (129.49,71.07) and (133.5,67.47) .. (137.57,66.15) .. controls (146.94,63.13) and (156.82,70.21) .. (156.29,82.41) .. controls (155.74,94.79) and (146.53,96.59) .. (137.57,95.75) .. controls (133.21,95.34) and (129.95,90.09) .. (125.64,89.24) .. controls (118.13,87.78) and (110.07,90.5) .. (102.85,87.62) .. controls (94.41,84.24) and (90.78,76.94) .. (99.05,71.68) ;

		\draw [color={rgb, 255:red, 0; green, 0; blue, 0 }  ,draw opacity=1 ][line width=0.75] [line join = round][line cap = round]   (217.87,71) .. controls (222.13,61.05) and (233.68,64.85) .. (239.72,68.99) .. controls (241.78,70.4) and (243.12,73.48) .. (245.46,74.03) .. controls (250.23,75.15) and (254.61,70.13) .. (258.97,67.64) .. controls (261.18,66.38) and (263.45,69.08) .. (265.01,69.99) .. controls (266.62,70.94) and (268.19,68) .. (269.89,67.3) .. controls (271.28,66.74) and (282.42,70.55) .. (283.4,71.34) .. controls (285.03,72.64) and (285.54,83.28) .. (285.13,84.79) .. controls (283.96,89.04) and (282,92.94) .. (280.24,96.9) .. controls (276.9,104.41) and (271.78,105.68) .. (264.43,104.3) .. controls (263.2,104.07) and (262.23,102.56) .. (260.98,102.62) .. controls (259.7,102.68) and (258.74,104.11) .. (257.54,104.64) .. controls (256.97,104.88) and (254.22,103.69) .. (253.8,103.63) ;

		\draw [color={rgb, 255:red, 0; green, 0; blue, 0 }  ,draw opacity=1 ][line width=0.75] [line join = round][line cap = round]   (253.8,103.63) .. controls (247.52,88.99) and (241.14,100.1) .. (234.84,97.7) .. controls (234.17,97.45) and (234.34,95.99) .. (233.65,95.79) .. controls (232.57,95.48) and (231.47,96.43) .. (230.35,96.43) .. controls (229.85,96.43) and (229.95,95.19) .. (229.46,95.15) .. controls (220.81,94.44) and (208.92,90.88) .. (209.99,78.24) .. controls (210.17,76.15) and (213.73,76.82) .. (215.38,75.68) .. controls (216.5,74.62) and (217.04,71.82) .. (217.87,71) ;

		\draw [color={rgb, 255:red, 0; green, 0; blue, 0 }  ,draw opacity=1 ][line width=0.75] [line join = round][line cap = round]   (404.93,71.33) .. controls (405.89,69.54) and (407.31,65.16) .. (410.03,65.89) .. controls (411.79,66.36) and (413.06,69.73) .. (414.77,69.06) .. controls (416.03,68.58) and (416.9,67.14) .. (418.05,66.34) .. controls (419.62,65.26) and (427.12,67.8) .. (429.35,66.8) .. controls (431.31,65.91) and (432.44,61.54) .. (434.45,62.26) .. controls (437.05,63.18) and (445.67,65.86) .. (449.03,64.07) .. controls (452.94,61.98) and (458.94,57.67) .. (463.6,59.54) .. controls (469,61.7) and (477.29,66.53) .. (476.36,75.87) .. controls (475.36,85.79) and (471.91,86.53) .. (466.52,92.2) .. controls (464.74,94.08) and (464.09,97.4) .. (462.14,99.01) .. controls (460.47,100.4) and (452.95,98.38) .. (450.85,98.56) .. controls (448.71,98.73) and (447.69,96.65) .. (445.38,97.19) .. controls (444.23,97.46) and (442.25,99.74) .. (440.64,99.01) .. controls (439.09,98.31) and (437.54,96.86) .. (435.91,97.19) .. controls (434.42,97.5) and (433.01,98.88) .. (431.53,98.56) .. controls (429.97,98.22) and (429.01,96.06) .. (427.53,95.38) .. controls (419.85,91.86) and (409.76,95.58) .. (402.38,91.3) .. controls (396.37,87.8) and (397.49,81.71) .. (400.19,76.32) .. controls (401.32,74.08) and (401.35,71) .. (404.93,71.33) ;
		
		\draw [color={rgb, 255:red, 0; green, 0; blue, 0 }  ,draw opacity=1 ]   (54,81) -- (88,81) ;
		\draw [shift={(90,81)}, rotate = 180] [color={rgb, 255:red, 0; green, 0; blue, 0 }  ,draw opacity=1 ][line width=0.75]    (10.93,-3.29) .. controls (6.95,-1.4) and (3.31,-0.3) .. (0,0) .. controls (3.31,0.3) and (6.95,1.4) .. (10.93,3.29)   ;
		\draw [color={rgb, 255:red, 0; green, 0; blue, 0 }  ,draw opacity=1 ]   (162,81) -- (201,81) ;
		\draw [shift={(203,81)}, rotate = 180] [color={rgb, 255:red, 0; green, 0; blue, 0 }  ,draw opacity=1 ][line width=0.75]    (10.93,-3.29) .. controls (6.95,-1.4) and (3.31,-0.3) .. (0,0) .. controls (3.31,0.3) and (6.95,1.4) .. (10.93,3.29)   ;
		\draw [color={rgb, 255:red, 0; green, 0; blue, 0 }  ,draw opacity=1 ]   (290,81) -- (329,81) ;
		\draw [shift={(331,81)}, rotate = 180] [color={rgb, 255:red, 0; green, 0; blue, 0 }  ,draw opacity=1 ][line width=0.75]    (10.93,-3.29) .. controls (6.95,-1.4) and (3.31,-0.3) .. (0,0) .. controls (3.31,0.3) and (6.95,1.4) .. (10.93,3.29)   ;
		\draw [color={rgb, 255:red, 0; green, 0; blue, 0 }  ,draw opacity=1 ]   (370,81) -- (390,81) ;
		\draw [shift={(392,81)}, rotate = 180] [color={rgb, 255:red, 0; green, 0; blue, 0 }  ,draw opacity=1 ][line width=0.75]    (10.93,-3.29) .. controls (6.95,-1.4) and (3.31,-0.3) .. (0,0) .. controls (3.31,0.3) and (6.95,1.4) .. (10.93,3.29)   ;
		\draw [color={rgb, 255:red, 0; green, 0; blue, 0 }  ,draw opacity=1 ]   (492,81) -- (524,81) ;
		\draw [shift={(526,81)}, rotate = 181.68] [color={rgb, 255:red, 0; green, 0; blue, 0 }  ,draw opacity=1 ][line width=0.75]    (10.93,-3.29) .. controls (6.95,-1.4) and (3.31,-0.3) .. (0,0) .. controls (3.31,0.3) and (6.95,1.4) .. (10.93,3.29)   ;
		
		\draw (450,189.5) -- (446.08,201.55) -- (435.83,209) -- (423.17,209) -- (412.92,201.55) -- (409,189.5) -- (412.92,177.45) -- (423.17,170) -- (435.83,170) -- (446.08,177.45) -- cycle ;
		\draw[shift={(-7,-1)}]     (361,236) -- (368,247) -- (360,265) -- (342,272) -- (332,260) -- (329,245) -- (336,236) -- (349,232) -- cycle ;
		\draw    (454,108) -- (443.41,159.04) ;
		\draw [shift={(443,161)}, rotate = 281.73] [color={rgb, 255:red, 0; green, 0; blue, 0 }  ][line width=0.75]    (10.93,-3.29) .. controls (6.95,-1.4) and (3.31,-0.3) .. (0,0) .. controls (3.31,0.3) and (6.95,1.4) .. (10.93,3.29)   ;
		\draw    (411,209) -- (366.8,230.14) ;
		\draw [shift={(365,231)}, rotate = 334.44] [color={rgb, 255:red, 0; green, 0; blue, 0 }  ][line width=0.75]    (10.93,-3.29) .. controls (6.95,-1.4) and (3.31,-0.3) .. (0,0) .. controls (3.31,0.3) and (6.95,1.4) .. (10.93,3.29)   ;
		\draw    (38,50) .. controls (78,20) and (397,13) .. (441,55) ;
		\draw [shift={(441,55)}, rotate = 218.93] [color={rgb, 255:red, 0; green, 0; blue, 0 }  ][line width=0.75]    (10.93,-4.9) .. controls (6.95,-2.3) and (3.31,-0.67) .. (0,0) .. controls (3.31,0.67) and (6.95,2.3) .. (10.93,4.9)   ;
		\draw    (39,111) .. controls (150,111) and (269.08,183.96) .. (322.21,226.37) ;
		\draw [shift={(323,227)}, rotate = 218.93] [color={rgb, 255:red, 0; green, 0; blue, 0 }  ][line width=0.75]    (10.93,-4.9) .. controls (6.95,-2.3) and (3.31,-0.67) .. (0,0) .. controls (3.31,0.67) and (6.95,2.3) .. (10.93,4.9)   ;
		\draw[shift={(-1,-2)}]   (476.5,82.75) -- (460.76,99.16) -- (435.29,92.89) -- (435.29,72.61) -- (460.76,66.34) -- cycle ;

		\draw (343,81) node [anchor=north west][inner sep=0.75pt]    {$...$};
		\draw (540,81) node [anchor=north west][inner sep=0.75pt]    {$...$};
		\draw (11,101.4) node [anchor=north west][inner sep=0.75pt]    {\hspace{-4em} \eqref{e:M_0} $M_{0}$};
		\draw (113,101.4) node [anchor=north west][inner sep=0.75pt]    {$M_{1}$};
		\draw (229,101.4) node [anchor=north west][inner sep=0.75pt]    {$M_{2}$};
		\draw (410,101.4) node [anchor=north west][inner sep=0.75pt]    {$M_{\ell }$};

		\draw (375,248.4) node [anchor=north west][inner sep=0.75pt]    {$6B^{j,\ell } \subset X$};
		\draw (454,185.4) node [anchor=north west][inner sep=0.75pt]    {$6B_{j,\ell } \subset \mathbb{R}^{n}$};
		\draw (374,198.4) node [anchor=north west][inner sep=0.75pt]    {$\alpha _{j}$};
		\draw (154,144.4) node [anchor=north west][inner sep=0.75pt]    {\hspace{-2em} $g_{\ell }$ \eqref{e:g-bar}};
		\draw (237,5.4) node [anchor=north west][inner sep=0.75pt]    {$F_{\ell }$ \eqref{e:map-F}};
		
		\draw (63,86.4) node [anchor=north west][inner sep=0.75pt]    {$f_{1}$};
		\draw (172,86.4) node [anchor=north west][inner sep=0.75pt]    {$f_{2}$};
		
		\draw (282,86.4) node [anchor=north west][inner sep=0.75pt]    {\eqref{e:def-f} $f_{3}$};
		
		\draw (437.3,75.01) node [anchor=north west][inner sep=0.75pt]    {$6D_{j}$};
		
		\draw (452,128.4) node [anchor=north west][inner sep=0.75pt]    {$\psi_j  \hspace{0.5em} 6D_j\to 6B_{j,\ell}$ \eqref{d:psi}};
		\draw (469,132) node [anchor=north west][inner sep=0.75pt]    {$.$};
		\draw (469.5,137.5) node [anchor=north west][inner sep=0.75pt]    {$.$};

	\end{tikzpicture}
	\caption{The map $g_\ell$}
	\label{f:figure-for-5.6}
\end{figure}

\begin{lem}\label{l:whereinX}
	Let $x \in 6B_{j_0}$, $0 \leq \ell \leq \ell(x)$ and $i \in J_{\ell}$ such that $x_\ell \in 6D^i$. Then, 
	\begin{align}\label{e:where1}
		d_X(g_\ell(x),\away_i \circ \psi_i(x_\ell)) \lesssim \ve r_\ell. 
	\end{align}
\end{lem} 

\begin{proof}
	Let $j = j_\ell(x).$ By definition $\psi_j(x_\ell) = I_{j,i}(\psi_i(x_\ell)) = \tilde{I}_{j,i}(\psi_j(x_\ell)) \in 6B_j.$ By \eqref{e:almost-id}, and Proposition \ref{p:coherent} we have $d_X(\away_j(\toward_j(x)),x) \leq \delta r_\ell \leq \ve r_\ell$ for all $x \in 10B^j$ and $\| \tilde{I}_{j,i}(x) - \toward_j\circ\away_i(x)\|_j \lesssim \ve r_\ell$ for all $x \in 8B_i.$ Also, by definition we have $\psi_i(x_\ell) \in 8B_i.$ Combing the above, we get  
	\begin{align}
		d_X(g_\ell(x),\away_i \circ \psi_i(x_\ell))  &=  d_X(\away_j(\psi_j(x_\ell)), \away_i(\psi_i(x_\ell))) \\
		&= d_X(\away_j(\tilde{I}_{j,i}(\psi_i(x_\ell))), \away_i(\psi_i(x_\ell))) \\
		&\leq d_X( \away_j(\toward_j(\away_i(\psi_i(x_\ell)))), \away_i(\psi_i(x_\ell))) +C\ve r_\ell \lesssim \ve r_\ell. 
	\end{align}
\end{proof}

\begin{lem}\label{l:gbar}
	For each $\ell \geq 0$ and $x \in 50B_{j_0}$ we have 
	\begin{align}
		d_X(g_\ell(x),g_{\ell+1}(x)) \lesssim \ve r_\ell. 
	\end{align}
	In particular, the limit $g \coloneqq \lim_{\ell \to \infty} g_\ell \colon 50B_{j_0} \to X$ exists, takes values in $50B^{j_0},$ and satisfies
	\[ d_X(g(x),g_\ell(x)) \lesssim \ve r_\ell.\]
\end{lem}

\begin{proof}
	If $x \in 50B_{j_0} \setminus 6B_{j_0}$ then $g_\ell(x) = g_0(x)$ for all $\ell \geq 0$ and there is nothing to show. Suppose then that $x \in 6B_{j_0}$ and fix some $\ell \geq 0$. If $\ell +1 > \ell(x)$ then $g_{\ell+1}(x) = g_\ell(x)$ and there is nothing to show. Let us suppose instead that $\ell +1 \leq \ell(x)$ so that $f_\ell(x_\ell) = x_{\ell+1} \in \bigcup_{j \in J_{\ell+1}} 6D^j.$ Let $j = j_{\ell+1}(x)$ as above \eqref{e:g-bar}. It follows from Lemma \ref{l:f-good} and \eqref{e:g-bar} that 
	\begin{align}
		g_{\ell+1}(x) = \alpha_j \circ \psi_j \circ f_\ell(x_\ell) = \alpha_j \circ \hat{H}_j \circ K_{j,i(j)} \circ \psi_{i(j)}(x_\ell).
	\end{align}
	Using this with \eqref{e:almost-id}, \eqref{e:h_ell3} and Lemma \ref{l:K-alpha}, we have 
	\begin{align}
		d_X(g_{\ell+1}(x),\away_{i(j)}(\psi_{i(j)}(x_\ell))) &\leq d_X(g_{\ell+1}(x) , \away_j \circ \toward_j \circ \away_{i(j)} \circ \psi_{i(j)}(x_\ell)) + C \ve r_{\ell} \\
		&\leq d_X(g_{\ell+1}(x) , \away_j \circ K_{j,i(j)} \circ \psi_{i(j)}(x_\ell)) + C \ve r_{\ell} \\
		&\leq d_X(g_{\ell+1}(x) , \away_j \circ H_j \circ K_{j,i(j)} \circ \psi_{i(j)}(x_\ell)) + C \ve r_{\ell}\\
		&\lesssim \ve r_\ell. 
	\end{align}
	By Lemma \ref{l:whereinX}, $d_X(g_\ell(x),\away_{i(j)}(\psi_{i(j)}(x_\ell)) \lesssim \ve r_{\ell}$ and the result now follow from the triangle inequality. 
\end{proof}

We now construct the semi-metric $\phi_\infty$ and the metric $\rho.$ We begin by defining a sequence of semi-metrics $\phi_\ell$ on $M_\ell$ as follows. For $\ell = 0$ and $x,y \in \R^n$ set 
\begin{align}\label{e:rho_0}
	\phi_0(x,y) \coloneqq \|x-y\|_{j_0}.
\end{align}
Let $\ell \geq 1$ and suppose we have defined $\phi_\ell$ on $M_\ell$. We define $\phi_{\ell+1}$ on $M_{\ell+1}$ by setting 
\begin{align}\label{d:rho}
	\phi_{\ell+1}(x,y) = 
	\begin{cases}
		d_{{\ell+1}}(x,y) & \mbox{ if } x,y \in 6D^j \mbox{ for some } j \in J_{\ell+1}; \\
		\phi_{\ell}(f_\ell^{-1}(x),f_\ell^{-1}(y)) & \mbox{ otherwise. } 
	\end{cases}
\end{align}
Recall the definition of $F_\ell$ in \eqref{e:map-F} and let $\bar{\phi}_\ell$ denote the pull-back of $\phi_\ell$ by $F_\ell$, that is,  
\begin{align}\label{e:bar-phi}
	\bar{\phi}_\ell = F^*_\ell \phi_\ell.
\end{align}

\begin{lem}\label{l:phi-cauchy}
	If $\ell \geq 0$ and $x,y \in M_\ell$ then
	\begin{align}\label{e:phi-cauchy}
		| \phi_{\ell+1}(f_\ell(x),f_{\ell}(y))  - \phi_\ell(x,y) | \lesssim \ve \min\{\phi_\ell(x,y),r_\ell\}. 
	\end{align}
	Furthermore, the limit $\phi_\infty \coloneqq \lim_{\ell \to \infty} \bar\phi_\ell$ exists and for each $\ell \geq 0$ and $w,z \in \R^n$ we have 
	\begin{align}\label{e:infty-ell}
		| \phi_\infty(x,y) - \phi_\ell(x_\ell,y_\ell) | \lesssim \ve r_\ell. 
	\end{align}
\end{lem}

\begin{proof}
	If there does not exist $j\in J_{\ell+1}$ such that $f_\ell(x),f_\ell(y) \in 6D^j$ then \eqref{e:phi-cauchy} is immediate from \eqref{d:rho}. Suppose instead that there exists $j \in J_{\ell+1}$ such that $f_\ell(x),f_\ell(y) \in 6D^j$. Lemma \ref{l:inclusion} gives $x,y \in 3D^{i(j)}$ so that, using \eqref{d:rho} and Lemma \ref{l:f-bilip1}, we have 
	\begin{align}
		| \phi_{\ell+1}(f_\ell(x),f_{\ell}(y))  - \phi_\ell(x,y) |  &= | d_{\ell+1}(f_\ell(x),f_{\ell}(y))  - d_\ell(x,y) | \\
		&\lesssim \ve d_\ell(x,y) = \ve \phi_\ell(x,y).
	\end{align}
	Since $x,y \in 3D^{i(j)}$ we also have $d_\ell(x,y) \lesssim r_\ell$. This and the above estimate complete the proof of \eqref{e:phi-cauchy}. To see that $\phi_\infty$ exists, we observe from \eqref{e:bar-phi} and \eqref{e:phi-cauchy} that  
	\begin{align}
		| \bar{\phi}_{\ell+1}(x,y) - \bar{\phi}_\ell(w,z) | = | \phi_{\ell+1}(f_\ell(F_\ell(x)),f_\ell(F_\ell(y))) - \phi_\ell(F_\ell(x),F_\ell(y)) | \lesssim \ve r_\ell. 
	\end{align} 
	From this we also get \eqref{e:infty-ell}, noting that $\sum_{m = \ell}^\infty r_m \lesssim r_\ell.$ 
\end{proof}

\begin{lem}\label{l:k(x,y)}
	For $x,y \in \R^n$ let 
	\[ k(x,y) = \sup\{\ell \geq 0 : x_k,y_k \in 6D^j \mbox{ for some } j \in J_\ell\}.\]
	For each $\ell \geq k(x,y)$ we have 
	\begin{align}\label{e:phi_infty1}
		\phi_\infty(x,y) = \bar{\phi}_\ell(x,y) = \phi_\ell(x_\ell,y_\ell) = d_\ell(x_\ell,y_\ell).
	\end{align}
\end{lem}

\begin{proof}
	This is immediate from \eqref{d:rho}
\end{proof}

\begin{lem}\label{l:phi_infty}
	Let $x,y \in 50B_{j_0}$ and $k(x,y)$ as in Lemma \ref{l:k(x,y)}. Then, 
	\begin{align}\label{e:phi_infty2}
		| d_X(g(x),g(y)) - \phi_\infty(x,y) | \lesssim \ve r_{k(x,y)}
	\end{align}	
	Furthermore, if $x_\ell \in \bigcup_{j \in J_\ell} 6D^j$ for some $\ell \geq 0$ then 
	\begin{align}\label{e:phi_infty3}
		| d_X(g(x),g(y)) - \phi_\infty(x,y) | \lesssim \ve \left(\phi_\infty(x,y) + r_\ell\right). 
	\end{align}	
\end{lem}

\begin{proof}
	We start with \eqref{e:phi_infty2}. If $k(x,y) = 0$ then $g_0(x,y) = \alpha_{j_0}(x,y)$ by definition and $\phi_\infty(x,y) = \|x-y\|_{j_0}$ by \eqref{e:rho_0} and Lemma \ref{l:k(x,y)}. Equation \eqref{e:phi_infty2} now follows since $\alpha_{j_0} \colon 100B_{j_0} \to 100B^{j_0}$ is a $\delta r_0$-GHA. Suppose instead that $k = k(x,y) \geq 1.$ Let $j \in J_k$ is such that $x_k,y_k \in 6D^j.$ By Lemma \ref{l:whereinX} we have 
	\[ d_X(g_k(x),\away_j \circ \psi_j(x_k)) \lesssim \ve r_k \quad \mbox{ and } \quad d_X(g_k(y),\away_j \circ \psi_j(y_k)) \lesssim \ve r_k. \]
	Using this with Lemma \ref{l:psi-bi-lip}, Lemma \ref{l:gbar}, \eqref{e:phi_infty1} and the fact that $\away_j$ is a $\delta r_k$ isometry, the left-hand side of \eqref{e:phi_infty2} is at most 
	\begin{align}
		&| d_X(g(x),g(y))  - d_X(g_k(x),g_k(y)) | \\
		&\hspace{2em}+ | d_X(g_k(x),g_k(y)) - d_X(\away_j \circ \psi_j(x_k),\away_j \circ \psi_j(y_k)) | \\
		&\hspace{4em} + |d_X(\away_j \circ \psi_j(x_k),\away_j \circ \psi_j(y_k)) - d_k(x_k,y_k) | \lesssim \ve r_k. 
	\end{align}
	Let us move onto \eqref{e:phi_infty3}. If $\ell \leq k+1$ then since $r_{k} \lesssim r_\ell.$ Suppose then that $\ell \geq k+2.$ Since $x_\ell \in \bigcup_{j \in J_\ell} 6D^j,$ Lemma \ref{l:outside} implies
	\begin{align}
		\phi_\infty(x,y) = d_k(x_k,y_k) > r_{k+1}/2 \gtrsim r_k. 
	\end{align}
	Using this with \eqref{e:phi_infty2} gives \eqref{e:phi_infty3}. 
\end{proof}

\begin{lem}\label{l:ball-D-phi}
	Let $\ell \geq 0$ and $x,y \in \R^n$ such that $\phi_\ell(x_\ell,y_\ell) \leq r_{\ell}/2$. Suppose also that there exists $j \in J_\ell$ satisfying $x_\ell \in 3D^{j,\ell}.$ Then, $y_{\ell} \in 4D^{j,\ell}.$ 
\end{lem}

\begin{proof}
	Let $k$ be the largest integer such that $x_k,y_k \in 6D^i$ for some $i \in J_k.$ Suppose, towards a contradiction, that $k \leq \ell-1.$ Since $x_\ell \in 3D^j$ we must have $d_k(x_k,y_k) > r_{k+1}/2$ by Lemma \ref{l:outside}. Lemma \ref{l:phi_infty} then gives
	\begin{align}
		r_\ell/2 \geq \phi_\ell(x_\ell,y_\ell) = \phi_k(x_k,y_k) = d_k(x_k,y_k) > r_{k+1}/2 \geq r_\ell/2
	\end{align}
	which is a contradiction. Thus, we have $k \geq \ell$, so there exists $i \in J_\ell$ satisfying $x_\ell,y_\ell \in 6D^i.$ It follows from \eqref{d:rho} that $d_{\ell}(x_\ell,y_\ell) = \phi_\ell(x_\ell,y_\ell) \leq r_\ell/2$ so that $y_\ell \in 4D^{j,\ell}$ by Lemma \ref{l:ball-D}. 
\end{proof}

We turn our attention to the metric $\rho,$ which we define below. 

\begin{defn}\label{d:rho'}
	Let $x,y \in \R^n.$ A \textit{chain} from $x$ to $y$ is a sequence of points $\{z^i\}_{i=1}^N$ in $\R^n$ such that $z^1 = x$ and $z^N = y.$ Let   
	\[ \rho(x,y)  = \inf \sum_{i=1}^{N-1} \phi_\infty(x_i,x_{i+1}),\]
	where the infimum is taken over all chains $\{z^i\}_{i=1}^N$ from $x$ to $y.$
\end{defn} 

The following says that $\rho$ is bi-Lipschitz equivalent to $\phi_\infty$ with constant close to 1.

\begin{lem}\label{l:pert}
	Let $\eta > 0.$ For $\ve$ small enough depending on $\eta$ we have  
	\begin{align}\label{e:pert}
		(1+C\eta)^{-1} \phi_\infty(x,y) \leq \rho(x,y) \leq  \phi_\infty(x,y)
	\end{align}
	for all $x,y \in \R^n.$ 
\end{lem}

\begin{proof}
	The second inequality is true by definition (consider the chain $z^1 = x$, $z^2 = y$) so let us prove the first. Before proving \eqref{e:pert} for all $x,y \in \R^n$ we will prove it for all $x,y \in 6B_{j_0}.$ We will show how to get \eqref{e:pert} for all $x,y \in \R^n$ at the end of the proof. Let $\eta > 0$, assume $x,y \in 6B_{j_0}$ are distinct and let 
	\[k = k(x,y)\]
	Let $\{z^i\}_{i=1}^N$ a chain from $x$ to $y$ such that 
	\begin{align}\label{e:optimal}
		\sum_{i=1}^{N-1} \phi_\infty(z^i,z^{i+1}) \leq (1+\ve)\rho(x,y) \leq 13r_k,
	\end{align}
	where the last inequality follows from Lemma \ref{l:psi-bi-lip}. For each $\ell \geq 0$ and $1 \leq i \leq N$ let
	\begin{align}
		k(i) = k(z^i,z^{i+1})\quad \mbox{ and } \quad z^i_\ell = F_{\ell}(z^i). 
	\end{align}
	It follows from Lemma \ref{l:phi_infty} that 
	\begin{align}\label{l:useful}
		\phi_\infty(x,y)  = d_k(x_k,y_k) \quad \mbox{ and } \quad 	\phi_\infty(z^i,z^{i+1}) = d_{k(i)}(z_{k(i)}^i,z_{k(i)}^{i+1}).
	\end{align}
	\noindent\textbf{Claim:} We have 
	\begin{align}
		\label{e:1} z^i \in 20B_{j_0} \quad &\mbox{ for all } 1 \leq i \leq N;\\
		\label{e:2} k(i) \geq k-2 \quad &\mbox{ for all } 1 \leq i \leq N-1. 
	\end{align}
	
	Indeed, suppose \eqref{e:1} is false and let $1 < i_0 < N$ be the smallest index such that $z^i \not\in 20B_{j_0}.$ Let $1 \leq i_1 < i_0$ be the largest integer such that $z^{i_0} \in 6B_{j_0}.$ It follows that $k(i) = 0$ for all $1 \leq i_1 \leq i_0.$ Then, by \eqref{e:optimal} and \eqref{l:useful}, 
	\begin{align}
		14 r_0 \leq d_0(z^{i_0},z^{i_1}) \leq \sum_{i=i_1}^{i_0-1} \phi_\infty(z^i,z^{i+1}) \leq 13r_k,
	\end{align}
	which is a contradiction. We now consider \eqref{e:2}. If $k \leq 2$ then this is immediate. Suppose then that $k \geq 3$ and suppose towards a contradiction that there exists $1 \leq i \leq N-1$ such that $k(i) \leq k-3$. Let $i_2$ be the smallest such index so that 
	\[ k(i_2) \leq k-3 \mbox{ and } k(i) \geq k-2 \mbox{ for all } 1 \leq i < i_2.\]
	We have 
	\begin{align}
		d_{k(i_2)}(z^{i_2}_{k(i_2)}, z^{i_2+1}_{k(i_2)}) = \phi_\infty(z^{i_2}, z^{i_2+1}) \leq (1+\ve)\rho(x,y)  \leq 13r_k \leq \frac{13r_{k(i_2)}}{1000}. 
	\end{align}
	It follows from Lemma \ref{l:outside} that 
	\[ {\dist}_k\left(z^{i_2}_k,\bigcup_{j \in J_k} 6D^j\right) \geq r_{k(i_2)+1}/2 \geq 50r_k .\]
	Since $k = k(x,y)$ and $z^1 = x,$ there exists $j \in J_k$ such that $z^1 \in 6D^j.$ In particular $i_2 \geq 2.$ Let $1 \leq i_1 < i_2$ be the largest integer such that $z^{i_3}_k \in \bigcup_{j \in J_k} 6D^j.$ It follows that if $i_3 \leq i < i_2$ then 
	\begin{align}
		k-2 \leq k(i) \leq k-1. 
	\end{align}
	Hence, 
	\begin{align}
		\sum_{i=i_3}^{i_2-1} \phi_\infty(z^i,z^{i+1}) \geq (1+C\ve)^{-1}\sum_{i=i_3}^{i_2-1} d_k(z^i_k,z^{i+1}_k) \geq (1+C\ve)^{-1}d_k(z^{i_2}_k,z^{i_3}_k) \geq 49r_k. 
	\end{align}
	This is a contradiction. \\
	
	Let $\eta > 0$ to be chosen small enough. We split the proof of \eqref{e:pert} into two cases. \\
	
	\noindent \textbf{Case 1:} Suppose first that $d_k(x_k,y_k) \leq \eta r_k.$ We claim that 
	\begin{align}\label{e:k(i)}
		k(i) -1 \leq k \leq k(i) + 2
	\end{align}
	for all $1 \leq i \leq N-1.$ Assuming \eqref{e:k(i)} for the moment, by applying \eqref{l:useful}, Lemma \ref{l:f-bilip1} and Lemma \ref{l:phi_infty} we get
	\begin{align}
		\phi_\infty(x,y) &= d_k(x_k,y_k) \leq \sum_{i=1}^{N-1} d_{k}(z_{k}^i,z_{k}^{i+1}) \leq (1+C\ve)  \sum_{i=1}^{N-1} d_{k(i)}(z_{k(i)}^i,z_{k(i)}^{i+1}) \\
		&\leq (1+C\ve)(	\tau r_k + \rho(x,y) ).
	\end{align}
	Since $\tau$ was arbitrary this finishes the proof of \eqref{e:pert} in Case 1 up to showing \eqref{e:k(i)}. We will prove \eqref{e:k(i)} below.

	We get the second inequality from \eqref{e:2}, so let us prove the first. Suppose towards a contradiction that there exists $1 \leq i \leq N-1$ such that $k(i) \geq k+2.$ Let $i_0$ be the smallest such index. We have 
	\begin{align}\label{e:zin}
		z_{k+2}^{i_0} \in \bigcup_{j \in J_{k+2}} 6D^j.
	\end{align}
	Furthermore, for any $1 \leq i < i_0$, since $k-1 \leq k(i) \leq k+1$ and $\phi_{k(i)}(z^i_{k(i)},z^{i+1}_{k(i)}) = d_{k(i)}(z^i_{k(i)},z^{i+1}_{k(i)})$, we have by Lemma \ref{l:f-bilip1} that 
	\[ (1+C\ve)^{-1}\phi_{k(i)}(z^i_{k(i)},z^{i+1}_{k(i)}) \leq  d_{k+2}(z_{k+2}^i,z_{k+2}^{i+1}) \leq (1+C\ve)\phi_{k(i)}(z^i_{k(i)},z^{i+1}_{k(i)}) \]
	Using this with \eqref{e:zin} and Lemma \ref{l:outside} gives
	\begin{align}
		(1+C\ve)\phi_{k(i_0-1)}(z_{k(i_0-1)}^{i_0-1},z_{k(i_0-1)}^{i_0}) &\geq d_{k+2}(z_{k+2}^{i_0-1},z_{k+2}^{i_0}) \\
		&\geq {\dist}_{k+2}\left(z_{k+2}^{i_0-1} , \bigcup_{j \in J_{k+2}} 6D^j\right) \\
		&\geq {\dist}_{k+2}\left(x_{k+2}, \bigcup_{j \in J_{k+2}} 6D^j\right) \\
		&\hspace{2em}- \sum_{i=1}^{i_0-2} d_{k+2}(z_{k+2}^i,z_{k+2}^{i+1}) \\
		&\geq r_{k+1}/2 - (1+C\ve)\sum_{i=1}^{i_0-2} \phi_{k(i)}(z_{k(i)}^i,z_{k(i)}^{i+1}) \\
		&\geq r_{k+2}/2. 
	\end{align}
	For $\ve$ small enough, this implies
	\begin{align}
		(1+\ve)\eta r_k \geq (1+\ve)\rho(x,y) \geq \phi_{k(i_0-1)}(z_{k(i_0-1)}^{i_0-1},z_{k(i_0-1)}^{i_0}) \geq r_{k}/4 \geq r_k,
	\end{align}
	which is a contradiction for $\eta$ small enough. This finishes the proof of the claim.  \\
	
	\noindent\textbf{Case 2:} Suppose that $d_k(x_k,y_k) > \eta r_k.$ We will find a new chain $\{w^i\}_{i=1}^K \subseteq 50B_{j_0}$ from $x$ to $y$ such that the following holds. 
	\begin{enumerate}
		\item We have 
		\begin{align}\label{e:optimal'}
			\sum_{i=1}^{K-1} \phi_\infty(w^i,w^{i+1}) \leq (1+C\eta) \rho(x,y) + C\eta \phi_\infty(x,y).
		\end{align}
		\item For $1 \leq i \leq K-1$ let $\tilde{k}(i)$ be the largest integer such that there exists $j \in J_{\tilde{k}(i)}$ satisfying $w^i_{\tilde{k}(i)},w^{i+1}_{\tilde{k}(i)} \in 6D^j.$ If
		\[I = \{1 \leq  i \leq K-1 : \phi_\infty(w^i,w^{i+1}) \leq \eta^2 r_{\tilde{k}(i)}\}\] then
		\begin{align}\label{e:sum-small}
			\sum_{i \in I} \ve r_{\tilde{k}(i)}  \lesssim \eta \phi_\infty(x,y). 
		\end{align}
	\end{enumerate}
	Before seeing how to construct $\{w^i\}_{i=1}^K$, let use it to complete the proof of \eqref{e:pert} in Case 2. Recall the map $g \colon 50B_{j_0} \to X$ constructed in Lemma \ref{l:gbar}. If $C'$ is the constant appearing in \eqref{e:phi_infty2} and $i \in I^c$ then by taking $\ve$ small enough with respect to $\eta$ we have 
	\begin{align}
		(1+\eta)\phi_\infty(w^i,w^{i+1}) &\geq (1+C'\ve \eta^{-2}) \phi_\infty(w^i,w^{i+1}) \geq  \phi_\infty(w^i,w^{i+1})  + C'\ve r_{\tilde{k}(i)} \\
		&\geq d_X(g(w^i),g(w^{i+1})). 
	\end{align}
	Combine this estimate with \eqref{e:phi_infty2}, \eqref{l:useful}, \eqref{e:sum-small}, using that $\phi_\infty(x,y) = d_k(x_k,y_k) > \eta r_k$ and taking $\ve$ small enough depending on $\eta,$ we have 
	\begin{align}
		\sum_{i=1}^{K-1} \phi_\infty(w^i,w^{i+1}) &= \sum_{i \in I} \phi_\infty(w^i,w^{i+1}) + \sum_{i \in I^c} \phi_\infty(w^i,w^{i+1})  \\
		&\geq \sum_{i \in I}[d_X(g(w^i),g(w^{i+1})) - C\ve r_{\tilde{k}(i)} ] \\
		&\hspace{2em}+ \sum_{i \in I^c} (1+C\ve\eta^{-2})^{-1}d_X(g(w^i),g(w^{i+1})) \\
		& \geq (1+C\eta)^{-1} \sum_{i=1}^{K-1} d_X(g(w^i),g(w^{i+1})) - C\eta \phi_\infty(x,y) \\
		&\geq (1+C\eta)^{-1} d_X(g(x),g(y)) - C\eta \phi_\infty(x,y) \\
		&\geq (1+C\eta)^{-1}(\phi_\infty(x,y) - C\ve r_k) -C\eta \phi_\infty(x,y) \\
		&\geq (1+C\eta)^{-1}\phi_\infty(x,y).
	\end{align}
	After rearranging, this and \eqref{e:optimal'} gives \eqref{e:pert}. \\
	
	For the remainder of the proof we describe how to construct such a chain. If $\phi_\infty(z^i,z^{i+1}) > \eta^2 r_{k(i)}$ for all $1 \leq i \leq N-1$ then we simply take $w^i = z^i$ for all $1 \leq i \leq N-1$ (in this case $I = \emptyset).$ Otherwise, we construct $\{w^i\}_{i=1}^{K}$ by removing certain links from the original chain. Let
	\begin{align}\label{e:calI}
		\calI = \{ 1 \leq i \leq N-1 : \phi_\infty(z^i,z^{i+1}) \leq \eta r_{k(i)}\} \neq \emptyset
	\end{align}
	and for each $i \in \calI$ let $i+1 \leq \bar{i} \leq N$ be the maximal index such that 
	\begin{align}\label{e:bari}
		d_{k(i)}(z_{k(i)}^i,z^{i'}_{k(i)}) \leq r_{k(i)+2} \mbox{ for all } i \leq i' \leq \bar{i}.
	\end{align}
	Let $i_1$ be the smallest index in $\calI.$ Then, supposing $i_m$ has been defined for some $m \geq 1,$ if 
	\[\{i \in \calI : i \geq \bar{i}_m\}\neq\emptyset\]
	let $i_{m+1}$ be the smallest such index. Otherwise, we stop. Let $\{i_m\}_{m=1}^{\Lambda}$, $\Lambda \in \N$, be the resulting sequence. For each $1 \leq m \leq \Lambda$ we remove from $\{z^i\}_{i=1}^N$ the links in the chain lying between $i_m$ and $\bar{i}_m$. More precisely, let 
	\begin{align}
		\calI_0 = 
		\begin{cases}
			\emptyset & \mbox{ if } i_1 = 1; \\
			\{1,\cdots,i_1-1\} & \mbox{ otherwise},
		\end{cases}
	\end{align}
	and for $1 \leq m \leq \Lambda,$ let 
	\begin{align}
		\calI_m = \begin{cases}
			\emptyset &\mbox{ if } m < \Lambda \mbox{ and } \bar{i}_m= i_{m+1};\\
			\{\bar{i}_m,\dots,i_{m+1}-1\} &\mbox{ if } m < \Lambda \mbox{ and } \bar{i}_m < i_{m+1}; \\
			\emptyset &\mbox{ if } m = \Lambda \mbox{ and } \bar{i}_m = N; \\
			\{\bar{i}_m,\dots,N\} &\mbox{ if } m=\Lambda \mbox{ and } \bar{i}_m < N. 
		\end{cases}
	\end{align}
	Then, set 
	\begin{align}\label{e:defw}
		Z = \{z^i\}_{i=1}^N \setminus \bigcup_{m=1}^\Lambda \bigcup_{i_m < i < \bar{i}_m} \{z^i\}  = \bigcup_{m=1}^\Lambda \{z^{i_m},z^{\bar{i}_m}\} \cup \bigcup_{m=0}^\Lambda \bigcup_{i \in \calI_m} \{z^i\}.
	\end{align}
	We let $\{w^i\}_{i=1}^K$ be an enumeration of $Z$ with order inherited from $\{z^i\}_{i=1}^N$ i.e. if $1 \leq i \leq j \leq K$ and $1 \leq k,\ell \leq N$ are such that $w^i = z^k$ and $w^j = z^\ell$, then $k \leq \ell.$ \\

	Before proving conditions (1) and (2) we make and prove several claims. 
	For $1 \leq m \leq \Lambda$ such that $k(i_m) \geq 1,$ choose $j(m) \in J_{k(i_m)-1}$ such that
	\begin{align}\label{e:j(i)}
		z^i_{k(i)-1}\in 3D^{j(m)},
	\end{align} 
	which exists by Lemma \ref{l:inclusion}. \\
	
	\noindent\textbf{Claim 1:} If $1 \leq m \leq \Lambda$ is such that $k(i_m) \geq 1$ and $i_m \leq i \leq \bar{i}_m$ then  
	\begin{align}\label{e:in5}
		z^i_{k(i_m)-1} \in 4D^{j(m)}.
	\end{align}
	If $1 \leq m \leq \Lambda$ and $i_m \leq i < \bar{i}_m$ then
	\begin{align}\label{e:k(i_m)}
		k(i_m) -1 \leq k(i) \leq k(i_m) +1.
	\end{align}
	If $\bar{k}(i_m) = k(z^{i_m},z^{\bar{i}_m})$ then
	\begin{align}\label{e:bark}
		k(i_m)-1 \leq \bar{k}(i_m) \leq k(i_m) + 1.
	\end{align}
	
	Equation \eqref{e:in5} follows from \eqref{e:bari}, \eqref{e:j(i)} and Lemma \ref{l:ball-D}. Suppose now that $1 \leq m \leq \Lambda$ and $i_m \leq i < i_m$. If $k(i_m) = 0$ then the first inequality in \eqref{e:k(i_m)} is clear. If $k(i_m) \geq 1,$ since \eqref{e:in5} holds for $i$ and $i+1$, the first inequality in \eqref{e:k(i_m)} follows from this and the definition of $k(i)$. For the second inequality, note that, by Lemma \ref{l:outside} and \eqref{e:bari}  we have 
	\begin{align}\label{e:out}
		{\dist}_{k(i_m)+2}\left(z_{k(i_m)+2}^{i_m}, \bigcup_{j \in J_{k+2}} 6D^j \right) \geq r_{k(i_m)+1}/2.
	\end{align}
	Also, by Lemma \ref{l:f-bilip1} and \eqref{e:bari}, we have 
	\begin{align}
		d_{k(i_m)+2}(z^{i_m}_{k(i_m)+2},z^i_{k(i_m)+2}) \leq (1+C\ve)r_{k(i_m)+2} \leq r_{k(i_m)+1}/4.
	\end{align}
	These two inequalities imply 
	\[z_{k(i_m)+2}^{i} \not\in \bigcup_{j \in J_{k(i_m)+2}} 6D^j \]
	and \eqref{e:k(i_m)} now follows from the definition of $k(i).$ The first inequality in \eqref{e:bark} holds by essentially the same argument used to establish the first inequality in \eqref{e:k(i_m)}. The second inequality follows from the definition of $\bar{k}(i_m)$ and \eqref{e:out}. \\

	Let  
	\begin{align} \calS_1 &= \left\{1 \leq m \leq \Lambda : \phi_{\infty}(z^{i_m},z^{\bar{i}_m}) > 10\eta^2 r_{k(i_m)}\right\}; \\
		\calS_2 &= \left\{1 \leq m \leq \Lambda : \phi_{\infty}(z^{i_m},z^{\bar{i}_m}) \leq 10\eta^2 r_{k(i_m)} \mbox{ and } \bar{i}_m < N\right\}; \\
		\calS_3 &= \left\{1 \leq m \leq \Lambda : \phi_{\infty}(z^{i_m},z^{\bar{i}_m}) \leq 10\eta^2 r_{k(i_m)} \mbox{ and } \bar{i}_m = N\right\} \subseteq \{\Lambda\}.
	\end{align} \\
	
	\noindent\textbf{Claim 2:} We have
	\begin{align}
		\sum_{m \in \calS_1} \phi_\infty(z^{i_m},z^{\bar{i}_m}) \leq (1+C\ve) \sum_{m \in \calS_1} \sum_{i=i_m}^{\bar{i}_m-1} \phi_\infty(z^i,z^{i+1}).
	\end{align}
	
	Suppose $m \in S_1.$ By \eqref{e:phi_infty1} and the definitions of $\bar{k}(i_m)$ and $k(i),$ we have 
	\begin{align}\label{e:ob'}
		\phi_\infty(z^{i_m},z^{\bar{i}_m}) = d_{\bar{k}(i_m)}(z_{\bar{k}(i_m)}^{i_m}, z_{\bar{k}(i_m)}^{\bar{i}_m}) \quad \mbox{ and } \quad \phi_\infty(z^i,z^{i+1}) = d_{k(i)}(z^i_{k(i)},z^{i+1}_{k(i)}) .
	\end{align} 
	Using this with Lemma \ref{l:f-bilip1}, \eqref{e:k(i_m)} and \eqref{e:bark}, we have 
	\begin{align}
		\phi_\infty(z^{i_m},z^{\bar{i}_m}) & \leq(1+C\ve) d_{k(i_m)}(z^{i_m}_{k(i_m)},z^{\bar{i}_m}_{k(i_m)}) \\
		&\leq (1+C\ve)\sum_{i=i_m}^{\bar{i}_m-1} d_{k(i_m)}(z^{i}_{k(i_m)},z^{i+1}_{k(i_m)}) \leq (1+C\ve) \sum_{i=i_m}^{\bar{i}_m-1} \phi_\infty(z^i,z^{i+1}).
	\end{align}
	The claim now follows by taking the sum over $\calS_1.$ \\
	
	\noindent\textbf{Claim 3:} We have 
	\begin{align}\label{e:M_2}
		\sum_{m \in \calS_2} \phi_\infty(z^{i_m},z^{\bar{i}_m}) \lesssim \eta \sum_{m \in \calS_2} \phi_\infty(z^{\bar{i}_m},z^{\bar{i}_m+1}) \leq \eta \sum_{m=1}^\Lambda \sum_{i \in \calI_m} \phi_\infty(z^{i},z^{i+1}). 
	\end{align}
	Let $m \in \calS_2$. The first inequality follows once we show
	\begin{align}\label{e:bigjump}
		\phi_\infty(z^{i_m},z^{\bar{i}_m}) \leq \eta \phi_\infty(z^{\bar{i}_m},z^{\bar{i}_m+1})
	\end{align}
	and the second inequality follows once we show 
	\begin{align}\label{e:iincali}
		\bar{i}_m \in \calI_m.
	\end{align}
	Crucial to both \eqref{e:bigjump} and \eqref{e:iincali} are the inequalities 
	\begin{align}\label{e:phi>eta}
		\phi_\infty(z^{\bar{i}_m},z^{\bar{i}_m+1}) > r_{k(\bar{i}_m)+4}/2 > \eta r_{k(\bar{i}_m)}
	\end{align} 
	and 
	\begin{align}\label{e:bari<i}
		k(\bar{i}_m) \leq k(i_m) + 4.
	\end{align}
	Before proving these, lets see how they imply \eqref{e:bigjump} and \eqref{e:iincali}. First, \eqref{e:phi>eta} and \eqref{e:bari<i} give \[\phi_{\infty}(z^{\bar{i}_m},z^{\bar{i}_m+1}) > \eta r_{k(\bar{i}_m)} \gtrsim \eta r_{k(i_m)}.\]
	Using this with Lemma \ref{l:f-bilip1}, \eqref{e:bark} and the definition of $\calS_2,$ we get 
	\begin{align}
		\phi_\infty(z^{i_m},z^{\bar{i}_m}) \leq 10  \eta^2 r_{k(i_m)} \lesssim \eta \phi_{\infty}(z^{\bar{i}_m},z^{\bar{i}_m+1})
	\end{align}
	which implies \eqref{e:bigjump}. Now for \eqref{e:iincali}. If $m = \Lambda$, the \eqref{e:iincali} follows immediately from the definitions of $\calS_2$ and $\calI_m.$ Suppose then that $m < \Lambda.$ If \eqref{e:phi>eta} holds then, recalling the definition of $\calI$ from \eqref{e:calI}, we have $\bar{i}_m \not\in \calI.$ 
	Since $i_{m+1}$ is the smallest index in $\{i \in \calI : i \geq \bar{i}_m\}$ this gives $\bar{i}_m < i_{m+1}$. Equation \eqref{e:iincali} then follows from the definition of $\calI_m.$ 
	
	We now prove \eqref{e:phi>eta} and \eqref{e:bari<i}, starting with \eqref{e:bari<i}. By maximality of $\bar{i}_m$ we have 
	\begin{align}\label{e:z1}
		d_{k(i_m)}(z_{k(i_m)}^{i_m},z^{\bar{i}_m+1}_{k(i_m)}) > r_{k(i_m)+2}.
	\end{align}
	Using Lemma \ref{l:f-bilip1} with \eqref{e:bark} and the fact that $m \in \calS_2,$ we also have 
	\begin{align}\label{e:z2}
		d_{k(i_m)}(z^{i_m}_{k(i_m)},z^{\bar{i}_m}_{k(i_m)}) &\leq (1+C\ve)d_{\bar{k}(i_m)}(z^{i_m}_{\bar{k}(i_m)},z^{\bar{i}_m}_{\bar{k}(i_m)}) \\
		&= (1+C\ve)\phi_{\infty}(z^{i_m},z^{\bar{i}_m}) \lesssim \eta^2 r_{k(i_m)}.
	\end{align}
	Combining this with \eqref{e:z1} gives \begin{align}\label{e:estimate}
		d_{k(i_m)}(z_{k(i_m)}^{\bar{i}_m},z_{k(i_m)}^{\bar{i}_m+1}) \geq r_{k(i_m) + 3},
	\end{align}
	which, after applying Lemma \ref{l:f-bilip1},  implies \[d_{k(i_m)+5}(z_{k(i_m)+5}^{\bar{i}_m},z_{k(i_m)+5}^{\bar{i}_m+1}) \geq (1+C\ve)^{-1}r_{k(i_m)+3} \geq 50r_{k(i_m)+5}.\]
	By Lemma \ref{l:psi-bi-lip} we have 
	\[{\diam}_{k(i_m)+5}(6D^j) \leq 13r_{k(i_m)+5}\]  
	for all $j \in J_{k(i_m)+5}$, and the above two inequalities imply \eqref{e:bari<i}. Let us prove \eqref{e:phi>eta}. The second inequality is immediate if $\eta$ is chosen small enough, so let us focus on the first inequality. Suppose first that $k(i_m) \geq k(\bar{i}_m) + 2$ (so that $k(i_m) \geq 2)$. By \eqref{e:in5} we have $z_{k(i_m)-1}^{\bar{i}_m} \in 4D^{j(i_m)}$ so (the contrapositive of) Lemma \ref{l:outside} gives the first inequality in \eqref{e:phi>eta}. Suppose instead that
	\begin{align}\label{e:bari>i}
		k(i_m)  \leq k(\bar{i}_m) +1  .
	\end{align}
	Then, using Lemma \ref{l:f-bilip1} with each of \eqref{e:bari<i}, \eqref{e:estimate} and \eqref{e:bari>i}, we get  
	\begin{align}
		\phi_\infty(z^{\bar{i}_m},z^{\bar{i}_m+1}) &= d_{k(\bar{i}_m)}(z^{\bar{i}_m}_{k(\bar{i}_m)},z^{\bar{i}_m+1}_{k(\bar{i}_m)}) \\
		&\geq (1+C\ve)^{-1} d_{k(i_m)}(z^{\bar{i}_m}_{k(i_m)},z^{\bar{i}_m+1}_{k(i_m)}) \geq r_{k(\bar{i}_m)+4}/2
	\end{align}
	and we are done.\\
	
	\noindent\textbf{Claim 4:} If $m \in \calS_3,$ then 
	\[ \phi_\infty(z^{i_m},z^{\bar{i}_m}) \lesssim \eta \phi_\infty(x,y).\]
	
	Indeed, recall in Case 2 that $\phi_\infty(x,y) = d_k(x_k,y_k) > \eta r_k.$ Then, by \eqref{e:2} and the definition of $\calS_3$, we have 
	\[ \phi_\infty(z^{i_m},z^{\bar{i}_m}) \leq 10 \eta^2 r_{k(i_m)} \lesssim \eta^2 r_k \leq \eta \phi_\infty(x,y).\]
	
	We now have everything we need to verify conditions (1) and (2). First, if we combine \eqref{e:defw} with Claim 2, Claim 3 and Claim 4, then 
	\begin{align}
		\sum_{i=1}^{K-1} \phi_\infty(w^i,w^{i+1}) &= \sum_{m=1}^\Lambda \phi_\infty(z^{i_m},z^{\bar{i}_m}) + \sum_{m=0}^\Lambda \sum_{i \in \calI_m} \phi_\infty(z^i,z^{i+1}) \\
		&\leq (1+C\ve)\sum_{m =1}^\Lambda \sum_{i=i_m}^{\bar{i}_m-1} \phi_\infty(z^i,z^{i+1}) \\
		&\hspace{2em} + (1+C\eta) \sum_{m=0}^\Lambda \sum_{i \in \calI_m} \phi_\infty(z^{i},z^{i+1}) + C\eta \phi_\infty(x,y) \\
		&\leq (1+C\eta)\sum_{i=1}^N \phi_\infty(z^i,z^{i+1}) + \eta \phi_\infty(x,y) \\
		&\leq (1+C\eta)\rho(x,y) + C\eta \phi_\infty(x,y) 
	\end{align}
	which gives (1). Let us focus our attention on (2). Recall the definition of $I$ from the statement of (2). If $I =\emptyset$ there is nothing to show, so let us assume $I \neq \emptyset$ and let $i \in I.$ We now show that  
	\begin{align}\label{e:w^i}
		w^i = z^{i_m} \mbox{ for some } m \in \calS_2 \cup \calS_3.
	\end{align}
	Suppose this is false and let $1 \leq j \leq N-1$ such that $w^i = z^j.$ By \eqref{e:defw} we have either $j = i_m$ for some $m \in \calS_1$ or $j \in \calI_m$ for some $m \geq 0.$ In the first case $w^{i+1} = z^{\bar{j}}$. The definition of $\calS_1$ implies then Lemma \ref{l:f-bilip1}, \eqref{e:bark} and the definition of $\calS_1$ imply 
	\begin{align}
		\phi_\infty(w^i,w^{i+1}) = \phi_\infty(z^{j},z^{\bar{j}}) >10 \eta^2 r_{k(j)} = 10\eta^2r_{\tilde{k}(i)},
	\end{align}
	which is a contradiction since $i \in I.$ In the second case, $w^{i+1} = z^{j+1},$ so that $\phi_\infty(w^i,w^{i+1}) = \phi_\infty(z^j,z^{j+1}) > \eta r_{k(j)} = \eta r_{\tilde{k}(i)}$ since $j \not\in \calI.$ This is again a contradiction.
	
	We now use \eqref{e:w^i} to finish the proof of (2). Recalling the definition of $\tilde{k}(i)$ and $\bar{k}(i_m),$ we see that $\tilde{k}(i) = \bar{k}(i_m).$ If $m \in \calS_2$ then by taking $\ve$ small enough and using \eqref{e:bark}, \eqref{e:phi>eta} and \eqref{e:bari<i} we have 
	\[\ve r_{\tilde{k}(i)}  = \ve r_{\bar{k}(i_m)} \lesssim \ve r_{k(\bar{i}_m)} \leq \eta \phi_\infty(z^{\bar{i}_m},z^{\bar{i}_m+1}).\]
	If $m \in \calS_3,$ then by taking $\ve$ small enough and using \eqref{e:2} and \eqref{e:bark} we have 
	\begin{align}
		\ve r_{\tilde{k}(i)} = \ve r_{\bar{k}(i_m)} \lesssim \ve r_{k(i_m)} \leq \ve r_k \leq \eta  \phi_\infty(x,y). 
	\end{align}
	Combing these estimates with \eqref{e:optimal} and the fact that $\rho(x,y) \leq \phi_\infty(x,y)$ we have 
	\begin{align}
		\sum_{i \in I} \ve r_{\tilde{k}(i)} \lesssim \eta \left(\sum_{m \in \calS_2} \phi_\infty(z^{\bar{i}_m},z^{\bar{i}_m+1}) + \phi_\infty(x,y)\right) \lesssim \eta \phi_\infty(x,y),
	\end{align}
	which completes the proof of (2) and the construction of $\{w^i\}_{i=1}^K.$ 
\end{proof}

\begin{cor}\label{c:rho-phi}
	Let $x,y \in \R^n$ and $\ell \geq 0$. Then, 
	\begin{align}\label{e:r1}
		| \rho(x,y) - \phi_\ell(x_\ell,y_\ell)| = | \rho(x,y) - \tilde{\phi}_\ell(x,y)| \lesssim \ve (r_\ell + \rho(x,y)). 
	\end{align}
	If $x_\ell \in \bigcup_{j \in J_\ell} 6D^j$ then 
	\begin{align}\label{e:r2}
		|d_X(g(x),g(y)) - \rho(x,y)| \lesssim	\ve (r_\ell + \rho(x,y))
	\end{align}
\end{cor}

\begin{proof}
	The equality in \eqref{e:r1} is immediate from \eqref{e:phi_infty1}. For the inequality, we apply Lemma \ref{l:phi-cauchy} and Lemma \ref{l:pert} to get 
	\begin{align}
		| \rho(x,y) - \tilde{\phi}_\ell(x,y) | &\leq | \rho(x,y) - \phi_\infty(x,y) | + | \phi_\infty(x,y) - \tilde{\phi}_\ell(x,y) | \\
		&\lesssim \ve \rho(x,y) + \ve r_\ell. 
	\end{align}
	For \eqref{e:r2}, we apply Lemma \ref{l:phi_infty} and Lemma \ref{l:pert} to get 
	\begin{align}
		| d_X(g(x),g(y)) - \rho(x,y) | &\leq | d_X(g(x),g(y)) - \phi_\infty(x,y) | + |\phi_\infty(x,y) - \rho(x,y)\\
		&\lesssim \ve \left(r_\ell + \rho(x,y)\right).
	\end{align}
\end{proof}

\bigskip

\subsection{Bi-H\"older, quasi-isometry, and Reifenberg Flatness}\label{s:quasi}

In the following two lemmas, we establish the properties of $\rho$ stated in Theorem \ref{t:Reif}. First, it is bi-H\"older equivalent to and a small perturbation of $\|\cdot\|_{j_0}$ as in \eqref{e:Reif-Holder} and \eqref{e:Reif1} respectively. Second, that $(\R^n,\rho)$ is Reifenberg flat as in \eqref{e:Reif-Reif}.

\begin{lem}
	The metric $\rho$ is bi-H\"older equivalent to $\|\cdot\|_{j_0}$ with exponent $\alpha = 1 - C\ve$ and 
	\begin{align}\label{e:dist-rho1}
		| \rho(x,y) - \|x-y\|_{j_0} | \leq C\ve 
	\end{align}
	for all $x,y \in 10B_{j_0}.$ 
\end{lem}

\begin{proof}
	For $x,y \in 10B_{j_0}$ we have by \eqref{e:infty-ell} and Lemma \ref{l:pert} that 
	\[ \rho(x,y) \lesssim \phi_\infty(x,y) \leq \|x-y\|_{j_0} + C\ve s_0 \lesssim s_0 \lesssim 1.\]
	Recalling from \eqref{e:rho_0} that $\phi_0(x,y) = \|x-y\|_{j_0},$ \eqref{e:dist-rho1} is then immediate from \eqref{e:r1}. By Lemma \ref{l:pert}, to show that $\rho$ is bi-H\"older equivalent to $\|\cdot\|_{j_0}$ with exponent $\alpha$, it is enough to show $\phi_\infty$ is bi-H\"older equivalent to $\|\cdot\|_{j_0}$ with the same exponent. This is our goal for the rest of the proof. Let $C'$ be the implicit constant appearing in \eqref{e:infty-ell} and let $x,y \in \R^n$. If $\| x - y \| > r_0$ then choosing $\ve$ small enough depending on $C'$, we have 
	\begin{align}
		| \phi_\infty(x,y) - \| x -y \|_{j_0} | \leq C'\ve r_0 \leq \tfrac{1}{2} \| x-y \|_{j_0}.
	\end{align}
	Suppose instead that $\| x-y \|_{j_0} \leq r_0.$ By Lemma \ref{l:phi-cauchy}, we have 
	\begin{align}\label{e:upper-phi}
		\phi_\ell(x,y) \leq (1+C\ve)^\ell \|x-y\|_{j_0}
	\end{align}
	for all $\ell \geq 0.$ If $m \geq 0$ is such that $\phi_\ell(x,y) \leq r_\ell$ for all $0 \leq \ell \leq m$ then Lemma \ref{l:phi-cauchy} also gives 
	\begin{align}\label{e:lower-phi}
		10^m = r_m \geq \phi_m(x,y) \geq (1+C\ve)^{-m}\| x-y \|_{j_0}.
	\end{align}
	Since the left hand-side decays more rapidly than the right, there is a maximal such $m \geq 0$, which we do not relabel. By maximality, we have 
	\begin{align}
		\phi_m(x,y) \geq (1+C\ve)^{-1} \phi_{m+1}(x,y) \geq \frac{r_{m+1}}{2}.
	\end{align}
	Using this estimate with \eqref{e:upper-phi} and \eqref{e:lower-phi}, gives
	\begin{align}
		\begin{split}\label{e:phi-infty-m-1}
		\phi_\infty(x,y) &\leq \phi_m(x,y) + C\ve r_m \leq (1+C\ve) \phi_m(x,y) \\
		&\leq (1+C\ve)^{m+1} \| x -y \|_{j_0} 
		\end{split}
	\end{align}
	and 
	\begin{align}
		\begin{split}\label{e:phi-infty-m-2}
		\phi_\infty(x,y) &\geq \phi_m(x,y) - C\ve r_m \geq (1+C\ve) \phi_m(x,y) \\
		&\geq (1+C\ve)^{-m-1} \|x-y\|_{j_0}. 
		\end{split}
	\end{align}
	After taking logarithms in \eqref{e:lower-phi} and rearranging, we obtain 
	\begin{align}
		m \leq \frac{1}{\log\left(\frac{10}{1+C\ve}\right)}\log\left( \frac{1}{\|x-y\|_{j_0}}\right) \leq C \log\left( \frac{1}{\|x-y\|_{j_0}}\right).
	\end{align}
	Combing this estimate for $m$ with \eqref{e:phi-infty-m-1} and \eqref{e:phi-infty-m-2}, we see 
	\begin{align}
		\abs{\log\left( \frac{\phi_\infty(x,y)}{\|x-y\|_{j_0}} \right) } \leq (m+1)\log(1+C\ve) \leq C\ve \log\left( \frac{1}{\|x-y\|_{j_0}} \right). 
	\end{align}
	After rearranging, it follows that 
	\begin{align}
		\|x-y\|^{1+C\ve}_{j_0} \leq \phi_\infty(x,y) \leq \|x-y\|_{j_0}^{1-C\ve}.
	\end{align}
\end{proof}

\begin{lem}\label{l:Reif-flat}
	The space $(\R^n,\rho)$ is $(C\ve,n)$-Reifenberg flat. 
\end{lem}

\begin{proof}
	Let $x \in \R^n$ and $0 < r < \infty$. We begin by defining a norm $\|\cdot\|$ and a mapping $\phi \colon B_\rho(x,r) \to B_{\|\cdot\|}(0,r).$ We will check afterwards that $\phi$ defines a $C\ve r$-GHA. For $\ell \geq 0$ set $B_\ell = F_\ell(B_\rho(x,r)) \subseteq M_\ell$ and note that $B_0 = B_\rho(x,r)$ by \eqref{e:map-F}. Furthermore, let 
	\begin{align}
		k = \max\{\ell \geq 0 \colon \mbox{there exists } j \in J_\ell \mbox{ such that } B_\ell \subseteq 6D^j\}.
	\end{align}
	We define the maximum of the empty set to be zero. If $k \geq 1$ fix some $i \in J_k$ satisfying $B_k \subseteq 6D^{i,k}$, let $\|\cdot\| = \|\cdot\|_{i,k}$ and $\psi  = \psi_{i,k} \colon 6D^{i,k} \to 6B_{i,k}$. Otherwise, set $\|\cdot\|=\|\cdot\|_{j_0}$ and $\psi = \text{Id}.$ In either case we have that $\psi$ is a $(1+C\ve)$-bi-Lipschitz map by Lemma \ref{l:psi-bi-lip}. By translation, we may assume 
	\begin{align}\label{e:F=0}
		\psi(F_k(x)) = 0,
	\end{align}
	noting that this does not affect the bi-Lipschitz constant. Let $p_{0,r}$ be the map from Lemma \ref{l:p-lambda} and set
	\begin{align}\label{e:phi}
		\phi \coloneqq p_{0,r} \circ \psi \circ F_k \colon B_0 \to B_{\|\cdot\|}(0,r).
	\end{align}
	
	Let $\eta > 0$ to be chosen sufficiently small. We now check that $\phi$ defines a $C\ve r$-GHA, the proof splits into two cases.   \\
	
	\noindent\textbf{Case 1:} Suppose that $d_k(x_k,p_k) \leq \eta r_k$ for all $p \in B_0.$ We begin by showing 
	\begin{align}\label{e:F_k-bi-Lip}
		F_k|_{B_0} : (B_0,\rho) \to (B_k,d_k) \mbox{ is a } (1+C\ve)\mbox{-bi-Lipschitz equivalence.} 
	\end{align}
	Indeed, by maximality of $k$ there exists a pair of points $p,q \in B_0$ such that if $\ell > k$ and $j \in J_\ell$ then either $p_\ell \not\in 6D^j$ or $q_\ell \not\in 6D^j.$ By assumption, $\diam(B_k) \leq 2 \eta r_k,$ so that $\diam(B_{k+2}) \leq 3\eta r_k$ by  Lemma \ref{l:f-bilip1}. Furthermore, Lemma \ref{l:outside} implies $\dist_{k+2}(p_{k+2},\bigcup_{j \in J_{k+2}} 6D^j) \geq r_{k+1}/2.$ Combining the previous two estimates with the triangle inequality gives
	\begin{align}\label{e:disjoint}
		B_{k+2} \cap \bigcup_{j \in J_{k+2}} 6D^j =\emptyset. 
	\end{align}	
	Let us see how \eqref{e:disjoint} implies \eqref{e:F_k-bi-Lip}. Let $a,b \in B_0$ and as above let us write $a_\ell = F_\ell(a)$ and $b_\ell = F_\ell(b)$ for $\ell \geq 0.$ Combining \eqref{e:disjoint} with \eqref{e:phi_infty1} we have $\phi_\infty(a,b) = d_{k+2}(a_{k+2},b_{k+2}).$ Applying Lemma \ref{l:pert} and Lemma \ref{l:f-bilip1} then gives
	\begin{align}
		| \rho(a,b) - d_k(a_k,b_k) | &\leq |\rho(a,b) - d_{k+2}(a_{k+2},b_{k+2})| + |d_{k+2}(a_{k+2},b_{k+2}) - d_k(a_k,b_k)| \\
		&\lesssim \ve (d_{k+2}(a_{k+2},b_{k+2}) + d_k(a_k,b_k)) \lesssim \ve d_k(a_k,b_k)
	\end{align}
	and we conclude \eqref{e:F_k-bi-Lip}. 
	
	Now we check that $\phi$ defines a $C\ve r$-isometry. Let $y,z \in B_0.$ Combining \eqref{e:F=0} with that fact that $\psi$ and $F_k$ are both $(1+C\ve)$-bi-Lipschitz, we know $\psi(F_k(y)) \in B_{\|\cdot\|}(0,(1+C\ve)r)$ so that $\| \phi(y) - \psi(F_k(y)) \| \leq C\ve r$ by \eqref{e:phi} and Lemma \ref{l:p-lambda}. The same argument also gives $\| \phi(z) - \psi(F_k(z)) \| \leq C\ve r.$ Using this, and again the fact that $\psi$ and $F_k$ are $(1+C\ve)$-bi-Lipschitz, we have  
	\begin{align}\label{e:g-1}
		| \| \phi(y) - \phi(z) \| - \rho(y,z) | &\leq C\ve r + | \| \psi(F_k(y)) - \psi(F_k(z)) \| - \rho(y,z) | \\
		&\leq C\ve r + C\ve \rho(y,z) \leq C\ve r. 
	\end{align}
	Since $y,z \in B_0$ were arbitrary, we are done. 
	
	It remains to check that $\phi$ is $C\ve r$-dense. Let $u \in B_{\|\cdot\|}(0,r)$ and set $v = \tfrac{u}{1+C\ve}$ and $w = F_k^{-1}(\psi^{-1}(v)).$ Then $\|u-v\| \leq C\ve r$ and $\| v \| \leq (1+C\ve)^{-1} r$. Recalling \eqref{e:F=0} and using that $F_k$ and $\psi$ are $(1+C\ve)$-bi-Lipschitz we have
	\begin{align}
		\rho(w,x) = \rho(F_k^{-1}(\psi^{-1}(v)),F_k^{-1}(\psi^{-1}(0))) \leq (1+C\ve) \| v  \| \leq r
	\end{align} 
	so that $w \in B_0.$ Since $v \in B_{\|\cdot\|}(0,r),$ Lemma \ref{l:p-lambda} implies $p_{0,r}(v) = v$. The triangle inequality, \eqref{e:phi} and the definition of $w$ then gives 
	\begin{align}\label{e:g-2}
		\hspace{2.5em}\| \phi(w) - u \| \leq \| \phi(w) - v \| + \| v - u \| = \| \psi(F_k(w)) - v \| + \| v - u \| \leq C\ve r. 
	\end{align}
	Since $u$ was arbitrary, this completes the proof in this case. \\
	
	\noindent\textbf{Case 2:} Suppose now that there exists $p \in B_0$ such that $d_k(x_k,p_k) > \eta r_k.$ For $\ve$ small enough depending on $\eta$, it follows from Lemma \ref{l:pert} and \eqref{e:infty-ell} that $\phi_k(x_k,p_k) \leq (1+C\ve\eta^{-1}) \rho(x,p) \leq 2r$ for such a point. Note that
	\begin{align}\label{e:=}
		d_k(y_k,z_k) = \phi_k(y_k,z_k) \mbox{ for all } y,z \in B_0.
	\end{align}
	This follows from \eqref{e:rho_0} if $k=0$ and from \eqref{d:rho} if $k \geq 1,$ recalling that $B_k \subseteq 6D^{i,k}.$ Combing the above two observations with our assumption on $d_k(x_k,p_k)$ gives 
	\[r \geq \eta r_k/2.\]
	
	Let us check that $\phi$ defines a $C\ve r$-isometry. Let $y,z \in B_0.$ By Corollary \ref{c:rho-phi} and \eqref{e:=} we have  
	\begin{align}
		d_k(x_k,y_k)\leq (1+C\ve) \rho(x,y) + C\ve r_k \leq (1+C\ve \eta^{-1}) r. 
	\end{align}
	Using \eqref{e:F=0} and the fact that $\psi$ is $(1+C\ve)$-bi-Lipschitz, this gives $\psi(F_k(y)) \in B_{\|\cdot\|}(0,(1+C\ve \eta^{-1})r).$ Lemma \ref{l:p-lambda} then implies $\| \phi(y) - \psi(F_k(y)) \| \leq C\ve \eta^{-1} r.$ Similarly, we have $\| \phi(z) - \psi(F_k(z)) \| \leq C\ve \eta^{-1} r.$ Using these estimates with \eqref{e:=} and Corollary \ref{c:rho-phi}, we get    
	\begin{align}\label{e:g-3}
		| \| \phi(y) - \phi(z) \| - \rho(y,z) | &\leq C\ve \eta^{-1} r +  | \| \psi(F_k(y)) - \psi(F_k(z)) \| - \rho(y,z) |  \\
		&\leq C\ve \eta^{-1} r \\
		&\hspace{2em}+  |  \| \psi(F_k(y)) - \psi(F_k(z)) \|  - d_k(F_k(y),F_k(z)) | \\
		&\hspace{4em} + | d_k(F_k(y),F_k(z)) - \rho(y,z) |  \\
		&\lesssim \ve \eta^{-1} r  + \ve \rho(y,z) + \ve r_k \lesssim \frac{\ve r}{\eta}. 
	\end{align}
	
	Let us now show that $\phi$ is $C\ve r$-dense. Let $u \in B_{\|\cdot\|}(0,r)$ and let $r' = \tfrac{(r - C\ve r_k)}{1+C\ve}$, $v = p_{z,r'}(u)$ and $w = \psi^{-1}(F_k^{-1}(v)).$ Notice that $\|v\| \leq r' $ and by Lemma \ref{l:p-lambda} that $\| u -v \| \leq  C\ve(r + r_k) \leq C\ve\eta^{-1} r$. Then, by \eqref{e:F=0}, Corollary \ref{c:rho-phi}, and the fact that $\psi$ is $(1+C\ve)$-bi-Lipschitz, we have 
	\begin{align}
		\rho(w,z) &= \rho(F_k^{-1}(\psi^{-1}(v)),F_k^{-1}(\psi^{-1}(0))) \leq C\ve r_k + (1+C\ve)d_k(\psi^{-1}(v),\psi^{-1}(0))\\
		&\leq C\ve r_k + (1+C\ve) \| v\| \leq r,
	\end{align}
	so that $w \in B_0.$ By exactly the same estimate as in \eqref{e:g-2}, we get
	\begin{align}\label{e:g-4}
		\| \phi(w) - u \| \leq C\ve r.
	\end{align}
	Since $u$ was arbitrary, this finishes the proof.  
\end{proof}

\bigskip

\subsection{The map $\bar{g}$}\label{s:bar-g}

	For each $\ell \geq 0$ and $j \in J_\ell$ it is possible, by Lemma \ref{l:GH-scale}, to find a $C\delta r_\ell$-GHA $\tilde{\beta}_{j,\ell} \colon 100B^{j,\ell} \to 100B_{j,\ell}$ which satisfies \eqref{e:centre-pos}, \eqref{e:almost-id} (with constant $C\delta r_\ell$) and $\tilde{\beta}_{j,\ell}(6B^{j,\ell}) \subseteq 6B_{j,\ell}$. Thus, we might as well (and will) assume that 
\begin{align}\label{r:assump} \beta_{j,\ell}(6B^{j,\ell}) \subseteq 6B_{j,\ell}.
\end{align}
This is needed for example in \eqref{e:g-def} below, since $\text{dom}(\psi_{j,\ell}^{-1}) = 6B_{j,\ell}.$

Let us now define the map $\bar{g} : 6B^{j_0} \to \R^n.$ For $x \in 6B^{j_0},$ let 
\begin{align}\label{e:k(x)}
	k(x) = \sup\left\{ \ell \geq 0 : x \in \bigcup_{j \in J_\ell} 6B^j\right\},
\end{align}
where again we set the supremum over the empty set to be zero. For each $0 \leq k \leq k(x)$ choose an arbitrary index $i_k(x)$ such that $x \in 6B^{i_k(x)}.$ Then, for $k \geq 0$, set 
\begin{align}\label{e:g-def}
	\bar{g}_k(x) = 
	\begin{cases}
		F_k^{-1} \circ  \psi^{-1}_{i_k(x)} \circ \toward_{i_k(x)}(x)  &\mbox{ if } \ell \leq k(x); \\
		F_{k(x)}^{-1} \circ  \psi^{-1}_{i_{k(x)}(x)} \circ \toward_{i_{k(x)}(x)}(x)   &\mbox{ if } k > k(x). 	
	\end{cases}
\end{align}
Note that $\bar{g}_k$ is well-defined by \eqref{r:assump}.

\begin{lem}\label{l:g-close}
	Let $k \geq 0$ and $i,j \in J_k$ such that $x \in 6B^i \cap 6B^j.$ Then, 
	\begin{align}
		\phi_k(\psi_i^{-1} \circ \toward_i(x),\psi_j^{-1} \circ \toward_j(x)) \lesssim \ve r_k. 
	\end{align}
\end{lem}

\begin{proof} 
	By assumption, $\beta_i$ is a $\delta r_i$-GHA and by Lemma \ref{l:psi-bi-lip}, $\psi_i$ is bi-Lipschitz. Using this with \eqref{e:almost-id} and Lemma \ref{l:tran-psi} gives 
	\begin{align}
		\phi_k(\psi_i^{-1} \circ \toward_i(x),\psi_j^{-1} \circ \toward_j(x)) &= \phi_k( \psi_i^{-1}\circ \toward_i(x)   , \psi_i^{-1} \circ I_{i,j} \circ \toward_j(x)   )\\
		&\lesssim \| \toward_i(x) - I_{i,j} \circ \toward_j(x) \|_i \\
		&\leq \| \toward_i(x) - \toward_i \circ \away_j \circ \toward_j(x) \|_i + C\ve r_k \lesssim \ve r_k. 
	\end{align}
\end{proof}

\begin{lem}\label{l:g-exists}
	For each $k \geq 0$ and $x \in 6B^{j_0}$ we have  
	\begin{align}\label{e:est-g}
		\rho(\bar{g}_k(x),\bar{g}_{k+1}(x)) \lesssim \ve r_k.
	\end{align}
	In particular, the limit $\bar{g} = \lim_{k \to \infty} \bar{g}_k$ exists, takes values in $6B_{j_0}$ and satisfies 
	\[ \rho(\bar{g}_{k}(x),\bar{g}(x)) \lesssim \ve r_k. \]
\end{lem}

\begin{proof}
	Let $k \geq 0$ and $x \in 6B^{j_0}.$ If $k +1 > k(x)$ then $\bar{g}_k(x) = \bar{g}_{k+1}(x)$ and there is nothing to show. Suppose then, that $k +1 \leq k(x).$ Let $i = i_{k}(x)$ and $j = i_{k+1}(x)$ as above \eqref{e:g-def}. By \eqref{r:assump}, we have $\toward_j(x) \in 6B_j$ and by \eqref{d:D} we have $\psi_j^{-1} \circ \toward_j(x) \in 6D^j.$ It follows from Lemma \ref{l:f-good} that 
	\begin{align}\label{e:g_k}
		\bar{g}_{k+1}(x) = F_k^{-1} \circ \psi_{i(j)}^{-1} \circ K_{i(j),j} \circ \hat{H}_j^{-1} \circ \toward_j(x),
	\end{align}
	where $i(j) \in J_k$ and $K_{i(j),j}$ are as in \eqref{e:i(j)} and Lemma \ref{l:K-alpha}, respectively, and $\hat{H}_j \colon \R^n \to \R^n$ satisfies 
	\begin{align}\label{e:est-H}
		\| \hat{H}_j- \text{Id} \|_{C^1(\R^n),r_{k+1}} \lesssim \ve.
	\end{align}
	
	Let us now estimate the left-hand side of \eqref{e:est-g}. Of course, we have  
	\begin{align}
		\rho(\bar{g}_k(x),\bar{g}_{k+1}(x)) \leq \rho(\bar{g}_k(x) , F_k^{-1} \circ \psi_{i(j)}^{-1} \circ \toward_{i(j)}(x)) +   \rho(F_k^{-1} \circ \psi_{i(j)}^{-1} \circ \toward_{i(j)}(x), \bar{g}_{k+1}(x)).
	\end{align}
	To estimate the first term we apply \eqref{e:r1} and Lemma \ref{l:g-close} which gives
	\begin{align}
		\rho(\bar{g}_k(x) , F_k^{-1} \circ \psi_{i(j)}^{-1} \circ \toward_{i(j)}(x)) \lesssim \phi_k(\psi_{i(j)} \circ \toward_{i(j)}(x) , \psi_i \circ \toward_i(x) ) + \ve r_k \lesssim \ve r_k. 
	\end{align}
	By Corollary \ref{c:rho-phi}, \eqref{e:almost-id} and Lemma \ref{l:K-alpha}, the second term is at most
	\begin{align}
		&\phi_k( \psi_{i(j)}^{-1} \circ \toward_{i(j)}(x) , F_k \circ \bar{g}_{k+1}(x)) + \ve r_k \\
		&\hspace{2em}\lesssim \phi_k(  \psi_{i(j)}^{-1} \circ \toward_{i(j)}\circ \away_j \circ \toward_j(x) ,  F_k \circ \bar{g}_{k+1}(x)) + \ve r_k \\
	&\hspace{2em}\lesssim \phi_k(\psi_{i(j)}^{-1} \circ K_{i(j),j} \circ \toward_j(x)   ,  F_k \circ \bar{g}_{k+1}(x) ) + \ve r_k. 
	\end{align}
	Making using of the definition of $\bar{g}_{k+1}(x)$ in \eqref{e:g_k} and \eqref{e:est-H}, the last term is at most a constant multiple of $\ve r_k.$ This now finishes the proof of the lemma.  
\end{proof}

The following proves \eqref{e:Reif3.5}. 

\begin{lem}\label{l:e-isometry}
	Let $k \geq 0, \ j \in J_k$ and $x,y \in 6B^j.$ Then, 
	\begin{align}
		| \rho(\bar{g}(x),\bar{g}(y)) - d_X(x,y) | \lesssim \ve r_k. 
	\end{align}
\end{lem}

\begin{proof}
	It suffices to show 
	\begin{align}\label{e:suf}
		| \rho(\bar{g}(x),\bar{g}(y)) - d_X(x,y) | \lesssim \ve (\rho(\bar{g}(x),\bar{g}(y)) + r_k).
	\end{align}
	Indeed, applying the above estimate twice and using that $x,y \in 6B^j$ (which implies $d_X(x,y) \lesssim r_k)$, we get  
	\begin{align}
		| \rho(\bar{g}(x),\bar{g}(y)) - d_X(x,y) | \lesssim \ve (d_X(x,y) + \ve r_k + r_k) \lesssim \ve r_k. 
	\end{align}
	To obtain \eqref{e:suf}, we first apply Corollary \ref{c:rho-phi}, Lemma \ref{l:g-exists} and the triangle inequality to get  
	\begin{align} \label{e:r3}
		| \rho(\bar{g}(x),\bar{g}(y)) - d_X(x,y) | &\lesssim \ve (\rho(\bar{g}(x),\bar{g}(y)) + r_k) \\
		&\hspace{2em} + 	| \phi_k(F_k(\bar{g}_k(x)),F_k(\bar{g}_k(y))) - d_X(x,y) |. 
	\end{align}
	Since $x \in 6B^j$, it follows from \eqref{e:g-def} and Lemma \ref{l:g-close} that 
	\[\phi_k(F_k(\bar{g}_k(x)) , \psi^{-1}_j(\toward_j(x))) \lesssim \ve r_k.\] Similarly, $\phi_k(F_k(\bar{g}_k(y)) , \psi^{-1}_j(\toward_j(y)) \lesssim \ve r_k$. Using this and the fact that $\psi_j$ is $(1+C\ve)$-bi-Lipschitz and $\toward_j$ is a $C\delta r_k$-GHA, the final term in \eqref{e:r3} is at most   
	\begin{align}
		C\ve r_k + | \phi_k(\psi_j^{-1}(\toward_j(x)),\psi_j^{-1}(\toward_j(y))) - d_X(x,y) |  + C\ve r_k \lesssim \ve r_k
	\end{align}
	which proves \eqref{e:suf}.
\end{proof}

The following proves \eqref{e:Reif2.0}.

\begin{lem}\label{l:g-inv1}
	Let $\ell \geq 0$ and suppose $x \in \bigcup_{j \in J_\ell} 6B^j$. Then,  
	\begin{align}\label{e:almost-inverse-ell}
		d_X(g_\ell(\bar{g}_\ell(x)),x) \lesssim \ve r_\ell. 
	\end{align}
	From this, it follows that 
	\begin{align}\label{e:almost-inverse}
		d_X(g(\bar{g}(x)),x) \lesssim \ve r_\ell. 
	\end{align}
\end{lem}

\begin{proof}
	Let us start with \eqref{e:almost-inverse-ell}. If $\ell = 0$ then $g_\ell(\bar{g}_\ell(x)) = \alpha_{j_0}(\beta_{j_0}(x))$ and \eqref{e:almost-inverse-ell} follows from \eqref{e:almost-id}. Suppose now that $\ell \geq 1.$  Since $x \in \bigcup_{j \in J_\ell} 6B^j$ we have $\ell \leq k(x)$ (recall \eqref{e:k(x)}). It follows from \eqref{e:g-def} that there exists $j \in J_\ell$ such that $\bar{g}_\ell(x)  = F_\ell^{-1} \circ \psi_{j}^{-1} \circ \toward_{j}(x)$. Since 
	\begin{align}\label{e:inD}
		F_\ell (\bar{g}_\ell(x)) = \psi_{j}^{-1} \circ \toward_{j}(x) \in 6D^j
	\end{align}
	we have $\ell \leq \ell(\bar{g}_\ell(x))$, recall the definition of $\ell(z)$ from above \eqref{e:g-bar}. In particular, by \eqref{e:g-bar}, there exists $i \in J_\ell$ such that $g_\ell(\bar{g}_\ell(x)) = \alpha_i \circ \psi_i \circ F_\ell(\bar{g}_\ell(x)).$ Using this with Lemma \ref{l:tran-psi} we get 
	\begin{align}
		g_\ell(\bar{g}_\ell(x)) = \away_{i} \circ \psi_{i} \circ \psi_j^{-1} \circ \toward_j(x) = \away_{i} \circ I_{i,j} \circ \toward_j(x).
	\end{align}
	Using \eqref{e:almost-id}, \eqref{e:prop} and the fact that $\away_j$ is a $\delta r_\ell$-GHA now gives \eqref{e:almost-inverse-ell} since 
	\begin{align}
		d_X(g_\ell(\bar{g}_\ell(x)),x) &\leq C\delta r_\ell + d_X(  \away_{i} \circ I_{i,j} \circ \toward_j(x) ,  \away_{i} \circ \toward_{i} \circ \away_{j} \circ \toward_j(x)) \\
		&\leq C\delta r_\ell + \| I_{i,j} \circ \toward_j(x) - \toward_{i} \circ \away_{j} \circ \toward_j(x)\|_{i} \lesssim \ve r_\ell .
	\end{align}

	Let us see how to get \eqref{e:almost-inverse} from \eqref{e:almost-inverse-ell}. Since $F_\ell(\bar{g}_\ell(x)) \in 6D^j,$ we have by Corollary \ref{c:rho-phi} and Lemma \ref{l:g-exists} that 
	\begin{align}
		d_X(g(\bar{g}(x)) , g(\bar{g}_\ell(x))) \lesssim \rho(\bar{g}(x),\bar{g}_\ell(x)) + \ve r_\ell \lesssim \ve r_\ell. 
	\end{align}
	Using this with Lemma \ref{l:gbar} and \eqref{e:almost-inverse-ell}, the left-hand side of \eqref{e:almost-inverse} is at most 
	\begin{align}
		d_X(g(\bar{g}(x)) , g(\bar{g}_\ell(x)) + d_X(g(\bar{g}_\ell(x)) , g_\ell(\bar{g}_\ell(x))  )   + d_X(  g_\ell(\bar{g}_\ell(x))  ,  x) \lesssim \ve r_\ell,
	\end{align}
	as required. 
\end{proof}

The following proves \eqref{e:Reif2.5}.

\begin{lem}
	Let $\ell \geq 0$, $x \in 6B_{j_0}$ and suppose $g(x) \in \bigcup_{j \in J_\ell} 6B^j.$ Then,  
	\begin{align}\label{e:almost-inverse2}
		\rho(\bar{g}(g(x)),x) \lesssim \ve r_\ell. 
	\end{align}
\end{lem}

\begin{proof}
	By Lemma \ref{l:g-exists} it suffices to show 
	\begin{align}\label{e:pho-g_ell}
		\rho(\bar{g}_{\ell}(g(x)),x) \lesssim \ve r_\ell.
	\end{align}
	By construction $F_\ell(\bar{g}_\ell(g(x))) \in \bigcup_{j \in J_\ell} 6D^j.$ It then follows from Lemma \ref{l:gbar}, Corollary \ref{c:rho-phi} and Lemma \ref{l:g-inv1} that 
	\begin{align}
		\rho(\bar{g}_\ell(g(x)),x) &\lesssim d_X(g(\bar{g}_\ell(g(x))) , g(x)) + \ve r_\ell \\
		&\lesssim d_X(g_\ell(\bar{g}_\ell(g(x))) , g(x)) + \ve r_\ell \lesssim \ve r_\ell. 
	\end{align}
\end{proof}

Equation \eqref{e:Reif3.0} follows from \eqref{e:Reif2.5} and \eqref{e:Reif3.5}. Indeed, setting $x = g(w)$ and $y = g(z)$ gives
\[| d_X(g(w),g(z)) - \rho(w,z) | \lesssim \ve r_\ell . \]

\bigskip

\subsection{The Ahlfors-regular case}\label{s:AR}

We now prove \eqref{e:regularity}. 

\begin{lem}
	If $X$ is Ahlfors $n$-regular then $(\R^n,\rho)$ is Ahlfors $n$-regular. Moreover, the regularity constant for $(\R^n,\rho)$ depends only on $n$ and the regularity constant for $X$.  
\end{lem}

\begin{proof}
	Let $x \in \R^n$ and $r > 0$. Let $B_0 = B_\rho(x,r)$ and for $\ell \geq 1$ set $B_\ell = F_\ell(B_0) \subseteq M_\ell$. Let 
	\begin{align}
		k = \max\{\ell \geq 0 \colon \mbox{there exists } j \in J_\ell \mbox{ such that } B_\ell \subseteq 6D^j\},
	\end{align}
	where the maximum of the empty set is defined to be zero. If $k \geq 1$ let $j \in J_k$ such that $B_k \subseteq 6D^{i,k}$. Let $\eta > 0.$ We split into two cases. \\
	
	\noindent\textbf{Case 1:} Suppose that $d_k(x_k,y_k) \leq \eta r_k$ for all $y \in B_0$. The setup in this case is exactly the same as in Case 1 of the proof of Lemma \ref{l:Reif-flat}. In particular, \eqref{e:F_k-bi-Lip} and \eqref{e:disjoint} hold for $B_k$. We will see that
	\begin{align}\label{e:inclusion}
		B(x_k,r/2) \subseteq B_k.
	\end{align}
	This is immediate for $k =0$ so let us suppose $k \geq 1.$ We begin by showing
	\[ r \leq 2\eta r_k.\]
	Indeed, suppose $y \in B_0$ is such that $\rho(x,y) = r$ and let $k(x,y)$ be as in Lemma \ref{l:k(x,y)}. By \eqref{e:disjoint} we know $k \leq k(x,y) \leq k+1$. It then follows from Corollary \ref{c:rho-phi} and Lemma \ref{l:f-bilip1} that 
	\begin{align}
		r &= \rho(x,y) \leq (1+C\ve)[d_{k(x,y)}(x,y) + C\ve r_{k(x,y)}] \\
		&\leq (1+C\ve)[d_{k}(x,y) + C\ve r_{k}] \leq 2\eta r_k. 
	\end{align}
	We return to \eqref{e:inclusion}. Let $w \in B(x_k,r/2)$ and let $z \in \R^n$ such that $w = z_k.$ Since $d_k(x_k,z_k) \leq r/2 \leq \eta r_k$ and $x_k \in 6D^j$ we have $k(x,z) \geq k-1.$ Using \eqref{e:disjoint} again, we have $k(x,z) \leq k + 1.$ Thus, 
	\begin{align}
		\rho(x,z) \leq (1+C\ve)\phi_\infty(x,z) \leq (1+C\ve)d_k(x,z) \leq (1+C\ve)r/2 \leq r. 
	\end{align}
	If follows that $z \in B_0$ and $w = z_k \in B_k.$ The measure estimates for $B_0$ now follow from \eqref{e:F_k-bi-Lip} since 
	\begin{align}
		r^n &\gtrsim \calH^n(6D^j) \gtrsim (1+C\ve)\calH^{n}(B_k) \geq  \calH^n(B_0) \geq (1+C\ve)^{-1} \calH^n(B_k)\\
		&\geq \calH^n(B(x_k,r/2)) \gtrsim r^n. 
	\end{align}
	\\
	\noindent\textbf{Case 2:} Suppose instead that there exists $y \in B_0$ such that $d_k(x_k,y_k) > \eta r_k.$ We begin by estimating $\calH^n(B_0)$ from above. Let 
	\[E_\infty = \{ x \in B_0 \colon \mbox{for every } \ell \geq 0 \mbox{ there exists } j \in J_\ell \mbox{ such that } x_\ell \in 6D^j\}\] 
	and, for each $\ell \geq k$ and $j \in J_\ell,$ let 
	\begin{align}\label{e:R}
		R^{j,\ell} = 6D^{j,\ell} \setminus f_\ell^{-1}\left( \bigcup_{i \in J_{\ell+1}} 6D^{i,\ell+1} \right) \mbox{ and } R^{j,\ell}_\infty = F_\ell^{-1}(R^{j,\ell}). 
	\end{align}
	Set
	\begin{align}
		J'_\ell = \{ j \in J_\ell \colon R^{j,\ell}_\infty \cap B_0 \neq \emptyset \} \mbox{ and } J' = \bigcup_{\ell \geq k} J_\ell'.
	\end{align}
	It follows that
	\begin{align}\label{e:B-infty-contained}
		B_0 \subseteq E_\infty \cup \bigcup_{\ell \geq k}\bigcup_{j \in J'_\ell} R^{j,\ell}_\infty.
	\end{align}
	Let us bound the measure of $E_\infty$ from above. Observe that $g|_{E_\infty}$ is $(1+C\ve)$-bi-Lipschitz. Indeed, suppose $w,z \in E_\infty$ and let $k' = k(w,z)$ as in Lemma \ref{l:k(x,y)}. Since $w,z \in E_\infty,$ we have by Lemma \ref{l:outside} that $\phi_\infty(w,z) = d_{k'}(w_{k'},z_{k'}) \gtrsim r_{k'}$ so that $\rho(w,z) \gtrsim r_{k'}$ by Lemma \ref{l:pert}. Since $w_{k'},z_{k'} \in 6D^j$ for some $j \in J_{k'},$ it follows from Corollary \ref{c:rho-phi} that
	\begin{align}
		|d_X(g(x),g(y)) - \rho(x,y) | \lesssim \ve (r_{k'} + \rho(x,y)) \lesssim \ve \rho(x,y). 
	\end{align} 
	Now that we know $g|_{E_\infty}$ is $(1+C\ve)$-bi-Lipschitz, we have that $g(E_\infty)$ is contained in a ball of radius $2r$ in $X$. Using Ahlfors regularity of $X$ and the bi-Lipschitz estimates of $g|_{E_\infty}$ again, we get
	\begin{align}\label{e:upper-E-infty}
		\calH^n(E_\infty) \leq (1+C\ve) \calH^n(g(E_\infty)) \lesssim r^n. 
	\end{align}
	Let us bound from above the measure of the second set in \eqref{e:B-infty-contained}. For each $j \in J_\ell',$ fix $z^{j,\ell} \in R^{j,\ell}_\infty \cap B_0$ and let 
	\[w^{j,\ell} =  \away_{j,\ell} (\psi_{j,\ell}(z^{j,\ell}_\ell)) \in X.\]
	Since $d_k(x_k,y_k) > \eta r_k$ for some $y \in B_0,$ we have $r \gtrsim \eta r_k$ by Lemma \ref{l:phi-cauchy} and Lemma \ref{l:pert} (see Case 2 in the proof of Lemma \ref{l:Reif-flat}). By taking $\ve$ small enough, this, Lemma \ref{l:whereinX}, Corollary \ref{c:rho-phi} (noting that $z^{j,\ell}_\ell \in 6D^{j,\ell}$) and the fact that $\ell \geq k$, give 
	\begin{align}
		\begin{split}\label{e:centre-close}
		d_X(g_k(x),w^{j,\ell}) &\leq d_X(g_k(x),g_\ell(z^{j,\ell})) + C\ve r_\ell \leq d_X(g_k(x),g_k(z^{j,\ell})) + C\ve r_k \\
		&\leq (1+C\ve)\rho(x,z^{j,\ell}) + C\ve r_k \leq (1+C\ve \eta^{-1}) \rho(x,z^{j,\ell})\\
		&\leq 2r.
		\end{split}
	\end{align}
	Define 
	\begin{align}\label{e:frak-B}
		\mathfrak{B}^{j,\ell} \coloneqq B_X(w^{j,\ell},\eta r_\ell) \subseteq B_X(g_k(x),3r),
	\end{align}
	where the final inclusion follows from \eqref{e:centre-close}. We claim that $\{\mathfrak{B}^j\}_{\ell \geq k, \ j \in J_\ell'}$ has bounded overlap i.e. there exists a constant $C \geq 1$ such that, for all $q \in X,$ we have 
	\begin{align}\label{e:boundedoverlap1}
		\sum_{\ell \geq k} \sum_{j \in J_\ell'} \mathds{1}_{\mathfrak{B}^{j,\ell}}(q) \leq C.
	\end{align}
	Indeed, fix $q \in X$ and suppose the left-hand side of \eqref{e:boundedoverlap1} is non-zero (otherwise there is nothing to show). 	We first show that
	\begin{align}\label{e:boundedoverlap3}
		| \{\ell \geq k : \sum_{j \in J_\ell'} \mathds{1}_{\mathfrak{B}^{j,\ell}}(q) \neq \emptyset \} | \leq 2. 
	\end{align}
	Let $\ell \geq k$ the smallest integer such that $\sum_{j \in J_\ell'} \mathds{1}_{\mathfrak{B}^{j,\ell}}(q) \neq \emptyset$ and let $i \in J_\ell'$ such that $q \in \mathfrak{B}^{i,\ell}.$ Suppose further that $m \geq \ell+2$ and $j \in J_m'.$ Recalling the definition of $k(\cdot,\cdot)$ from Lemma \ref{l:k(x,y)}, we have $k'' = k(z^{i,\ell},z^{j,\ell}) \leq \ell.$ Since $z^{j,m} \in \bigcup_{j \in J_m} 6D^j$ and $m \geq \ell+2 \geq k''+2$ we have by Lemma \ref{l:outside} that 
	\begin{align}
		\phi_\infty(z^{i,\ell},z^{j,m}) = d_{k''}(z^{i,\ell}_{k''},z^{j,m}_{k''}) \geq r_{k''+1}/2 \geq r_{\ell+1}/2.
	\end{align}
	Lemma \ref{l:whereinX}, Lemma \ref{l:gbar} and Lemma \ref{l:phi_infty} then give  
	\begin{align}
		d_X(y^{i,\ell},y^{j,m}) &\geq d_X(\bar{g}_\ell(z^{i,\ell}),\bar{g}_{m}(z^{j,m})) - C\ve r_\ell \geq d_X(\bar{g}(z^{i,\ell}),\bar{g}(z^{j,m})) - C\ve r_\ell \\
		&\geq \phi_\infty(z^{i,\ell},z^{j,m}) - C\ve r_\ell \geq r_{\ell+1}/8 \geq \eta r_\ell + \eta r_m. 
	\end{align}
	This implies $\mathfrak{B}^{i,\ell} \cap \mathfrak{B}^{j,m} = \emptyset$ so that $x \not\in \mathfrak{B}^{j,m}.$ Since $m \geq \ell+2$ and $j \in J_m'$ are arbitrary, we get \eqref{e:boundedoverlap3}. 
	
	Next, we show that 
	\begin{align}\label{e:boundedoverlap2}
		\sum_{j \in J_m'} \mathds{1}_{\mathfrak{B}^{j,m}}(q) \lesssim 1 \mbox{ for all } m \in \{\ell,\ell+1\}.
	\end{align}
	Fix $i \in J_m'$ such that $q \in \mathfrak{B}^{i,m}$ (we may assume such an index exists otherwise there is nothing to show). Suppose $j \in J_{m}'$ is such that $q \in \mathfrak{B}^{i,m} \cap \mathfrak{B}^{j,m}$. By definition $w^{i,m} \in 6B^{i,m}$ and $w^{j,m} \in 6B^{j,m}$. It follows that $q \in 7B^{i,m} \cap 7B^{j,m}$ and $x_{j,m} \in 15B^{i,m}.$ Thus, \[\{x_{j,m} : j \in J_{m}', \  x \in \mathfrak{B}^{j,m}\}\]
	is an $r_m$-separated collection of points contained in $15B^{i,m}.$ Since $X$ is Ahlfors $n$-regular, this implies \eqref{e:boundedoverlap2}. 
	
	Using \eqref{e:boundedoverlap1} and \eqref{e:frak-B} we have
	
	\begin{align}\label{e:upper-R}
		\sum_{\ell \geq k} \sum_{j \in J_\ell'} \calH^n(R^{j,\ell}_\infty) &= \sum_{\ell \geq k} \sum_{j \in J_\ell'} \calH^n(R^{j,\ell}) \leq \sum_{\ell \geq k} \sum_{j \in J_\ell'} \calH^n(6D^{j,\ell}) \\
		&\lesssim_\eta \sum_{\ell \geq k} \sum_{j \in J_\ell'} \calH^n(\mathfrak{B}^{j,\ell}) \lesssim r^n. 
	\end{align}
	Combining \eqref{e:B-infty-contained}, \eqref{e:upper-E-infty} and \eqref{e:upper-R} gives 
	\begin{align}
		\calH^n(B_0) \lesssim r^n. 
	\end{align}
	
	Let us now prove the lower Ahlfors estimate. Let $z = \psi_j(x_k).$ We construct a continuous mapping $\sigma : B_{\|\cdot\|_j}(z,r/4) \to B_{\|\cdot\|_j}(z,r/4)$ such that $\sigma = \sigma_2 \circ \sigma_1$ for some continuous $\sigma_1 \colon B_{\|\cdot\|_j}(z,r/4) \to B_0$ and some $C(n)$-Lipschitz $\sigma_2 \colon B_0 \to \R^n_j$. We will do this in such a way that 
	\begin{align}\label{e:f-identity}
		\| \sigma(w) - w \|_j \lesssim \ve \eta^{-1} r \mbox{ for all } w \in B_{\|\cdot\|_j}(z,r/4).
	\end{align}
	
	Before constructing $\sigma$, let us see how the existence of such a map gives the lower Ahlfors bound. First, Lemma \ref{l:topology} implies $B_{\|\cdot\|_j}(z,r/8) \subseteq \sigma(B_{\|\cdot\|_j}(z,r/4)).$ Then, since $\sigma(B_{\|\cdot\|_j}(z,r/4)) \subseteq \sigma_2(B_0)$ and $\sigma_2$ is $C(n)$-Lipschitz, we have 
	\begin{align}
		r^n \lesssim \calH^n(B_{\|\cdot\|_j}(z,r/8)) \leq \calH^n(\sigma_2(B_0)) \lesssim_n \calH^n(B_0).
	\end{align}
	
	To define the mapping we first need to know
	\begin{align}\label{e:in}
		B_{\|\cdot\|_j}(z,r/4) \subseteq \psi_j(B_k). 
	\end{align}
	Since $\psi_j$ is $(1+C\ve)$-bi-Lipschitz it is enough to show $B(x_k,r/2) \subseteq B_k \subseteq 6D^j.$ If $k =0$ then this is immediate. Suppose instead that $k \geq 1.$ Recall from above \eqref{e:centre-close} that $r \gtrsim \eta r_k$. Thus, if $p \in B_0$ is such that $\rho(x,p) = r,$ then for $\ve$ small enough, Corollary \ref{c:rho-phi} gives 
	\begin{align}
		13r_k \geq d_k(x_k,p_k) \geq (1+C\ve)^{-1}\rho(x,p) - C\ve r_k \geq r/2.  
	\end{align} 
	In the first inequality we used that $\diam(6D^j) \leq 13r_k.$ Now suppose $z \in \R^n$ such that $z_k \in B(x_k,r/2)$ and let us prove \eqref{e:in}. By the above inequality and Lemma \ref{l:f-bilip1} we have $d_{k-1}(z_{k-1},x_{k-1}) \leq 14r_k \leq 2r_{k-1}.$ By Lemma \ref{l:inclusion} we have $x_{k-1} \in 3D^{i(j)}$ so that $z_{k-1} \in 4D^{i(j)}$ by Lemma \ref{l:ball-D}. All of this is to say that $k(x,z) \geq k-1$ and so $\phi_k(x_k,z_k) \leq (1+C\ve)d_k(x_k,z_k) \leq r/4$ by Lemma \ref{l:f-bilip1}. Now, using Corollary \ref{c:rho-phi} and choosing $\ve$ small enough depending on $\eta,$ we have 
	\[ \rho(x,z) \leq (1+C\ve)(\phi_k(x_k,z_k) + C\ve \eta^{-1} r)  \leq r. \]
	This gives $z \in B_0$ which implies $z_k \in B_k,$ and we are done. 
	
	Now let us define the function $\sigma.$ Set $\sigma_1 = F_k^{-1} \circ \psi_j^{-1} \colon B_{\|\cdot\|_j}(z,r/4) \to \R^n.$ Then, by Lemma \ref{l:phi-cauchy}, Lemma \ref{l:pert} and since $\psi_j$ is $(1+C\ve)$-Lipschitz, we have 
	\begin{align}\label{e:almost-isom}
		| \rho(\sigma_1(p) , \sigma_1(q) ) - \| p -q \|_j | \leq C\ve r_k \leq C_1\ve \eta^{-1} r. 
	\end{align} 
	for all $p,q \in B_{\|\cdot\|_j}(x_k,r/4).$ Thus, 
	\begin{align}\label{e:in-B-infty}
		\sigma_1(B_{\|\cdot\|_j}(z,r/4)) \subseteq B_0.
	\end{align}
	Let $\mathcal{N}$ be a maximal $\eta^2 r_k$-separated net in $B_{\|\cdot\|_j}(z,r/4)$ such that $z \in \mathcal{N}.$ By \eqref{e:almost-isom}, $\sigma_1|_{\mathcal{N}}$ is $(1+C\ve)$-bi-Lipschitz. Let $\mathcal{N}' = \sigma_1(\mathcal{N}),$ which is contained in $B_0$ by \eqref{e:in-B-infty}, and extend $(\sigma_1|_{\mathcal{N}})^{-1}$ to a $C(n)$-Lipschitz map $\tilde{\sigma}_2 \colon B_0 \to \R^n_j.$ Finally, let 
	\[ \sigma_2 = p_{z,r/4} \circ \tilde{\sigma}_2,  \]
	where $p_{z,\lambda}$ is the map defined in Lemma \ref{l:p-lambda}. Since $\tilde{F}_2$ is $C(n)$-Lipschitz and $p_{z,r/4}$ is 2-Lipschitz, it only remains to show \eqref{e:f-identity}. Fix $w \in B_{\|\cdot\|_j}(z,r/4)$ and let $q \in \mathcal{N}$ such that $\| q - w \|_j \leq \eta^2 r_k \lesssim \eta r.$ Since $q \in \mathcal{N}$ we have $q = \tilde{\sigma}_2 \circ \sigma_1(p).$ Hence, 
	\begin{align}\label{e:-w}
		\| \tilde{\sigma}_2(\sigma_1(w)) - w \|_j &\leq \|  \tilde{\sigma}_2(\sigma_1(w)) - \tilde{\sigma}_2(\sigma_1(q)) \| + \| q - w \|_j  \lesssim_n d( \sigma_1(w) , \sigma_1(q) ) + \eta r \\
		&\leq \|q-w\|_j + \eta r \lesssim \eta r. 
	\end{align}
	Since $w \in B_{\|\cdot\|_j}(z,r/4),$ this implies $\| \tilde{\sigma}_2(\sigma_1(w))  - z\| \lesssim \eta r.$ Now, Lemma \ref{l:p-lambda} and \eqref{e:-w} gives  
	\begin{align}
		\| \sigma(w) - w \|_j \leq \| p_{z,r/4}(\tilde{\sigma}_2(\sigma_1(w)) ) - \tilde{\sigma}_2(\sigma_1(w)) \|_j + \| \tilde{\sigma}_2(\sigma_1(w))  - w \|_j \lesssim \ve r. 
	\end{align}
\end{proof}

\bigskip 

\subsection{Proof of Theorem \ref{t:Reif-intro}}\label{s:last}

In this section we see how Theorem \ref{t:Reif} implies Theorem \ref{t:Reif-intro}. So, fix $\ve > 0$ and suppose $X$ is $(\delta,n)$-Reifenberg flat up to scale $s_0 > 0$ and let $x \in X$ and $0 < s_0 < r.$ By scaling we may suppose $r =1.$ As in Theorem \ref{t:Reif}, we let $r_\ell = 10^{-\ell}$. Let $k \in \N$ such that  
\begin{align}\label{e:r-size}
	r_{k+1} \leq s_0 < r_k.
\end{align}

For $\ell = 0$, define the index set $J_0 = \{j_0\}$ the point $x_{j_0,0} = x.$ Then, for $1 \leq \ell \leq k$, let $\{x_{j,\ell}\}_{j \in J_\ell}$ be a maximal $r_\ell$-separated net in $B_X(x,1)$. Conditions (1) - (3) are immediate to verify. Since $X$ is $(\delta,n)$-Reifenberg flat up to scale $s_0$ and $100 r_\ell \geq s_0$ for all $0 \leq \ell \leq k,$ we can find norms $\|\cdot\|_{j,\ell}$ satisfying (4) with constant $C\delta$ (note, the maps obtained from the Reifenberg condition do not automatically satisfy \eqref{e:centre-pos}, however, at the cost of using the slightly worse constant $C\delta$, we can modify these maps so that \eqref{e:centre-pos} holds by Lemma \ref{l:GH-scale}). 

Choose $\delta$ small enough so that the conclusion of Theorem \ref{t:Reif} holds with constant $\ve/20$ and let $\rho$ be the metric and $g,\bar{g}$ the maps we obtain. The statement concerning Ahlfors regularity is immediate from the corresponding statement in Theorem \ref{t:Reif}. The fact that $\bar{g}$ maps $B_X(x,r)$ into $B_\rho(0,2r)$ follows from \eqref{e:Reif1} and \eqref{e:Reif3.5}. To verify \eqref{e:coarse-bi-lip} we proceed as follows. Let $y,z \in B_X(x,r)$ and let $0 \leq \ell \leq k$ be the largest integer such that there exists $j \in J_\ell$ with $y,z \in 6B^{j,\ell}.$ Suppose to begin with that $\ell < k.$ By maximality there exists $i \in J_{\ell+1}$ such that $d(y,x_{i,\ell+1}) \leq r_{\ell+1}$, in particular, $y \in 6B^{i,\ell+1}.$ By maximality of $\ell$ we have $z \not \in 6B^{i,\ell+1}.$ It now follows that $d(y,z) \geq 5r_{\ell+1}.$ Using this with \eqref{e:Reif3.5} gives 
\begin{align}
	| \rho(\bar{g}(y),\bar{g}(z)) - d_X(y,z) | \leq \ve r_\ell /20 \leq \ve d(y,z) \leq \ve(d_X(y,z) + s_0).
\end{align}
Suppose instead that $\ell = k.$ Using \eqref{e:Reif3.5} again, with \eqref{e:r-size}, we have 
\begin{align}
	| \rho(\bar{g}(y),\bar{g}(z)) - d_X(y,z) | \leq \ve r_k/20 \leq \ve s_0 \leq \ve(d_X(y,z) + s_0).
\end{align} 
Let's now verify \eqref{e:last}. Pick some $u \in B_\rho(0,1-\ve).$ By \eqref{e:g-near-centre} we have $d_X(g(0),x) \leq \ve/2$ (recall we specified the conclusion of Theorem \ref{t:Reif} to hold with constant $\ve/20$). By \eqref{e:Reif1} we have $u \in B_{j_0,0}$ so that $g(u) \in 6B^{j_0,0}$. Furthermore, by \eqref{e:g-near-centre} we have $d_X(g(0),x) \leq \ve/2$ (recall we specified the conclusion of Theorem \ref{t:Reif} to hold with constant $\ve/20$) so that $g(0) \in 6B^{j_0,0}.$ Applying \eqref{e:Reif3.0} now gives 
\begin{align}
	d_X(g(u),x) \leq d_X(g(u),g(0)) + d_X(g(0),x) \leq \rho(u,0) + \ve/20 + \ve/20 \leq 1.
\end{align}
This implies $g(u) \in B_X(x,1)$, so, by maximality, there exists $j \in J_k$ such that $g(u) \in 6B^{j,k}.$ Using \eqref{e:Reif2.5} and \eqref{e:r-size} gives now
\begin{align}
	{\dist}_\rho(u,\bar{g}(B_X(x,1))) \leq \rho(u,\bar{g}(g(u))) \leq \ve r_k /20 \leq \ve s_0.
\end{align}

\newpage

\section{RF metric spaces and Corona Decompositions (CD)}\label{s:Reif}
\etocsettocstyle{Contents of this section}{}
\etocsettocmargins{.01\linewidth}{.01\linewidth}

\localtableofcontents

\bigskip

\subsection{Main result}
In this section we show how to use Theorem \ref{t:Reif} to prove some interesting results concerning Reifenberg flat metric spaces, recall Definition \ref{d:RF}. The definition is stated in terms of $\bilat_X$ which can be found in Definition \ref{d:bilat}. 

Our main goal is the following. The notion of corona decomposition by normed space will be defined in Section \ref{s:corona-def}. 

\begin{thm}\label{t:corona-reif}
	Suppose $X$ is a $(\delta,n)$-Reifenberg flat and Ahlfors
	$(C_0,n)$-regular metric space with a system of Christ-David cubes
	$\calD.$ Provided $\delta$ is small enough, depending only on $n$, then
	$X$ admits a corona decomposition by normed spaces with constants
	$\Corona_1,\Corona_2$ depending only on $n$. The Carleson norms in
	Definition \ref{d:corona-Y} depend additionally on $C_0,n$ and $\delta.$
\end{thm}

\bigskip

\subsection{RF  metric spaces are bi-H\"older to  normed  spaces}\label{r:proof-RF}
The first result on Reifenberg flat metric spaces is a generalization of \cite[Theorem A.1.1]{cheeger1997structure}. For Cheeger-Colding, Reifenberg flat means that every ball is flat with respect to $(\R^n,|\cdot|)$ i.e. $\R^n$ equipped with the Euclidean norm. For us, we allow the norm to vary as different scales and locations. The result in \cite{cheeger1997structure} is itself a metric space generalization of Reifenberg's topological disk theorem which appears in \cite{reifenberg1960solution}. First, let us state a version of Theorem \ref{t:Reif} tailored to Reifenberg flat metric spaces.

\begin{lem}\label{l:main-Reif}
	Let $\ve,\delta > 0.$ Suppose $(X,d)$ is $(\delta,n)$-Reifenberg flat, $x \in X$, $r > 0$ and $\|\cdot\|$ is a norm on $\R^n$ such that $\bilat_X(x,100r,\|\cdot\|) \leq 100\delta r.$ If $\delta$ is chosen small enough, depending on $\ve$ and $n$, there exists a metric $\rho$ on $\R^n$ which is bi-H\"older equivalent to $\|\cdot\|$ with exponent $1-\ve$ such that $(\R^n,\rho)$ is $(\ve,n)$-Reifenberg flat and 
	\begin{align}\label{e:Reif1'}
		\abs{ \rho(y,z) - \|y-z\|} \leq \ve r
	\end{align} 
	for all $y,z \in \R^n.$ Furthermore, there are maps $g :B_{\|\cdot\|}(0,6r) \to B(x,6r)$ and $\bar{g} \colon B(x,6r) \to B_{\|\cdot\|}(0,6r)$ such that $g|_{B_{\rho}(0,2r)}$ and $\bar{g}|_{B(x,2r)}$ are $(1+C\ve)$-bi-Lipschitz. We also have $g(0) \in B(x,10\ve r)$ and $\bar{g}(x) \in B_{\rho}(0,10\ve r).$ If $y \in B_{\rho}(0,2r)$ and $z \in B(x,2r),$ then 
	\begin{align}\label{e:Reif2'}
		\bar{g}(g(y)) = y \quad \mbox{ and } \quad g(\bar{g}(z)) = z. 
	\end{align}
	If $y,z \in B_{\|\cdot\|}(0,6r)$ $($resp. $y,z \in B(x,6r))$ then 
	\begin{align}\label{e:Reif3'}
		\abs{ d(g(y),g(z)) - \rho(y,z) } \leq \ve r
		\quad (\text{resp.}  \ \abs{ \rho(\bar{g}(y),\bar{g}(z)) - d(y,z) } \leq \ve r	) . 
		\end{align}
	If additionally $X$ is Ahlfors $n$-regular, then $(\R^n,\rho)$ is Ahlfors $n$-regular with regularity constant depending on $n$ and the regularity constant for $X$. 

\end{lem}

\begin{proof}
	By scale invariance we may assume $r = 1.$ Let $J_0 = \{j_0\}$ and set $x_{j_0,0} = x.$ Then, supposing we have defined collections of points $\{x_{j,\ell}\}_{j \in J_\ell}$ up to some $\ell \geq 0,$ obtain $\{x_{j,\ell+1}\}_{j \in J_{\ell+1}}$ by completing $\{x_{j,\ell}\}_{j \in J_\ell}$ to a maximal $r_{\ell+1}$-separated net in $B_X(x,3).$ For each $\ell \geq 0$ and $j \in J_{\ell}$ choose a norm $\|\cdot\|_{j,\ell}$ such that $\bilat(x_{j,\ell},100r_\ell,\|\cdot\|_{j,\ell}) \leq 100\delta r_\ell.$ It is immediate that the hypotheses of Theorem \ref{t:Reif} hold for the pair $(\{x_{j,\ell}\},\{\|\cdot\|_{j,\ell}\})_{\ell \geq 0, j \in J_\ell}$ with constant $C\delta.$ By choosing $\delta$ small enough, we can find a metric $\rho$ and maps $g,\bar{g}$ satisfying the conclusion of Theorem \ref{t:Reif} with constant $\ve.$ We immediately get that $\rho$ is bi-H\"older to $\|\cdot\|$ with the correct exponent and \eqref{e:Reif1'}. The statement concerning Ahlfors regularity is also immediate. The fact that $g(0) \in B(x,10\ve)$ and $\bar{g}(x) \in B_{\|\cdot\|}(0,10\ve)$ follows from \eqref{e:g-near-centre}. Equation \eqref{e:Reif3'} follows by combining \eqref{e:Reif3.0} and \eqref{e:Reif3.5}. It only remains to check the bi-Lipschitz conditions and \eqref{e:Reif2'}.
	
	 Let $y,z \in B_{\rho}(0,2).$ By \eqref{e:g-near-centre} we know $g(0) \in B(x,10\ve)$. Combing this with \eqref{e:Reif3.0} implies $g(y),g(z) \in B(x,3).$ Let $\ell \geq 0$ such that there exists $j \in J_\ell$ satisfying $g(y),g(z) \in 6B^j.$ Let $i \in J_\ell$ such that $d_X(x_{i,\ell+1},g(y)) \leq r_{\ell+1}$ (we know such a index exists by maximality because $g(y) \in B(x,3)$). By maximality of $\ell$, we know $g(z) \not\in 6B^{i},$ hence, $d_X(g(y),g(z)) \gtrsim r_\ell.$ If now follows from \eqref{e:Reif3.0} that 
	\begin{align}
		\abs{ d_X(g(y),g(z)) - \rho(y,z) } \leq \ve r_\ell \lesssim \ve d(g(y),g(z)). 
	\end{align}
	Suppose now that $w,x \in B(x,1).$ Let $k \geq 0$ be the largest integer such that there exists $i \in J_k$ such that $w,x \in 6B^i.$ By the same reasoning above, we have $d_X(w,x) \gtrsim r_\ell.$ Using \eqref{e:Reif3.5} again we get 
	\begin{align}
		\abs{ \rho(\bar{g}(w),\bar{g}(x) ) - d_X(w,x) } \leq \ve r_\ell \lesssim \ve d_X(w,x). 
	\end{align}
	
	Now for \eqref{e:Reif2'}. Suppose $y \in B_\rho(0,2)$ and $z \in B_X(x,2).$ As above, we know $g(y) \in B_X(x,3).$ By maximality, we know for each $\ell \geq 0$ that there exists $i,j \in J_\ell$ such that $g(y) \in 6B^i$ and $z \in 6B^j.$ Equation \eqref{e:Reif2'} then follows from \eqref{e:Reif2.0} and \eqref{e:Reif2.5}. 
\end{proof}

\begin{prop}\label{p:Reif-classical}
		Let $\ve,\delta > 0$. Suppose $(X,d)$ is $(\delta,n)$-Reifenberg flat, $x \in X$, $r > 0$ and $\|\cdot\|$ is a norm on $\R^n$ such that $\bilat(x,100r,\|\cdot\|) \leq 100\delta r.$ If $\delta$ is small enough, depending on $\ve$, then there exists a map $h \colon B_{\|\cdot\|}(0,r) \to B_X(x,r)$ which is bi-H\"older with exponent $\alpha = 1-\ve$ such that $h(0) \in B_X(x,10\ve r),$ 
		\begin{align}\label{e:Reif-dist}
			| d(h(y),h(z)) - \| y-z \| \, | \leq \ve r 
		\end{align} 
		for all $y,z \in B_{\|\cdot\|}(0,r)$ and 
		\begin{align}\label{e:g-inclusion}
			h(B_{\|\cdot\|}(0,r)) \supseteq B_X(0,(1-\ve)r).
		\end{align}
		
\end{prop}

\begin{proof}
	Assume without loss of generality that $r =1.$ Choose $\delta$ small enough so that the conclusion of Lemma \ref{l:main-Reif} holds with constant $\tfrac{\ve}{2}.$ Let $\rho$ be the metric and $g \colon B_{\|\cdot\|}(0,6r) \to B(x,6r)$ the map we obtain. Let $s \colon \R^n \to \R^n$ be the map defined by $s(x) = \tfrac{x}{1+\ve}$ and set $h \coloneqq g \circ s.$ Let us check that $h$ is bi-H\"older. Clearly $s$ is $(1+\ve)$-bi-Lipschitz and $s(B_{\|\cdot\|}(0,1)) \subseteq B_{\|\cdot\|}(0,1).$ Thus, it suffices to check $g$ is bi-H\"older on $B_{\|\cdot\|}(0,1).$ Let $y,z \in B_{\|\cdot\|}(0,1).$ By \eqref{e:Reif1'} we know $B_{\|\cdot\|}(0,1) \subseteq B_\rho(0,2)$. Then, since $g$ is bi-Lipschitz on $B_\rho(0,2)$ and $\rho$ is bi-H\"older equivalent to $\|\cdot\|$ with exponent $\alpha = 1-\ve,$ we have 
	\begin{align}
		d_X(g(y),g(z)) \lesssim \rho(y,z) \leq \| y-z\|^{1-\ve}
	\end{align}
	and 
	\begin{align}
		d_X(g(y),g(z)) \gtrsim \rho(y,z) \gtrsim \| y-z \|^{1+\ve}. 
	\end{align}
	Equation \eqref{e:Reif-dist} follows from \eqref{e:Reif1'}, \eqref{e:Reif3'}, and the fact that 
	\[\abs{ \|y-z\| - \| s(y) - s(z) \|} \lesssim \ve \]
	for all $y,z \in B_{\|\cdot\|}(0,1).$ Finally, we check \eqref{e:g-inclusion} (with a slightly worse constant). 
	Pick a point $u \in B_X(x,1-13\ve)$ and set $v = s^{-1}(\bar{g}(u)) = (1+\ve)\bar{g}(u).$ By \eqref{e:Reif2'} we know $g(\bar{g}(u)) = u$ so that $h(v) = u.$ By \eqref{e:Reif1'}, \eqref{e:Reif3'} and using that fact that $\bar{g}(x) \in B_\rho(0,10\ve),$ we have  
	\begin{align}
		\| \bar{g}(u) \| &\leq \rho(\bar{g}(u),0) + \ve \leq \rho(\bar{g}(u),\bar{g}(x)) + \rho(\bar{g}(x),0) + \ve \\
		&\leq d_X(u,x) + 12\ve \leq 1 - \ve \leq \tfrac{1}{1+\ve}. 
	\end{align}
	It now follows that $\|v\| = (1+\ve)\|\bar{g}(u)\| \leq 1$ so that $v \in B_{\|\cdot\|}(0,1).$
\end{proof}

\bigskip

\subsection{AR+RF implies UR}\label{s:RF-UR}
One major goal of this section is to show that upper Ahlfors regular Reifenberg flat metric spaces are UR (Proposition \ref{p:reif-UR}). The proof of this fact is reminiscent of the proof that \textit{Weak Geometric Lemma} and \textit{Big Projections} imply \textit{Big Pieces of Lipschitz Graphs} in \cite{david1993quantitative}. Crucial for us is the use of Theorem \ref{t:Reif-intro}. It allows us to show that Reifenberg flat metric spaces have a suitable metric space analogue of the big projections property (Corollary \ref{c:Reif-lower-reg}) and satisfy certain topological conditions required to apply a result of Guy C. David \cite{david2016bi} (see Theorem \ref{t:david}) on the bi-Lipschitz decomposition of Lipschitz functions defined on certain metric spaces.

To state the result of Guy C. David, we need the following definition.

\begin{defn}\label{d:LLC}
	A metric space $X$ is called \textit{linearly locally contractible} (LLC) if there exist constant $A \geq 1$ and $R > 0$ such that every ball $B$ of radius $0 < r < R$ is contractible in $AB$ (i.e. there exists a continuous map $H \colon B \times [0,1] \to AB$ such that $H(\cdot,0)$ is the identity map on $B$ and $H(\cdot,1)$ is a constant map).
\end{defn} 

\begin{thm}[{\cite[Theorem 1.1]{david2016bi}}]\label{t:david}
	Let $X$ be an Ahlfors $n$-regular, LLC, complete, oriented topological $n$-manifold and $\|\cdot\|$ a norm on $\R^n.$ Let $\calD$ be a system of Christ-David cubes for $X$ and $z : Q_0 \to (\R^n,\|\cdot\|)$ be Lipschitz. For each $\ve > 0$ there exist constants $L,N \geq 1$ depending only on $\ve, n$, the regularity constant of $X$ and the Lipschitz constant of $z,$ and measurable subsets $E_1 , \dots, E_N$ such that $z|_{E_i}$ is $L$-bi-Lipschitz and 
	\begin{align}
		\calH^n\left(z\left(Q_0 \setminus \bigcup_{i=1}^N E_i\right)\right) \leq \ve \calH^n(Q_0). 
	\end{align}
\end{thm}

\begin{rem}
	The statement in \cite{david2016bi} is more general than that stated above. In particular, it allows for Ahlfors $s$-regular (i.e. non-integer) domains and Ahlfors $s$-regular targets spaces possessing $n$-\textit{manifold weak tangents}, see \cite[Definition 1.5]{david2016bi}. It is easy to check that $(\R^n,\|\cdot\|)$ has $n$-manifold weak tangents, see the first point in \cite[Remark 1.7]{david2016bi} for an explanation in the case of $\R^n$ equipped with Euclidean norm. The argument for general norms is exactly the same. 
\end{rem}

In order to make use of Theorem \ref{t:david}, we of course need to know Reifenberg flat metric spaces are LLC. This is a consequence of Proposition \ref{p:Reif-classical} and is proven below.

\begin{prop}\label{p:Reif-LLC}
	Suppose $X$ is a $(\delta,n)$-Reifenberg flat metric space. If $\delta > 0$ is chosen small enough then $X$ is LLC with constants $A = 2$ and $R = \infty.$ 
\end{prop}

\begin{proof}
	Let $x \in X$, $r > 0$ and choose a norm $\|\cdot\|$ such that $\bilat_X(x,200r,\|\cdot\|) \leq 200\delta r.$ Let $\ve > 0$ and choose $\delta$ small enough so the conclusion of Proposition \ref{p:Reif-classical} holds with constant $\ve.$ Let $h \colon B_{\|\cdot\|}(0,2r) \to B_X(x,2r)$ be the map we obtain. By \eqref{e:g-inclusion} we know $h^{-1}$ is well-defined on $B_X(x,r).$ Since $h(0) \in B_X(x,20 \ve r)$ we know by \eqref{e:Reif-dist} that $h^{-1}(x) \in B_{\|\cdot\|}(0,22\ve r).$ Using \eqref{e:Reif-dist} again implies $h^{-1}(B_X(x,r)) \subseteq B_{\|\cdot\|}(0,(1+24\ve)r).$ Now consider the map $H : B(x,r) \times [0,1] \to X$ defined by 
	\begin{align}
		H(z,t)  = h \left( h^{-1}(x) + t(h^{-1}(z) - h^{-1}(x)) \right).
	\end{align}
	We observe first that $H$ takes values in $B_X(x,2r).$ Indeed, suppose $z \in B(x,r)$ and $0 \leq t \leq 1.$ By the discussion above we know both $h^{-1}(x),h^{-1}(z) \in B_{\|\cdot\|}(0,(1+24\ve)r) \subseteq B_{\|\cdot\|}(0,3r/2).$ By convexity, this implies $h^{-1}(x) + t(h^{-1}(z) - h^{-1}(x)) \in B_{\|\cdot\|}(0,3r/2).$ Since $h(0) \in B_X(0,20\ve r),$ equation \eqref{e:Reif-dist} now implies $H(z,t) \in B_X(x,2r),$ as required. Since $h$ is bi-H\"older it is continuous and its inverse $h^{-1}$ is continuous on $B_X(x,r).$ Since the map $(y,s) \mapsto h^{-1}(x) + s(y  - h^{-1}(x))$ is continuous on $\R^n \times [0,1]$, it follows that $H$ is continuous $B(x,r) \times [0,1].$ Finally, note that $H(z,0) = x$ and $H(z,1) = z$. Recalling Definition \ref{d:LLC}, this finishes the proof of Proposition \ref{p:Reif-LLC}. 
\end{proof}

Proposition \ref{p:Reif-classical} also gives us lower bounds on the Hausdorff measure of balls in Reifenberg flat sets (this is why we only need to assume upper Ahlfors regularity). 

\begin{cor}\label{c:Reif-lower-reg}
	Suppose $X$ is $(\delta,n)$-Reifenberg flat. If $\delta$ is small enough then the following holds. Let $x \in X, \ r > 0$ and choose a norm $\|\cdot\|$ and a $C(n)$-Lipschitz map $\varphi \colon B(x,100r) \to B_{\|\cdot\|}(0,100r)$ such that $\bilat_X(x,100r,\|\cdot\|,\vp) \lesssim \delta$ (such objects exist by Lemma \ref{l:Lipschitz-GHA}). Then, 
	\begin{align}\label{e:lower-reg}
		 \calH^n(B(x,r)) \gtrsim_n \calH^n(\vp(B(x,r))) \gtrsim_n r^n.
	\end{align}
	In particular, $X$ is Ahlfors lower $n$-regular, with regularity constant depending only on $n$. 
\end{cor}

\begin{proof}
	The first inequality in \eqref{e:lower-reg} since we assumed $\vp$ to be $C(n)$-Lipschitz. It only remains to prove the second inequality. Let $\ve > 0$ and choose $\delta$ small enough such that Proposition \ref{p:Reif-classical} holds with constant $\ve$. Let $h : B_{\|\cdot\|}(0,r) \to B_X(x,r)$ be the map we obtain. Now consider the map $f \coloneqq p_{0,r} \circ \vp \circ h$ which is continuous from $B_{\|\cdot\|}(0,r)$ to itself. Recall, $p_{z,\lambda}$ is the map from Lemma \ref{l:p-lambda}. Suppose $y,z \in B_{\|\cdot\|}(0,r).$ By definition $h(y),h(z) \in B_X(x,r)$ and $\vp(g(y)),\vp(g(z)) \in B_{\|\cdot\|}(0,r+100\delta r).$ So, by Lemma \ref{l:p-lambda} and \eqref{e:Reif-dist} we have $\abs{ \|f(y) - f(z) \| - \|y-z \| } \leq  C\ve r.$ It follows from this that 
	\begin{align}
		\calH^n_{\|\cdot\|}(f(B_{\|\cdot\|}(0,r))) \gtrsim r^n
	\end{align}
	(this is a consequence of Brouwer's Fixed Point Theorem, see \cite[Lemma 7.1]{bate2020purely} for a proof involving the unit Euclidean ball). Finally, using that $p_{0,r}$ is 2-Lipschitz to get
	\begin{align}
		\calH^n_{\|\cdot\|}(\vp(B(x,r))) \geq \calH^n_{\|\cdot\|}(\vp(h(B_{\|\cdot\|}(0,r)))) \gtrsim \calH^n_{\|\cdot\|}(f(B_{\|\cdot\|}(0,r))) \gtrsim r^n. 
	\end{align}
\end{proof}

We now have everything we need to prove UR.

\begin{prop}\label{p:reif-UR}
	Suppose $X$ is $(\delta,n)$-Reifenberg flat and Ahlfors upper $(C_0,n)$-regular. If $\delta > 0$ is small enough, depending only on $n$, then $X$ is uniformly $n$-rectifiable with \emph{UR} constants depending only on $\delta, C_0$ and $n$.  
\end{prop}

\begin{proof}
	Let $\ve > 0$ (to be chosen small enough) and choose $\delta$ small enough so that Corollary \ref{c:Reif-lower-reg} and conclusion of Lemma \ref{l:main-Reif} holds with constant $\ve.$ It is immediate that $X$ is Ahlfors $n$-regular and it remains to check $X$ has big pieces of Lipschitz images. Let $x \in X$, $R > 0$ and choose a norm $\|\cdot\|$ such that $\bilat_X(x,100R,\|\cdot\|) \leq 100\delta.$ Now, let $\rho$ be the metric and $g$ the map we obtain by applying Lemma \ref{l:main-Reif}. Let $r = R/2.$ \\
	
	\noindent\textbf{Claim:} There exists $A \subseteq B_{\|\cdot\|}(0,r)$ and an $L(n)$-Lipschitz map $f \colon A \to B_{\rho}(0,r)$ such that 
	\[\calH^n(f(A) \cap B_\rho(0,r) ) \gtrsim r^n.\] 
	Before proving the claim, let us see how it finishes the proof of the lemma. Recalling Remark \ref{r:BM}, after rescaling we can find a $C(n)$-bi-Lipschitz map $T \colon (\R^n,\|\cdot\|) \to \R^n$ such that $B_{\|\cdot\|}(0,r) \subseteq B_{|\cdot|}(0,R).$ Let $A' = T(A) \subseteq B_{|\cdot|}(0,R).$ Combing the fact that $g(0) \in B(x,10\ve R)$ with \eqref{e:Reif3'} we know $g(B_\rho(0,r)) \subseteq B(x,R).$ Since $g$ is $(1+\ve)$-bi-Lipschitz on $B_\rho(0,r)$, the map $F \coloneqq f \circ g \circ T^{-1} \colon A' \to B(x,R)$ is $C(n)$-Lipschitz. We now get 
	\begin{align}
		\calH^n(F(A') \cap B(x,R)) \gtrsim \calH^n( f(A) \cap B_\rho(0,r)) \gtrsim r^n \gtrsim R^n. 
	\end{align}
	
	For the remainder of the proof we focus on proving the claim. Since is $(\ve,n)$-Reifenberg flat, by applying Lemma \ref{l:Lipschitz-GHA} we can find a norm $\|\cdot\|$ and a $C(n)$-Lipschitz mapping $\vp \colon B_\rho(0,100r) \to B_{\|\cdot\|}(0,100r)$ satisfying \[\bilat_{(\R^n,\rho)}(0,100r,\|\cdot\|,\vp) \leq C\ve.\]
	For $\ve$ small enough, Lemma \ref{c:Reif-lower-reg} implies 
	\begin{align}\label{e:rho-big}
		\calH^n_{\|\cdot\|}(\vp(B_\rho(0,r))) \gtrsim r^n.
	\end{align}
	 By Theorem \ref{t:Reif} and Proposition \ref{p:Reif-LLC} we know $(\R^n,\rho)$ satisfies the hypotheses of Theorem \ref{t:david}. Let $\eta > 0$ (to be chosen small enough) and find constant $L,N \geq 1$, depending on $\eta, n$ and $C_0$ such that $\varphi|_{E_i}$ is $L$-Lipschitz for each $i = 1,\dots,N$ and 
	\begin{align}\label{e:small-image}
		\calH^n\left(\varphi\left(B_\rho(0,r) \setminus \bigcup_{i=1}^N E_i\right)\right) \leq \eta \calH^n(B_\rho(0,r)).
	\end{align}
	 Combining \eqref{e:rho-big}, \eqref{e:small-image} and choosing $\eta$ small enough, we have 
	\begin{align}
		\calH^n\left(\varphi\left(\bigcup_{i=1}^N E_i\right)\right) \geq \calH^n(\varphi(B_\rho(0,r))) - \eta \calH^n(B_\rho(0,r)) \gtrsim r^n.
	\end{align}
	In particular, there exists some $i \in  \{ 1,\dots,N\}$ satisfying $\calH^n(\varphi(E_i)) \gtrsim 1/N.$ Let $A = \vp(E_i) \subseteq B_{\|\cdot\|}(0,r)$ and set $f = \vp^{-1}|_A \colon A \to B_\rho(0,r).$ It follows that $f$ is $C(n)$-bi-Lipschitz and 
	\begin{align}
		\calH^n(f(F_i)) \gtrsim_n \calH^n(E_i) \gtrsim r^n,
	\end{align}
	as required. 
\end{proof}

\bigskip

\subsection{Corona Decompositions and CD($\mathscr{Y}$)} \label{s:corona-def}
The corona decomposition for subsets of Euclidean space was introduced in \cite{david1991singular}. Roughly speaking, a set $E \subseteq \R^n$ admits a corona decomposition if any system of Christ-David cubes may be partition into good and bad cubes which satisfy the following conditions. First, the number of bad cubes is controlled (in the sense that they satisfy a Carleson packing condition). Second, the good cubes may be decomposed into a controlled number of stopping-time regions such that for each of these regions there exists a Lipschitz graph which well-approximated the region. The definition of stopping-time region in given below. Again, controlled means that the top cubes satisfy a Carleson packing condition. For us, we need an alternative to the approximation by Lipschitz graphs.

We take the opportunity here to introduce a version of the corona decomposition for metric spaces. The final part of this section will be showing the existence of such a decomposition for Reifenberg flat metric spaces.

\begin{defn}[Corona Decomposition by $\mathscr{Y}$]\label{d:corona-Y}
	Suppose $(X,d)$ is an Ahlfors $n$-regular metric space and $\calD$ a system of Christ-David cubes on $X$. Let $\mathscr{Y}$ be a family of Ahlfors $n$-regular metric spaces with uniform regularity constant. We say $X$ admits a \textit{Corona Decomposition by $\mathscr{Y}$} (CD($\mathscr{Y}$)), with constants $\Corona_1,\Corona_2 \geq 1$, if for each $\theta > 0$ there exists a partition $\calD = \mathcal{G} \cup \mathcal{B}$ such that the following conditions hold. 
	\begin{enumerate}
		\item There exists a partition of $\mathcal{G}$ into disjoint stopping-time regions $\{S\}_{S \in \calF}.$ 
		\item The families $\calB$ and $\{Q(S)\}_{S \in \calF}$ satisfy a Carleson packing condition with Carleson constants depending only on $\theta.$ 
		\item For each $S \in \calF$ there exists $Y_S \in \mathscr{Y}$, a point $y \in Y$ and a map $\varphi_S \colon 3B_{Q(S)} \to B_Y(y,3\ell(Q(S)))$ such that if $Q \in S$ and $x,y \in 3B_Q$ satisfy $d(x,y) \geq \theta \ell(Q),$ then 
		\begin{align}
			\frac{1}{\Corona_1}d(x,y) \leq d_{Y_S}(\varphi_{S}(x), \varphi_{S}(y) ) \leq \Corona_2 d(x,y).
		\end{align}
	\end{enumerate}
\end{defn}

\begin{defn}\label{d:CDV}
	We say $X$ admits a \textit{Corona Decomposition by normed spaces} (CD(normed spaces)) with constant $\Lambda_1,\Lambda_2$ if it admits a corona decomposition by 
	\[\mathscr{Y}  = \{(\R^n,\|\cdot\|) : \|\cdot\| \mbox{ is a norm on } \R^n\} \]
	with constant $\Lambda_1,\Lambda_2$. We say $X$ admits a \textit{Corona Decomposition by Ahlfors $(C,n)$-regular $(\ve,n)$-Reifenberg flat spaces} with constant $\Lambda_1,\Lambda_2$ if $X$ admits a corona decomposition by 
	\[\mathscr{Y}(C,\ve) = \{ Y : Y \mbox{ is Ahlfors } (C,n)\mbox{-regular and } (\ve,n)\mbox{-Reifenberg flat}\}\]
	with constant $\Corona_1,\Corona_2.$ We say $X$ admits a \textit{Corona Decomposition by Ahlfors regular Reifenberg flat spaces} with constant $\Corona_1,\Corona_2$ if there exists $C \geq 1$ such that $X$ admits a corona decomposition by Ahlfors $(C,n)$-regular $(\ve,n)$-Reifenberg flat spaces with constant $\Corona_1,\Corona_2$ for all $\ve > 0.$ 
\end{defn}

\begin{rem}\label{r:usual-corona}
	Let us describe the relationship between the usual corona decomposition and the corona decomposition by normed spaces. Items (1) and (2) are lifted verbatim from the definition of corona decomposition in \cite{david1991singular}, with (2) quantifying the control on the number of bad cubes and stopping-times regions, respectively. For subsets of $\R^n$, item (3) is replaced by the following condition: 
	\begin{enumerate}[label={(\arabic*$'$})]
		\setcounter{enumi}{2}
		\item For each $S \in \calF$, there exists a Lipschitz graph $\Gamma_S \subseteq \R^n$ with constant at most $\eta > 0$ such that if $Q \in S$ and $x \in 3B_Q$ then $\dist(x,\Gamma_S) \leq \theta \ell(Q).$
	\end{enumerate}
	Condition (3) is an intrinsic condition, but it implies an extrinsic condition in the spirit of ($3'$) in the following way. Let $\mathcal{N}$ be a maximal net in $3B_{Q(S)}$ such that if $Q \in S$ and $x,y \in \mathcal{N} \cap Q$ then $d(x,y) \geq \theta \ell(Q).$ Item (3) tells us that $\mathcal{N}$ is bi-Lipschitz equivalent to a subset $A \subseteq B_{\|\cdot\|_S}(0,3\ell(Q(S))).$ Applying Lemma \ref{l:extension}, we can find a Banach space $\calB$, an isometric embedding on $X$ into $\calB$ and a surface $\Sigma_S \subseteq \calB$ which is bi-Lipschitz equivalent to $\R^n$ such that, after identifying $X$ with its image in $\calB$, we have $\mathcal{N} \subseteq \Sigma_S.$ By maximality, if $z \in Q \in S$ then $\dist_\calB(z,\Sigma_S) \leq \dist_X(z,\mathcal{N}) \leq \theta \ell(Q).$ In particular, (3) implies a version of ($3'$) with bi-Lipschitz images of $\R^n$ (with possibly large constant) replacing the role of Lipschitz graphs with small constants. 
	\end{rem}

	\begin{rem}\label{r:bad}
		Suppose $X$ is an Ahlfors $n$-regular metric space, $\calD$ is a system of Christ-David cubes for $X$ and suppose $X$ admits a corona decomposition by $\mathscr{Y}$. It is easy to obtain from this a local corona decomposition by $\mathscr{Y}$ i.e. for each $Q \in \calD$ a decomposition of $\calD(Q)$ satisfying Definition \ref{d:corona-Y} (1) - (3) with uniform estimates. Suppose $Q \in \calD$ and that we have such a decomposition of $\calD(Q)$ into bad cubes $\calB$ and stopping-time regions $\{S\}_{S \in \calF'}.$ Viewing each cube in $\calB$ as a stopping-time region consisting of a single cube, we get a decomposition of $\calD(Q)$ into stopping-time regions $\{S\}_{S \in \calF}$ whose top cubes satisfy a Carleson packing condition. Furthermore, each $S \in \calF$ which is not a singleton satisfies Definition \ref{d:corona-Y} (3). This reformulated decomposition will be used in Section \ref{s:corona-UR}. 
	\end{rem}

		\begin{rem}
		In Euclidean space, the fact that existence of a corona decomposition implies UR is a crucial step in showing BWGL implies UR. The same is true for us here. The proof of this fact in metric space is very similar to the argument given in \cite{david1991singular} and the details are given in Section \ref{s:corona-UR}. 
	\end{rem}

\bigskip

\subsection{UR+RF implies CD(normed spaces)}
%

Recalling Proposition \ref{p:reif-UR} and Corollary \ref{c:Reif-lower-reg} above, Theorem \ref{t:corona-reif} is an easy consequence of the following proposition. 

\begin{prop}\label{p:UR+BP-corona}
	Let $X$ be uniformly $n$-rectifiable and $\calD$ a system of Christ-David cubes for $X$. Suppose for each $\ve > 0$ and each $Q \in \calD$ with $\bilat_X(Q) \leq \ve$ that there exists a norm $\|\cdot\|$ and a map $\varphi \colon 3B_Q \to B_{\|\cdot\|}(0,3\ell(Q))$ satisfying $\bilat(Q,\|\cdot\|,\varphi) \leq 2\ve$ and $\calH^n_{\|\cdot\|}(\varphi(3B_Q)) \geq c\ell(Q)^n.$ Then $X$ admits a corona decomposition by normed spaces with constants $\Corona_1, \Corona_2$ depending only on $n$. The Carleson norms in Definition \ref{d:corona-Y} depend additionally2 on $n$ and the constants in Definition \ref{d:UR-intro-metric}. 
\end{prop}

	Let us say a few words about the proof of Proposition \ref{p:UR+BP-corona}. Let $Q_0 \in \calD$ and let $\vp$ be its associate map. By modifying $\vp$ slightly we may assume it is $C(n)$-Lipschitz. Since $X$ is UR it satisfies a Carleson condition for the $\gamma$-coefficients so $\vp$ is approximated by affine functions at most scales and locations. Let's assume for simplicity that $\vp$ is approximated by affine functions at \textit{all} scales and locations (additional considerations are needed to deal with the places where $\vp$ is not approximated by affine functions). We run a stopping-time procedure, stopping at sub-cubes $Q \subseteq Q_0$ if there exists a pair of $\theta\ell(Q)$-separated points which are squished under $\vp.$ In this stopping-time region the upper bound in Definition \ref{d:corona-Y} holds because we assume $\vp$ to be Lipschitz and the lower bound holds because of the stopping condition. We get our collection of stopping-time regions by restarting on the stopped cubes (with new $\vp$ maps). We also need to control how often we stop. Suppose $Q$ is a stopped cube in the region associated to $Q_0$. Since there are two points in $Q$ which are squished under $\vp$ and since $\vp$ is approximated by affine functions, this forces the squishing of the whole of $Q$ in one direction. In particular, $\vp(Q)$ has small measure in the image. Since we are assuming large projection for $Q_0,$ this prevents there from being to many stopped cubes. A similar argument over all stopping-time regions gives the desired Carleson packing condition.

\begin{rem}\label{r:local-corona}
	We will prove a local decomposition, namely, for each $Q_0 \in \calD$ we will find a decomposition of the sub-cubes $\calD(Q_0) = \calG \cup \calB$ satisfying conditions (1) - (3) with estimates independent of $Q_0$ (recall the definition of $\calD(Q_0)$ from \eqref{d:subcubes}). The fact that such a local decomposition yields a global decomposition is already known and relies only on the properties of Christ-David cubes, see \cite[Page 38]{david1991singular}.
\end{rem}

Now we can begin. For the remainder of this section, fix a metric space $(X,d)$ satisfying the hypotheses of Proposition \ref{p:UR+BP-corona} with $\calD$, a system of Christ-David cubes on $X$, and $Q_0 \in \calD$. Let $\theta > 0$ (the constant appearing in Definition \ref{d:corona-Y}) and let $\ve > 0$, which will be chosen small enough depending on $\theta.$ 

Recall the definitions of $\bilat_X(AQ)$ and $\gamma_{X,f,\|\cdot\|}^K(AQ)$ from Definition \ref{d:notation-cubes}. We start by setting 
\begin{align}\label{e:G,B}
	\mathcal{G} = \{Q \in \calD(Q_0) : \bilat_X(3Q) \leq \ve\} \quad \mbox{ and } \quad \calB = \calD(Q_0) \setminus \calG.
\end{align}

\begin{rem}\label{r:packing}
In the previous section we showed UR implies BWGL (recall Lemma  \ref{l:WALAM-BWGL} and Theorem \ref{t:approx-UR}). Thus, by applying Lemma \ref{l:cubes-BWGL} we know that $\calB$ satisfies a Carleson packing condition as required for part of Definition \ref{d:corona-Y} (2).
\end{rem}

Our next step is to define the family of stopping-time regions which partition $\calG$. By Lemma \ref{l:Lipschitz-GHA}, for each $Q \in \calG$ there exists a norm $\|\cdot\|_Q$ and a $C(n)$-Lipschitz map $\pi_Q \colon 3B_Q \to B_{\|\cdot\|_Q}(0,3\ell(Q))$ such that 
\[\bilat_X(3Q,\|\cdot\|_Q,\pi_Q) \leq \bilat_X(3Q) + \ve \leq 2\ve.\]
Let $K \geq 1$ be the constant obtained from Theorem \ref{t:approx-UR} for $X.$ Then set 
\begin{align} \calG'(Q,\ve) &= \{ R \in \calD(Q) : \gamma_{X,\pi_Q,\|\cdot\|_Q}^K(3R) \leq \ve \mbox{ for all } R' \in \text{Sibling}(R)\}; \\
\calB'(Q,\ve) &= \calD(Q) \setminus \calG(Q,\ve).
\end{align}

\begin{defn}\label{d:stop-time}
	Let $\Corona \geq 1,$ to be chosen sufficiently large. For each $Q \in \calG$ and $R \in \calG$ such that $R \subseteq Q,$ define a stopping-time region $S_Q(R)$ as follows. First, add $R$ to $S_Q(R).$ Then, inductively, add cubes $T \subset R$ to $S_Q(R)$ if each of the following conditions hold. 
	\begin{enumerate}
		\item $\text{Parent}(T) \in S_Q(R)$.
		\item $T \in \calG'(Q,\ve).$
		\item For each $T' \in \text{Sibling}(T)$ we have $\| \pi_Q(x) - \pi_Q(y) \|_Q \geq d(x,y)/\Corona$ for all $x,y \in 3B_{T'}$ such that $d(x,y) \geq \theta \ell(T').$ 
	\end{enumerate} 
\end{defn}
If $T \in \min(S_Q(R))$ then there exists $T' \in \text{Child}(T)$ such that either $T' \in \calB'(Q,\ve)$ or $\| \pi_Q(x) - \pi_Q(y) \|_Q < d(x,y)/\Corona$ for some $x,y \in 3B_{T'}$ with $d(x,y) \geq \theta \ell(T').$ We partition $\text{min}(S_Q(R))$ by setting 
\begin{align}
	{\min}_1(S_Q(R)) &= \{ T \in \min(S_Q(R)) : \mbox{ there exists } T' \in \text{Child}(T) \cap \calB'(Q,\ve)\}; \\
	{\min}_2(S_Q(R)) &= \min(S_Q(R)) \setminus {\min}_1(S_Q(R)). 
\end{align}

\begin{rem}
In words, if $T \in \text{min}_1(S_Q(R))$ then there exists a child $T'$ of $T$ such that $\pi_Q$ is not well-approximated by affine functions in $T'.$ If $T \in \text{min}_2(S_Q(R))$ then $\pi_Q$ is well-approximated by affine functions on each of its children, however there exists a child $T'$ such that $\pi_Q$ contracts a pair of well-separated points in $T'.$ We will use this property later to show $\calH^n(\pi_Q(T))$ is small for such a $T$. 
\end{rem}

Before defining the family of stopping-time regions $\calF,$ we first define a family $\calF_Q$ associated to $Q \in \calG$ and prove some properties of it. Let $\tau > 0$. 

\begin{lem}
	There exists a constant $N = N(C_0,n,\ve,\tau)$ such that 
	\begin{align}\label{e:WALA-proof}
		\calH^n\left(\left\{ x \in Q : \sum_{R \in \calB'(Q,\ve)} \mathds{1}_R(x) \geq N\right\}\right) \leq \tau \ell(Q)^n \quad \mbox{ for all } Q \in \calG.
	\end{align}
\end{lem}

\begin{proof}
	Let $Q \in \calG.$ If $R \in \calB'(Q,\ve)$ then there exists $T \in \text{Sibling}(R)$ such that $\gamma_{X,\pi_Q,\|\cdot\|_Q}^K(T) > \ve.$ By Ahlfors regularity each cube in $\calD$ has a bounded number of siblings, depending on the regularity constant and $n$. Then, by Theorem \ref{t:approx-UR}, Ahlfors regularity of $X$ and Chebyshev's Inequality, the left-hand side of \eqref{e:WALA-proof} is at most 
	\begin{align}
		\frac{1}{N} \sum_{R \in \calB'(Q,\ve)} \ell(R)^n \lesssim_{C_0,n} \frac{1}{N}  \sum_{\substack{T \subseteq Q \\ \gamma_{X,\pi_Q,\|\cdot\|_Q}^K(T) > \ve}} \ell(R)^n \lesssim_{C_0,\ve,n} \frac{\ell(Q)^n}{N}. 
	\end{align}
	Choosing $N$ large enough gives the required result. 
\end{proof}

From now on we fix $N$ as in the above lemma. For $Q \in \calG,$ let $\text{Top}_Q^0 = \{Q\}$ and $\calF_Q^0 = \{S_Q(Q)\}$. Then, supposing $\text{Top}_Q^k$ and $\calF_Q^k$ have been defined for some integer $0 \leq k \leq N-2,$ let 
\begin{align}\label{e:F^k}
	\text{Top}_Q^{k+1} &= \bigcup_{R \in \text{Top}_Q^k} \bigcup_{T \in \text{min}_1(S_Q(R))} \text{Child}(T); \\
	\calF_Q^{k+1} &= \{S_Q(R) : R \in \text{Top}_Q^{k+1}\}.
\end{align} 
Now, let 
\begin{align}
	\calF_Q &\coloneqq \bigcup_{k=1}^{N-1} \calF_Q^k
\end{align}
and
\begin{align}
	\text{min}(\calF_Q) &= \bigcup_{S_Q(R) \in \calF_Q} {\min}_2(S_Q(R)) \cup \bigcup_{S_Q(R) \in \calF_Q^{N-1}} {\min}_1(S_Q(R)); \\
\end{align}

\begin{lem}\label{l:dis-FQ}
	Let $Q \in \calG.$ Then $\calF_Q$ is a disjoint collection of stopping-time regions. 
\end{lem}

\begin{proof}
	Fix distinct stopping-time regions $S,S' \in \calF_Q.$ Let $R,R' \subseteq Q$ such that $S = S_Q(R)$ and $S' = S_Q(R').$ Since $S$ and $S'$ are distinct, we must have $R \neq R'.$ If $R \cap R' \neq \emptyset$ there is nothing to show. Suppose then that $R \cap R' \neq \emptyset$ and suppose without loss of generality that $R \subseteq R'.$ By construction we have $R \subseteq R''$ for some $R'' \in \min(S')$ and so $S,S'$ are disjoint
\end{proof}

We will show that the image of $\min(\calF_Q)$ under $\pi_Q$ has small $\calH^n$-measure, starting with the cubes in $\min_1.$ 

\begin{lem}
	Let $Q \in \calG.$ Then, 
	\begin{align}\label{e:min_1}
		\calH^n\left(\pi_Q\left(\bigcup_{S \in \calF_Q^{N-1}}\bigcup_{R \in\emph{min}_1(S)} R\right)\right) \lesssim_{C_0,n} \tau \ell(Q)^n. 
	\end{align}
\end{lem}

\begin{proof}
	Suppose $S \in \calF_Q^{N-1}$ and $R \in \text{min}_1(S).$ By construction there exists a $T_R \in \text{Child}(R)$ which is contained in exactly $N$ cubes from $\calB'(Q,\ve).$ Recalling that $\pi_Q$ is $C(n)$-Lipschitz and using \eqref{e:WALA-proof} with Ahlfors regularity and the fact that $\min(\calF_Q)$ is a collection of disjoint cubes, the left-hand side of \eqref{e:min_1} is at most
	\begin{align}
		\sum_{S \in \calF_Q^{N-1}} \sum_{R \in \text{min}_1(S)} \calH^n(\pi_Q(R)) &\lesssim_n \sum_{S \in \calF_Q^{N-1}} \sum_{R \in \text{min}_1(S)} \calH^n(R) \\
		&\lesssim_{C_0,n} \sum_{S \in \calF_Q^{N-1}} \sum_{R \in \text{min}_1(S)} \calH^n(T_R) \\		
		&\leq \calH^n\left(\left\{ x \in Q : \sum_{R \in \calB'(Q,\ve)} \mathds{1}_R(x) \geq N\right\}\right) \\
		&\leq \tau \ell(Q). 
	\end{align}
\end{proof}

To show that the image of cubes in $\min_2$ have small measure we need the following preliminary lemma. 

\begin{lem}\label{l:in-box}
	Let $\ve, \eta > 0.$ Suppose $\|\cdot\|_1$ and $\|\cdot\|_2$ are norms on $\R^n$ and $A : (\R^n,\|\cdot\|_1) \to (\R^n,\|\cdot\|_2)$ is an $L$-Lipschitz affine map such that $\| A(x) - A(y) \|_2 \leq \eta \| x - y \|_1$ for some $x,y \in \R^n$. Then $\calH^n(\mathcal{N}_{\ve r_B}(A(B))) \lesssim_{L,n} (\ve + \eta) r_B^n$ for all balls $B \subseteq (\R^n,\|\cdot\|_1)$, where $\mathcal{N}_\delta(F)$ denotes the $\delta$-neighbourhood of $F.$ 
\end{lem}

\begin{proof}
	By translating and rotating we may also assume $B$ is centred at the origin and $A$ is linear.	It suffices to prove the case in which both the domain and image are equipped with the Euclidean norm. Indeed, suppose we have Lemma \ref{l:in-box} for this case and suppose $\|\cdot\|_1$ and $\|\cdot\|_2$ are norms and $A$ is a map satisfying the hypotheses above with constant $\eta$ and $L$. By Remark \ref{r:BM} there exist $C(n)$-bi-Lipschitz linear maps $T_i : (\R^n,\|\cdot\|_i) \to (\R^n,|\cdot|)$, $i =1,2.$ It is immediate to check that $A' = T_2 \circ A \circ T_1^{-1} \colon \R^n \to \R^n$ satisfies the hypotheses of Lemma \ref{l:in-box} with constant $\eta'  \sim \eta$ and $L' \sim_n L.$ If $B \subseteq (\R^n,\|\cdot\|_1)$ is a ball then $T_1(B)$ is contained in a Euclidean ball $B'$ centred at the origin satisfying $r_B \sim r_B.$ Hence
	\begin{align}
		\calH^n(\mathcal{N}_{\ve r_B}(A(B))) &\leq \calH^n(\mathcal{N}_{\ve r_B}(A(T_1^{-1}(B')))) \lesssim_n \calH^n(\mathcal{N}_{C\ve r_{B'}}(A'(B'))) \\
		&\lesssim_{n,L'}(\ve + \eta') r_{B'}^n \lesssim_{n,L} (\ve + \eta) r_B^n
	\end{align}
	
	We proceed now with the proof of Lemma \ref{l:in-box} in this special case. It is enough to show that $\mathcal{N}_{\ve r_B}(A(B))$ is contained in a cuboid of dimension
		\begin{align}\label{e:dimensions}
			 \underbrace{r_B \times \cdots \times r_B}_\text{$(n-1)$-times}  \times (\ve + \eta)r_B
		\end{align}
		up to some constant depending on $L$. So, let $u = x-y$ and set $V$ to be the $(n-1)$-dimensional subspace which is perpendicular to $u$. Let $W = A(V)$ (also an $(n-1)$-dimensional subspace) and denote by $\pi_W$ the orthogonal projection onto $W$ and $\pi_{W^\perp}$ the orthogonal projection onto $W^\perp.$ Let $p \in \mathcal{N}_{\ve r_B}(A(B))$ and suppose $z \in B$ is such that $| A(z) - p | \leq \ve r_B.$ Write $z = \lambda u + \mu v$ for some $v \in V$. Since $A$ is $L$-Lipschitz we have  
		\begin{align}
			|\pi_W(p)| \leq |\pi_W(A(z)) | +\ve r_B \leq | A(z) | +\ve r_B \leq L | z | + \ve r_B \leq 2L r_B. 
		\end{align}
		Since $A$ is linear and $A(\mu v) \in W$ we know $\pi_{W^\perp}(A(z)) = \pi_{W^\perp}(A(\lambda u)).$ Since $\lambda u$ is the orthogonal projection of $z$ onto the line through $u$ we know $| \lambda u | \leq |z| \leq r_B.$ Combining these facts gives 
		\begin{align}
			| \pi_{W^\perp}(p)| \leq 
			| \pi_{W^\perp}(A(z)) | +\ve r_B = | \pi_{W^\perp}(A(\lambda u)) | +\ve r_B \leq \eta | \lambda u | +\ve r_B \leq (\eta + \ve)r_B.
		\end{align}
		Since $p$ was arbitrary, it follows that $\mathcal{N}_{\ve r_B}(A(B))$ is contained in a cuboid with the same dimensions as in \eqref{e:dimensions} up to some constant depending only on $L$.  
\end{proof}

\begin{lem}
	Let $Q \in \calG.$ Then, 
	\begin{align}\label{e:small-M}
		\calH^n\left(\pi_Q\left(\bigcup_{S \in \calF_Q}\bigcup_{R \in\emph{min}_2(S)} R\right)\right) \lesssim \frac{1}{\Corona} \ell(Q)^n. 
	\end{align}
\end{lem}

\begin{proof}
	Suppose $S \in \calF_Q$ and $R \in \min_2(S)$. 
	We claim 
	\begin{align}\label{e:small-proj}
		\calH^n(\pi_Q(R)) \lesssim \frac{\ell(R)^n}{\Corona}. 
	\end{align}
	Recalling the definition of $\text{min}_2(S),$ there exists $T \in \text{Child}(R)  \cap \calG'(Q,\ve)$ and $x_0,y_0 \in 3B_{T}$ such that $d(x_0,y_0) \geq \theta \ell(T)$ and 
	\begin{align}\label{e:contract}
		\| \pi_Q(x_0) - \pi_Q(y_0) \|_Q < \frac{d(x_0,y_0)}{\Corona}.
	\end{align} 
	By definition of $\calG'(Q,\ve),$ there exists a norm $\|\cdot\|$, a map $\varphi \colon 3B_T \to B_{\|\cdot\|}(0,3\ell(T))$ and a $K$-Lipschitz affine function $A \colon (\R^n,\|\cdot\|) \to (\R^n,\|\cdot\|_Q)$ such that 
	\begin{align}\label{e:flat}
		\sup_{x,y \in 3B_T} \abs{ d(x,y) - \| \varphi(x) - \varphi(y) \| } \leq 6\ve \ell(T) 
	\end{align}
and
	\begin{align}\label{e:affine}
		\sup_{x \in 3B_T} \| \pi_Q(x) - A(\varphi(x)) \|_Q \leq 6\ve \ell(T). 
	\end{align}
	Let $p_0 = \varphi(x_0)$ and $q_0 = \varphi(y_0).$ For $\ve$ small enough, depending on $\theta,$ we have 
	\begin{align}
		\|p_0-q_0\| \geq d(x_0,y_0) - 6\ve \ell(T) \geq \theta \ell(T) - 6\ve \ell(T) \geq \frac{\theta \ell(T)}{2}. 
	\end{align}
	Using this, with \eqref{e:contract}, \eqref{e:flat} and \eqref{e:affine}, and choosing $\ve$ small enough, depending on $\theta$ and $\Corona$, we have 
	\begin{align}
		\begin{split}\label{e:contractions}
		\|A(p_0) - A(q_0) \|_Q &\leq \|A(p_0) - \pi_Q(x_0) \|_Q + \|\pi_Q(x_0) - \pi_Q(y_0) \|_Q  \\
		&\hspace{2em} + \|\pi_Q(y_0) - A(q_0)\|_Q \\
		&\leq 12\ve \ell(T) + \frac{1}{\Corona}d(x_0,y_0) \\
		&\leq \frac{24 \ve}{\theta} \| p_0 -q_0 \| + \frac{1}{\Corona} \left(\|p_0-q_0\| + 6\ve \ell(T)\right) \\
		&\leq \left( \frac{36 \ve}{\theta}  + \frac{1}{\Corona} \right) \|p_0-q_0\| \lesssim \frac{1}{\Corona} \|p_0-q_0\|. 
		\end{split}
	\end{align}
	Since $T$ is a child of $R$ we know $\ell(T) \sim \ell(R)$. Then, by \eqref{e:affine} we have $\pi_Q(R) \subseteq \mathcal{N}_{C\ve \ell(R)}(A(B))$ for some ball $B \subseteq (\R^n,\|\cdot\|)$ with radius comparable to $\ell(R).$ This and \eqref{e:contractions} now imply \eqref{e:small-proj} by Lemma \ref{l:in-box}. 
	Using \eqref{e:small-proj}, Ahlfors regularity of $X$, and the fact that $\{R : R \in S, \ S \in \calF_Q\}$ is a disjoint collection of cubes in $Q$, the left-hand side of \eqref{e:small-M} is at most 
	\begin{align}
		\sum_{S \in \calF_Q} \sum_{R \in \text{min}_2(S)} \calH^n(\pi_Q(R)) \lesssim \frac{1}{\Corona} 	\sum_{S \in \calF_Q} \sum_{R \in \text{min}_2(S)} \calH^n(R) \leq \frac{\calH^n(Q)}{\Corona}. 
	\end{align}
\end{proof}

Recalling the definition of $\text{min}(\calF_Q),$ the previous two lemmas combine to give the following. 

\begin{cor}\label{c:small-proj}
	Let $Q \in \emph{Top}.$ Then, 
	\begin{align}
		\calH^n\left(\pi_Q\left( \bigcup_{R \in \min(\calF_Q)} R\right)\right) \lesssim \left( \delta + \frac{1}{\Corona} \right) \ell(Q)^n. 
	\end{align}
\end{cor}

We can now define the family $\calF.$ Let $\text{Top}_0$ be the set of maximal cubes in $\calG$ and 
\[\calF_0 = \bigcup_{Q \in \text{Top}_0} \calF_Q .\]
Then, supposing $\text{Top}_k$ and $\calF_k$ have been defined for some integer $k \geq 0,$ let $\text{Top}_{k+1}$ be the set of maximal cubes $Q \in \calG(Q_0)$ such that there exists $R \in \text{Top}_k$, $T \in \min(\calF_R)$ and $T' \in \text{Child}(T)$ with $Q \subseteq T'.$ Let 
\[ \calF_{k+1} = \bigcup_{Q \in \text{Top}_{k+1}} \calF_Q = \bigcup_{Q \in \text{Top}} \bigcup_{i=1}^{N-1} \calF_Q^i.\]
Finally, let 
\begin{align}\label{e:F}
	\calF = \bigcup_{k=0}^\infty \calF_k \quad \mbox{ and } \quad \text{Top} = \bigcup_{k=0}^\infty \text{Top}_k.
\end{align}

\begin{lem}\label{l:dis-F}
	The collection $\{S\}_{S \in \calF}$ forms a disjoint partition of $\calG.$
\end{lem}

\begin{proof}
	We begin by checking $\{S\}_{S \in \calF}$ is disjoint. Suppose $S,S' \in \calF$ are distinct. Let $Q,Q' \in \text{Top}$, $R \in \text{Top}_Q$ and $R' \in \text{Top}_{Q'}$ such that $S = S_Q(R)$ and $S' = S_{Q'}(R')$. If $Q \cap Q' = \emptyset$ then there is nothing to show. Suppose instead that $Q \cap Q' \neq\emptyset$ and suppose without loss of generality that $Q \subseteq Q'.$ By maximality, $\text{Top}_m$ is a disjoint collection of cubes for each $m \geq 0.$ Thus $Q \in \text{Top}_k$ and $Q' \in \text{Top}_\ell$ for some $\ell \geq k.$ By construction, if $\ell > k$ then $Q \subset T$ for some $T \in \min(S_{Q'}(R'))$. It follows that $S$ is disjoint from $S'$. If $k = \ell$ then $Q = Q'$ since $\text{Top}_k$ is disjoint. In particular, $S,S' \in \calF_Q.$ Lemma \ref{l:dis-FQ} now implies $S$ is disjoint from $S'$.  
	
	Let us now check that 
	\begin{align}\label{e:G=} 
		\calG = \bigcup_{S \in \calF} S.
	\end{align}
	The fact that $\bigcup_{S \in \calF} S \subseteq \calG$ is clear from Definition \ref{d:stop-time} (2). For the forward inclusion, fix some $T \in \calG.$  Let $k \geq 0$ be the largest integer for which there exists a cube $Q \in \text{Top}_k$ satisfying $T \subseteq Q$. Let $m \geq 0$ be the largest integer for which there exists $R \in \text{Top}_Q^m$ satisfying $T \subseteq R.$ By maximality there exists a cube in $\text{Top}_0$ which contains $T$ and so $k$ and $Q$ are well-defined. Since $T \subseteq Q$ and $\text{Top}_Q^0 = \{Q\}$ we have that $m$ and $R$ are also well-defined. We will show
	\begin{align}\label{e:in-S}
		 T \in S_Q(R),
	\end{align}
	which is enough for \eqref{e:G=}. First, if $T \cap R' = \emptyset$ for all $R' \in \min(S_Q(R))$ then \eqref{e:in-S} is true by Definition \ref{StoppingTime} (2) (the child of $R'$ must be in $S_Q(R)$). Suppose then that there exists $R' \in \min(S_Q(R))$ such that $T \cap R' \neq \emptyset$. Let $R'' \in \text{Child}(R')$ such that $T \cap R'' \neq\emptyset.$ By Definition \ref{StoppingTime}, it is enough to show $T \supseteq R'$ or equivalently $T \supset R''.$ To see this final inclusion, we proceed as follows. Suppose first that $m \leq N-2.$ Then $R'' \in \text{Top}_{Q}^{m+1}$ and $T \supset R''$ by maximality of $m$. Suppose instead that $m = N-1.$ Since $T \in \calG(Q_0),$ if $T \subseteq R''$ then by maximality we would have $T \in \text{Top}_{k+1}.$ This contradicts that maximality of $k$ and we conclude that $T \supset R''.$

\end{proof}

	\begin{lem}\label{l:z}
	For $Q \in \emph{Top}$ define
	\begin{align}
		z(Q) = Q \setminus \bigcup_{R \in \min(\calF_Q)} R. 
	\end{align}
		If $Q,Q' \in  \emph{Top}$ are distinct then $z(Q) \cap z(Q') = \emptyset.$  
	\end{lem}

	\begin{proof}
		Since $z(Q) \subseteq Q$ and $z(Q') \subseteq Q',$ if $Q \cap Q' = \emptyset$ then the result follows immediately. Suppose instead that $Q \cap Q' \neq\emptyset$ and suppose towards a contradiction that there exists $x \in z(Q) \cap z(Q')$. By definition, we can find $S \in \calF_Q$ and $S' \in \calF_{Q'}$ such that $x$ is contained in arbitrarily small cubes from each stopping-time region. Let $k \geq 0$ such that $5\varrho^k \leq \min\{\ell(Q(S)),\ell(Q(S'))\}$ and choose $R \in S \cap \calD_k$ and $R' \in S' \cap \calD_k$ such that $x \in R \cap R'.$ Since cubes in $\calD_k$ are disjoint it follows that $R = R',$ however, this contradict Lemma \ref{l:dis-F}.
	\end{proof}

\begin{lem}\label{l:top-control}
	If $\Corona$ is chosen sufficiently large then for each $R \in \calD(Q_0)$ we have 
	\begin{align}
		\sum_{\substack{S \in \calF \\ Q(S) \subseteq R} }\ell(Q(S))^n \lesssim \ell(R)^n. 
	\end{align}
\end{lem}

\begin{proof}
	Fix some $R \in \calD(Q_0).$ Using the Ahlfors regularity of $X$ and the definition of $\calF,$ we have 
	\begin{align}
		\sum_{\substack{S \in \calF \\ Q(S) \subseteq R}} \ell(Q(S))^n =  \sum_{\substack{Q \in \text{Top} \\ Q \subseteq R}} \sum_{i=0}^{N-1} \sum_{S \in \calF_Q^i} \ell(Q(S))^n \lesssim \sum_{\substack{Q \in \text{Top} \\ Q \subseteq R}} \sum_{i=0}^{N-1} \sum_{S \in \calF_Q^i} \calH^n(Q(S)).
	\end{align}
	By the construction in \eqref{e:F^k}, for each fixed $Q \in \text{Top}$ and $i \in \{0,\dots,N\}$ we know $\{Q(S) : S \in \calF_Q^i\}$ is a disjoint collection of cubes contained in $Q.$ Using this fact and the Ahlfors regularity of $X$ again, we get  
	\[\sum_{\substack{S \in \calF \\ Q(S) \subseteq R}} \ell(Q(S))^n \lesssim_N \sum_{\substack{Q \in \text{Top} \\ Q \subseteq R}} \calH^n(Q) \lesssim \sum_{\substack{Q \in \text{Top} \\ Q \subseteq R}} \ell(Q)^n. \]
	Thus, it suffices to show 
	\begin{align}
		\sum_{\substack{Q \in \text{Top} \\ Q \subseteq R}} \ell(Q)^n \lesssim \ell(R)^n. 
	\end{align}
	Suppose $Q \in \text{Top}$ is such that $Q \subseteq R.$ Applying Corollary \ref{c:small-proj} and choosing $\delta$ small enough and $\Corona$ large enough, we have 
	\begin{align}
		\calH^n(z(Q)) &\gtrsim_n \calH^n(\pi_Q(z(Q))) \geq \calH^n(\pi_Q(Q)) - \calH^n\left(\pi_Q\left( \bigcup_{T \in \min(\calF_Q)} T\right)\right) \\
		&\geq \left( \frac{1}{2} - C\left(\delta + \frac{1}{\Corona}\right)\right) \ell(Q)^n \gtrsim \ell(Q)^n.  
	\end{align}
	By Lemma \ref{l:z}, $\{z(Q)\}_{Q \in \text{Top}}$ forms a disjoint collection of subsets of $R.$ Hence, 
	\begin{align}
		\sum_{Q \in \text{Top}} \ell(Q)^n \lesssim \sum_{Q \in \text{Top}} \calH^n(z(Q)) \leq \calH^n(R) \lesssim \ell(R)^n. 
	\end{align}
\end{proof}

\begin{proof}[Proof of Proposition \ref{p:UR+BP-corona}]
	By Remark \ref{r:local-corona} is suffices to prove a decomposition of the cubes in $\calD(Q_0).$ We decompose $\calD(Q_0) = \calG \cup \calB$ with $\calG,\calB$ as in \eqref{e:G,B}. Defining the family of stopping-time regions $\calF$ as in \eqref{e:F}, Definition \ref{d:corona-Y} (1) follows from Lemma \ref{l:dis-F}.  Definition \ref{d:corona-Y} (2) follows directly from Remark \ref{r:packing} and Lemma \ref{l:top-control}. Finally, observe by construction that if $S \in \calF$ then $S = S_Q(R)$ for some $Q \in \calG$ and $R \subseteq Q.$ If $T \in S$ then the lower bound in Definition \ref{d:corona-Y} (3) follows from Definition \ref{d:stop-time} with constant $\Corona_1 = \Corona.$ The upper bound follows with constant $\Corona_2$, depending only on $n$, since we assume $\pi_{Q(S)}$ to be $C(n)$-Lipschitz. This finishes the proof. 
\end{proof}

\newpage

\section{BWGL implies $\text{CD}(\text{RF metric})$ and $\text{CD(CD}$(normed spaces))}\label{s:final}

In this section we prove the following.

\begin{thm}\label{t:final}
	Let $\ve > 0$ and $\Corona_1,\Corona_2 > 1.$ Let $X$ be an Ahlfors $n$-regular metric space. For $\delta > 0$ sufficiently small, if $X$ satisfies the \emph{BWGL($\delta$)} then $X$ admits a corona decomposition by Ahlfors regular $(\ve,n)$-Reifenberg flat metric spaces with constants $\Corona_1,\Corona_2.$
\end{thm}

Suppose Theorem \ref{t:final} for the minute. Let $\ve > 0$ be sufficiently small for Theorem \ref{t:corona-reif} and let $\delta > 0$ be sufficiently small such that Theorem \ref{t:final} holds for this choice of $\ve.$ Recalling Definition \ref{d:CDV}, Theorem \ref{t:final} implies that $X$ admits a $\text{CD}(\mathscr{Y})$ with $\mathscr{Y}$ the family of all Ahlfors $(C,n)$-regular and $(\ve,n)$-Reifenberg flat metric spaces. This and Theorem \ref{t:corona-reif} give the following corollary. 

\begin{cor}\label{c:BWGL implies CD^2}
	Let $X$ be an Ahlfors $n$-regular metric space. For $\delta > 0$ sufficiently small, if $X$ satisfies the \emph{BWGL}$(\delta)$, then $X$ admits a $\emph{CD}(\mathscr{Y})$ for a family $\mathscr{Y}$ of uniformly Ahlfors $n$-regular metric spaces each admitting corona decomposition by normed spaces with uniform constant. 
\end{cor}

Let's return to the proof of Theorem \ref{t:final}. Fix $\Corona_1,\Corona_2 > 1$ and let $\ve > 0$. We will find a constant $C$, independent of $\ve,$ and show that $X$ admits a corona decomposition  by the family of Ahlfors $(C,n)$-regular $(\ve,n)$-Reifenberg flat metric spaces with constants $\Corona_1,\Corona_2$. 

Let $Q_0 \in \calD.$ By Remark \ref{r:local-corona} it suffices to construct a corona decomposition of the cubes $\calD(Q_0).$ Now, let $\theta > 0$ be the coarse parameter appearing in Definition \ref{d:corona-Y} and we may suppose $\delta > 0$ is a parameter to be chosen sufficiently small depending on $\ve$ and $\theta.$ Note that the Carleson estimates in Definition \ref{d:corona-Y} (2) are allowed to depend on $\delta.$ Let's now define the families of good and bad cubes, and the family of stopping time regions. First, let
\begin{align}
	\calG = \{Q \in \calD(Q_0) :  \bilat_X(100Q) \leq \delta \} \quad \mbox{ and } \quad \calB = \calD(Q_0) \setminus \calG.
\end{align}

\begin{defn}
	For $Q \in \calG$, let $S_Q$ be a stopping-time region defined inductively as follows. First, add $Q$ to $S_Q.$ Then, add cubes $R \subset Q$ to $S_Q$ if $\text{Parent}(R) \in S_Q$ and $\bilat_X(100R') \leq \delta$ for all $R' \in \text{Sibling}(R).$ 
\end{defn}
Let $\text{Top}_0$ be the maximal collection of cubes $T \in \calG$ and $\calF_0 = \{S_Q\}_{Q \in \text{Top}_0}.$ Then, supposing we have define $\calF_k$ and $\text{Top}_k$ up to some integer $k \geq 0,$ let $\text{Top}_{k+1}$ be the maximal collection of cubes $T \in \calG$ such that there exists $Q \in \text{Top}_k$ and $R \in \min(S_Q)$ satisfying $T \subset R$. We also let
\[ \calF_{k+1} = \bigcup_{i=0}^{k+1} \{ S_Q : Q \in \text{Top}_i\}. \]
Finally, define
\[ \calF = \bigcup_{k=0}^\infty \calF_k. \]

One can easily verify that $\{S\}_{S \in \calF}$ defines a collection of disjoint stopping-time regions in much the same way as we did in Lemma \ref{l:dis-FQ}. Similarly, we can check it forms a partition of $\calG$ as in Lemma \ref{l:dis-F}. In particular, we have Definition \ref{d:corona-Y} (1).

\begin{lem}
	Definition \ref{d:corona-Y} (2) holds with Carleson norms depending on $\delta.$
\end{lem} 

\begin{proof}
	The fact that $\calB$ satisfies a Carleson packing condition with Carleson norm depending on $\delta$ follows immediately from our assumption that $X$ satisfies the BWGL and Lemma \ref{l:cubes-BWGL}.  It only remains to check the Carleson packing condition on $\{Q(S)\}_{S \in \calF}$. By construction, if $S \in \calF$ then $Q(S)$ is the child of some $R_S \in \bigcup_{Q \in \text{Top}} \min(S_Q).$ Since children cubes are disjoint, if $R \in \bigcup_{Q \in \text{Top}} \min(S_Q)$ then $\{Q(S) : S \in \calF \mbox{ and } R_S = R\}$ is a disjoint collection of cubes. In particular, by Ahlfors regularity, 
	\begin{align}
		\sum_{\substack{S \in \calF \\ R_S = R}} \ell(Q(S))^n \lesssim \ell(R)^n. 
	\end{align}
	For each $R \in \bigcup_{Q \in \text{Top}} \min(S_Q)$ there exists $T_R \in \text{Child}(R)$ such that $\bilat_X(100T_R) > \delta.$ Since $T_R \in \text{Child}(R)$ we have $\ell(R) \sim \ell(T_R).$ Since every cube has a unique parent, if $T \in \calD$ then 
	\[\left|\left\{R \in \bigcup_{Q \in \text{Top}} \min(S_Q) : T_R = T\right\}\right| \leq 1.\]
	Combining the above and using that $X$ satisfies the BWGL, we get 
	\begin{align}
		\sum_{S \in \calF} \ell(Q(S))^n  &= \sum_{Q \in \text{Top}} \sum_{R \in \min(S_Q)} \sum_{\substack{S \in \calF \\ R_S = R}} \ell(Q(S))^n \lesssim \sum_{Q \in \text{Top}} \sum_{R \in \min(S_Q)} \ell(R)^n\\
		&\lesssim \sum_{Q \in \text{Top}} \sum_{R \in \min(S_Q)} \ell(T_R)^n \leq \sum_{\substack{T \in \calD \\ \bilat_X(100T) > \delta}} \ell(T)^n  \lesssim_\delta \ell(Q_0)^n. 
	\end{align}
\end{proof}

The proof of Theorem \ref{t:final} is concluded with the following lemma. 

\begin{lem}
	If $\delta$ is chosen small enough, depending on $\ve$, $\theta$ and $\Corona_1,\Corona_2$, then the following holds. There exists $C_*$ independent of $\ve$ and for each $S \in \calF$ a metric $\rho_S$ and a map $\vp_S \colon 3B_{Q(S)} \to \R^n$ such that $(\R^n,\rho_S)$ is an Ahlfors $(C_*,n)$-regular $(\ve,n)$-Reifenberg flat metric space and if $Q \in S$ and $x,y \in 3B_Q$ are such that $d_X(x,y) > \theta \ell(Q)$ then 
	\begin{align}\label{e:final-eq}
		\frac{1}{\Corona_1} d_X(x,y) \leq \rho_S(\vp_S(x),\vp_S(y)) \leq \Corona_2d_X(x,y). 
	\end{align}
\end{lem}

\begin{proof}
	Let $S \in \calF.$ By scaling we may assume $Q(S) \in \calD_0$ so that $B_{Q(S)} = B(x_{Q(S)},1)$ and $3\ell(Q(S)) =15.$ We construct $\rho_S$ by applying Theorem \ref{t:Reif}. First, we need to define a collection of points and norms which satisfy the hypotheses. 
	
	For each $\ell \geq 0$, let $k(\ell) \in \Z$ such that $5\varrho^{k(\ell)} \leq r_\ell < 5\varrho^{k(\ell)-1}.$ Then, let $\{x_{j,\ell}\}_{j \in J_{\ell}}$ be a maximal $r_\ell$-separated net in $\{x_Q : Q \in S \cap \calD_{k(\ell)} \}$. Since $S \cap \calD_0 = \{Q(S)\}$ we have Theorem \ref{t:Reif} (1). The separation condition implies Theorem \ref{t:Reif} (2). 
	
	Now we consider Theorem \ref{t:Reif} (3). For $\ell \geq 0$ and $j \in J_\ell$ set $B^{j,\ell} = B(x_{j,\ell},r_\ell)$. Fix $\ell \geq 1$ and $j \in J_\ell.$ Suppose first that $k(\ell) = k(\ell-1).$ By maximality, there exists $i \in J_{\ell-1}$ such that $d(x_{j,\ell},x_{i,\ell-1}) \leq r_{\ell-1},$ which implies $x_{j,\ell} \in B^{i,\ell-1}.$ Suppose instead that $k(\ell) > k(\ell-1).$ There exists $Q \in S \cap \calD_{k(\ell-1)}$ such that $d(x_{j,\ell} , x_Q) \leq \ell(Q) \leq r_{\ell-1}$ and, by maximality, $i \in J_{\ell-1}$ such that $d(x_Q, x_{i,\ell-1}) \leq r_{\ell-1}.$ By the triangle inequality we have $x_{j,\ell} \in 2B^{i,\ell-1}.$ 
	
	Finally, let us check Theorem \ref{t:Reif} (4). Let $\ell \geq 0$, $j \in J_\ell$ and let $Q \in S \cap \calD_{k(\ell)}$ such that $x_{j,\ell} = x_Q.$ Since $100B^{j,\ell} \subseteq 100B_Q$ and $r_\ell \sim \ell(Q)$ we have \begin{align}\label{e:C'}
		\bilat_X(100B^{j,\ell}) \leq C\bilat_X(100Q) \leq C\delta.
	\end{align}
	Choose a parameter $\gamma > 0$, depending on $\theta$ and $\Corona_1,\Corona_2,$ such that 
	\begin{align}\label{e:gamma}
		\Corona_1^{-1} \leq (1-\gamma/\theta) \quad \mbox{ and } \quad (1+\gamma/\theta) \leq \Corona_2. 
	\end{align}
	Let $\delta' = C\delta$ (with $C$ as in \eqref{e:C'}) and choose $\delta'$ (hence $\delta$) small enough, depending on $\ve$ and $\gamma$, such that the conclusion of Theorem \ref{t:Reif} holds with constant $\min\{\ve,\gamma\}/10.$ Let $\rho_S = \rho$ and $\vp_S = \bar{g}$ be the metric and map we obtain by applying Theorem \ref{t:Reif} with this choice of parameter. 
	
	Since $X$ is Ahlfors $n$-regular, $(\R^n,\rho_S)$ is Ahlfors $n$-regular with regularity constant $C_*$ depending only on the regularity constant of $X$. Since $\ve < \min\{\ve,\gamma\}/10,$ it is also $(\ve,n)$-Reifenberg flat. It remains to check \eqref{e:final-eq}. Suppose $Q \in S$ and $x,y \in 3B_Q$ are such that $d_X(x,y) > \theta \ell(Q).$ Let $\ell \geq 0$ be the largest integer such that $\ell(Q) \leq r_\ell.$ In this way, if $k \geq 0$ is such that $Q \in \calD_k$ then $k(\ell) = k.$ Note that by maximality of $\ell$ we also have $r_\ell \leq 10\ell(Q).$ By the maximality of $\{x_{j,\ell}\}_{j \in J_\ell}$ there exists $j \in J_\ell$ such that $d(x_{j,\ell},x_Q) \leq r_\ell$ such that $3B_Q \subseteq 4B^{j,\ell}.$ In particular, by applying \eqref{e:Reif3.5}, we have 
	\begin{align}
		| \rho_S(\vp_S(x),\vp_S(y)) - d_X(x,y) | \leq \frac{\gamma r_\ell}{10} \leq \gamma \ell(Q) \leq \frac{\gamma d_X(x,y)}{\theta}.
	\end{align}
	The estimate \eqref{e:final-eq} follows from this and the estimates on $\gamma$ in \eqref{e:gamma}. 
\end{proof}

\newpage


\section{$\text{CD}^2$(normed spaces) implies CD(normed spaces)}\label{s:corona-corona}

Recall the notion of $\text{CD}(\mathscr{Y})$ and corona decomposition by normed spaces from Definition \ref{d:corona-Y} and Definition \ref{d:CDV}, respectively. The main goal of this section is to prove the following, which states that a corona decomposition by metric spaces with corona decompositions implies a corona decomposition. 

\begin{prop}\label{prop:corona-of-corona-1}
	Let $\Corona_1, \Corona_2,\Corona_1',\Corona_2' \geq 1$ be given. Suppose $\mathscr{Y}$ is a family of Ahlfors $n$-regular sets with uniform regularity constant each of which admits a corona decomposition by normed spaces with constants $\Corona_1,\Corona_2.$ If $X$ admits a $\emph{CD}(\mathscr{Y})$ with constants $\Corona_1',\Corona_2'$, then $X$ admits a corona decomposition by normed spaces with constant $\Corona_1'',\Corona_2''$, where $\Corona_1'' = \Corona_1\Corona_1'$ and $\Corona_2'' = \Corona_2\Corona_2'.$ 
\end{prop}

Proposition \ref{prop:corona-of-corona-1} is a variation of \cite[Lemma 3.41]{david1993analysis}. The proof there is said to be a variation their proof of \cite[Lemma 3.21]{david1993analysis}. Let us remark that what is needed is in fact a variation of \cite[Proposition 3.32]{david1993analysis}. Our proof below is almost identical to that of \cite[Proposition 3.32]{david1993analysis}, where we only deviate towards the end. At various points we will make use of \cite[Lemmas 3.22-3.29]{david1993analysis} which are about \textit{coronization} (and not corona decomposition) and hold in the metric setting.

\begin{proof}[Proof of Proposition \ref{prop:corona-of-corona-1}]
	Let $\COARSE_1,\COARSE_2 > 0.$ Let $(\cB, \cG, \calF)$ be a coronization of $X$ satisfying the assumptions in Definition \ref{d:CDV} with constant $\COARSE_1$. In particular, for each $S\in \calF$ there is a $Y=Y_S \in \mathscr{Y}$ and 
	$\phi_{X,S}:3B_{Q(S)}\to W$ such that if $Q\in S$ and $x,y\in 3B_Q$ with 
	$\dist(x,y)>\COARSE_1\ell(Q)$ then 
	\begin{align}\label{e:corona-map}
		\frac1{\Corona_1'}d_X(x,y)\leq d_Y(\phi_{X,S}(x),\phi_{X, S}(y))\leq \Corona_2' d_X(x,y).
	\end{align}
	We also have a decomposition of $Y_S$ into Christ-David cubes $\calD(Y_S)$ and a coronization $(\cB(Y_S), \cG(Y_S), \calF(Y_S))$ of $Y_S$
	such that for every $T\in \calF(Y_S)$ there is $$\phi_{Y_S, T}:3B_{Q(T)}\to \bR^n_{\|\cdot\|_{Q(T)}}$$
	such that if $R\in T$ and $x,y\in 3B_{R}$ with $\dist(x,y)>\COARSE_2 \ell(R)$ 
	then
	$$\frac1{\Corona_1}d_{Y_S}(x,y)\leq \|\phi_{Y_S, T}(x)-\phi_{Y_S,T}(y)\|_{Q(T)}\leq \Corona_2 d_{Y_S}(x,y).$$
	Note that $R$ is a cube in the dyadic structure of $Y_S$ and not in the dyadic structure of $X$.

	Let $A$ be a large constant (to be chosen later). We say two set $Q,R$ in some common space are \textit{$A$-close} if 
	\begin{align}
		A^{-1} \diam(Q) \leq \diam(R) \leq A \diam(Q)
	\end{align}
	and 
	\begin{align}
		\dist(Q,R) \leq A(\diam(Q) + \diam(R)). 
	\end{align}
	For each $S\in \calF$ we introduce a new bad set $\cB_1(Y_S)$ composed of all cubes in $\calD(Y_S)$ which are $A$-close to some cube in $\cB(Y_S)$ \textit{or} some top cube of a stopping time region in $\cT(Y_S)$ \textit{or} are in $\cG(Y_S)$ and are $A$-close to some other cube of $\cG(Y_S)$ that belongs to a different stopping time region of $\cT(Y_S)$. It follows from the proof of \cite[Lemma 3.26]{david1993analysis} that $\cB_1(Y_S)$ satisfies a Carleson packing condition with controlled constants.
	
	Let $\cB_2(Y_S)$ be the subset of $\cB_1(Y_S)$ which are $A$-close to elements of 
	$\phi_{X,S}(S)$. 
	Note that $Q\in S$ implies that $\phi_{X,S}(Q)$ is of comparable diameter and we may talk about cubes in $\calD(Y_S)$ that have comparable diameter to that and are not too far (up to $A$ times the diameter).

	Let $\cB_1(X)$ be cubes in $\calD(X)$ that satisfy at least one of:
		 \begin{enumerate}
		 \item $A$-close to an element of $\cB$
		 \item $A$-close to an element of $Q(S)$ for $S\in \calF$;
			 \item have a $\phi_{X,S}$ image that is $A$-close to a cube of
			$\cB_2(Y_S)$ for $S\in \calF$.
			\end{enumerate} We claim that $\calB_1(X)$ satisfies a Carleson packing condition. Indeed, one can prove a packing condition of the cubes in $\calB_1(X)$ satisfying (1) and (2) by  \cite[Lemma 3.27]{david1993analysis} which says that $A$-closeness does not ruin a packing condition. Thus, it suffices to prove a packing condition for those cubes satisfying (3). Fix $Q_0 \in \calD(X)$ and for $S \in \calF$ let $\calB(S)$ be those cubes in $\calB_1(X)$ satisfying (3).  We want to show 
		\begin{align}\label{e:packing'}
			\sum_{S \in \calF} \sum_{\substack{Q \in \calB(S) \\ Q \subseteq Q_0}} \ell(Q)^n \lesssim \ell(Q_0)^n. 
		\end{align}
		Suppose that $S \in \calF$ contributes to the left-hand side of \eqref{e:packing'} and that $Q(S) \subseteq Q_0.$ By definition of $A$-closeness, if $Q \in \calB(S)$ then there exists $R_Q \in \calB_2(Y_S)$ with $\ell(R_Q) \sim_A \ell(Q)$ such that $\dist(R_Q,\phi_{X,S}(Q)) \lesssim_A \ell(Q).$ Using this and \eqref{e:corona-map}, there exists $z \in Y_S$ and a constant $C$ depending $A$ and $\Corona_2$ such that $R_Q \subseteq B(z,C\ell(Q(S))).$	Furthermore, the correspondence $Q \mapsto R_Q$ is at most $C(A)$ to 1. It now follows from the packing conditions on $\calB_2(Y_S)$ and $\{Q(S)\}_{S \in \calF}$ that 
		\begin{align}\label{e:packing''}
			\sum_{\substack{S \in \calF \\ Q(S) \subseteq Q_0}}  \sum_{\substack{Q \in \calB(S) \\ Q \subseteq Q_0}} \ell(Q)^n \lesssim \sum_{\substack{S \in \calF \\ Q(S) \subseteq Q_0}} \ell(Q(S))^n \lesssim \ell(Q_0)^n. 
	\end{align}
	If $S \in \calF$ contributes to the left-hand side of \eqref{e:packing'} but $Q_0 \subseteq Q(S)$ then we proceed as follows. Since there exists $Q \in \calB(S) \subseteq S$ such that $Q \subseteq Q_0 \subseteq Q(S)$ we have $Q_0 \in S$ by Definition \ref{StoppingTime} (2). In particular, there is at most one such $S$. As below \eqref{e:packing'}, using the definition of $A$-closeness and \eqref{e:corona-map}, for each $Q \in \calB(S)$ there exists $R_Q \in \calB_2(Y_S)$ satisfying the same properties except $R_Q \subseteq B(z,C\ell(Q_0)).$ The packing condition on $\calB_2(Y_S)$ then gives 
	\begin{align}\label{e:packing'''}
		\sum_{\substack{Q \in S \\ Q_0 \subseteq Q(S)}}  \sum_{\substack{Q \in \calB(S) \\ Q \subseteq Q_0}} \ell(Q)^n \lesssim \ell(Q_0)^n. 
\end{align} 
	Combining \eqref{e:packing''} and \eqref{e:packing'''} finishes the proof of the claim.
	
	By \cite[Lemma 3.22]{david1993analysis} there is a coronization $(\cB'(X), \cG',\calF')$ of $X$ such that $\cB_1(X)\subset \cB'(X)$. Let $S'\in \calF'$ be given. We want to find a norm $\|\cdot\|_{Q(S')}$ and a map $\phi_{S'}: 3B_{Q(S')} \to \bR^n_{\|\cdot\|_{Q(S')}}$ such
	that if $Q\in S'$ and $x,y\in 3B_{Q}$ with $d_X(x,y)>\COARSE' \ell(Q)$ 
	then
	$$\frac1{\Corona_1''}d_X(x,y)\leq \|\phi_{S'}(x)-\phi_{S'}(y)\|\leq \Corona_2'' d_X(x,y).$$

	By the construction of $\cB_1(X)$ and the procedure in \cite[Lemma 3.22]{david1993analysis},  we have that there is an $S\in \calF$ such that $S'\subset S$. There is then a $Y=Y_S \in \mathscr{Y}$ and a $\phi_{X,S}:3B_{Q(S)}\to Y$ such that if $Q\in S$ and $x,y\in 3B_Q$ satisfy
	$\dist(x,y)>\COARSE_1\ell(Q)$ then 
	$$\frac1{\Corona_1}d_X(x,y)\leq d_Y(\phi_{X,S}(x), \phi_{X, S}(y) )\leq \Corona_2 d_X(x,y).$$
	Thus, for any $R\in S'$ there is $Q_R\in \calD(Y_S)$ such that $\dist(\phi_{X,S}(R),Q_R)\lesssim \ell(R)\sim \ell(Q_R)$. If $A$ is large enough, then $Q_R\notin \cB_1(Y_S)$ whenever $R\in S'$.  
	Further, all the neighbours (in the stopping-time region $S'$) are associated (by the map $R\mapsto Q_R$) to the same $T\in \calF(Y_S)$, which we now denote by $T(S')$.
	
	Let $\|\cdot\|_{Q(S')} = \|\cdot\|_{Q(T(S'))}$ and let $\phi_{S'}$ be defined for $x\in 3B_{Q(S')}$ by
	$$\phi_{S'}(x) \coloneqq \phi_{Y_{S}, T(S')}\circ \phi_{X,S}(x).$$
	Now, if $R\in S'$ and  $x,y\in 3B_R$ satisfy
	$$d_X(x,y)>\COARSE' \ell(R)$$
	then,
	provided 
	$\COARSE'>\COARSE_1$,
	we have
	$$d_{Y_S}(\phi_{X,S}(x),\phi_{X,S}(y))>\frac1{\Corona_1}d_X(x,y).$$
	Since $\ell(Q_R)\sim \ell(R)$,
	if we also have
	$\COARSE'/\Corona_1 \gtrsim \COARSE_2$ then this gives \[d_{Y_S}(\phi_{X,S}(x),\phi_{X,S}(y)) > \theta_2 \ell(Q_R).\]
	Thus, we have
	$$\|\phi_{S'}(x)-\phi_{S'}(y)\|\geq 
	\frac1{\Corona_2} d_{Y_S}(\phi_{X,S}(x),\phi_{X,S}(y))>\frac1{\Corona_1\Corona_2}d_X(x,y).$$  
	The upper bound on $\|\phi_{S'}(x)-\phi_{S'}(y)\|$ follows similarly.
	
	
	
\end{proof}

	\newpage

\section{CD(normed spaces) implies UR}\label{s:corona-UR}

In this section we prove the following. 

\begin{thm}\label{t:corona-UR}
	Let $(X,d)$ be an Ahlfors $n$-regular metric space which admits a corona decomposition by normed spaces with constants $\Corona_1,\Corona_2$. Then $X$ is UR.
\end{thm}

This theorem is an immediate consequence of the following proposition. 

\begin{prop}\label{p:Cor-UR}
	Let $\ve > 0$ and $\Corona \geq 1.$ Suppose $(X,d)$ is Ahlfors $n$-regular and admits a corona decomposition with constants $\Corona_1,\Corona_2$. Then, there is $L \geq 1,$ depending only on $\ve,\Corona_1$ and $\Corona_2$, such that for each $x \in X$ and $r > 0$ there exists $F \subseteq B(x,r)$ satisfying $\calH^n(B(x,r) \setminus F) \leq \ve r^n$ and an $L$-bi-Lipschitz map $f : F \to \R^n$.
\end{prop}

\begin{rem}\label{r:similar-DS}
	This proposition is analogous to \cite[Proposition 16.1]{david1991singular}. Although our definition of corona decomposition differs from the one used in \cite{david1991singular}, the proof of Proposition \ref{p:Cor-UR} follows a very similar method and only differs in our use of Definition \ref{d:corona-Y} (3). We will give a brief outline of the proof and give the details of any significant changes. 
\end{rem}

Fix constants $\ve,\Corona_1,\Corona_2$ and fix a point $x \in X$ and radius $r > 0$ for the remainder of the section. Our first step is to define $F$. Roughly speaking, $F$ will be the parts of $B(x,r)$ which are not contained in too many stopping-time regions from corona decomposition and which do not lie too near to the boundaries of certain cubes. 

Let's be more precise. Let $\calD$ be a system of Christ-David cubes on $X$. Fix $k_0 \in \N$ such that $5\varrho^{k_0+1} < r \leq 5\varrho^{k_0}$ and let $\{Q_{0,i}\}_{i \in I}$ be the cubes in $\mathcal{D}_{k_0}$ which intersect $B(x,r).$ Since $X$ is Ahlfors $n$-regular, we know that 
\begin{align}\label{e:I-bound}
	|I| \lesssim_{C_0,n} 1. 
\end{align}
Let 
\begin{align}\label{e:theta}
	\theta = 2\varrho \leq 2c_0,
\end{align}
where $\varrho$ and $c_0$ are the constants appearing in Definition \ref{cubes}. Since $X$ admits a corona decomposition by normed spaces with constants $\Corona_1,\Corona_2$ (recalling as well Remark \ref{r:bad}), for each $i \in I$ there exists a collection of stopping-time regions $\mathcal{F}_i$ such that 
\begin{align}\
	\{Q \in \calD : Q \subseteq Q_{0,i}\} &= \bigcup_{S \in \calF_i} S,  \label{e:S_i} \\
	\sum_{S \in \calF_i } \ell(Q(S))^d  &\lesssim_\theta \ell(Q_{0,i})^d, \label{e:control}
\end{align} 
and for each $S \in \calF_i$ which is not a singleton, there is a norm $\|\cdot\|_S$ and a map $\pi_S \colon 3B_{Q(S)} \to B_{\|\cdot\|_S}(0,3\ell(Q(S)))$ such that if $Q \in S$ and $x,y \in 3B_Q$ satisfy $d(x,y) \geq \theta \ell(Q),$ then 
\begin{align}\label{e:(4)}
	 \frac{1}{\Corona_1}d(x,y) \leq \| \pi_S(x) - \pi_S(y) \|_S \leq \Corona_2d(x,y).
\end{align}

Let
\[ R_0 = \bigcup_{i \in I} Q_{0,i} \quad \mbox{ and } \quad \calF = \bigcup_{i \in I} \{S : S \in \calF_i\}. \]
Using the same terminology as in \cite{david1991singular}, we call a cube $Q \subseteq R_0$ a \textit{transition cube} if it is the top cube, or a minimal cube of some $S \in \calF$. Let $T$ denote the set of transition cubes. Combining \eqref{e:I-bound} with \eqref{e:control}, we have 
\begin{align}\label{e:bound-T}
	\sum_{Q \in T} \ell(Q)^n \lesssim_\theta \sum_{i \in I} \ell(Q_{i,0})^n \lesssim_{C_0,n} r^n. 
\end{align}
For $Q \in \mathcal{D},$ let $N(Q)$ denote the number of transition cubes strictly containing $Q$ (this is the quantity $\ell(Q)$ in \cite{david1991singular}). Let $\eta > 0$ (which will be chosen small momentarily) and define the boundary of a cube $Q \in \calD_k$ as 
\[ \sigma(Q) =  \{ x \in Q : \dist(x,E \setminus Q) \leq \eta \varrho^k \}. \] 
Let $N \geq 1$ (to be chosen large) and define $F \subseteq B(x,r)$ by setting 
\[ F = B(x,r) \cap \left(\bigcup_{Q \in T} \sigma(Q) \right)^c \cap\left( \bigcup_{\substack{Q \in T\\ N(Q) \geq N}} Q \right)^c. \]

\begin{lem}
	For $N$ large enough and $\eta$ small enough, both depending on $C_0,\ve$ and $\theta$, we have $\calH^n(B(x,r) \setminus F) \leq \ve r^n.$
\end{lem}

\begin{proof}
	Ahlfors regularity, Lemma \ref{cubes} (4) and \eqref{e:bound-T} imply, after choosing $\eta$ small enough depending on $C_0,\ve$ and $\theta$, that
	\begin{align}
		\calH^n\left(\bigcup_{Q \in T} \sigma(Q)\right) \lesssim \sum_{Q \in T} \eta^\frac{1}{C} \ell(Q)^n \lesssim_{C_0,\theta} \eta^\frac{1}{C} r^n \leq \ve r^n/2. 
	\end{align}
	By taking $N$ large enough, depending on $C_0,\ve$ and $\theta,$ we also get 
	\begin{align}
		\calH^n\left(\bigcup_{\substack{Q \in T \\ N(Q) \geq N}} Q\right) \leq \frac{1}{N} \sum_{Q \in T} \calH^n(Q) \lesssim_{C_0,n,\theta} \frac{1}{N} r^n \leq \ve r^n/2. 
	\end{align}
	The result now follows.
\end{proof}

From now on, we fix $N$ large enough such that the above lemma holds. The next step in the proof of Proposition \ref{p:Cor-UR} is to construct the bi-Lipschitz map from $F$ to $\R^n$. This is done by stitching together the maps associated to each stopping-time region $S \in \calF$. The amount of stitching is controlled by the constant $N$. We begin by defining a preliminary map $g$ on the set of cubes $Q$ contained in $R_0$.

\begin{lem}
	There exists a map $g$ from the set of transition cubes $T$ such that if $Q \in T$ then $g(Q)$ is a cube in $\R^d$ with $\diam(g(Q)) = C_3^{-N(Q)} \ell(Q)$ for some large constant $C_3 > 1$. Furthermore, for any $Q,Q' \in T$ satisfying $Q \subseteq Q'$ we have $g(Q) \subseteq \tfrac{1}{2}g(Q').$
\end{lem}

\begin{proof}
	We define $g$ inductively, starting with the top cubes $Q_{0,i}.$ We set $\{g(Q_{0,i})\}_{i \in I}$ to be a collection of cubes in $\R^n$ each with diameter $5\varrho^{k_0}$ and with mutual distances between $5\varrho^{k_0}$ and $5C\varrho^{k_0}$ for some $C > 1$ (with ${k_0}$ as defined after the statement of Remark \ref{r:similar-DS}). Such a prescription is possible if $C$ is chosen large enough by \eqref{e:I-bound}. Since $N(Q_{i,0}) = 0$ for each $i \in I,$ the cubes $g(Q_{i,0})$ have the correct diameter.  
	
	We now define $g(Q)$ inductively for the remaining transition cubes. Suppose $g(Q)$ has been defined for some $Q \in T \cap \calD_k$ with $k \geq k_0.$ If $Q$ is a minimal cube of a stopping-time region, each of its children is in $T$. We want to $g$ for these cubes. By definition $N(R) = N(Q) + 1$ for each child $R$ of $Q$. Since $X$ Ahlfors $n$-regular, $Q$ has a bounded number of children (with constant depending only on $n$ and the regularity constant). Thus, for $C_3> 1$ large enough, we can choose $g(R)$ to be a collection of cubes with diameters $C_3^{-N(R)}\ell(R)$ contained in $\tfrac{1}{2}g(Q)$ such that 
	\begin{align}\label{e:g-estimate}
		\dist(g(R),g(R')) \geq C_3^{-N(R)} \ell(R).
	\end{align} 
	
	Assume now that $Q$ is the top cube of a stopping-time region $S \in \calF$, which is not a singleton. Each minimal cube of $S$ is in $T$ and we want to define $g$ for each one. At this point we deviate from what is written in \cite{david1991singular}, we make instead use of our definition of corona decomposition.

	By Remark \ref{r:BM}, we can find a linear isomorphism $T_S : (\R^n,\|\cdot\|_S) \to \R^n$ such that $\|T_S\|_{\op} \leq 1$ and $\|T^{-1}_S\|_{\op} \leq n$. Set $h_S = T_S \circ \pi_S.$ By \eqref{e:(4)} we have that $\pi_S(Q)$ is contained in a ball of radius $\Corona_2\ell(Q)$ in $(\R^n,\|\cdot\|_S).$ Since $\|T_S\|_\op \leq 1,$ the same is true of $h_S(Q)$ in $\R^n.$ By assumption, we have $\diam(g(Q)) = C_3^{-N(Q)}\ell(Q),$ so there exists an affine map $\phi_S \colon \R^n \to \R^n$ such that
	\begin{align}\label{e:in-g(Q)}
		\phi_S(h_S(Q)) \subseteq \frac{1}{3}g(Q)
	\end{align}
	and 
	\begin{align}\label{e:phi-est} C_4^{-1}C_3^{-N(Q)}|p-q| \leq |\phi_S(p) - \phi_S(q)| \leq C_3^{-N(Q)}|p-q| 
	\end{align}
	each $p,q \in \R^n$ and some $C_4 \geq 1$ chosen large enough depending on $\Corona_2$. Then, if $R \in \min(S)$, define $g(R)$ to be the cube centred at $\phi_S(h_S(x_R))$ of size $C_3^{-N(Q)-1} \ell(R) = C_3^{-N(R)}\ell(R).$ It is easy to see that $g(R) \subseteq \tfrac{1}{2}g(Q)$ by the triangle inequality. 
\end{proof}

The following lemma lets us estimate relative distance between cubes $g(R)$ and $g(R').$

\begin{lem}
	If $C_3$ is chosen large enough depending on $\Corona_1$ and $\Corona_2$ we get the following. Let $S \in \calF$ be a stopping-time that is not a singleton and suppose $R,R' \in \min(S).$ Then,  
	\begin{align}\label{e:g-lip} 
		C_3^{-N(Q) -1}    [\ell(R) + \ell(R') + \dist(R,R')] &\leq \dist(g(R),g(R')) \\
		&\hspace{2em}\leq   3C_3^{-N(Q)} [\ell(R) + \ell(R') + \dist(R,R')] .
	\end{align}	
\end{lem}

\begin{rem}
	The proof of the above lemma can be found on \cite[Page 104]{david1991singular} or in the proof of \cite[Lemma 7.6]{hyde2022d} with some more details. In our notation, the crucial estimates for these proofs are
	\begin{align}
		d(x_R , x_{R'}) \geq C^{-1} (\diam R + \diam R' + \dist(R,R'))
	\end{align}
	and 
	\begin{align}\label{e:lower-bound}
		\| \pi_S(x_R) - \pi_S(x_{R'}) \|_S \geq \frac{1}{2}d(x_R , x_{R'} ).
	\end{align}
	See above \cite[Equation 16.11]{david1991singular} for the corresponding estimates. The first estimate is immediate from the construction of the Christ-David cubes. For the second estimate, David and Semmes appeal to condition (3') in Remark \ref{r:usual-corona}. For us, we get the estimate with $\tfrac{1}{\Corona_1}d(x_R,x_{R'})$ on the right-hand side, by Definition \ref{d:corona-Y} (3). Indeed, let $k \geq k_0$ be the largest integer such that there exists $Q \in S \cap \calD_k$ with  $x_R,x_{R'} \in 3B_Q$. Suppose there exists $Q' \in S \cap \calD_{k+1}$ such that $x_R \in Q'$ or $x_{R'} \in Q$. By maximality of $k$ and the triangle inequality we have $\dist(x_R,x_{R'}) \geq 10\varrho^{k+1} = 2\varrho \ell(Q).$ The estimate above then follows from \eqref{e:(4)} and our choice of $\theta$ in \eqref{e:theta}. If there does not exists $Q' \in S \cap \calD_{k+1}$ satisfying the above it must be that $R,R' \in \calD_k.$ In this case we get $d(x_R,x_{R'}) \geq 2c_0\ell(Q)$ and the estimate again follows from \eqref{e:(4)} and our choice of $\theta$. The fact that we prove \eqref{e:lower-bound} with a worse constant does not cause any significant difficulties in the proof, one just needs to take the constant $C_3$ large enough depending on $\Corona_1$. 
\end{rem}

Now that we have defined $g,$ we can use it to define the mapping $f$. By definition, for each $x \in F$ there is a minimal transition cube, which we denote by $Q(x),$ such that $x \in Q(x).$ It must be that $Q(x)$ is the top cube of a stopping-time region $S(x) \in \calF$ (which is not a singleton), otherwise $x$ would be contained in a smaller transition cube. Thus, $x$ is contained in arbitrarily small cubes from $S(x)$. We define
\[ f(x) = \phi_{S(x)}(h_{S(x)}(x)). \] 
By \eqref{e:in-g(Q)}, $f(x) \in g(Q(x)).$ \\

The proof of Theorem \ref{t:corona-UR} is concluded with the following lemma. 

\begin{lem}
	The map $f: F \rightarrow \R^n$ is $L$-bi-Lipschitz, where $L$ depends on $C_0,\ve,\theta, \Corona_1$ and $\Corona_2$.  
\end{lem}

\begin{proof}
	Let $x,y \in F$ be distinct points. We split the estimates on $|f(x) - f(y)|$ into several cases. Let $Q_1$ be the maximal cube contained in $R_0$ such that $x \in Q_1$ but $y \not \in Q_1.$ \\
	
	\noindent\textbf{Case 1:} Suppose that $Q_1$ is contained in some stopping-time region $S \in \calF$, but that $Q_1 \neq Q(S)$. By maximality, any cube containing $Q_1$ contains both $x$ and $y$, in particular, this holds true for $Q(S).$ \\
	
	\noindent\textbf{Case 1(i):} Suppose further that there are cubes $R_x,R_y \in \min(S)$ such that $x \in R_x$ and $y \in R_y.$ This implies, first, that 
	\begin{align}\label{e:R_x-low}
		d(x,y) \leq \ell(R_x) + \ell(R_y) + \dist(R_x,R_y).
	\end{align}  
	Secondly, since $x \not \in \sigma(R_x)$ and $y \not\in \sigma(R_y)$ (by the definition of $F$), we also have 
	\begin{align}\label{e:R_x-up} 
		d(x,y) \gtrsim  \ell(R_x) + \ell(R_y) + \dist(R_x,R_y).
	\end{align}
	The estimates on $f$ then follow from the fact that $f(x) \in g(R_x),$ $f(y) \in g(R_y),$ $N(Q) \leq N$ and the estimates \eqref{e:g-lip}, \eqref{e:R_x-low}, and \eqref{e:R_x-up}, since
	\begin{align}
		\begin{split}\label{e:done}
		d(x,y) &\leq  \ell(R_x) + \ell(R_y) + \dist(R_x,R_y) \lesssim_{C_3,N} \dist(g(R_x),g(R_y) \\
		&\leq |f(x) - f(y)| \leq \dist(g(R_x),g(R_y)) + \diam(R_x) + \diam(R_y) \\
		&\lesssim_{C_3,N} \ell(R_x) + \ell(R_y) + \dist(R_x,R_y) \lesssim d(x,y).
	\end{split}
	\end{align} 
	Recall that $C_3$ depends only on $\Corona_1$ and $\Corona_2$ while $N$ depends only on $C_0,\ve$ and $\theta.$ \\
	
	\noindent\textbf{Case 1(ii):} Suppose that $x \in R_x \in \min(S)$ but $y$ is not contained in any cubes from $\text{min}(S)$ or that $y \in R_y \in \min(S)$ but $x$ is not contained in any cube from $\text{min}(S).$ By symmetry, we may suppose that $x \in R_x$ for some $R_x \in \text{min}(S).$ If follows that $f(x) = \phi_{S}(h_S(x))$ and by a similar argument to \eqref{e:g-lip}, it can be shown that
	\[ C_3^{-N(Q) -1}    [ \ell(R_y) + \dist(x,R_y)] \leq \dist(x,g(R_y)) \leq   3C_3^{-N(Q)} [\ell(R_y) + \dist(x,R_y)]. \]
	Using the same estimates as in \eqref{e:done} (after replacing $\dist(g(R_x),g(R_y))$ with $\dist(x,g(R_y))$) proves the bi-Lipschitz estimates for $f$ in this case. \\

	\noindent\textbf{Case 1(iii):} Suppose that neither $x$ nor $y$ are contained in a minimal cubes of $S$. In this case $f(x) = \phi_S(h_S(x))$ and $f(y) = \phi_S(h_S(y))$ and the bi-Lipschitz estimates for $f$ follow from the bi-Lipschitz estimate for $T_S$, \eqref{e:phi-est}, and \eqref{e:(4)}. \\
	
	\noindent\textbf{Case 2:} Suppose that $Q_1$ is the top of a stopping-time region. Let $Q_2$ be the parent of $Q_1.$ By construction, $Q_2$ is the minimal cube of some stopping-time region, hence, a transition cube. By maximality, it follows that $x,y \in Q_2.$ Since $x,y \in Q_2$, $x \in Q_1\setminus \sigma(Q_1)$ and $y \not\in Q_1$ we have $d(x,y) \sim \ell(Q_1).$
	To estimate $|f(x) - f(y)|$, first note that since $x,y \in Q_2,$ we have $f(x),f(y) \in g(Q_2)$ and so 
	\[|f(x) - f(y)| \leq \diam(g(Q_2)) \lesssim \ell(Q_2) \sim \ell(Q_1) \sim d(x,y) .\]
	Let $Q_1'$ be the sibling of $Q_1$ with $y \in Q_1'.$ By \eqref{e:g-lip} we have
	\[ |f(x) - f(y)| \geq \dist(g(Q_1),g(Q_1')) \gtrsim \ell(Q_1) \sim d(x,y) \]
	and this finishes the proof of the bi-Lipschitz estimates. 
	
	\end{proof}

\newpage

\section{An example and application}\label{s:example}

We conclude by giving an example of a discrete use of Theorem \ref{t:UR-BWGL} to construct a UR metric space
from the BWGL. Conversely this gives a condition for testing whether a discrete set has bi-Lipschitz coordinate charts.

Let $(Y,d_Y)$ be a discrete metric space and, for each $y\in Y$,
suppose that $r_y\in(0,\diam(Y)]$ is such that $B(y,r_y)=\{y\}$.
Let $\nu$ be a measure on $Y$ and suppose that for all $y\in Y$ and
$r\in (r_y,\diam(Y)]$ we have $C^{-1} r^n \leq \nu(B(y,r))\leq C r^n$ for some fixed $C \geq 1.$ 

For each $y\in Y,$ let $D_{y,r_y}$ be the ball of radius $r_y$ centred at the origin in a
normed space $(\R^n, \|\cdot\|_y)$ and consider the metric measure space
given by taking a disjoint union of the $D_{y,r_y}$.
That is, let $X=\bigsqcup_{y\in Y} D_{y,r_y}$ and for $x\in D_{y,r_y}$ and $w\in
D_{z,r_z}$ let 
\[d_X(x,w):=\|x\|_y + d_{Y}(z,y) + \|w\|_z,\]
endowing $X$ with the $n$-dimensional Hausdorff measure.
Note that $X$ is Ahlfors $n$-regular (Definition \ref{d:ADR}) with constant comparable to $C$. 
For $x\in D_{y,r_y}$ we will write $y(x):=y$.
The claim is that, correctly interpreted, $Y$ is uniformly rectifiable if
and only if $X$ is.

One may say that $Y$ is UR if there are constants $\theta, L>0$ such that,
for every $y\in Y$ and $r\in (r_y,\diam(Y)]$ we have a set
$F_0\subset B(0,r)\subset \bR^n$ and an $L$-Lipschitz map
$f_0:F_0\to Y$ such that
$\nu(B(y,r)\cap f_0(F_0)) \geq \theta r^n$.
It is easy to see that if
$X$ is UR then   $Y$ is UR.
Conversely,
if $Y$ is UR then  $X$ is UR.
Indeed,  the main issue is (with minor abuse of notation) with  $y$
being the center of $D_{y,r_y}$ and $r\in (r_y,\diam(Y)]$.
In this case consider the map $f_0\colon F_0\to Y$ given by the UR of $Y$
associated to $B(y,r)\subset Y$.
For each $z$ in the image of $f_0$,  every element of $f_0^{-1}(z)$ has a ball around it 
of radius of size at least $r_z/L$ which is disjoint from $F_0\setminus
f_0^{-1}(z)$.
Thus, by increasing the Lipschitz constant of $f$ by a factor of $2L$, we
obtain a map $f\colon F\to X$ which
covers all balls $\{D_{z,r_z} \colon z\in B(y,r)\}$.
Here, $F\supset F_0$ may be taken to be
\[\bigcup_{a\in F_0} B(a,\frac1{2L} r_{f_0(a)})\subset \bR^n.\]

For $\ve>0$, we may say that $Y$ satisfies  BWGL$(\ve)$
if the set 
$$\mathcal Y:=\{(y,r) \in Y \times (0,\diam(Y)) : r>r_y/\ve\textrm{ and }
\bilat_{Y}(y,r) > \ve\}$$
is a Carleson set (see Definition \ref{d:Carleson-set}).
It is clear that if  $X$ satisfies BWGL$(\ve)$ then $Y$ satisfies BWGL$(2\ve)$. 
Conversely, suppose $Y$ satisfies BWGL$(\ve).$ We verify that $X$ satisfies BWGL$(2\ve)$. 
Let $x\in X$ and $R\in [0,\diam(X)]$.
Consider
$$\mathscr{A}:=\{(w,r) \in B_{X}(x,R) \times (0,R) : \bilat_{X}(w,r) > 2\ve\}.$$
Using  the geometry of the ball $D_{y(w),r_{y(w)}}$ and  the
Ahlfors-regularity of $Y$ and $X$, it is easy to check that
\begin{equation}\label{e:example-1}
	\mathscr{A}_1:=\{(w,r) \in \mathscr{A} : r\leq  r_{y(w)} \} 
\end{equation}
is a Carleson set. 
Note that the condition $\bilat_{X}(w,r)>2\ve$ prevents balls $B(w,r)$ which are disjoint from the  the boundary of $D_{y(w),r_{y(w)}}$ from having $(w,r)\in \mathscr{A}_1$.

Consider now
\begin{equation}\label{e:example-2}
	\mathscr{A}_2:=\{(w,r) \in \mathscr{A} : r_{y(w)}<r<\frac1{\ve}r_{y(w)} \} .
\end{equation}
Then the Carleson measure of $\mathscr{A}_2$ is controlled by a constant multiple of  $\log(\ve) R^n$.  
This only uses the controlled range of $r$ and the Ahlfors-regularity of $X$, and does not make use of the condition $\bilat>2\ve$.

Finally, consider 
\begin{equation}\label{e:example-3}
	\mathscr{A}_3:=\{(w,r) \in \mathscr{A} :\frac1{\ve}r_{y(w)}<r \}. 
\end{equation}
We consider three subsets of $\mathscr{A}_3$ as follows.
First let
\begin{equation}\label{e:example-3.1}
	\mathscr{A}_{3.1}\coloneqq\{(w,r) \in \mathscr{A}_3: \mbox{ there is no }y\in Y\mbox{ such that } 0<d_Y(y,y(w))\leq r\}
\end{equation}
and note that if $(w,r)\in \mathscr{A}_{3.1}$ then $\xi_Y(w,r)>\ve$ as $B(y(w),r)$ is a singleton.
Thus, the Carleson measure of $\mathscr{A}_{3.1}$ is bounded above by the Carleson measure of $\mathcal Y$
Second, consider
\begin{equation}\label{e:example-3.2}
	\mathscr{A}_{3.2}\coloneqq\{(w,r) \in \mathscr{A}_3: \mbox{ there is  }y\in Y\mbox{ such that } d_Y(y,y(w))\leq r, \mbox{ and } r<r_y/\ve\}.
\end{equation}
For $y$ as in the RHS of  \eqref{e:example-3.2} we have that 
$r\in [r_y,r_y/\ve]$
and we may proceed as in $\mathscr{A}_2$ 
to control the Carleson measure of $\mathscr{A}_{3.2}$.
Finally let
\begin{equation}\label{e:example-3.3}
	\mathscr{A}_{3.3}\coloneqq\{(w,r) \in \mathscr{A}_3: \mbox{ if   }y\in Y\mbox{ such that } d_Y(y,y(w))\leq r, \mbox{ then } r\geq r_y/\ve\}.
\end{equation}
For $(w,r) \in \mathscr{A}_{3.3}$
it is easy to check, since $\xi_X(w,r) > 2\ve$, that $\xi_Y(y(w),r) > \ve.$
As for $\mathscr{A}_{3.1}$, the Carleson measure of $\mathscr{A}_{3.3}$ is bounded above by the Carleson measure of $\mathcal Y$.
Therefore, we see that $\mathcal Y$ is a Carleson set.

Combining these definitions, we see that Theorem \ref{t:UR-BWGL} states that $Y$ is UR if and
only if it satisfies BWGL.

\newpage

\appendix

	\section{Table of selected notation}

\begin{itemize}
	\setlength\itemsep{1em}
	\item $X$ -- A fixed metric space from statement of main theorems.
	\item $C_0$ -- The regularity constant of $X$. 
	\item $c_0$ -- The constant appearing in Definition \ref{cubes}.
	\item GHA, $\delta$-Gromov-Hausdorff Approximation -- Definition \ref{d:GHA}.
	\item RF, 
	$(\delta,n)$-Reifenberg flat -- Definition \ref{d:RF}. 
	\item UR, Uniformly Rectifiable -- Definition \ref{d:UR-intro-metric}.
	\item BWGL, Bi-lateral Weak Geometric Lemma -- Definition \ref{d:BWGL}.
	\item CD($\mathscr{Y}$), Corona Decomposition by $\mathscr{Y}$ for $\mathscr{Y}$ a family of metric spaces -- Definition \ref{d:corona-Y}.
	\item $\Phi({\rm point,radius, norm})$ -- Definition \ref{d:bilat}.
	\item $\unilat_X({\rm point, radius, norm,map})$ -- Definition  \ref{d:bilat}. See also Definition \ref{d:notation-cubes}.
	\item $\eta_X({\rm point, radius, norm,map})$ -- Definition  \ref{d:bilat}. See also Definition \ref{d:notation-cubes}.
	\item $\bilat_X({\rm point, radius, norm,map}$) -- Definition  \ref{d:bilat}. See also Definition \ref{d:notation-cubes}.
	\item $\bilat_X({\rm point, radius})$ -- Definition  \ref{d:bilat}. See also Definition \ref{d:notation-cubes}.
	\item $\Omega_{X,{\rm map, norm}}^{{\rm constant}}({\rm point, radius, another \  norm, another \  map})$ -- Definition \ref{d:gamma}. See also Definition \ref{d:notation-cubes}.
	\item $\gamma_{X,{\rm map, norm}}^{{\rm constant}}({\rm point, radius})$ -- Definition \ref{d:gamma}. See also Definition \ref{d:notation-cubes}.
	
	\item $B^{j}$ -- Balls in the metric space $X$.
	\item $B_{j} $ -- Balls in $\R^n$ endowed with the norm  $\|\cdot\|_{j}$.
	\item $W_\ell$ -- A (possibly non-connected) manifold constructed in Section \ref{s:disconnected}.
	\item $M_\ell$ -- A connected manifold constructed from Section \ref{s:M} onwards.
	\item  $\text{md}_f(I)$, $\text{md}_f^L(I)$ -- See before the statement of Theorem \ref{t:metric-diff}.
\end{itemize}

\newpage

\bibliographystyle{alpha}
\bibliography{/Users/matthewhyde/Documents/BibTex/Ref.bib}

\end{document}